\documentclass[11pt]{article}

\usepackage{arxiv}

\usepackage{amsthm,amsmath,amsfonts,amssymb}
\usepackage{amsfonts}
\usepackage{graphicx}
\usepackage{enumerate}
\usepackage{natbib}
\usepackage{booktabs}
\usepackage{url}  

\usepackage[colorlinks=true,citecolor=blue]{hyperref}
 
\usepackage{cleveref}
\Crefname{appendix}{Appendix}{Appendices}
\crefname{appendix}{Appendix}{Appendices}

\usepackage[normalem]{ulem}
\usepackage[dvipsnames]{xcolor}
\usepackage{enumitem}  
\usepackage{algorithm}

\usepackage{algorithmicx}
\usepackage{algpseudocode}

\theoremstyle{plain}

\newtheorem{theorem}{Theorem}
\newtheorem{lemma}{Lemma}

\newtheorem{corollary}{Corollary}
\newtheorem{proposition}{Proposition}

\newtheorem{remark}{Remark}

\newtheorem{condition}{Condition}
\def\ba{{\boldsymbol a}}

\def\be{{\boldsymbol e}}

\def\bg{{\boldsymbol g}}
\def\bh{{\boldsymbol h}}

\def\bq{{\boldsymbol q}}

\def\bu{{\boldsymbol u}}
\def\bv{{\boldsymbol v}}

\def\bx{{\boldsymbol x}}
\def\by{{\boldsymbol y}}

\def\bT{{\boldsymbol T}}

\def\bX{{\boldsymbol X}}

\def\bZ{{\boldsymbol Z}}

\def\beps{{\boldsymbol \eps}}

\def\bxi{{\boldsymbol \xi}}
\def\btheta{{\boldsymbol \theta}}
\def\bbeta{{\boldsymbol \beta}}

\def\bSigma{{\boldsymbol \Sigma}}

\newcommand{\eps}{\varepsilon}

\def\noverp{\delta}
\def\noverpmass{\chi}

\crefname{theorem}{theorem}{theorems}
\Crefname{theorem}{Theorem}{Theorems}
\crefname{lemma}{lemma}{lemmas}
\Crefname{lemma}{Lemma}{Lemmas}
\crefname{proposition}{proposition}{propositions}
\Crefname{proposition}{Proposition}{Propositions}
\crefname{corollary}{corollary}{corollaries}
\Crefname{corollary}{Corollary}{Corollaries}

\crefname{definition}{definition}{definitions}
\Crefname{definition}{Definition}{Definitions}
\crefname{assumption}{assumption}{assumptions}
\Crefname{assumption}{Assumption}{Assumptions}
\crefname{condition}{condition}{conditions}
\Crefname{condition}{Condition}{Conditions}

\title{Regularization Using Synthetic Data for High-Dimensional Inference}

\author{
 Weihao Li \\
  National University of Singapore\\
   \texttt{weihao.li@u.nus.edu}
   \And
   Dongming Huang \\
   National University of Singapore\\
   \texttt{stahd@nus.edu.sg}
}

\renewcommand{\headeright}{}
\renewcommand{\undertitle}{}
\renewcommand{\shorttitle}{Regularization Using Synthetic Data}

\begin{document}
\maketitle

\begin{abstract}
To address the challenges of obtaining reliable inference in high-dimensional models, we introduce the Synthetic-data Regularized Estimator (SRE). Unlike traditional regularization methods, the SRE regularizes the complex target model via a weighted likelihood based on synthetic data generated from a simpler, more stable model. This method provides a theoretically sound and practically effective alternative to parameter penalization. We establish key theoretical properties of the SRE in generalized linear models, including existence, stability, consistency, and minimax rate optimality.
 We leverage the Convex Gaussian Min-max Theorem to derive precise asymptotic characterizations in high-dimensional linear regimes where $n/p \to \delta > 0$, both for noninformative synthetic data and for informative auxiliary data in a transfer learning setting. Our asymptotic results characterize how performance depends on the signal strength and the similarity between target and auxiliary data sources.
Building upon these results, we develop practical methodologies for high-dimensional inference, including tuning parameter selection, confidence interval construction, and calibrated variable selection. The effectiveness of the SRE is demonstrated through simulation studies and real-data applications.
\end{abstract}

\keywords{ synthetic data   \and regularization \and high-dimensional inference \and exact asymptotics \and generalized linear models}

\section{Introduction}

\label{sec:introduction}

A pervasive challenge in modern data analysis is making reliable statistical inferences from high-dimensional datasets where the number of variables ($p$) is comparable to or larger than the number of observations ($n$).
 In such cases, standard methods like maximum likelihood estimation (MLE)  can become unstable or biased, leading to unreliable inferences.
 For example, in logistic regression, the MLE may not exist or may be biased with high variability when the dimension is comparable to the sample size. Under an asymptotic setting where the ratio $n/p$ approaches a constant, the behavior of the MLE for logistic regression is investigated in \cite{sur2019modern,candes2020phase}. The finite-sample existence of the MLE for logistic regression and its finite-sample properties are investigated in \cite{albert1984existence,firth1993bias,heinze2002solution}.

Penalty-based regularization methods are widely used to control model complexity by directly imposing a penalty function on the parameter vector \citep{wainwright2014structured}. Classical examples include ridge \citep{hoerl1970ridge}, LASSO \citep{tibshirani1996regression}, SCAD \citep{fan2001variable}, group LASSO \citep{yuan2006model}, MCP \citep{zhang2010nearly}, and others.
To address the bias introduced by penalization, a line of work on debiased estimators and post-selection inference has been developed \citep{zhang2014confidence,van2014asymptotically,javanmard2014confidence,lee2016exact}.
Although these methods work well when their underlying assumptions (such as sparsity) hold, they face several challenges in practice. For instance, penalty-based methods often require specialized optimization algorithms, can be highly sensitive to the scaling of the parameters, and become unreliable when the underlying assumptions fail.

\subsection{Synthetic-data regularization and catalytic priors}\label{sec:intro_catalytic}

In response to the limitations of these existing methods, we introduce the Synthetic-data Regularized Estimator (SRE),
a novel frequentist regularization technique that rethinks the regularization mechanism.
Instead of penalizing parameters directly, the SRE regularizes a complex target model by supplementing the observed data with weighted synthetic data generated from a fitted simpler model.

Let $\mathcal{D}$ denote the observed dataset and let $L(\btheta;\mathcal{D})$ be the likelihood function of the target model with the parameter $\btheta\in \mathbf{\Theta}$.
Suppose that we have already generated a synthetic dataset $\mathcal{D}^{*}$ of size $M$ and the likelihood function of the target model based on $\mathcal{D}^{*}$ is denoted by $L(\btheta;\mathcal{D}^*)$.
We define the SRE as the maximizer of the weighted likelihood
\begin{equation}\label{eq:general_SRE}
\widehat{\btheta}=\underset{\btheta\in \mathbf{\Theta}} {\arg\max} \left[ L(\btheta;\mathcal{D})L(\btheta;\mathcal{D}^*)^{\frac{\tau}{M}} \right],
\end{equation}
where $\tau$ is a positive tuning parameter.
The downweighted likelihood based on synthetic data $L(\btheta;\mathcal{D}^*)^{\frac{\tau}{M}}$ can be viewed as a data-centric regularizer.

The SRE is closely related to the catalytic prior method for Bayesian prior specification \citep{huang_catalytic_2020}.
Given a synthetic dataset $D^*$, the catalytic prior takes the form $\pi(\btheta) \propto L(\btheta; D^*)^{\tau / M}$.
Under this prior, the posterior mode coincides with the maximizer in \eqref{eq:general_SRE}.
Unlike Bayesian inference that relies on posterior distributions for estimation and uncertainty quantification, our focus is on developing frequentist estimation and inference methods for the SRE, including confidence interval construction, variable selection, and tuning parameter selection.

The SRE can be defined broadly for general likelihood-based models as in \eqref{eq:general_SRE}. For the sake of concreteness and tractability, in this paper we develop theory and methods for generalized linear models (GLMs).
We present the logistic case in the main paper as a canonical example and provide detailed extensions to other GLMs in \Cref{supp:extension_GLM_section}.

\subsection{Connections to existing approaches}\label{sec:connections-existing}

Penalty-based regularization methods leverage structural assumptions like sparsity or smoothness to enhance statistical and computational efficiency \citep{bickel2006regularization,wainwright2014structured}.
They are preferred when those structural assumptions hold.
Our synthetic-data regularization complements them because it does not rely on specific structural assumptions, which makes it beneficial when such assumptions are questionable.
Connections between catalytic priors and Ridge, LASSO, and elastic net are explored in \citet[Section 4]{huang2022catalytic}.

There are also several existing approaches to achieving regularization without explicit penalties. Below we briefly review their ideas and their differences from the SRE method.

\paragraph{Borrowing from related data}

When auxiliary data are available, a number of existing approaches borrow information from related datasets. In the Bayesian literature, such data are often called historical data and are incorporated through power priors \citep{chen2000power}. In transfer learning, they are often called source data and are used to improve performance on a target task \citep{torrey2010transfer}.

In Bayesian inference, the power prior \citep{chen2000power} incorporates historical data $\mathcal{D}_0$ through a prior proportional to $L(\btheta;\mathcal{D}_0)^{a_0}\pi_0(\btheta)$, where $a_0 \in [0,1]$ controls the degree of borrowing and $\pi_0(\cdot)$ is a baseline prior. When the baseline prior is flat, the posterior mode under this power prior is given by
\begin{equation}\label{eq:power-mode}
\widehat{\boldsymbol{\btheta}}_{\mathrm{pow}}
=
\arg\max_{\btheta\in \mathbf{\Theta}}
\left\{
\log L(\boldsymbol{\btheta};\mathcal{D})
+
a_0\log L(\boldsymbol{\btheta};\mathcal{D}_{0})
\right\}.
\end{equation}
The SRE coincides with this construction if the synthetic dataset $D^*$ is replaced by historical data $D_0$, and $\tau/M = a_0$.
Consequently, the theory developed for the SRE also applies to this posterior mode estimator. In \Cref{sec:exact_asymptocis_infor}, we study this estimator in the linear proportional regime and derive a precise asymptotic characterization, which is used to develop high-dimensional inference tools.
Power priors are Bayesian tools for incorporating genuine historical data, with emphasis on prior specification and posterior analysis. In contrast, the SRE uses synthetic data generated from a simpler model, and we study the resulting estimator from a frequentist perspective.

The weighted source-target objective in \Cref{eq:power-mode} is also closely related to weighted empirical risk minimization in transfer learning. For example, \cite{ben2010theory} study domain adaptation for binary classification and analyze estimators that minimize a convex combination of the empirical target and source losses, for which they derive VC-type upper bounds on the target prediction error. The formulation in \Cref{eq:power-mode} has the same weighted source-target structure, but it is likelihood-based and tailored to parametric models.
Furthermore, our focus is on regularized estimation and statistical inference for parametric models while weighted empirical risk minimization concerns generalization performance.

For GLMs, \cite{hector2024turning} propose an estimator similar to \Cref{eq:power-mode}, but they do not use the source responses directly. Instead, they first fit a GLM to the source data and then replace the source responses in the weighted likelihood term of \Cref{eq:power-mode} by the fitted means. Their estimator coincides with the SRE in the special case where the synthetic covariates are resampled from the empirical distribution of the source covariates, the synthetic responses are generated from the fitted source GLM, and $M\to\infty$.
Beyond this connection, \cite{hector2024turning} study fixed-dimensional inference, whereas we develop a broader framework with high-dimensional theory.

\paragraph{Data augmentation}

In the absence of suitable auxiliary data, data augmentation and feature noising can be interpreted as regularization \citep{matsuoka1992noise,bishop1995training,rifai2011adding};
for an overview of data augmentation methods in machine learning, see \citet{shorten2019survey}.
For GLMs, \cite{wager2013dropout} show that dropout and additive feature corruption schemes induce a label-free quadratic penalty that behaves like an $L_2$-regularizer scaled by the diagonal Fisher information, and
\cite{li2022adaptive} propose a framework that iteratively generates parameter-dependent noisy data so that the augmented loss approximates a pre-specified penalty such as lasso and SCAD.
However, both works rely on second-order Taylor expansions of the loss, which are exact only for linear regression models, and neither develops statistical inference in high-dimensional regimes where the number of parameters is comparable to or larger than the sample size.

More generally, the SRE method differs from data augmentation in how synthetic data are generated and how they are used.
Data augmentation aims to improve prediction accuracy by increasing the diversity of training samples; it typically generates synthetic data via geometric transformations, noise injection, interpolation, or generative models, and then treats them as additional training data. In contrast, the SRE uses synthetic data generated from simpler models to regularize complex models. It combines real and synthetic data through a weighted likelihood, where the synthetic-data weight plays the role of a regularization parameter, and it is designed to improve estimation and uncertainty quantification for model-based statistical inference.

\subsection{Review of synthetic data}
The term synthetic data is broad, and its meaning depends on the role played by the generated data.
Besides the uses in machine learning discussed in \Cref{sec:connections-existing}, prominent lines of work in statistics include both synthetic datasets released for external use and synthetic data constructed internally as part of a statistical procedure. These lines of work have different goals and use synthetic data in different ways. We briefly recall these two statistical perspectives in order to clarify the scope of the present paper.

A major use of synthetic data in statistics arises in disclosure control and public-use data release. \citet{rubin1993statistical} proposed releasing synthetic microdata in place of the original confidential records, and subsequent work developed inferential procedures for public-use synthetic datasets that account for the uncertainty introduced by synthesis \citep{Reiter2002SyntheticDataSets,RaghunathanReiterRubin2003Disclosure,reiter2005releasing}. Recent reviews place this literature in a broader statistical framework and emphasize synthetic data as a tool for widening access to sensitive data while supporting statistical analysis that accounts for synthesis uncertainty \citep{Raghunathan2021SyntheticData,DrechslerHaensch2024ThirtyYears}. This line of work studies synthetic data as externally released surrogates for confidential records, which is not the goal of the present paper.

Closer to the present paper is a line of work in which synthetic or imaginary data are introduced internally for prior construction or inferential stabilization.
In the field of prior specification for Bayesian inference, conditional means priors have been proposed to incorporate additional synthetic data derived from experts' knowledge \citep{bedrick1996new,bedrick1997bayesian}, and expected-posterior priors average posterior distributions over imaginary training samples drawn from a predictive distribution \citep{iwaki1997posterior,perez2002expected,neal2001transferring}.
More recently, catalytic priors generate synthetic observations from a fitted simpler model and incorporate them through a down-weighted likelihood \citep{huang_catalytic_2020}. The SRE adopts this internal-use perspective, but our focus is on regularized frequentist estimation and inference, rather than on public release of synthetic data or on Bayesian posterior inference.

\subsection{Contributions, organization, and notation}
Our work establishes the synthetic-data regularization as a theoretically sound and practically powerful frequentist method. Specifically, we achieve the following:

\begin{enumerate}
    \item We show that the SRE can be constructed to exist
    even when the MLE does not exist, and we demonstrate that the SRE is stable against the randomness in the synthetic data.

  \item We establish that, over the asymptotic regimes covered by our theory, the SRE achieves the estimation error rate
    $\min(p/n, 1)$, which is minimax optimal. This shows that incorporating
    synthetic data does not degrade performance. In particular, the SRE is
    consistent when $p/n \to 0$.
    
    \item
    We characterize the precise asymptotic behavior of the SRE when $n/p \to \delta>0$. Our analysis covers both noninformative synthetic data and informative auxiliary data (as formulated in \eqref{eq:power-mode}). The resulting formulas show how the limiting performance of the SRE depends on signal strength and regularization level and, in the informative case, on the degree of similarity between data sources.
    
     \item Building on the precise asymptotic theory, we develop practical methods for estimating the signal strength and the similarity between data sources. We leverage these asymptotic results to design SRE-based confidence intervals and variable selection strategies that remain effective even in scenarios where MLE fails to exist.
     
\end{enumerate}

The paper is structured as follows.
\Cref{sec:sec:SRE_GLM} introduces the construction of the SRE for GLMs.
\Cref{sec:properties} establishes theoretical properties of the SRE.
\Cref{sec:linear_asymptotic_regime} characterizes the asymptotic behavior of the SRE as $p$ and $n$ diverge linearly.
 We develop practical methodologies and conduct numerical studies in \Cref{sec:adjust_inference}.
Finally, \Cref{sec:discussion} concludes with a discussion of our findings and future directions.
For ease of exposition, we present the theory for logistic regression while deferring extensions to other GLMs, proofs, and additional numerical studies to Section~\ref{supp:extension_GLM_section}, \ref{supp:sec:proof}, and~\ref{supp:sec:numerical}, respectively.

Hereafter, we use $\|\boldsymbol{v}\|_q$ for the $\ell_q$ norm of a vector $\boldsymbol{v}$,
$\|B\|_{\text {op }}$ for the operator norm of a matrix $B$,
and $\lambda_{\min }(A)$ and $\lambda_{\max}(A)$ for the smallest and largest eigenvalues of a positive definite matrix $A$, respectively.
For a positive integer $n$, $[n]$ denotes the set $\{1, 2, \dots, n\}$.
We write the indicator of a statement $\mathcal{E}$ as $\mathbf{1}\{\mathcal{E}\}$, and the point mass at $x$ as $\chi_x$.
We write $\rightsquigarrow$ for weak convergence and $\xrightarrow{\mathbb{P}}$ for convergence in probability.  For any function $\rho:\mathbb{R}\mapsto\mathbb{R}$ and any positive scalar $\gamma$, we define the proximal operator as $\operatorname{Prox}_{\gamma \rho(\cdot)}(x):=\arg\min_{t\in \mathbb R} \left[\rho(t)+\frac{1}{2\gamma}(x-t)^2\right]$.
The sub-gaussian norm of a random variable $W$ is defined as $\|W\|_{\psi_2}=\sup _{t \geq 1} t^{-1 / 2}\left(\mathbb{E}|W|^t\right)^{1 / t}$ and the sub-gaussian norm of a $p$-dimensional random vector $\boldsymbol{W}$ is defined as $\|\boldsymbol{W}\|_{\psi_2}=\sup_{x \in S^{p-1}}\|\langle \boldsymbol{W}, x\rangle\|_{\psi_2}$.
For sequences $\{a_k\}$ and $\{b_k\}$, we write $a_k\lesssim b_k$ if there exists some constant $C$ such that $a_k\leq C b_k$.
We write $a_k\asymp b_k$ if $a_k\lesssim b_k$ and $b_k\lesssim a_k$.

\section{Synthetic-data regularization in GLMs}
\label{sec:sec:SRE_GLM}

In this section we specialize the general SRE construction from \Cref{sec:intro_catalytic} to GLMs and denote the regression parameter by $\bbeta$ instead of the general notation $\btheta$.
We first specify the GLM setup and define the SRE, then give an interpretation in terms of an expected KL penalty, and finally describe how synthetic data are generated.
Rigorous properties of $\widehat{\bbeta}_M$ are developed in \Cref{sec:properties,sec:linear_asymptotic_regime}.

Let $\mathcal{D}=\left\{\left(Y_i, \boldsymbol{X}_i\right)\right\}_{i=1}^n$ be $n$ independent pairs of observed data, where $Y_i$ is a response and $\boldsymbol{X}_i$ is a $p$-dimensional covariate vector.
The GLM assumes that the conditional density of $Y_i$ given $\bX_i$ w.r.t. some base measure is
$$
{f}\left(y \mid \boldsymbol{X}_i, \bbeta\right)=\exp\left[y\eta_i-\rho(\eta_i)\right], \quad \eta_i=\boldsymbol{X}_i^\top  \bbeta,
$$
where $\bbeta\in \mathbb{R}^p$ is the model parameter and $\rho(\cdot)$ is the log-partition function.
The likelihood derived from observed data is
\begin{equation*}
    L(\bbeta; \mathcal{D})=\prod_{i=1}^n f\left(Y_i \mid \boldsymbol{X}_i, \bbeta\right)=\exp \left(\sum_{i=1}^n\left[Y_i \boldsymbol{X}_i^{\top} \bbeta-\rho\left(\boldsymbol{X}_i^{\top} \bbeta\right)\right]\right).
\end{equation*}

When $p$ is comparable to or larger than $n$, the MLE may fail to exist or be unstable, so we seek a regularized estimator for $\bbeta$ based on synthetic data.

\subsection{SRE for GLM}

Let $\mathcal{D}^*=\left\{\left(Y_i^*, \boldsymbol{X}_i^*\right)\right\}_{i=1}^M$ be a synthetic dataset, generated from a simpler model that is stably fitted using the observed data; concrete generation schemes are given in \Cref{sec:synthetic_generation}.
Given $\mathcal{D}^*$, the likelihood function derived from the synthetic data is
\begin{equation}\label{eq: synthetic-likelihood}
  L(\bbeta; \mathcal{D}^*)=\prod_{i=1}^{M}f\left(Y^*_i \mid \boldsymbol{X}^*_i, \bbeta\right)
=\exp\left(\sum_{i=1}^M\left[Y^*_i{\boldsymbol{X}_i^*}^\top \bbeta-\rho\left({\boldsymbol{X}_i^*}^\top\bbeta\right)\right]\right).
\end{equation}
We combine the observed and synthetic data through the weighted likelihood $L(\bbeta; \mathcal{D}) L(\bbeta; \mathcal{D}^*)^{\tau/M}$
so the synthetic data contribute total weight $\tau$ regardless of $M$.
This leads to our construction of the Synthetic-data Regularized Estimator (SRE), which is the maximizer of the weighted sum of the observed and synthetic log-likelihoods:
\begin{equation}\label{eq: SRE_def}
    \widehat{\bbeta}_{M} = \arg \max _{\bbeta \in \mathbb{R}^p} S_{M}(\bbeta)
\end{equation}
where
\begin{equation}\label{eq:SM_def}
\begin{aligned}
S_{M}(\bbeta)  & =  \log L(\bbeta; \mathcal{D}) + \frac{\tau}{M} \log L(\bbeta; \mathcal{D}^*) \\
& = \sum_{i=1}^n \left[Y_i\boldsymbol{X}_i^\top \bbeta-\rho\left(\boldsymbol{X}_i^\top \bbeta\right)\right]+\frac{\tau}{M}\sum_{i=1}^M \left[Y^*_i{\boldsymbol{X}_i^*}^\top \bbeta-\rho\left({\boldsymbol{X}_i^*}^\top\bbeta\right)\right].
\end{aligned}
\end{equation}

Because the SRE is defined through a weighted likelihood, it can be computed using standard software by augmenting the observed data with weighted synthetic observations.
Its effortless implementation contrasts with other regularization methods that demand specialized optimization algorithms.

The invariance, existence, and uniqueness of $\widehat{\bbeta}_M$ are investigated in \Cref{sec:uniqueness_MAP_finite_M}.

\subsection{Data-centric regularization and large-$M$ limit}\label{sec:EKL}

Unlike traditional regularization schemes that impose an explicit penalty on $\bbeta$, the SRE regularizes through a weighted synthetic likelihood: the synthetic dataset encodes a simpler data-generating mechanism, and the weight $\tau/M$ controls how strongly the estimator is pulled toward model parameters whose predictions match this data-generating mechanism.

To understand the nature of this regularization mechanism, let $Q$ be the joint distribution used to generate each synthetic data pair $(\bX^*, Y^*)$.
We call the associated conditional distribution of $Y^*$ given $\bX^*$
\textit{the synthetic response generating distribution} and suppose it has a density $g\left(\cdot \mid \bX^*\right)$ w.r.t. the same base measure as $f\left(y \mid \boldsymbol{X}, \bbeta\right)$.
We define the \textit{expected KL divergence} between the synthetic data generation and the target model at $\bbeta$ as
$$
\text{EKL}(\bbeta)=\mathbb{E}_Q\left[ \log \frac{g( Y^* \mid \bX^*) }{f(Y^* \mid \bX^*,\bbeta)}\right]=\mathbb{E}_{X^*}\left[\operatorname{KL}\left(g\left(\cdot \mid X^*\right) \| f\left(\cdot \mid X^*, \bbeta\right)\right)\right],
$$
where the expectation $\mathbb{E}_Q$ is taken w.r.t. $Q$ and $\mathbb{E}_{X^*}$ w.r.t. the marginal distribution of $X^*$. Here the Kullback-Leibler divergence is given by
\[
  \mathrm{KL}\big(g(\cdot \mid \bX^*) \,\|\, f(\cdot \mid \bX^*, \bbeta)\big)
  = \mathbb{E}_{g}
    \bigg[
      \log \frac{g(Y^* \mid \bX^*)}{f(Y^* \mid \bX^*, \bbeta)}
    \bigg],
\]
where the expectation $\mathbb{E}_{g}$ is taken w.r.t. the conditional distribution of $Y^*$ given $\bX^*$.

Recall the likelihood based on the synthetic data defined in \eqref{eq: synthetic-likelihood} and assume that the synthetic observations $\{(X_i^*, Y_i^*)\}_{i=1}^M$ are i.i.d. samples from $Q$.
For each fixed $\bbeta$, the law of large numbers implies that
$$
\begin{aligned}
\frac{1}{M} \log L(\bbeta; \mathcal{D}^*)
~ = & ~ \frac{1}{M}\sum_{i=1}^M \log f(Y_i^*\mid \bX^*_i,\bbeta) \\
\stackrel{\text{a.s.}}{\rightarrow} & ~~ \mathbb{E}_{Q}
    \bigg[
      \log f(Y^* \mid \bX^*, \bbeta)
    \bigg] ~= ~ C- \operatorname{EKL}(\bbeta), \qquad M\rightarrow \infty,
\end{aligned}
$$
where $C=\mathbb{E}_Q\big[\log g(Y^* \mid \bX^*)\big]$ does not depend on $\bbeta$.
It follows that for large $M$, we have
$$
S_{M}(\bbeta) \approx S_\infty(\bbeta) := \log L(\bbeta; \mathcal{D})  - \tau \, \mathrm{EKL}(\bbeta) + \tau C.
$$
Thus, for large $M$, maximizing $S_{M}(\bbeta)$ is approximately equivalent to maximizing the penalized likelihood
\begin{equation}
  \label{eq:EKL_penalized}
  \log L(\bbeta; \mathcal{D})  - \tau \, \mathrm{EKL}(\bbeta) ,
\end{equation}
where the term $ - \tau \, \mathrm{EKL}(\bbeta)$ acts as a regularizer that penalizes model parameters with large expected KL divergence from the synthetic data generating mechanism.

The above large-$M$ limit holds for each fixed $\bbeta$ rather than uniformly over $\bbeta$, and it does not quantify how $M$ affects the SRE in finite samples.
In particular, it is of interest to see how $\widehat{\bbeta}_M$ differs from the \textit{population Synthetic-data Regularized Estimator} (pSRE) defined as
\begin{equation}
\label{cat_betahat_Minfty}
\begin{aligned}
\widehat{\bbeta}_{\infty}&=\arg\max_{\bbeta \in \mathbb{R}^p} S_\infty(\bbeta) \\
&=\arg \max _{\bbeta \in \mathbb{R}^p} \left\{\sum_{i=1}^n \left[Y_i\boldsymbol{X}_i^\top \bbeta-\rho\left(\boldsymbol{X}_i^\top \bbeta\right)\right]+\tau \mathbb E\left[Y^*\boldsymbol{X}^{*\top}\bbeta-\rho(\boldsymbol{X}^{*\top}\bbeta)\right] \right\},
\end{aligned}
\end{equation}
where the expectation is taken w.r.t. the synthetic data generation.
The pSRE $\widehat{\bbeta}_{\infty}$ serves as a non-random benchmark for theoretical purposes rather than a practical method.
A detailed analysis on the difference between $\widehat{\bbeta}_{M}$ and $\widehat{\bbeta}_{\infty}$ is provided in \Cref{sec:stability_MAP_finite_M}.

\subsection{Generation of synthetic data and default choices}\label{sec:synthetic_generation}

We briefly describe how we generate the synthetic data $(\bX^*, Y^*)$ in practice.
The general principle is to choose a simpler model that can be stably fitted to the observed data, and then to use this fitted model as the synthetic response generating distribution $g(\cdot \mid \bX^*)$, where $\bX^*$ is generated by resampling from the observed covariates with appropriate modifications.

To streamline the discussions in later sections, we adopt the following default pipeline:

\begin{enumerate}
  \item \emph{Fit a simpler model.}
  Choose a low-dimensional or otherwise stable model for $Y$ given $\bX$, for example, a submodel of the GLM that only includes the intercept or a small subset of predictors.
  Fit this model to the observed data $\mathcal{D}$ and set the predictive distribution as $g(\cdot \mid \bX^*)$.

  \item \emph{Generate synthetic covariates.}
  Generate $\bX_1^*, \dots, \bX_M^*$ from a design distribution that is easy to sample from and compatible with the target model.
  Typical choices include independently resampling covariate coordinates from the observed covariates, or sampling from a multivariate normal distribution with estimated mean and covariance matrix based on the observed covariates.

  \item \emph{Generate synthetic responses.}
  For each synthetic covariate $\bX_i^*$, draw $Y_i^*$ from the fitted simpler model, that is
  \[
    Y_i^* \mid \bX_i^* \sim g(\cdot \mid \bX_i^*),
    \qquad i = 1, \dots, M.
  \]
\end{enumerate}

We refer to \cite{huang_catalytic_2020} for more strategies for generating synthetic data.
We emphasize that the simpler model does not need to be well-specified or closely approximate the target model; for example, the target model can be a high-dimensional linear regression model, whereas the simpler model can be a regression tree with a few predictors.
In our numerical studies, the simpler model is the intercept-only model by default, which results in satisfactory performance.

In practice, users must also choose the synthetic sample size $M$ and the total weight $\tau$. In our numerical studies, we set $M$ proportional to the dimension $p$, for example $M \ge 4p$, and choose $\tau$ either proportional to $p$ or via cross-validation. These choices are justified theoretically in \Cref{sec:properties} and empirically in \Cref{sec:adjust_inference}.

\section{Theoretical properties of SRE}\label{sec:properties}

This section examines the theoretical properties of the SRE.
We will focus on logistic regression with $\rho(t) = \log(1 + e^t)$ and extend to other GLMs in \Cref{supp:extension_GLM_section}.

\subsection{Existence and uniqueness} \label{sec:uniqueness_MAP_finite_M}

\cite{albert1984existence} showed that the MLE in logistic regression does not exist if the dataset $\{(\boldsymbol{X}_i, Y_i)\}_{i=1}^n$ is separable, meaning that there exists a hyperplane that perfectly separates the covariate vectors with $Y_i=0$ from those with $Y_i=1$.
The next result guarantees the SRE exists if the synthetic data are not separable.

\begin{theorem}\label{thm:MAP_uniqueness}

If the synthetic data $\{(\boldsymbol {X}^*_i,{Y}^*_i)\}_{i=1}^M$ are not separable, equivalently,
$$
\max_{\|\be\|=1} \min _i\left(2 Y_i^*-1\right) \boldsymbol {X}_i^{* \top} \be < 0,
$$
and the synthetic covariate matrix has full column rank, then the SRE in \eqref{eq: SRE_def} exists and is unique.

\end{theorem}

\Cref{thm:MAP_uniqueness} guarantees the existence and uniqueness of the SRE for any sample size $n$. In contrast, the MLE often fails when $2p>n$ \citep{candes2020phase}, and the Maximum Diaconis-Ylvisaker prior penalized likelihood (MDYPL) estimator does not exist when $p>n$ \citep{sterzinger2023diaconis}.
\Cref{app:estimation-comparison} includes a high-dimensional example with $p>n$, where the SRE remains feasible while both the MLE and the MDYPL estimator fail to exist. This illustrates a practical advantage of the SRE in high-dimensional settings.

The condition in \Cref{thm:MAP_uniqueness} is numerically verifiable.
Furthermore, since we have full control over synthetic data generation, it can always be achieved by choosing the synthetic generation scheme.

\begin{proposition}[Equivariance under reparametrization]\label{prop:invariance}
For any bijective function $\vartheta$ of $\bbeta$, the SRE of
$\vartheta(\bbeta)$ is $\vartheta(\widehat{\bbeta}_{M})$.
\end{proposition}
\Cref{prop:invariance} follows directly from the fact that the SRE maximizes a weighted likelihood, and it holds for any parametric model, not only for GLMs.
This property is desirable because the regularization induced by the synthetic data is not tied to a particular coordinate system. Therefore, routine transformations such as rescaling covariates or recoding categorical variables yield the corresponding transformed estimator. In contrast, penalties based on the $\ell_1$ or $\ell_2$ norm are written directly in the parameter coordinates, so estimators such as ridge and lasso generally are not equivariant under such transformations.

\subsection{Regularity conditions}
To apply \Cref{thm:MAP_uniqueness}, it is of interest to study the weighted likelihood based on synthetic data-generating distributions.
\cite{huang_catalytic_2020} study the properness of catalytic priors for synthetic-covariate generating distributions that are \textit{norm-recoverable}, which means
$$\forall \bbeta \in \mathbb{R}^p, \quad \mathbb{E}\left|\bbeta^{\top} \boldsymbol{X}^* \right| \geq c_*\|\bbeta\|, \quad \frac{1}{M} \sum_{i=1}^M\left| \bbeta^{\top} \boldsymbol{X}^* \right| \geq c_*^\prime\|\bbeta\|$$
holds for some constants $c_*$ and $c_*^\prime$.
They show that if the coordinates of $\widetilde{\boldsymbol{X}}^*$ are independent and uniformly bounded, then the synthetic-covariate generating  distribution is norm-recoverable.
However, such a sufficient condition is too restricted. To relax it, we introduce the following condition for synthetic data.

\begin{condition}[Synthetic covariates and responses]
\label{conditions:synthetic_X_Y}
Conditional on the observed data $D$, the synthetic pairs $(X_i^*, Y_i^*)$ are i.i.d. copies of $(\boldsymbol{X}^*, Y^*)$.
Write the synthetic covariate vector $\boldsymbol{X}^* \in \mathbb{R}^p$  as
$\boldsymbol{X}^*=(1, \widetilde{\boldsymbol{X}}^{*\top})^\top$, where the first coordinate corresponds to the intercept term and
$\widetilde{\boldsymbol{X}}^* \in \mathbb{R}^{p-1}$ is the stochastic component.
The pair $(\widetilde{\boldsymbol{X}}^*, Y^*)$ satisfies the following conditions:
\begin{enumerate}
\item[(C1)] (Centering) $\mathbb{E}\bigl(\widetilde{\boldsymbol{X}}^*\bigr)=\mathbf{0}$.

\item[(C2)] (Covariance) Let
$\boldsymbol{\Sigma}^* := \mathbb{E}\bigl(\widetilde{\boldsymbol{X}}^* \widetilde{\boldsymbol{X}}^{*\top}\bigr)$.
There exist constants $\kappa_- ,\kappa_+>0$ such that
$$
\begin{aligned}
\kappa_- \leq \lambda_{\min}(\boldsymbol{\Sigma}^*) \leq \lambda_{\max}(\boldsymbol{\Sigma}^*) \leq \kappa_+ .
\end{aligned}
$$

\item[(C3)] (Sub-gaussian tail)
There exists a constant $K_X>0$ such that for
every vector
$\boldsymbol{u}\in \mathbb{R}^{p-1}$,
$$
\begin{aligned}
\|\boldsymbol{u}^{\top}\widetilde{\boldsymbol{X}}^*\|_{\psi_2} \leq K_X \|\boldsymbol{u}^{\top}\widetilde{\boldsymbol{X}}^*\|_{L^2}.
\end{aligned}
$$

\item[(C4)] (Clipped synthetic response probabilities) There exists a fixed $q\in(0,1/2]$ such that
$$
\begin{aligned}
\mathbb{P}\bigl(Y^*=1\mid \boldsymbol{X}^*,\mathcal{D}\bigr)\in[q,1-q].
\end{aligned}
$$
\end{enumerate}
\end{condition}
\Cref{conditions:synthetic_X_Y} is mild and can always be satisfied since we have full control over the generation of synthetic data.
(C1) can always be assumed for convenience since the SRE is invariant to reparametrization.
(C2) requires the stochastic components to have well-conditioned covariances, while (C3) imposes a sub-gaussian tail bound on all linear combinations of $\widetilde{\boldsymbol{X}}^*$.
These requirements are satisfied, for instance,
by Gaussian synthetic covariates with well-conditioned covariance, or by independently resampling observed covariate coordinates after coordinate-wise truncation with bounded truncation levels.
(C4) is also mild and can always be satisfied.
For example, if we generate synthetic responses independently from a symmetric Bernoulli distribution, the condition is satisfied with $q=0.5$.

\begin{proposition}\label{prop:properties_synthetic_X}
    Under \Cref{conditions:synthetic_X_Y}, the following statements hold:
    
    \begin{enumerate}
    
        \item For all $t\ge 0$,
$$
\begin{aligned}
\left\| \frac{1}{M}\sum_{i=1}^{M} \boldsymbol{X}_i^*\boldsymbol{X}_i^{*\top } \right\|_{op}
 \leq \left\{1 + \kappa_{+}^{1/2}\left[1+C K_X^2(\sqrt{(p-1)/M}+t)\right] \right\}^2
\end{aligned}
$$
holds with probability at least $1-2\exp(-M t^2)$ where $C>0$ is a universal constant.

\item There exist positive constants $\rho_0$, $\eta_0$, and $r_0$ that only depend on $(\kappa_-,\kappa_+,  K_X)$ such that for every $\bbeta\in\mathbb{R}^p$ with $\|\bbeta\|_2=1$,
\begin{equation}\label{eq:main-spread}
    \begin{aligned}
\mathbb{P}\left(\left|\boldsymbol{X}^{*\top}\bbeta\right|>\eta_0\right)\geq \rho_0.
\end{aligned}
\end{equation}

Furthermore, if $M\geq r_0 p$, then with probability at least $1-2e^{-M \min(1, ~~ \rho_0^2/4) }$,
the synthetic covariate matrix $\mathbb{X}^*$ has full column rank and
$$
\inf _{\|\bbeta\|=1} \frac{1}{M} \sum_{i=1}^M\left|\boldsymbol{X}_i^{*\top} \bbeta\right| \geq \frac{\eta_0 \rho_0}{4}.
$$
\item There exist positive constants $r_1$ and $c_1$ depending only on $q$,
    such that if $M\geq r_1 p$, then the synthetic data $\{(\boldsymbol {X}^*_i,{Y}^*_i)\}_{i=1}^M$ are not separable with probability at least $1-2e^{-c_1 M}$.
\end{enumerate}
\end{proposition}

In \eqref{eq:main-spread}, $\eta_0$ is a margin level and $\rho_0$ is the probability mass beyond that margin, so together they quantify that the synthetic covariates are sufficiently spread out and not concentrated near any hyperplane.

\Cref{prop:properties_synthetic_X} guarantees that under \Cref{conditions:synthetic_X_Y}, if the ratio $M/p\geq \max(r_0, r_1)$, then the condition in \Cref{thm:MAP_uniqueness} holds with high probability.

\subsection{Consistency of SRE when $p$ diverges}\label{sec:Consistency}
This section establishes the consistency of the SRE in the regime where the dimension $p$ can diverge to infinity with $p = o(n)$.
We begin with the following conditions on the true regression coefficients and the observed covariates.

\begin{condition}
\label{condition:constant_signal}
The true coefficient vector $\bbeta_0$ satisfies  $\|\bbeta_0\|_2\leq C_3$.
\end{condition}
\begin{condition}\label{condition:moment_X_bound}
    $\mathbb E\left(\|\bX_i\|^2_2  \right)\leq C_2 p$ for all $i\in \{1,2,\cdots n\}$.
\end{condition}

\begin{condition} \label{condition:SubgaussianX}
There exist positive constants $c_1$, $c_2$, $\zeta$, and $N_0$ such that for any $n>N_0$ and any subset $S\subseteq \{1,2,\cdots,n\} $ with $|S|\geq (1-\zeta) n$, the following inequality holds:
        $$c_1 |S|\leq \lambda_{\min}\left(\sum_{i\in S} \bX_i\bX_i^\top \right)\leq \lambda_{\max}\left(\sum_{i\in S} \bX_i\bX_i^\top \right)\leq c_2 |S| . $$
\end{condition}

\Cref{condition:constant_signal} is a standard regularity condition.
\Cref{condition:moment_X_bound} is a moment condition weaker than common boundedness assumptions in the literature (see, e.g., \cite{portnoy1984asymptotic, liang2012maximum}).
\Cref{condition:SubgaussianX} ensures the Hessian matrix remains well-conditioned when $p$ diverges, which is a mild condition.

\begin{theorem}
	\label{thm:post_mode_consistency}
Consider the estimators $\widehat{\bbeta}_{M}$ defined in \eqref{eq: SRE_def} and $\widehat{\bbeta}_{\infty}$
 defined in \eqref{cat_betahat_Minfty} in logistic regression.
 Suppose $p=o(n)$ and the tuning parameter is chosen such that $\tau \leq C_4 p$ for some fixed constant $C_4<\infty$.
 Under \Cref{conditions:synthetic_X_Y,condition:constant_signal,condition:moment_X_bound,condition:SubgaussianX}, we have
 $$ \|\widehat{\bbeta}_{\infty }-\bbeta_0\|_2^2=O_p\left(\frac{p}{n}\right).$$
 If $p^2=O(Mn)$, then we further have
 $$\|\widehat{\bbeta}_{M}-\bbeta_0\|_2^2=O_p\left(\frac{p}{n}\right).$$

\end{theorem}

\Cref{thm:post_mode_consistency} shows that when $p=o(n)$ and $p=O(M)$, both $\widehat{\bbeta}_M$ and $\widehat{\bbeta}_{\infty}$ converge to $\bbeta_0$ at the rate $O_p\bigl(\tfrac{p}{n}\bigr)$.
This rate matches the minimax lower bound of $O(p/n)$ for the estimation error in GLMs \citep{chen2016bayes} when $p=o(n)$.

The $\tau=O(p)$ requirement in \Cref{thm:post_mode_consistency}  ensures that the synthetic-data regularization does not overwhelm the information in the observed data.

\subsection{Nonasymptotic boundedness} \label{sec:nonasymptotic_bound}

The consistency results in \Cref{sec:Consistency} require $p=o(n)$.
When this requirement is not met, it remains interesting to establish nontrivial bounds on the estimators.
This section establishes that
the SREs remain bounded even when $p$ exceeds $n$.

For the boundedness, we only impose conditions on the synthetic data and the tuning parameter $\tau$, which are both operational.

\begin{condition}
\label{condition:sufficient_regualrization}
$\tau$ is chosen such that $\tau \geq c_* p $ where $c_*$ is any positive constant.
\end{condition}

\Cref{condition:sufficient_regualrization} ensures effective regularization when $p$ is large, which aligns with the principle that models with more parameters require more regularization to prevent overfitting \citep{hastie2009elements}.

\begin{theorem}\label{thm:MAP_bounded}
Consider the estimators $\widehat{\bbeta}_{M}$ defined in \eqref{eq: SRE_def} and $\widehat{\bbeta}_{\infty}$
 defined in \eqref{cat_betahat_Minfty} in logistic regression.
Suppose \Cref{conditions:synthetic_X_Y,condition:sufficient_regualrization} hold and $p > \omega_* n$ for some $\omega_*>0$.

Then, there are some positive constants $\Tilde{c},\Tilde{C}, \overline{C}$ such that the following statements hold:

(i) If $M\geq \Tilde{C}p$, the SRE  $\widehat{\bbeta}_{M}$ satisfies that
$\|\widehat{\bbeta}_{M}\|_2\leq 4\overline{C}$
with probability at least $1-2\exp(-\Tilde{c}M)$.

(ii)  The pSRE $\widehat{\bbeta}_{\infty}$ satisfies that
    $\|\widehat{\bbeta}_{\infty}\|_2\leq \overline{C}
    $.

\end{theorem}

\Cref{thm:MAP_bounded} requires only mild conditions on the synthetic data, which we fully control, and no assumptions on the observed data.
In contrast, the boundedness of MLEs and MDYPL estimators requires stricter conditions, such as $n > p$, normality assumptions on the observed covariates, and a full-rank design matrix \citep{sur2019likelihood, sterzinger2023diaconis}.
These distinctions highlight the robustness and broader applicability of our method.

\Cref{thm:MAP_bounded} reveals that sufficient regularization (i.e., $\tau \geq c_* p$) ensures the norms of the SREs $\|\widehat{\bbeta}_{M}\|_2$ and $\|\widehat{\bbeta}_{\infty}\|_2$ remain bounded.
This regularization condition is compatible with the condition for consistency (i.e., $\tau \leq C_4 p$) in \Cref{thm:post_mode_consistency}: since choosing $\tau \propto p$ satisfies both, this serves as a default choice regardless of the relationship between $p$ and $n$.
Choosing $\tau$ proportional to $p$ also aligns with the empirical recommendation made in \cite{huang_catalytic_2020}.

To illustrate the idea behind the proof of \Cref{thm:MAP_bounded}, we briefly outline how part (ii) can be established.
By \Cref{prop:properties_synthetic_X}, there exist positive constants $\eta_0$ and $\rho_0$ such that for any unit vector $\bu\in \mathbb{R}^p$,
$$
\mathbb{P}\left(\left|\bX^{* \top} \bu\right|>\eta_0\right) \geq \rho_0.
$$
If we define $\nu:=\min (q, 1-q) \rho_0$, we can establish the following coercivity condition:
$$
\mathbb{E} \max \left\{0,-\left(2 Y_i^*-1\right) \bX_i^{* \top} \bbeta \right\} \geq \eta_0 \nu \|\bbeta\|,  \quad \text { for all }\bbeta\in \mathbb{R}^p.
$$
On the other hand, based on the optimality of the objective function in \eqref{cat_betahat_Minfty} and compared with the naive estimator  $\bbeta=\mathbf{0}$, we can obtain
$$\tau \mathbb{E} \max \left\{0,-\left(2 Y^*-1\right) \bX^{* \top} \widehat{\bbeta}_{\infty}\right\} \leq(n+\tau) \log 2.
$$
Combining these two results yields the desired bound on $\|\widehat{\bbeta}_{\infty}\|_2$.

We discuss the implications of  \Cref{thm:MAP_bounded}.
The boundedness property verifies the radius condition required by the stability bound in \Cref{sec:stability_MAP_finite_M}.
Furthermore,  the boundedness property serves as an essential condition for the  high-dimensional exact asymptotic analysis in \Cref{sec:linear_asymptotic_regime}.

Lastly, the boundedness result  implies the following corollary regarding the estimation error.

\begin{corollary}\label{corollary:error_O1}

 Suppose \Cref{condition:constant_signal} and the conditions in \Cref{thm:MAP_bounded} hold. Then, there are positive constants $\Tilde{C}_1$, $\Tilde{C}_2$, and $\Tilde{c}$
 such that
(1) $\|\widehat{\bbeta}_{\infty}-\bbeta_0\|_2^2\leq \Tilde{C}_1,$
 and (2) if $M\geq \Tilde{C}_2 p$, then
	$\|\widehat{\bbeta}_{M}-\bbeta_0\|_2^2\leq \Tilde{C}_1$
with probability at least $1-2\exp(-\Tilde{c}M)$.
\end{corollary}

\Cref{corollary:error_O1} shows the error of the SRE remains bounded even when $p$ grows as fast as or faster than $n$,  unlike the MLE and MDYPL estimator, whose error is unbounded when $p/n$ is large.

\Cref{corollary:error_O1} and \Cref{thm:post_mode_consistency} together imply that, over the asymptotic regimes considered here,
the SRE with $\tau\propto p$ attains the error rate of order $\min(\frac{p}{n}, 1)$, which matches the rate of the minimax lower bound for estimation error in GLMs when no structural assumption is imposed \citep{chen2016bayes}.
This rate can also be attained by other regularized estimators, and it provides a baseline guarantee:
incorporating synthetic data through the SRE, with appropriate tuning, does not worsen the performance.
In \Cref{app:estimation-comparison}, we provide additional simulations to compare the SRE with ridge and Lasso estimators. In the settings considered there, the SRE is competitive with these methods and can improve on them as the dimension increases.

\subsection{Stability against finite $M$}
\label{sec:stability_MAP_finite_M}

A concern with synthetic-data regularization is the potential for instability due to randomness in synthetic data generation.
To address this concern, we examine the impact of the randomness in synthetic data in this section.

As discussed in \Cref{sec:EKL}, the pSRE $\widehat{\bbeta}_{\infty}$ serves as a non-random benchmark, and the computable SRE $\widehat{\bbeta}_{M}$ is expected to converge to $\widehat{\bbeta}_{\infty}$ as $M \to \infty$.
Moreover, when the weight parameter is set as $\tau \asymp p$, \Cref{thm:MAP_bounded,thm:post_mode_consistency} ensure that both estimators $\widehat{\bbeta}_M$ and $\widehat{\bbeta}_{\infty}$ are bounded. This suggests that we can restrict our attention to a compact set to avoid unnecessary technical complications.

Concretely, we define  $\mathcal{B}_K:= \{\bbeta\in \mathbb R^p: \|\bbeta \|_2\leq K\}$ for any $K>0$ and define the constrained estimators
\begin{equation}\label{eq: SRE_BL}
\left\{
\begin{aligned}
\widehat{\bbeta}_{M}^{(K)} &= \arg \max _{\bbeta \in \mathcal{B}_K} S_{M}(\bbeta), \\
\widehat{\bbeta}_{\infty}^{(K)} &= \arg \max _{\bbeta \in \mathcal{B}_K} S_{\infty}(\bbeta).
\end{aligned}
    \right.
\end{equation}
\Cref{thm:MAP_bounded,thm:post_mode_consistency} ensure that under their respective conditions, there is some $K$ such that
$\widehat{\bbeta}_M = \widehat{\bbeta}_{M}^{(K)}$ and $\widehat{\bbeta}_{\infty} = \widehat{\bbeta}_{\infty}^{(K)}$
with probability tending to one.
Therefore, our stability analysis can be conducted on $\widehat{\bbeta}_{M}^{(K)}$ and $\widehat{\bbeta}_{\infty}^{(K)}$ without loss of generality.
In the following, we treat the observed data as fixed and regard the synthetic data as the only source of randomness.

A key ingredient in our stability analysis is that the synthetic-data component supplies
\emph{uniform curvature} of the objective. For GLMs, the log-likelihood is twice differentiable
and the Hessian quantifies the \emph{local sensitivity} of the log-likelihood to perturbations of $\bbeta$.
Equivalently, it measures local identifiability: when the objective is strongly concave near its maximizer, the maximizer is stable under small perturbations of the objective.

For the synthetic likelihood, the (negative) Hessian takes the Fisher-information form
\[
\forall \bbeta\in \mathcal{B}_K, \quad
\left\{
\begin{aligned}
\boldsymbol{H}(\bbeta)
:= &  - \nabla^2 \mathbb{E}_{Q}
    \bigg[
      \log f(Y^* \mid \bX^*, \bbeta)
    \bigg]
=
\mathbb{E}\!\left(\rho^{\prime\prime}(\bX^{*\top}\bbeta)\bX^*\bX^{*\top}\right), \\
\widehat{\boldsymbol{H}}_M(\bbeta)
:= & - \nabla^2 \frac{1}{M} \log L(\bbeta; \mathcal{D}^*)
=
\frac{1}{M}\sum_{i=1}^M \rho^{\prime\prime}(\bX_i^{*\top}\bbeta)\bX_i^*\bX_i^{*\top}.
\end{aligned}
\right.
\]
In logistic regression, $\rho^{\prime\prime}(t)=\rho^\prime(t)[1-\rho^\prime(t)]$, so the weight $\rho^{\prime\prime}(\bX^{*\top}\bbeta)$
is largest when the synthetic responses are most uncertain (probabilities around 0.5), and it becomes small when the linear predictors are extreme (probabilities close to 0 or 1). Thus, the matrix $\widehat{\boldsymbol{H}}_{M}(\bbeta)$ aggregates
directional information $\bX_{i}^{*}\bX_{i}^{*\top}$, weighted by how informative each synthetic observation is at $\bbeta$.

These Hessian matrices enter the stability analysis through the strong-concavity modulus of the full SRE objective:
\begin{equation}\label{eq:hessian-decomposition}
\forall \bbeta\in \mathcal{B}_K, \quad
-\nabla^{2} S_{M}(\bbeta)
=
\underbrace{\sum_{i=1}^{n}\rho^{\prime\prime}(\bX_{i}^{\top}\bbeta)\,\bX_{i}\bX_{i}^{\top}}_{\text{observed information}}
\;+\;
\underbrace{\tau\,\widehat{\boldsymbol{H}}_{M}(\bbeta)}_{\text{synthetic information}},
\end{equation}
where $\tau$ controls the curvature coming from the synthetic term.
The stability analysis relies on the property that this synthetic Fisher information is non-degenerate \emph{uniformly over the region where the
estimator lives}.
Concretely, we have the following result.

\begin{proposition}
\label{prop:strong_convex_SRE}
Suppose \Cref{conditions:synthetic_X_Y} holds and fix any $K>0$.
There exist constants $C_0$ and $c_K$ such that
$$
\inf_{\bbeta\in\mathcal{B}_K}\lambda_{\min}\!\left(\boldsymbol{H}(\bbeta)\right)\ge 2c_K,
$$
and if $M\geq C_0[p+\log(1/\varepsilon)]$, then with probability at least $1-\varepsilon$,
\begin{equation}\label{eq:empirical-hessian-lowerbound}
\inf_{\bbeta\in\mathcal{B}_K}\lambda_{\min}\!\left(\widehat{\boldsymbol{H}}_{M}(\bbeta)\right)\ge c_K.
\end{equation}
\end{proposition}

Intuitively, two features of the synthetic design ensure $c_K>0$:
(i) the synthetic covariates are spread out so that every direction $\bv$ has nontrivial mass in $(\bX^{*\top}\bv)^{2}$,
and (ii) within $\mathcal{B}_K$ the linear predictors $\left|\bX^{*\top}\bbeta\right|$ are not too large,
so $\rho^{\prime\prime}(\bX^{*\top}\bbeta)$ is not systematically near $0$.
An explicit statement and detailed proofs (including a closed-form choice of $c_K$) are given in \Cref{proof_sec:logitic_stability}.

For the observed information in \eqref{eq:hessian-decomposition}, we define
$$
\lambda_{n,K}:=\inf_{\bbeta\in \mathcal{B}_K}\lambda_{\min}\left(\sum_{i=1}^n \rho^{\prime\prime}(\bX_i^\top \bbeta)\bX_i \bX_i^\top\right) ~ \geq 0.
$$
We are now ready to present the stability bound.

\begin{theorem}\label{thm:stability_finite_M}
Suppose that $\tau>0$, $K>0$, and \Cref{conditions:synthetic_X_Y} holds.
There exist constants $C_0$ and $C_1$ depending on $(\kappa_{-},\kappa_{+},K_X)$, and a constant $c_K$ that additionally depends on $K$, such that for any $\epsilon\in (0,1)$,
if $M\geq C_0[p+\log(1/\varepsilon)]$, then the following holds with probability at least $1-\epsilon$ (with respect to the randomness of the synthetic data):
$$
\begin{aligned}
\|\widehat{\bbeta}_M^{(K)}-\widehat{\bbeta}_\infty^{(K)}\|_2
\le
\frac{\tau C_1}{ \lambda_{n,K}+\tau c_K/2 } \sqrt{\frac{p+\log(4/\epsilon)}{M}}.
\end{aligned}
$$
In particular, the inequality can be written as $\|\widehat{\bbeta}_{M}^{(K)}-\widehat{\bbeta}_{\infty}^{(K)}\| \lesssim
 \min(\frac{\tau}{\lambda_{n,K}}, 1) \sqrt{\frac{p+\log(4/\epsilon)}{M}}$.
 
\end{theorem}

\Cref{thm:stability_finite_M} shows that $\|\widehat{\bbeta}_{M}-\widehat{\bbeta}_{\infty}\|_2^{2}$ decays at rate $1/M$.
\Cref{sec:experiment-var-M} numerically illustrates this decay rate in more general settings.
This result suggests that increasing $M$ effectively enhances the stability of the SRE against the randomness of synthetic data.

\textbf{The roles of $\tau$ and $\lambda_{n,K}$}.
In low-dimensional settings where $p=o(n)$, it is typical that $\lambda_{n,K}$ grows linearly with $n$.
If we choose $\tau\propto p$, the upper bound reduces to $O\left( \frac{p^{3/2}}{n \sqrt{M}}\right)=o(\sqrt{p/M})$, so choosing $M \asymp p$ is sufficient to ensure stability.
In high-dimensional settings where $p$ is comparable to or larger than $n$,
$\lambda_{n,K}$ is often near zero and the observed information has limited curvature.
The SRE remains stable in this case: in \eqref{eq:hessian-decomposition}, the synthetic term contributes an additional curvature $\tau c_K$, so the effective strong-convexity modulus grows at least linearly with $\tau$.
This ensures the objective function of the SRE remains strongly concave, and thus the SRE is near the pSRE as long as $p/M$ is small.

 \section{Characterization in the linear asymptotic regime}
\label{sec:linear_asymptotic_regime}

This section studies the behavior of the SRE in the linear asymptotic regime, defined by $\lim n/p =\noverp\in (0, \infty)$.
While consistency is impossible for any estimation method in this setting without additional assumptions such as sparsity, we establish a precise asymptotic characterization of the SRE.
The structure of this section is as follows.
In \Cref{sec:exact_asym_M_finite_non_informative}, we consider the setting where $Y^*_i \mid \boldsymbol{X}^*_i \sim \text{Bern}\left(\rho^{\prime}(\boldsymbol{X}_i^{*\top} \bbeta_s) \right)$ with $\bbeta_s=\textbf{0}$.
\Cref{sec:exact_asymptocis_infor} studies the general case in which $\bbeta_s$ is nonzero and exhibits a nontrivial correlation with the true regression coefficient.
\Cref{sec:numerical_verify} presents numerical experiments that verify the theoretical results in \Cref{sec:exact_asym_M_finite_non_informative,sec:exact_asymptocis_infor}.
Finally, \Cref{sec:road_map_maintext} provides a roadmap for the proofs of the main results in this section.
Throughout this section, we assume there is no intercept term.

\subsection{Precise asymptotics of SRE}
\label{sec:exact_asym_M_finite_non_informative}
In this and the next sections, we will focus on the asymptotic behavior of the SRE under noninformative synthetic data and informative auxiliary data respectively. Informally, in this section,  we demonstrate that
\begin{equation}\label{informal_MAP_linear_asym}
\widehat{\bbeta}_M\stackrel{}{\approx}\alpha_*\bbeta_0+p^{-1/2}\sigma_* \bZ  ,
\end{equation}
where $\bZ$ is a standard normal vector, and $(\alpha_*, \sigma_*)$ are constants that depend on $\noverp$, $\tau$, and the data generation process.
This suggests that asymptotically the SRE  is centered around $\alpha_*\bbeta_0$ with some additive Gaussian noise.
To proceed with rigorous justification, we introduce some scaling parameters and necessary conditions.
\begin{condition}\label{condition:proper_scaling}
The parameters $\tau$ and $M$ scale linearly with $n$ such that $\tau/n = \tau_0$, $M/n= m$, and $p/n= 1/\noverp$ for fixed constants
$\tau_0\in (0,\infty)$, $m\in (0,\infty)$, and $\noverp\in (0,\infty)$.
\end{condition}

\Cref{condition:proper_scaling} is motivated by our previous findings:
as shown in \Cref{sec:Consistency,sec:nonasymptotic_bound}, choosing $\tau$ proportional to $p$ is crucial for achieving optimal rates in estimation;
\Cref{sec:stability_MAP_finite_M} suggests that the estimator is stable for sufficiently large  $M/p$.
This condition also echoes the practical guidelines provided by \cite{huang_catalytic_2020}.

\begin{condition}\label{condition:dist_condition(modify)}
    $\left\{\boldsymbol{X}_i\right\}_{i=1}^n \stackrel{\text { i.i.d. }}{\sim} \mathcal{N}\left(\mathbf{0},  \mathbf{I}_p\right)$, $Y_i\mid \bX_i \sim \text{Bern}\left(\rho^\prime(\bX_i^{\top}  \bbeta_0) \right)$ and    there is a constant $\kappa_1>0$, such that $\lim_{p\rightarrow\infty}\|\bbeta_0\|^2=\kappa_1^2$.
\end{condition}

 \Cref{condition:dist_condition(modify)} imposes a strong condition on the covariate matrix, assuming that its entries are independent standard Gaussian random variables.
 Standard Gaussian design conditions are commonly imposed in the study of precise asymptotics such as those discussed in \Cref{rem: literature on precise asymptotics}.
 Some recent works attempt to relax the standard Gaussian design condition in various settings to allow general covariance structures \citep{zhao2022asymptotic,celentano2023lasso} and replace the normality assumption with moment conditions \citep{el2018impact,han2023universality}.
The independence in \Cref{condition:dist_condition(modify)} is relaxed in \Cref{coro:arbitray_cov_exact_cat_M_MAP_noninformative}.
Furthermore, we expect that it is possible to relax the Gaussian design condition for our result and we provide empirical justification in \Cref{sec:beyond_gaussian_empirical_justification}, which suggests that the same convergence seems to hold if the entries of $\boldsymbol{X}_i$'s are independent with zero mean, unit variance, and a finite fourth moment.
However, the development will be much more complicated than the current work and we leave it for future study.

\begin{condition}\label{condition:non-informative_syn_data}
    $\left\{\boldsymbol{X}^*_i\right\}_{i=1}^M \stackrel{\text { i.i.d. }}{\sim} \mathcal{N}\left(\mathbf{0},  \mathbf{I}_p\right)$ and $\left\{Y^*_i\right\}_{i=1}^M \stackrel{\text { i.i.d. }}{\sim}\text{Bern}(0.5)$.
\end{condition}

\Cref{condition:non-informative_syn_data} essentially assumes $Y^*_i$ are generated under the noninformative coefficients $\bbeta_s=\boldsymbol{0}$, which is always achievable since we have full control over the synthetic data generation.
\Cref{sec:exact_asymptocis_infor} relaxes this condition to allow for general $\bbeta_s$.

To make the statement in \eqref{informal_MAP_linear_asym} rigorous, the constants $\alpha_*$ and $\sigma_*$ therein are taken from the solution to the following system of equations in three variables ($\alpha,\sigma,\gamma$):
\begin{equation}
\label{nonlinear_three_equation(modify)}
\left\{\begin{aligned}
\frac{\sigma^2}{2 \noverp} & = \, \mathbb{E}\left[\rho^{\prime}\left(-\kappa_1 Z_1\right)\left(W-\operatorname{Prox}_{\gamma \rho(\cdot)}\left(W\right)\right)^2 + \frac{m}{2}\left(W-\operatorname{Prox}_{{m^{-1} \gamma \tau_0} \rho(\cdot)}\left(W\right)\right)^2\right] \\
1-\frac{1}{\noverp}  & =\,
 \mathbb{E}\left[\frac{2 \rho^{\prime}\left(-\kappa_1 Z_1\right)}{1+\gamma \rho^{\prime \prime}\left(\operatorname{Prox}_{\gamma \rho(\cdot)}\left(W\right)\right)}-\frac{\gamma \tau_0 \rho^{\prime \prime}\left(\operatorname{Prox}_{{m^{-1} \gamma \tau_0} \rho(\cdot)}\left(W\right)\right)}{1+{m^{-1} \gamma \tau_0} \rho^{\prime \prime}\left(\operatorname{Prox}_{{m^{-1} \gamma \tau_0} \rho(\cdot)}\left(W\right)\right)}\right], \\
-\frac{\alpha}{2 \noverp} & =\,
 \mathbb{E}\left[\rho^{\prime \prime}\left(-\kappa_1 Z_1\right) \operatorname{Prox}_{\gamma \rho(\cdot)}\left(W\right)\right],
\end{aligned}\right.
\end{equation}
where $Z_1$ and $Z_2$ are independent standard normal variables, $W:=\kappa_1 \alpha Z_1+\sigma Z_2$.
The system of equations in \eqref{nonlinear_three_equation(modify)} arises as the first-order optimality conditions of a limiting scalar saddle point problem, and the solution with positive $\sigma$ and $\gamma$ is unique; see \Cref{sec:unique-saddle-R}.

We can now precisely characterize the asymptotic behavior of the SRE.

\begin{theorem} \label{thm:exact_cat_M_MAP_noninformative(modify)}
Consider the SRE $\widehat{\bbeta}_{M}$ defined in \eqref{eq: SRE_def}.
   Suppose that \Cref{condition:proper_scaling,condition:dist_condition(modify),condition:non-informative_syn_data} hold and $m\noverp>2$.
   Suppose the  positive parameters ($\kappa_1,\noverp,\tau_0,m$) are such that the system of equations  \eqref{nonlinear_three_equation(modify)} has a solution $(\alpha_*,\sigma_*,\gamma_*)$ with positive $\sigma_*$ and $\gamma_*$. Then, the following statements hold:
   
   (1) For any fixed index set $\mathcal{S} \subset$ $\{1, \ldots, p\}$ with $\sqrt{p}\left\|\bbeta_{0, \mathcal{S}}\right\|_2=O(1)$, we have
$$
\frac{\sqrt{p}}{\sigma_*} \left(\widehat{\bbeta}_{M, \mathcal{S}}-\alpha_* \bbeta_{0, \mathcal{S}}\right)\stackrel{d}{\longrightarrow} \mathcal{N}\left(\mathbf{0}, \boldsymbol{I}_{|\mathcal{S}|}\right) .
$$

 (2) Suppose $\frac{1}{p}\sum_{j=1}^p \noverpmass_{\sqrt{p}\bbeta_{0,j}} \rightsquigarrow \Pi$ for a distribution $\Pi$ with $\mathbb{E}_{\Pi}[\beta^2]=\kappa_1^2$. For any locally-Lipschitz  function\footnote{A function $\Psi: \mathbb{R}^m \rightarrow \mathbb{R}$ is said to be locally-Lipschitz   if there exists a constant $L>0$ such that for all $\boldsymbol{t}_0, \boldsymbol{t}_1 \in \mathbb{R}^m,$ $\left\|\Psi\left(\boldsymbol{t}_0\right)-\Psi\left(\boldsymbol{t}_1\right)\right\| \leq L\left(1+\left\|\boldsymbol{t}_0\right\|+\left\|\boldsymbol{t}_1\right\|\right)\left\|\boldsymbol{t}_0-\boldsymbol{t}_1\right\|$.}
 $\Psi(a,b)$ or for the indicator function $\Psi(a,b)=\mathbf{1}\{|a/\sigma_*|\leq t\}$ with fixed $t>0$, we have
 $$	\frac{1}{p} \sum_{j=1}^p \Psi\left(\sqrt{p}(\widehat{\bbeta}_{M,j}-\alpha_* \bbeta_{0,j}), \sqrt{p}\bbeta_{0,j}\right) \stackrel{\mathbb P}{\longrightarrow} \mathbb{E}[\Psi( \sigma_*Z, \beta)] ,$$
 where $Z\sim N(0,1)$ is independent of  $\beta\sim  \Pi$.
\end{theorem}

The proof of \Cref{thm:exact_cat_M_MAP_noninformative(modify)} follows the same argument as that of \Cref{thm:exact_cat_M_MAP_informative}, but is much simpler, so we defer the discussion until after the latter theorem.

\Cref{thm:exact_cat_M_MAP_noninformative(modify)} reveals that
in the linear asymptotic regime, the SRE $\widehat{\bbeta}_{M}$ is centered around the scaled true coefficient vector $\alpha_*\bbeta_0$ and
$\sqrt{p}\left(\widehat{\bbeta}_{M}-\alpha_*\bbeta_0\right)$ is approximately normal with independent entries with standard deviation $\sigma_*$.
\Cref{thm:exact_cat_M_MAP_noninformative(modify)} implies various asymptotic relationships between the estimator $\widehat{\bbeta}_{M}$ and the true coefficients $\bbeta_0$ by varying the locally-Lipschitz function $\Psi$. Here are some examples:

 \textbf{Squared error and cosine similarity.}
By taking $\Psi(a,b)=(a+(\alpha_{*}-1)b)^2$ and Slutsky's theorem, we have
\begin{align}
   & \|\widehat{\bbeta}_M-\bbeta_0\|^2\stackrel{\mathbb{P}}{\longrightarrow} \sigma_*^2+(\alpha_{*}-1)^2\kappa_1^2. \label{eq:exact_cat_M_MAP_noninformative_MSE}\\
   &\frac{\langle \widehat{\bbeta}_M,\bbeta_0 \rangle}{\|\widehat{\bbeta}_M\|_2\|\bbeta_0\|_2} \xrightarrow{\mathbb{P}}\frac{\alpha_*\kappa_1}{\sqrt{\alpha_*^2\kappa_1^2+\sigma_*^2}}. \label{eq:exact_cat_M_MAP_noninformative_similarity}
\end{align}
In \Cref{sec:numberical_verify_non_infor}, we plot the theoretical limit against the value of $\tau_0$ and reveal a bias and variance trade-off phenomenon for the regularization using synthetic data.

These choices of $\Psi$ have previously been explored in the literature and we remark that these results continue to hold without the condition that $\frac{1}{p}\sum_{j=1}^p \noverpmass_{\sqrt{p}\bbeta_{0,j}} \rightsquigarrow \Pi$.
Other examples, including \textit{generalization error} and \textit{predictive deviance}, are discussed in \Cref{proof:deviance_generalization}.

\textbf{Oracle confidence intervals.}
\Cref{thm:exact_cat_M_MAP_noninformative(modify)} also yields an oracle inference result for individual coordinates.
For each $j\in[p]$, consider
\begin{equation}\label{eq:exact_cat_M_MAP_noninformative_CI}
\mathrm{CI}_j=
\left[
\frac{\widehat{\bbeta}_{M,j}-1.96\,\sigma_*/\sqrt{p}}{\alpha_*},
\frac{\widehat{\bbeta}_{M,j}+1.96\,\sigma_*/\sqrt{p}}{\alpha_*}
\right].
\end{equation}
\Cref{thm:exact_cat_M_MAP_noninformative(modify)} implies two types of asymptotic validity.
First, for any fixed $j$ with $\sqrt{p}\bbeta_{0,j}=O(1)$, part (1) gives
$$
\mathbb{P}\bigl(\bbeta_{0,j}\in \mathrm{CI}_j\bigr)\to 0.95.
$$
Second, by applying part (2) with
$\Psi(a,b)=\mathbf{1}\{|a/\sigma_*|\le 1.96\}$, we obtain
$$
\frac{1}{p}\sum_{j=1}^p \mathbf{1}\{\bbeta_{0,j}\in \mathrm{CI}_j\}
\xrightarrow{\mathbb{P}} 0.95.
$$
Since $\alpha_*$ and $\sigma_*$ depend on unknown parameters, the intervals in \Cref{eq:exact_cat_M_MAP_noninformative_CI} are not directly implementable.
\Cref{sec:adjust_inference} develops a feasible version by estimating the required quantities.

In \Cref{thm:exact_cat_M_MAP_noninformative(modify)}, the condition that $\operatorname{Cov}(\boldsymbol{X})=\mathbb{I}_p$ can be relaxed to allow for a general covariance matrix, as stated in the following corollary.

\begin{corollary}\label{coro:arbitray_cov_exact_cat_M_MAP_noninformative}
Consider the logistic regression model and the SRE $\widehat{\bbeta}_{M}$ defined in \Cref{sec:sec:SRE_GLM} under \Cref{condition:proper_scaling} and the condition $m\noverp > 2$.
Suppose $\boldsymbol{X}_i \stackrel{\text { i.i.d. }}{\sim} \mathcal{N}\left(\mathbf{0}, \mathbf{\Sigma}\right)$ for $i\in [n]$ and $\boldsymbol{X}^*_i\stackrel{\text { i.i.d. }}{\sim} \mathcal{N}\left(\mathbf{0},  \mathbf{\Sigma}\right)$  for $i \in [M]$, where $\mathbf{\Sigma}$ is a positive definite matrix.
Let $v_j^2=\operatorname{Var}\left(X_{i, j} \mid \boldsymbol{X}_{i,-j}\right)$ denote the conditional variance of $X_{i, j}$ given all other covariates. Furthermore, assume that the empirical distribution $\frac{1}{p}\sum_{j=1}^p \noverpmass_{\sqrt{p} v_j \bbeta_{0,j}} $ converges weakly to a distribution $\Pi$ with a finite second moment, $\|\mathbf{\Sigma}^{1/2}\bbeta_0\|^2\xrightarrow{\mathbb{P}} \kappa_1^2$, and $\sum_{j=1}^p v_j^2 \bbeta_{0,j}^2 \xrightarrow{\mathbb{P}} \mathbb{E}\left[\beta^2\right]$ for $\beta \sim \Pi$.
Given the parameters ($\kappa_1,\noverp,\tau_0,m$) are such that the system of equations  \eqref{nonlinear_three_equation(modify)} has a   solution $(\alpha_*,\sigma_*,\gamma_*)$. Then, for any locally-Lipschitz  function
 $\Psi: \mathbb{R} \times \mathbb{R} \rightarrow \mathbb{R}$ or for the indicator function $\Psi(a,t)=\mathbf{1}\{|a/\sigma_*|\leq t\}$ with any fixed $t>0$, we have
    $$ \frac{1}{p} \sum_{j=1}^p \Psi\left(\sqrt{p}v_j(\widehat{\bbeta}_{M,j}-\alpha_*\bbeta_{0,j}), \sqrt{p}v_j\bbeta_{0,j}\right) \stackrel{\mathbb P}{\longrightarrow}  \mathbb{E}[\Psi( \sigma_*Z, \beta)]. $$
    where $Z\sim N(0,1)$ independent of  $\beta\sim  \Pi$.
\end{corollary}
\Cref{coro:arbitray_cov_exact_cat_M_MAP_noninformative} makes our theory more applicable in practice.  It can be proved by employing the argument in \cite{zhao2022asymptotic} and we omit the details here.

\subsection{Extension to informative auxiliary data}\label{sec:exact_asymptocis_infor}

In many practical scenarios, \textit{informative auxiliary data} (e.g., data from different but similar studies) may be available.
As discussed in \Cref{sec:connections-existing}, when the synthetic data are replaced by such auxiliary data, the SRE coincides with the posterior mode under the power prior.
From a frequentist perspective, this estimator can also be viewed as a source-target weighted estimator in transfer learning.
In this section, we derive a precise asymptotic characterization for this estimator in the proportional regime and quantify how its performance depends on the similarity between the target and auxiliary data sources.

To be concrete, suppose $(\bX_i^*,Y_i^*)$ are informative auxiliary data such that $Y^*_i \mid \boldsymbol{X}^*_i \sim \text{Bern}\left(\rho^{\prime}(\boldsymbol{X}_i^{*\top} \bbeta_s) \right)$ and $\bbeta_s$ is correlated with $\bbeta_0$.
In addition to \Cref{condition:proper_scaling,condition:dist_condition(modify)}, further assume $\|\bbeta_s\|_2\to \kappa_2$ and $\langle\bbeta_s,\bbeta_0\rangle/(\|\bbeta_s\|_2\|\bbeta_0\|_2) \to \xi$.
The parameter $\xi$ measures the alignment between the directions of $\bbeta_s$ and $\bbeta_0$ and admits a natural cosine-type geometric interpretation.
In analogy to \eqref{informal_MAP_linear_asym}, we have the following for the SRE $\widehat{\bbeta}_{M}$ using informative auxiliary data:
\begin{equation}\label{informal_MAP_informative_linear_asym}
    \widehat{\bbeta}_M\approx \alpha_{1*} \bbeta_0+ \frac{\alpha_{2*}}{\sqrt{1-\xi^2}}(\bbeta_s-\xi\frac{\|\bbeta_s\|}{\|\bbeta_0\|}\bbeta_0)+p^{-1/2}\sigma_{*}\bZ ,
\end{equation}
where $\bZ$ is a standard normal vector, and $(\alpha_{1*},\alpha_{2*},\sigma_{*})$ depend on $\noverp,\tau,M$ and the data generation process.
Compared with \eqref{informal_MAP_linear_asym}, the SRE is not centered at scaled $\bbeta_0$ but a linear combination of $\bbeta_0$ and $\bbeta_s$.

\begin{remark}
Many methods for transfer learning have been investigated recently from statistical perspectives; see for example \cite{bastani2021predicting,reeve2021adaptive,li2022transfer,li2023transfer,tian2023transfer,li2023estimation,zhang2023transfer}.
These developments usually assume that the difference between the target model and the source model is sufficiently small.
In contrast, our analysis accommodates any fixed similarity level $\xi\in (-1,1)$.
This parameter is related in spirit to recent angle-based notions of source-target alignment, although the formulations and conclusions are different.
\citet{gu2025robust} study high-dimensional linear regression and introduce source-target alignment through a random coefficient model for the target and source regression coefficients, where the common entrywise correlation parameter plays a role analogous to our $\xi$.
Their results give upper and lower bounds on the limiting expected predictive risk.
\cite{tian2025learning} study multi-task regression under a low-rank representation assumption, where task similarity is quantified through the maximum principal angle. In the rank one case, this becomes a function of the angle between the regression coefficient vectors. Their results are nonasymptotic upper bounds and minimax lower bounds, so the conclusions are at the level of rates rather than an exact proportional asymptotic characterization.
\end{remark}

To make the normal approximation in \eqref{informal_MAP_informative_linear_asym} rigorous, we continue to assume \Cref{condition:proper_scaling,condition:dist_condition(modify)}, and we suppose the auxiliary data satisfy the following condition.

\begin{condition}\label{condition:informative_syn_data}
    The covariate vectors  $\left\{\boldsymbol{X}^*_i\right\}_{i\in [M]}  \stackrel{\text { i.i.d. }}{\sim} \mathcal{N}\left(\mathbf{0},  \mathbf{I}_p\right)$.
    Given these covariate vectors, the auxiliary responses are conditionally independent with $Y^*_i \mid \boldsymbol{X}^*_i \sim \text{Bern}\left(\rho^{\prime}(\bbeta_s^{\top} \boldsymbol{X}^*_i) \right)$.
    There is a  constant  $\kappa_2> 0$, and $\xi\in (-1,1)$, such that  $\lim_{p\rightarrow\infty}\|\bbeta_s\|^2=\kappa_2^2$ and $\lim_{p\rightarrow\infty}\frac{1}{\|\bbeta_0\|\|\bbeta_s\|}\langle \bbeta_0,\bbeta_s\rangle=\xi$.
\end{condition}

Similar to \eqref{nonlinear_three_equation(modify)}, we introduce an important system of equations in four variables ($\alpha_1,\alpha_2,\sigma,\gamma$), which includes an extra variable $\alpha_2$ to track the influence of informative auxiliary data.

To present the new system of equations, let  $Z_1,Z_2,Z_3$ be i.i.d. standard normal random variables.  The variable $W_I$ is defined as  a linear combination of $Z_1,Z_2$ and $Z_3$, specifically  $W_I:=\kappa_1 \alpha_1 Z_1+\kappa_2 \alpha_2 Z_2 +\sigma Z_3$.
The system of equations is given as follows.

\begin{equation}
\label{nonlinear_four_equation}
\left\{\begin{aligned}
\frac{\sigma^2}{2 \noverp} & =\mathbb{E}\left[\rho^{\prime}\left(-\kappa_1 Z_1\right)\left(W_I-\operatorname{Prox}_{\gamma \rho(\cdot)}\left(W_I\right)\right)^2\right] \\
& \quad  +m \mathbb{E}\left[\rho^{\prime}\left(-\kappa_2 \xi Z_1-\kappa_2 \sqrt{1-\xi^2}Z_2\right)\left(W_I-\operatorname{Prox}_{{m^{-1} \gamma \tau_0} \rho(\cdot)}\left(W_I\right)\right)^2\right],\\
1-\frac{1}{\noverp}+m  & =\mathbb{E}\left[\frac{2 \rho^{\prime}\left(-\kappa_1 Z_1\right)}{1+\gamma \rho^{\prime \prime}\left(\operatorname{Prox}_{\gamma \rho(\cdot)}\left(W_I\right)\right)}+\frac{2m \rho^{\prime}\left(-\kappa_2 \xi Z_1-\kappa_2 \sqrt{1-\xi^2}Z_2\right)}{1+{m^{-1} \gamma \tau_0} \rho^{\prime \prime}\left(\operatorname{Prox}_{{m^{-1} \gamma \tau_0} \rho(\cdot)}\left(W_I\right)\right)}\right] \\
-\frac{\alpha_1}{2 \noverp} & =\mathbb{E}\left[\rho^{\prime \prime}\left(-\kappa_1 Z_1\right) \operatorname{Prox}_{\gamma \rho(\cdot)}\left(W_I\right)\right] \\
&\quad +m \xi \frac{\kappa_2}{\kappa_1}\mathbb{E}\left[\rho^{\prime \prime}\left(-\kappa_2 \xi Z_1-\kappa_2 \sqrt{1-\xi^2}Z_2\right) \operatorname{Prox}_{{m^{-1} \gamma \tau_0} \rho(\cdot)}\left(W_I\right)\right] , \\
-\frac{\alpha_2}{2 \noverp} & =m \sqrt{
1-\xi^2
} \mathbb{E}\left[\rho^{\prime \prime}\left(-\kappa_2 \xi Z_1-\kappa_2 \sqrt{1-\xi^2}Z_2\right) \operatorname{Prox}_{{m^{-1} \gamma \tau_0} \rho(\cdot)}\left(W_I\right)\right].
\end{aligned}\right.
\end{equation}
Similar to \eqref{nonlinear_three_equation(modify)}, the system of equations in \eqref{nonlinear_four_equation} arises as the first-order optimality conditions of a limiting scalar saddle point problem. This system admits a unique admissible solution, denoted by $(\alpha_{1*},\alpha_{2*},\sigma_{*},\gamma_{*})$,
where admissibility means $\sigma_{*}>0$ and $\gamma_{*}>0$. In practice, the solution can be computed by fixed-point iteration. The uniqueness of the admissible solution is established in Appendices D.6.7 and D.6.8.

We are now ready to make the statement in \eqref{informal_MAP_informative_linear_asym} rigorous.
\begin{theorem} \label{thm:exact_cat_M_MAP_informative}
	Consider the SRE defined in \eqref{eq: SRE_def}. Suppose \Cref{condition:proper_scaling,condition:dist_condition(modify),condition:informative_syn_data} hold.
    Suppose the parameters $\noverp, \kappa_1>0,\kappa_2>0,\tau_0>0,m>0$, and $\xi$ are such that the system of equations  \eqref{nonlinear_four_equation} has an admissible solution $(\alpha_{1*},\alpha_{2*},\sigma_{*},\gamma_{*})$.
    Assume further that $m\delta>2$ and
$\kappa_2<\overline{\kappa}_{\mathrm{MLE}}(m\delta)$,
where $\overline{\kappa}_{\mathrm{MLE}}(\cdot)$ is the logistic MLE phase-transition boundary defined in \Cref{lemma:MAP_bounded_gaussian_design}.
Then, the following statements hold:

      (1) For any fixed index set $\mathcal{S} \subset$ $\{1, \ldots, p\}$ with $\sqrt{p}\left\|\bbeta_{0, \mathcal{S}}\right\|_2=O(1)$ and $\sqrt{p}\left\|\bbeta_{s, \mathcal{S}}\right\|_2=O(1)$, we have
$$
\frac{\sqrt{p}}{\sigma_*}\left(\widehat{\bbeta}_{M, \mathcal{S}}-\alpha_{1*}\bbeta_{0,\mathcal{S}}-\frac{\alpha_{2*}}{\sqrt{1-\xi^2}}(\bbeta_{s,\mathcal{S}}-\xi \frac{\kappa_2}{\kappa_1}\bbeta_{0,\mathcal{S}})\right) \stackrel{d}{\longrightarrow} \mathcal{N}\left(\mathbf{0}, \boldsymbol{I}_{|\mathcal{S}|}\right) .
$$

 (2) Suppose $\frac{1}{p}\sum_{j=1}^p \noverpmass_{\sqrt{p}\bbeta_{0,j}} \rightsquigarrow \Pi$ for a distribution $\Pi$ with $\mathbb{E}_{\Pi}[\beta^2]=\kappa_1^2$. For any locally-Lipschitz  function
 $\Psi(a,b)$ or for the indicator function $\Psi(a,b)=\mathbf{1}\{|a/\sigma_*|\leq t\}$ with fixed $t>0$, we have
\begin{equation}\label{eq:infor_theorem_xi_neq_one}
    \frac{1}{p} \sum_{j=1}^p \Psi\left(\sqrt{p}[\widehat{\bbeta}_{M,j}-\alpha_{1*}\bbeta_{0,j}-\frac{\alpha_{2*}}{\sqrt{1-\xi^2}}(\bbeta_{s,j}-\xi \frac{\kappa_2}{\kappa_1}\bbeta_{0,j})], \sqrt{p}\bbeta_{0,j}\right) \xrightarrow{\mathbb{P}} \mathbb{E}[\Psi( \sigma_{*}Z, \beta)],
\end{equation}
where $Z\sim N(0,1)$ is independent of $\beta\sim \Pi$.
Furthermore, if we allow $\xi=1$, \eqref{eq:infor_theorem_xi_neq_one} continues to hold after replacing the left-hand side with $\frac{1}{p} \sum_{j=1}^p \Psi\left(\sqrt{p}(\widehat{\bbeta}_{M,j}-\alpha_{1*} \bbeta_{0,j}), \sqrt{p}\bbeta_{0,j}\right)$.
\end{theorem}

Our proof of \Cref{thm:exact_cat_M_MAP_informative} relies on a novel application of the Convex Gaussian Min-max Theorem (CGMT) \citep{thrampoulidis2018precise}.
While CGMT has been used for regularized M-estimators with separable regularization\footnote{A regularization function $h(\boldsymbol{b})$ is said to be separable if $h(\boldsymbol{b})=\sum_{j=1}^p h\left(b_j\right)$ for some convex function $h(\cdot)$. E.g.: $h(\boldsymbol{b})=\|\boldsymbol{b}\|_1=\sum_{i}|b_i|$ and $h(\boldsymbol{b})=\|\boldsymbol{b}\|_2^2=\sum_{i}b_i^2$ are separable regularization functions.}, the existing techniques do not apply to the non-separable regularization in \eqref{eq: SRE_def}.
To apply CGMT, it is generally necessary to reduce the optimization problem to an ancillary optimization (AO) over compact sets of variables and then analyze the optima of the AO.
Traditional analyses of CGMT  proceed through a rank-one projection of the optima on the direction of $\bbeta_0$. These analyses are applicable only to separable regularization.
For our optimization \eqref{eq: SRE_def}, the regularization is non-separable and we need to project $\bbeta$ into a space spanned by $\bbeta_0$ and $\bbeta_s$, where the traditional argument fails to work.
We overcome this challenge by employing a new strategy, which  extends the application of CGMT to non-separable regularization.

We provide a brief overview of our proof.
Suppose the Gram-Schmidt process yields two orthonormal vectors $\be_1,\be_2$ in the space spanned by $\bbeta_0$ and $\bbeta_s$. We decompose the SRE as follows:
\begin{align*}
\widehat \bbeta_M&=(\be_1^\top\widehat \bbeta_M) \be_1+(\be_2^\top\widehat \bbeta_M) \be_2+\mathbf{P}^{\perp} \widehat \bbeta,
\end{align*}
where $\mathbf{P}^{\perp}$ is the projection matrix onto the orthogonal complement of the space spanned by $\bbeta_0$ and $\bbeta_s$.
Next, we develop a novel reduction of the AO problem to track the limits of $\be_1^\top\widehat \bbeta_M$ and $\be_2^\top\widehat \bbeta_M$.
Finally, we demonstrate that $\mathbf{P}^\perp \widehat{\bbeta}_M$ will be asymptotically equal to $p^{-1/2}\sigma_* \mathbf{P}^\perp \bZ$.
We believe this novel decomposition could be of independent interest and applicable in other analyses where it is necessary to project the optimization variable into a multidimensional space.

\Cref{thm:exact_cat_M_MAP_informative}  can be viewed as an extension of
\Cref{thm:exact_cat_M_MAP_noninformative(modify)} by incorporating an
additional $\bbeta_s$-component. Formally, the noninformative synthetic-data
setting in
\Cref{thm:exact_cat_M_MAP_noninformative(modify)} corresponds to the boundary case $\boldsymbol{\beta}_s=\mathbf{0}$ and
$\kappa_2=0$, for which the alignment
parameter $\xi$ and the factor $\alpha_{2*}$ become vacuous.
After removing the
last equation related to $\alpha_{2*}$ from the system of equations in \eqref{nonlinear_four_equation} and setting
$W_I=\kappa_1\alpha_1 Z_1+\sigma Z_3$, the remaining system reduces to the system of equations in
\eqref{nonlinear_three_equation(modify)}.
Similarly, after removing the second term related to $\bbeta_s$ and $\alpha_{2*}$ from the asymptotic representation in
\Cref{informal_MAP_informative_linear_asym},
the representation formally reduces to that
in \Cref{informal_MAP_linear_asym}.

When $\kappa_2$ is nonzero,
\Cref{thm:exact_cat_M_MAP_informative} shows that the SRE admits the approximate decomposition
\begin{equation}\label{eq:exact_informative_decomp}
 \widehat{\bbeta}_M\approx \underbrace{\alpha_{1*} \bbeta_0}_{\text{signal}}+ \underbrace{\frac{\alpha_{2*}}{\sqrt{1-\xi^2}}\left(\bbeta_s-\xi\frac{\kappa_2}{\kappa_1}\bbeta_0\right)}_{\text{bias}}+\underbrace{p^{-1/2}\sigma_{*}\bZ}_{\text{noise}},
\end{equation}
where $\bZ$ is a standard normal vector.
The first term is the signal component along $\bbeta_0$, with magnitude governed by $\alpha_{1*}$.
The second term is proportional to
\begin{equation}\label{eq:ortho_source}
    \frac{\bbeta_s-\xi(\kappa_2/\kappa_1)\bbeta_0}{\sqrt{1-\xi^2}},
\end{equation}
which is the renormalized component of the source coefficient that is orthogonal to $\bbeta_0$.
Its magnitude is governed by $\alpha_{2*}$, and it represents the source-specific bias that arises from auxiliary information not aligned with the target signal.
The third term is the high-dimensional isotropic noise of order $p^{-1/2}$, with magnitude $\sigma_*$.

This decomposition clarifies the role of the similarity parameter $\xi$.
When $\xi \to 1$, the source coefficient $\bbeta_s$ becomes parallel to $\bbeta_0$.
The unnormalized orthogonal component
$\bbeta_s-\xi(\kappa_2/\kappa_1)\bbeta_0$
has asymptotic norm \(\kappa_2\sqrt{1-\xi^2}\), which vanishes as
\(\xi\to 1\).
Furthermore, the last equation in \eqref{nonlinear_four_equation} implies that $\alpha_{2*}=O(\sqrt{1-\xi^2})$, so the orthogonal bias term in \Cref{eq:exact_informative_decomp} vanishes and the approximation reduces to $\widehat{\bbeta}_M\approx \alpha_{1*} \bbeta_0+p^{-1/2}\sigma_{*}\bZ$.
When $\xi \to 0$, the source coefficient $\bbeta_s$ becomes orthogonal to $\bbeta_0$, so the auxiliary data contain no aligned information about the target signal. In this sense, $\xi=0$ corresponds to a noninformative source for estimating $\bbeta_0$, although the resulting estimator may still differ from the MLE through the bias-variance tradeoff.

Although our theory applies to any $\xi\in (-1,1)$, the case $\xi<0$ corresponds to anti-alignment, which may lead to negative transfer and is not our main focus.
The negative transfer phenomenon is illustrated numerically in \Cref{app:negative_transfer}, where the squared error for $\xi<0$ is larger than that for the benchmark case $\xi=0$.

In \Cref{sec:apply_theory_estimate_tau_MSE}, we utilize the asymptotic characterization in \Cref{thm:exact_cat_M_MAP_informative} to illustrate that when $\xi$ is sufficiently large, the SRE based on informative auxiliary data can be substantially better than the version based on noninformative synthetic data.

\begin{remark}\label{rem: literature on precise asymptotics}
In the past two decades, new theoretical frameworks have been developed to characterize the precise asymptotic behavior of MLEs and regularized estimators in the linear asymptotic regime.
These frameworks have been successfully employed in linear models \citep{bayati2011dynamics,el2013robust,thrampoulidis2015regularized,el2018impact} and binary regression models \citep{sur2019modern,salehi2019impact,taheri2020sharp,deng2022model}.
The main technical tools for these frameworks include approximate message passing (AMP) \citep{donoho2009message,bayati2011dynamics}, Convex Gaussian Min-max Theorem (CGMT) \citep{thrampoulidis2015regularized,thrampoulidis2018precise}, and the leave-one-out analysis \citep{el2013robust,el2018impact}.
In particular, our precise characterization of the SRE is based on CGMT. Although CGMT is a powerful tool for reducing the analysis of a min-max optimization to a much simpler optimization with the same optimum, the analysis of the reduced optimization is problem-specific and often challenging.
To deal with the synthetic data in the SRE, novel probabilistic analyses have to be developed.
\end{remark}

\begin{remark}\label{rem:sk-compare}
\cite{sterzinger2023diaconis} employed AMP to analyze the asymptotic behavior of the MDYPL estimator with noninformative pseudo-responses for the same observed covariates, but their analysis focused on the case where $p<n$.
We point out that AMP is not suitable for the analysis of the SRE because the AMP algorithm becomes too complex when applied to synthetic data with covariates that are different from the observed ones.
See  \Cref{proof_sec:amp_comparison} for more details.
In contrast, our novel application of CGMT can accommodate this situation.
\end{remark}

Similar to \Cref{thm:exact_cat_M_MAP_noninformative(modify)}, \Cref{thm:exact_cat_M_MAP_informative} yields several asymptotic characterizations of the relationship between the SRE and the true coefficients as discussed in \Cref{sec:exact_asym_M_finite_non_informative}. In particular, for squared error, we have
\begin{equation}\label{eq:exact_cat_M_MAP_informative_square_error}
    \|\widehat{\bbeta}_M-\bbeta_0\|_2^2\xrightarrow{\mathbb{P}}(\alpha_{1*}-1)^2\kappa_1^2+\alpha_{2*}^2\kappa_2^2+\sigma_*^2.
\end{equation}
The detailed derivation is provided in \Cref{sec:exact_limit_square_error}.
Furthermore, for the cosine similarity, we obtain
\begin{equation}\label{eq:exact_cat_M_MAP_informative_similarity}
    \frac{\langle \widehat{\bbeta}_M,\bbeta_0 \rangle}{\|\widehat{\bbeta}_M\|_2\|\bbeta_0\|_2} \xrightarrow{\mathbb{P}}\frac{\alpha_{1*}\kappa_1}{\sqrt{\alpha_{1*}^2\kappa_1^2+\alpha_{2*}^2\kappa_2^2+\sigma_*^2}}
\end{equation}

The asymptotic characterization provides a direct way to evaluate the effect of informative auxiliary data through the limiting risk.
Formally, the target-only likelihood corresponds to the boundary case as $\tau_0\to 0$.
Hence, for any fixed similarity level $\xi$, one can numerically compare the limiting risk at any $\tau_0>0$ with that of the target-only MLE.
Moreover, by minimizing the limiting risk over $\tau_0$, one can compare the best achievable limiting risk under informative auxiliary data with the corresponding limiting risk under noninformative synthetic-data regularization studied earlier in \Cref{sec:exact_asym_M_finite_non_informative}.
We illustrate these comparisons in \Cref{sec:num_verify_infor_syn}.
Although we do not obtain a closed-form condition that guarantees general improvement, the exact formulas make it possible to numerically identify favorable regimes where informative auxiliary data can lead to smaller limiting risk.

For practical inference, we develop a data-driven method for estimating $\xi$ in \Cref{sec:est_xi1}. Combined with the asymptotic comparisons above, this provides a principled way to assess the potential usefulness of the available auxiliary information in any given application.

\subsection{Numerical illustration}\label{sec:numerical_verify}

In this section, through simulation experiments, we test the finite-sample accuracy of our theoretical results on the SRE in \Cref{thm:exact_cat_M_MAP_noninformative(modify),thm:exact_cat_M_MAP_informative}.
We focus on the squared error $\|\widehat{\bbeta}_M-\bbeta_0\|_2^2$ and the cosine similarity  $\frac{\langle \widehat{\bbeta}_M,\bbeta_0 \rangle}{\|\widehat{\bbeta}_M\|_2\|\bbeta_0\|_2}$ and we compare the theoretical prediction on these quantities with the finite-sample counterparts.
Throughout this section, the synthetic sample size is set to $M=20p$ and the SRE is computed with tuning parameter $\tau=p \tau_0\noverp$ for some sequence of values for $\tau_0$.
To get the solutions from the systems of equations  \eqref{nonlinear_three_equation(modify)} and \eqref{nonlinear_four_equation} , we use the fixed-point iterative method \citep[Ch 1.2]{berinde2007iterative}.

\subsubsection{Noninformative synthetic data}\label{sec:numberical_verify_non_infor}

We consider the setting in \Cref{sec:exact_asym_M_finite_non_informative} where the SRE is constructed with noninformative synthetic data.
In the experiments, we pick different combinations of parameters $\delta$ and $\kappa_1$, and fix $p$ at $250$ so that $n$ is $250\noverp$.
The observed data $\{\bX_i,Y_i\}_{i=1}^n$ and the synthetic data $\{\bX^*_i,Y^*_i\}_{i=1}^M$ are generated under the conditions of \Cref{thm:exact_cat_M_MAP_noninformative(modify)}.
For the true coefficients $\bbeta_0$, we first generate $T_j\sim t_3$ independently for each $j\in [p]$ and  then set $\bbeta_{0j}=\frac{\kappa_1}{\sqrt{3p}}T_j$.
The limiting values of the squared error and the cosine similarity are given in \eqref{eq:exact_cat_M_MAP_noninformative_MSE} and \eqref{eq:exact_cat_M_MAP_noninformative_similarity}, respectively.

\begin{figure}[hbtp]
	\centering
	\includegraphics[scale=0.35]{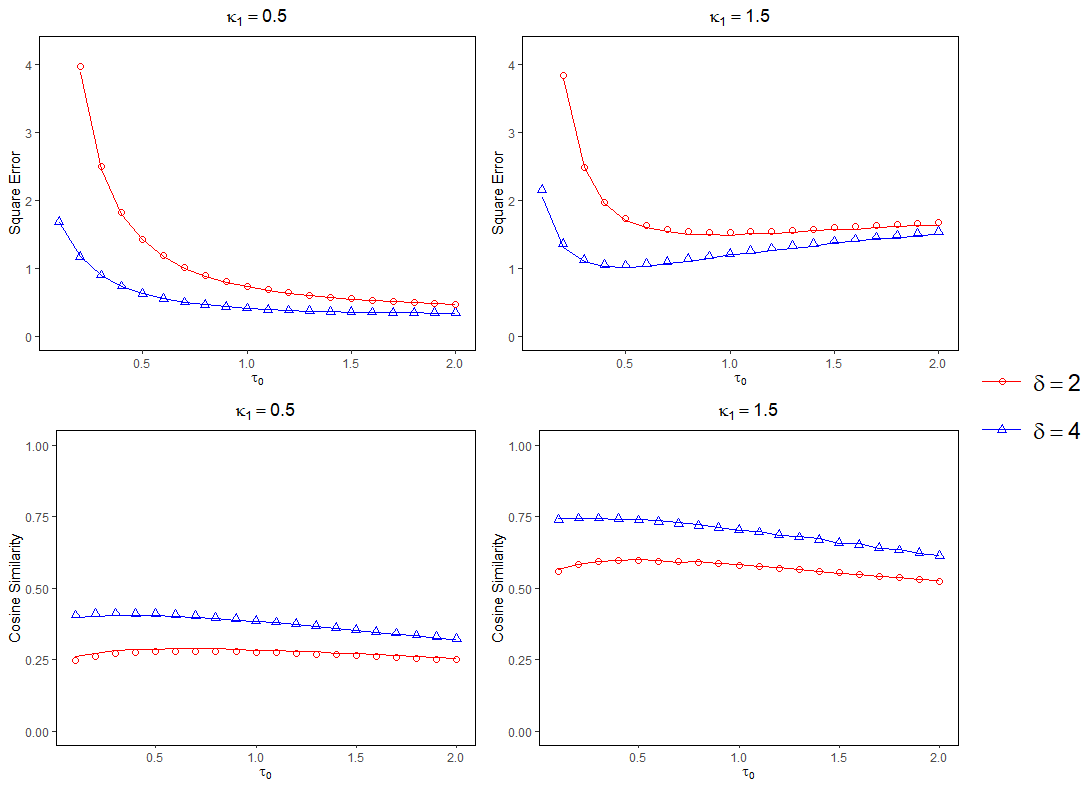}
		\caption{
 Performance of the SRE with noninformative synthetic data as a function of $\tau_0=\tau/n$.
 Each point is obtained by averaging the performance metric of the SRE over 50 simulation replications.
 The solid lines represent the corresponding theoretical prediction.
 }
		\label{fig:num_non_infor_syn_mse_cor_kappa1_half_onehalf}
\end{figure}

For $\kappa_1=0.5$ and $\kappa_1=1.5$, we plot the finite-sample averaged squared error and cosine similarity as points and we draw the limiting values as curves in \Cref{fig:num_non_infor_syn_mse_cor_kappa1_half_onehalf}, where the x-axis shows the value of $\tau_0$.
Results for $\kappa_1=1$ and $2$ are provided in \Cref{supp:sec:estimation_kappa1}.
In these plots, the points align well with the curves, which demonstrates that our asymptotic theory has desirable finite-sample accuracy.
Furthermore, the U-shaped curve of the squared error suggests that for the bias-variance tradeoff, the optimal value of $\tau$ should have the same order as the dimension $p$, which aligns with the practical suggestion in \cite{huang2022catalytic}.

\subsubsection{Informative auxiliary data}\label{sec:num_verify_infor_syn}
We consider the setting in \Cref{sec:exact_asymptocis_infor} where the auxiliary data are generated using regression coefficients $\bbeta_s$ that have nonzero cosine similarity $\xi$ with the true regression coefficients $\bbeta_0$.
In the experiments, we pick different combinations of parameters $\delta$ and $\kappa_1$, and fix $\kappa_2=1$, $\xi=0.9$, and $p=250$ so that $n=p\delta$.
The observed data and true regression coefficients $\bbeta_0$ are generated as in \Cref{sec:numberical_verify_non_infor}.
We set $\bbeta_s=\xi\frac{\kappa_2}{\kappa_1}\bbeta_0+\kappa_2\sqrt{1-\xi^2}\Tilde{\beps}$ with $\xi=0.9$, where $\tilde \beps$ is a random vector independent of $\bbeta_0$ and the entries of $\Tilde{\beps}$ are independently generated from the scaled t-distribution with 3 degrees of freedom and mean zero and variance $1/p$.
This particular choice guarantees that $\lim_{p\rightarrow\infty}\|\bbeta_s\|_2^2=\kappa_2^2$ and $\lim_{p\rightarrow\infty}\frac{1}{\|\bbeta_0\|_2\|\bbeta_s\|_2}\langle \bbeta_0,\bbeta_s\rangle=\xi$.   Then we generate informative auxiliary data as in \Cref{condition:informative_syn_data}.
The limiting values of the squared error and the cosine similarity are given in \eqref{eq:exact_cat_M_MAP_informative_square_error} and \eqref{eq:exact_cat_M_MAP_informative_similarity} respectively.

\begin{figure}[hbtp]
	\centering
	\includegraphics[scale=0.35]{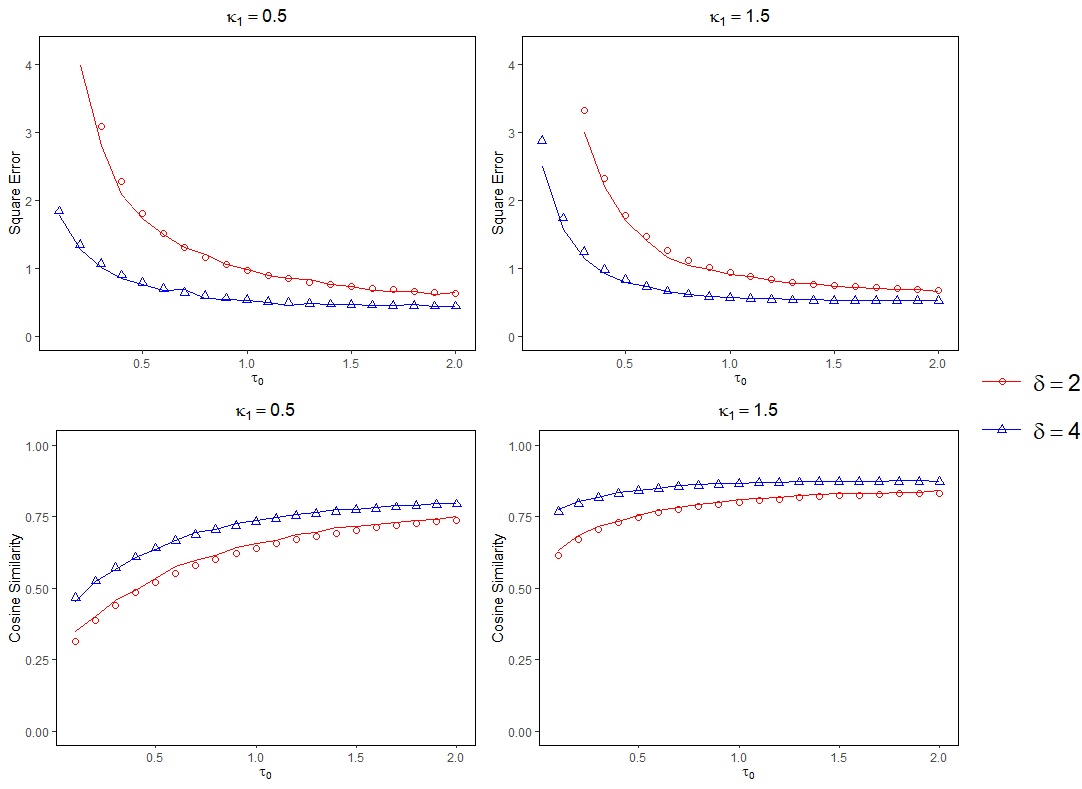}
		\caption{
 Performance of the SRE with informative auxiliary data $(\kappa_2=1,\xi=0.9)$ as a function of $\tau_0=\tau/n$.
 Each point is obtained by averaging the performance metric of the SRE over 50 simulation replications. 
 The solid lines represent the corresponding theoretical prediction.
  }
  \label{fig:num_infor_syn_mse_cor_kappa1_0.5_1.5}
\end{figure}

For $\kappa_1=0.5$ and $\kappa_1=1.5$, we plot the finite-sample averaged squared error and cosine similarity as points and we draw the limiting values as curves in \Cref{fig:num_infor_syn_mse_cor_kappa1_0.5_1.5}, where the x-axis shows the value of $\tau_0$.
Results for $\kappa_1=1$ and $2$ are provided in \Cref{supp:sec:estimation_kappa1}.
In these plots, the points align well with the curves, which demonstrates that our asymptotic theory has desirable finite-sample accuracy.

When compared with the experiments in \Cref{sec:numberical_verify_non_infor}, \Cref{fig:num_infor_syn_mse_cor_kappa1_0.5_1.5} demonstrates that incorporating additional informative auxiliary data can significantly reduce estimation errors.
For example, consider the case with parameters $(\noverp=2,\kappa_1=1.5)$.
In \Cref{fig:num_non_infor_syn_mse_cor_kappa1_half_onehalf}, the lowest MSE is approximately 1.5. In contrast, \Cref{fig:num_infor_syn_mse_cor_kappa1_0.5_1.5} shows a reduction in this value to below 1.
Similarly, we observe that the maximum cosine similarity improves from 0.6 to 0.8.
These observations indicate the effectiveness of transferring valuable information from informative auxiliary data in enhancing the estimation accuracy of the SRE.

\subsection{Brief road-map of the proof}
\label{sec:road_map_maintext}
This section outlines the high-level strategy for characterizing the asymptotically exact SRE behavior (\Cref{thm:exact_cat_M_MAP_informative}). The complete proof is provided in  \Cref{proof_sec:logitic_exact}.

\textbf{First step: Reformulation of original problem.}
To make our optimization problem more suitable for exact asymptotic analysis, we execute a series of transformations on the original optimization problem. By integrating these transformation steps, we reach an equivalent formulation known as the Primal Optimization (PO) problem:
$$
\begin{aligned}
	\min _{\bbeta_S \in \mathcal{S}_{\bbeta}, \bbeta_{S^{\perp}} \in \mathcal{S}_{\bbeta}, \mathbf{u}_1 \in \mathbb{R}^n, \mathbf{u}_2 \in \mathbb{R}^M} \max _{\bv \in \mathcal{S}_{\bv}}&\left(\frac{1}{n} \mathbf{1}^T \rho\left(\mathbf{u}_1\right)-\frac{1}{n} \mathbf{y}_1^T \mathbf{u}_1+\frac{\tau_0}{M} \mathbf{1}^T \rho\left(\mathbf{u}_2\right)-\frac{\tau_0}{M} \mathbf{y}_2^T \mathbf{u}_2\right.\\
	&\left. +\frac{1}{\sqrt{n}} \bv^T\left(\left[\begin{array}{l}
\mathbf{u}_1 \\
\mathbf{u}_2
\end{array}\right]-  \mathbf{H} \bbeta_S\right)-\frac{1}{\sqrt{n}}   \bv^T \mathbf{H} \bbeta_{S^{\perp}}\right)
\end{aligned}
$$
where $\mathbf{H}$ is a matrix with entries that are i.i.d. standard normal,  $\bbeta_{S}:=\mathbf{P}\bbeta$ and $\bbeta_{S^{\perp}}:=\mathbf{P}^\perp \bbeta$, where $\mathbf{P}$ is the projection matrix onto the column space spanned by $\bbeta_0$ and $\bbeta_s$ and $\mathbf{P}^\perp$ is the projection onto the orthogonal complement of that space.

\smallskip
\textbf{Second step: Reduction  to an Auxiliary Optimization (AO) problem.}
The particular form of PO allows us to use the Convex Gaussian Min-max Theorem \citep{thrampoulidis2015regularized}, which characterizes the exact asymptotic behavior of min-max optimization problems that are affine in Gaussian matrices.
This result enables us to characterize the properties of $\widehat{\bbeta}_M$ by studying the asymptotic behavior of the following, arguably simpler, Auxiliary Optimization (AO) problem:
$$
\begin{aligned}
	\min _{\bbeta_S \in \mathcal{S}_{\bbeta}, \bbeta_{S^{\perp}} \in \mathcal{S}_{\bbeta}, \mathbf{u}_1 \in \mathbb{R}^n, \mathbf{u}_2 \in \mathbb{R}^M} \max _{\bv \in \mathcal{S}_{\bv}}&\left(\frac{1}{n} \mathbf{1}^T \rho\left(\mathbf{u}_1\right)-\frac{1}{n} \mathbf{y}_1^T \mathbf{u}_1+\frac{\tau_0}{M} \mathbf{1}^T \rho\left(\mathbf{u}_2\right)-\frac{\tau_0}{M} \mathbf{y}_2^T \mathbf{u}_2\right.\\
	&\left. +\frac{1}{\sqrt{n}} \bv^T\left(\left[\begin{array}{l}
\mathbf{u}_1 \\
\mathbf{u}_2
\end{array}\right]- \mathbf{H} \bbeta_S\right) -\frac{1}{ \sqrt{n}}\left(\bv^T \mathbf{h}\left\|\mathbf{P}^{\perp} \bbeta\right\|+\|\bv\| \mathbf{g}^T \mathbf{P}^{\perp} \bbeta\right)\right)
\end{aligned}
$$
where $\mathbf{h} \in \mathbb{R}^{n+M}$ and $\mathbf{g} \in \mathbb{R}^p$ have i.i.d. standard normal entries.

\smallskip
\textbf{Third step: Scalarization of the Auxiliary Optimization problem.}
We further simplify AO to an optimization over some scalar variables. Specifically, we demonstrate that the asymptotic behavior of AO can be captured through the following optimization problem:
$$\begin{aligned}
\min _{\substack{ \alpha_1\in \mathbb{R} , \alpha_2\in \mathbb{R}\\v,\sigma>0}}\max _{r>0} & \quad \left( -\frac{r\sigma}{\sqrt{\noverp}}+\frac{r}{2v} -\frac{1}{4rv}-\kappa_1^2\alpha_1\mathbb{E}(\rho^{\prime\prime}(\kappa_1 Z_1))-\frac{\tau_0^2}{4rvm}      \right.\\
 & -\tau_0 \kappa_2 \mathbb{E}(\rho^{\prime\prime}(\kappa_2 \xi Z_1+\kappa_2 \sqrt{1-\xi^2}Z_2))(\alpha_1\kappa_1\xi+\alpha_2\kappa_2\sqrt{1-\xi^2} )\\
 &+\mathbb{E}(M_{\rho(\cdot)}(\kappa_1\alpha_1 Z_1+\kappa_2\alpha_2 Z_2+\sigma Z_3+\frac{1}{rv}Ber(\rho^{\prime}(\kappa_1 Z_1)),\frac{1}{rv}))\\
 +\tau_0\mathbb{E}&(M_{\rho(\cdot)}(\kappa_1\alpha_1 Z_1+\kappa_2\alpha_2 Z_2+\sigma Z_3+\frac{\tau_0}{rvm}Ber(\rho^{\prime}(\kappa_2\xi Z_1+\kappa_2 \sqrt{1-\xi^2}Z_2)),\frac{\tau_0}{rvm}))
 \end{aligned}$$
By checking the first-order optimality conditions of the above scalar optimization, we can derive the system of equations \eqref{nonlinear_four_equation}.

\section{Statistical inference based on precise asymptotics}\label{sec:adjust_inference}

This section develops practical methods for estimation and inference based on the asymptotic theory in \Cref{sec:linear_asymptotic_regime}.
\Cref{sec:estimation_kappa1} focuses on the noninformative case ($\bbeta_s = \mathbf{0}$) and estimates the signal strength $\kappa_1$, which yields plug-in estimates of $\alpha_*$ and $\sigma_*$ for use in adjusted confidence intervals.
\Cref{sec:est_xi1} considers the general case with informative auxiliary data and develops estimation for the similarity parameter $\xi$.
\Cref{sec:apply_theory_estimate_tau_MSE} studies selection of the tuning parameter and illustrates the performance of the SRE through a real data example.
\Cref{sec:apply_theory_variable_selection_FDR} considers variable selection based on the SRE and demonstrates its performance on another real data example.

\subsection{Estimation of signal strength}\label{sec:estimation_kappa1}

The precise asymptotic characterization in \Cref{thm:exact_cat_M_MAP_noninformative(modify)} depends on the unknown signal strength $\kappa_1$.
\cite{sur2019modern} proposed a method for estimating $\kappa_1$ called \textit{ProbeFrontier} based on an asymptotic theory of the existence of the MLE, but their method only works when $p/n<1/2$. Our method introduced below works for any value of $p/n>0$.

\begin{figure}[!t]
\caption{
  Relationship between $\eta^2_{M}$ and $\kappa_1$ across different values of $\noverp$. For each $\noverp$, $\eta^2_M$ is computed using a grid of $\kappa_1$ values, with $\tau_0 = 1/4$ and $m = 20/\noverp$.}
	\centering
 	\includegraphics[scale=0.35]
 {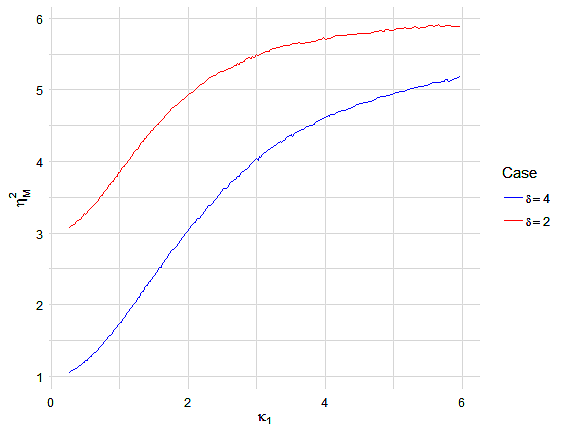}

	\label{fig:dictionary_M20}
\end{figure}

Our method is based on the precise limit of the SRE.
For any given $(\noverp, \tau_0, m)$ and any $\kappa_1$, let $\alpha_*(\kappa_1)$ and $\sigma_*(\kappa_1)$ be  the solutions of \eqref{nonlinear_three_equation(modify)}.
Intuitively, if the norm of the true coefficients (i.e., signal strength $\kappa_1$) increases,
the norm of the SRE increases accordingly. This is in light of the result proved in \cite{candes2020phase} that a large $\kappa_1$ makes the norm of the MLE unbounded.
This intuition can be justified by plotting the limiting value of $\|\widehat{\bbeta}_M\|_2^2$ with respect to $\kappa_1$.
\Cref{thm:exact_cat_M_MAP_noninformative(modify)} suggests that the squared norm of the SRE converges to $\eta_{M}^2:= \alpha_*^2(\kappa_1)\kappa_1^2+\sigma_*^2(\kappa_1)$.
We illustrate the relationship between $\eta_{M}^2$ and $\kappa_1$ in \Cref{fig:dictionary_M20}, which suggests that $\eta_{M}^2$ is increasing in $\kappa_1$.
We denote this relationship as $\eta_{M}=g_{\noverp}(\kappa_1)$, where we omit the dependence on  $\tau_0$ and $m$ because the values of $\tau_0$ and $m$ are manipulable and can be pre-chosen.
Although it could be challenging to estimate $\kappa_1$ directly, it is straightforward to estimate $\eta_{M}$ by $\widehat{\eta}_{M}:=\|\widehat{\bbeta}_{M}\|_2$, the norm of the SRE with noninformative synthetic data of size $M=m n$ and with total weight parameter $\tau=\tau_0 n$.
Given $\widehat{\eta}_{M}$, we set $\widehat{\kappa}_1$ to be the smallest minimizer of $|g_{\noverp}(\kappa) - \widehat{\eta}_{M}|$ over a bounded interval chosen to cover the plausible range of signal strengths; in practice, this minimization is carried out numerically over a grid on that interval.
Given the value of $\widehat{\kappa}_1$, the corresponding solution to the system of equations  \eqref{nonlinear_three_equation(modify)} will be denoted by $\left(\widehat{\alpha}_*, \widehat{\sigma}_*,\widehat{\gamma}_*\right)$.
The accuracy of our estimation method for parameters $(\kappa_1,\alpha_*,\sigma_*)$ is empirically demonstrated in \Cref{supp:sec:estimation_kappa1}.

Substituting the unknown parameters in \eqref{eq:exact_cat_M_MAP_noninformative_CI} with these estimates, we construct the following $95\%$ adjusted confidence intervals (ACI):
$
\widehat{\mathrm{CI}}_j=\left[\frac{\widehat{\bbeta}_{M,j}-1.96 \widehat{\sigma}_*/\sqrt{p}}{\widehat{\alpha}_*}, \frac{\widehat{\bbeta}_{M,j}+1.96 \widehat{\sigma}_*/\sqrt{p}}{\widehat{\alpha}_*}\right], j\in [p].
$
We investigate the performance of the ACI  in \Cref{table:adjust_CI_delta2}  with $\delta=2$, and the cases with $\delta=4$ are provided in  \Cref{supp:extra_CI}.
For $\delta=2$, the MLE does not exist, and consequently, methods relying on the MLE, such as classical asymptotic confidence intervals and adjusted confidence intervals, cannot be applied.
In contrast, our adjusted confidence intervals achieve desirable average coverage for the true regression coefficients.

\begin{table}[ht]
\centering
\caption{ Coverage rates of 95\% adjusted confidence intervals based on $\widehat{\bbeta}_M$ with $\noverp=2$ (MLE does not exist) in \Cref{sec:estimation_kappa1}. The results are averaged over 50 independent experiments.}
\begin{tabular}{ |c|c|c|c|c|}
\hline
  $p$ & $\kappa_1=0.5$ & $\kappa_1=1$ & $\kappa_1=1.5$ & $\kappa_1=2$ \\ \hline
  100        & 0.947                 & 0.948               & 0.948                  & 0.942                \\ \hline
 400        & 0.948                 & 0.950               & 0.946                  & 0.946                \\ \hline
\end{tabular}

\label{table:adjust_CI_delta2}
\end{table}

For the general covariance setting considered in \Cref{coro:arbitray_cov_exact_cat_M_MAP_noninformative},  \Cref{supp:est_kappa_general_covariance} provides an estimation method for the signal strength.

\subsection{Estimation of similarity}\label{sec:est_xi1}

In addition to signal strength parameters,
the precise asymptotic characterization in \eqref{eq:infor_theorem_xi_neq_one}  also depends on the unknown similarity $\xi$.
In the following, we introduce an estimation method for the similarity between the underlying regression coefficients for two datasets.
Suppose we have two independent datasets: target dataset $\{\bX_{i0},Y_{i0}\}_{i=1}^{n_0}$ and source dataset $\{\bX_{is},Y_{is}\}_{i=1}^{n_s}$, both satisfy \Cref{condition:dist_condition(modify)} with true regression coefficients $\bbeta_0$ and $\bbeta_s$ respectively. Furthermore, we assume $\|\bbeta_0\|_2=\kappa_1$, $\|\bbeta_s\|_2=\kappa_2$, and $\frac{1}{\|\bbeta_0\|\|\bbeta_s\|}\langle \bbeta_0,\bbeta_s\rangle=\xi$.
For each original dataset, we generate an independent noninformative synthetic dataset of size $M$ and then construct the SRE separately.
For simplicity, we choose the tuning parameter $\tau=n_0\tau_0$ (or $\tau=n_s \tau_0$) for a fixed $\tau_0$.
The resultant estimators are denoted by $\widehat{\bbeta}_{M,0}$ for the target dataset and $\widehat{\bbeta}_{M,s}$ for the source dataset.
\Cref{thm:exact_cat_M_MAP_noninformative(modify)} suggests that
\begin{equation*}
    \label{estimation_xi_rational}
    \begin{aligned}
    &\widehat{\bbeta}_{M,0}\approx \alpha_{*1}\bbeta_0+p^{-1/2}\sigma_{*1}\bZ_1,
	&\widehat{\bbeta}_{M,s}\approx \alpha_{*2}\bbeta_s+p^{-1/2}\sigma_{*2}\bZ_2,
    \end{aligned}
\end{equation*}
where $(\alpha_{*1} ,\sigma_{*1})$ are the solutions to the system \eqref{nonlinear_three_equation(modify)} based on the parameter tuple $(\noverp_0=n_0/p,\kappa_1,\tau_0,M/n_0)$, $(\alpha_{*2} ,\sigma_{*2})$ are based on the parameter tuple $(\noverp_s=n_s/p,\kappa_2,\tau_0,M/n_s)$, and the entries of $\bZ_1$ and $\bZ_2$ are independent $N(0,1)$ random variables.
Based on this relationship, we have
$ \langle \widehat{\bbeta}_{M,0},\widehat{\bbeta}_{M,s}\rangle\approx  \alpha_{*1} \alpha_{*2} \cdot  \langle \bbeta_0,\bbeta_s \rangle \approx \alpha_{*1} \alpha_{*2} \kappa_1 \kappa_2 \xi$.
This leads to the following estimator for $\xi$: $$\widehat{\xi}=\frac{\langle \widehat{\bbeta}_{M,0},\widehat{\bbeta}_{M,s}\rangle}{ \alpha_{*1} \alpha_{*2} \kappa_1 \kappa_2}, $$
where $\kappa_1$ and $\kappa_2$ can be estimated by the method introduced in \Cref{sec:estimation_kappa1} if unknown.
\Cref{supp:numerical_est_cosine_similar} provides a numerical study to illustrate the accuracy of our estimation of $\xi$.
Once $\xi$, $\kappa_1$, and $\kappa_2$ are estimated, we can use \eqref{eq:infor_theorem_xi_neq_one} to perform downstream inference, such as constructing ACIs using the SRE with informative auxiliary data.

Since $\widehat{\xi}$ is constructed from two SREs computed on separate datasets, it inherits uncertainty from both. The noise in each SRE is governed by its own sample-size-to-dimension ratio: a smaller $\delta_0$ or $\delta_s$ yields a noisier estimator, and the accuracy of $\widehat{\xi}$ is ultimately limited by the less favorable of the two ratios.
The numerical study in \Cref{supp:numerical_est_cosine_similar} confirms that when $\delta_0<\delta_s$, increasing $\delta_0$ reduces the estimation error of $\widehat{\xi}$ more effectively than increasing $\delta_s$.
This pattern is consistent with the common phenomenon in transfer learning that the potential efficiency gain is limited by the smaller sample size.

\subsection{Selection of tuning parameter}\label{sec:apply_theory_estimate_tau_MSE}

The tuning parameter $\tau$ controls the bias-variance tradeoff for the SRE.
This section discusses several methods for selecting the value of $\tau$ and compares the performance of the resulting estimators.

A widely used strategy for selecting $\tau$ is cross-validation, which requires data-splitting and recomputing the estimator on subsets of data \citep[Section 7.10]{hastie2009elements}.
Here, we describe leave-one-out cross-validation and propose an efficient approximation.
The validation error (VE) is measured using the deviance as follows:
 \begin{equation*}
     \text{VE}(\tau) = -\sum_{i=1}^n\left\{Y_i\bX_i^\top \widehat{\bbeta}_{M,-i} -\rho(\bX_i^\top \widehat{\bbeta}_{M,-i}) \right\},
 \end{equation*}
where $\widehat{\bbeta}_{M,-i} $ denotes the SRE in \eqref{eq: SRE_def} computed using all observed data except the $i$-th observation.
Since computing all $\widehat{\bbeta}_{M,-i} $ is computationally intensive,
it is beneficial to only compute $\widehat{\bbeta}_M$ once (for each value of $\tau$).
Motivated by the leave-one-out estimators in \cite{sur2019modern}, we propose an accurate approximation to $\text{VE}(\tau)$.
To be concrete, let  $\mathcal{I}_{-i}=[n]\setminus\{i\}$, and approximate  $\bX_i^\top \widehat{\bbeta}_{M,-i}$ by
$$
\tilde l_i:=\bX_i^\top \widehat{\bbeta}_{M}+\bX_i^{\top}\left( H_{\tau} + \rho^{\prime\prime}\left(\widehat{\bbeta}_M^{\top} \bX_i\right)\bX_i\bX_i^{\top}    \right)^{-1} \bX_i\left(Y_i-\rho^{\prime}\left(\bX_i^\top \widehat{\bbeta}_{M} \right)\right),
$$
where $H_{\tau}$ is the Hessian matrix of the objective in \eqref{eq: SRE_def},
i.e.,
$H_{\tau}=-\sum_{j \in [n]} \rho^{\prime\prime}\left(\widehat{\bbeta}_M^{\top} \bX_j\right)\bX_j\bX_j^{\top}- \frac{\tau}{M}\sum_{j\in [M]}\rho^{\prime\prime}\left(\widehat{\bbeta}_M^{\top} \bX_j^*\right)\bX_j^*\bX_j^{*\top}$.
The matrix inversion in the above display can be computed efficiently using the Sherman-Morrison inverse formula \citep{sherman1950adjustment}.
Subsequently, we approximate $\text{VE}(\tau)$ by $\widetilde{\text{VE}}(\tau):= -\sum_{i=1}^n\left\{Y_i\tilde l_i-\rho(\tilde l_i) \right\}$.
In \Cref{appendix_sec:LOOCV}, we provide a detailed derivation and summarize the algorithm for selecting $\tau$ by minimizing $\widetilde{VE}(\tau)$.
The SRE resulting from this selection of $\tau$  is named the \textbf{S}RE with \textbf{L}eave-one-out \textbf{C}ross  \textbf{V}alidation (SLCV).

Another way to select $\tau$ is to minimize the theoretical limit of the squared error given by \Cref{thm:exact_cat_M_MAP_noninformative(modify)}.
Using the estimator $\widehat{\kappa}_1$ from \Cref{sec:estimation_kappa1}, we compute the corresponding solutions to  \eqref{nonlinear_three_equation(modify)} for any $\tau_0 = \tau / n$, denoted as $\left(\widehat{\alpha}_*(\tau), \widehat{\sigma}_*(\tau), \widehat{\gamma}_*(\tau)\right)$.
We can then estimate the limit of the squared error by \eqref{eq:exact_cat_M_MAP_noninformative_MSE} for a fixed grid of values of $\tau$ and select the one that minimizes the estimated limit.
The SRE resulting from this selection of $\tau$ is named the \textbf{S}RE with \textbf{E}stimated \textbf{S}quared  \textbf{E}rror (SESE).
For comparison, we also consider the optimal $\tau$ that minimizes the limit of the squared error based on the true value of $\kappa_1$, and call the resulting estimator \textbf{S}RE with \textbf{T}rue \textbf{S}quared \textbf{E}rror (STSE).

We provide numerical experiments to highlight the effectiveness of our proposed tuning parameter selection methods in \Cref{supp:sec:apply_theory_estimate_tau_MSE}.
The results show that both SESE and SLCV perform comparably to the benchmark STSE.
Furthermore, these methods, when applied with informative auxiliary data, demonstrate a significant improvement in estimation accuracy compared to using noninformative synthetic data.

\textbf{Real data illustration.}
To illustrate the practical benefits of our methods, particularly the advantage of informative auxiliary data and our tuning parameter selection, we analyze the Wisconsin Diagnostic Breast Cancer dataset \citep{street1993nuclear}, which consists of $n=569$ observations. The response variable is binary, indicating whether a tumor is malignant or benign, and we have $p=10$ standardized covariates measuring various tumor characteristics.

We simulate a transfer learning scenario by partitioning the data randomly into three subsets: target training set ($n_{train}=50$), target test set ($n_{test}=119$), and source set ($n_s=400$).
We compare the classification performance of different methods on the test set with predicted label $\hat{Y}=1\{\bX_{test}^\top \widehat{\bbeta}>0\}$ for any estimator $\widehat{\bbeta}$. We consider two SREs: (1) \textbf{SRE(I)}, which uses the source set as informative auxiliary data; and (2) \textbf{SRE(N)}, which uses   synthetic data  $\{\bX_i\}_{i=1}^M\stackrel{i.i.d.}{\sim} N(\mathbf{0},\mathbf{I}_p), \{Y_i^*\}_{i=1}^M\stackrel{i.i.d.}{\sim}$ \text{Bern}(0.5) with $M=n_s$.
Since the tuning procedure of SLCV is computationally much more efficient and SLCV has comparable numerical performance to the SESE, we tune both SRE(I) and SRE(N) in the same way as SLCV.
For comparison, we consider $\ell_2$-penalized MLE (\textbf{ridge}) and $\ell_1$-penalized MLE (\textbf{Lasso}) (note that the MLE does not exist for the target training set).
We employ TransGLM \citep{tian2023transfer} as the benchmark for incorporating source data, which is implemented using  \texttt{glmtrans}.
Lasso and ridge estimators are implemented using the $\mathrm{R}$ package \texttt{glmnet} \citep{simon_regularization_2011}.
Tuning parameters for these methods are selected using their respective default cross-validation procedures.

\begin{table}[ht]
\caption{\label{realdata:classification_error}
Average classification error over 50 random splits 
of the Wisconsin Diagnostic Breast Cancer dataset. Standard errors are shown in parentheses.}
    \centering
  \begin{tabular}{l|ccccc}
\hline
& SRE(N) & SRE(I) & Lasso & Ridge & TransGLM\\
\hline \hline Error & $0.078(0.003)$ & $\textbf{0.069}(0.003)$ & $0.084(0.004)$ & $0.074(0.003)$  &$0.081(0.004)$\\ \hline
\end{tabular}
\end{table}

\Cref{realdata:classification_error} summarizes the classification errors of different methods. Our method, SRE(I), achieves the lowest error among the five methods, which demonstrates the benefit of incorporating informative auxiliary data in the SRE.
SRE(N) also performs competitively, showing the regularization effect of synthetic data even without specific prior information.
TransGLM and Lasso exhibit higher errors, potentially due to their use of the
$\ell_1$ penalty, which may not be suitable for this dataset.

\subsection{Variable selection}\label{sec:apply_theory_variable_selection_FDR}

Our precise asymptotic characterization of the SRE can be applied to variable selection with  False Discovery Rate (FDR) control
using the data-splitting method introduced by \cite{dai2023scale}.
While the original method requires the MLE to exist on split datasets, our extension lifts this restriction and applies more broadly.

The index set of null (irrelevant) variables is denoted by $S_0$ and the index set of relevant variables by $S_1$; for logistic regression, $S_0=\{j\in [p]: \bbeta_{0,j} = 0\}$ and $S_1=[p]\setminus S_0$.
Let $\widehat{S}$ be the index set of selected variables.
The False Discovery Proportion (FDP) and False Discovery Rate (FDR) are defined as
$$\text{FDP}=\frac{\#( S_0\cap \widehat{S}) }{\# \widehat{S} },\quad \text{FDR}=\mathbb{E} [\text{FDP}].$$
\cite{dai2023scale} considered a variable selection framework based on mirror statistics $M_j$'s that are constructed for all $j\in [p]$.
A mirror statistic exhibits two key features: (1) large values indicate potentially important variables, and (2) it is symmetrically distributed around zero for null variables. Thus, variables can be ranked by the magnitude of their mirror statistics, and those exceeding a chosen cutoff are selected.
The second property suggests an estimated upper bound for FDP for each $t$, which is given by $\frac{\#\left\{j: M_j<-t\right\}}{\#\left\{j: M_j>t\right\}}$. Following these two intuitions, the cutoff with a preassigned FDR level $q \in(0,1)$ is given by
$\text{Cutoff}(q,\{M_j\}_{j=1}^p):=\inf \left\{t>0:   {\#\left\{j: M_j<-t\right\}}/{\#\left\{j: M_j>t\right\}} \leq q\right\}$,
and we select variables with mirror statistics greater than the above cutoff value.

Next, we construct the mirror statistic that satisfies the two properties mentioned above. According to \Cref{coro:arbitray_cov_exact_cat_M_MAP_noninformative}, for each $j$ we have ${v_j}\widehat{\bbeta}_{M,j} \approx {v_j}\alpha_* \bbeta_{0,j} + {\sigma_*}Z_j$, where $Z_j\sim N(0,1/p)$ and $v^2_j=\operatorname{Var}\left(X_j \mid \bX_{-j}\right)$ is the conditional variance.
Adapting the data-splitting method in \cite{dai2023scale}, we split the observed data into two equal-sized halves, and compute the SRE for each half with separately generated synthetic data.
This leads to
\begin{equation}\label{ds_split_logic}
    {v_j}\widehat{\bbeta}^{(1)}_{M,j} \approx {v_j}\alpha_* \bbeta_{0,j} + {\sigma_*}Z^{(1)}_j \quad \text{and} \quad {v_j}\widehat{\bbeta}^{(2)}_{M,j} \approx {v_j}\alpha_* \bbeta_{0,j} + {\sigma_*}Z^{(2)}_j,
\end{equation}
where $(\widehat{\bbeta}^{(1)}_{M,j},Z^{(1)}_j)$ is independent of  $(\widehat{\bbeta}^{(2)}_{M,j},Z^{(2)}_j)$ due to data splitting.
\eqref{ds_split_logic} enables us to define  the mirror statistic as $M_j:=v_j^2\widehat{\bbeta}^{(1)}_{M,j}\widehat{\bbeta}^{(2)}_{M,j}$, which will be large in magnitude when $\bbeta_{0,j}\neq 0$ and its distribution will be symmetric around 0 when $\bbeta_{0,j}=0$.
When $v_j^2$'s are unknown, we estimate them using either node-wise regression or the diagonal entries of the inverse of the sample covariance matrix $\widehat{\mathbf{\Sigma}}=\frac{1}{n}\sum_{i=1}^n \bX_i\bX_i^\top$.
To overcome the power loss due to data splitting,
\cite{dai2023scale} introduced the Multiple Data-Splitting (MDS) procedure that aggregates multiple selection results via repeated sample splits; see Algorithm 2 therein.

In addition to variable selection via mirror statistics, we can consider the adjusted Benjamini-Hochberg (ABH) procedure and the adjusted Benjamini-Yekutieli (ABY) procedure \citep{benjamini1995controlling,benjamini2001control}.
Both procedures rely on the adjusted p-values, which are given by $2\Phi(-|\widehat{v}_j\sqrt{p}\widehat{\bbeta}_{M,j}/\widehat{\sigma}_*|)$ for $j\in [p]$, where $\Phi(\cdot)$ is the cumulative distribution function of standard Gaussian, $\widehat{v}^2_j$ is an estimate of the conditional variance $\operatorname{Var}\left(X_j \mid \bX_{-j}\right)$, and $\widehat{\sigma}_*$ is an estimate of  $\sigma_*$ defined in \Cref{coro:arbitray_cov_exact_cat_M_MAP_noninformative}.
Here $\widehat{\sigma}_*$ can be obtained by solving \Cref{nonlinear_three_equation(modify)} with $\kappa_1$ estimated using the method introduced in  \Cref{supp:est_kappa_general_covariance}.

We conduct numerical experiments across different settings to compare the performance of the aforementioned variable selection methods based on SREs in terms of FDR and power.
See the caption of \Cref{fig:FDR_extra_noMLE} for details of the experiments.
In each simulation, we numerically verified that the MLE does not exist, so MLE-based methods are inapplicable in all these experiments.
We have the following observations from \Cref{fig:FDR_extra_noMLE}.
When the signal strength is fixed and the correlation $r$ of the covariate matrix is varied,
the MDS procedure based on the SRE effectively controls the FDR when $r\leq 0.2$, but it suffers from inflation of FDR when $r\geq 0.3$.
This is probably due to the difficulty of estimating $v_j$'s in the presence of high correlations.
In addition, ABH is more powerful than MDS in every case, although it lacks theoretical guarantees on FDR control. In contrast, ABY comes with a theoretical guarantee, but it is too conservative and has the lowest power across all settings.
When $r$ is fixed at 0.2 while the signal strength is increasing, all three methods have decreasing FDR and increasing power since it becomes easier to distinguish the relevant variables from the null ones.

To compare with the variable selection methods based on the MLE, we also reproduce the numerical experiments in \citet[Section 5.1.1]{dai2023scale} where the MLE exists.
The results are presented in \Cref{supp:FDR_numerical} and they reveal that the selection methods based on SREs perform similarly to the MLE-based methods.

\begin{figure}[!t]
 \caption{
    Empirical FDRs and powers in a logistic regression with $p=200$ and $n=500$.
    The covariate vectors are sampled from a normal distribution $N(0, \Sigma)$, where $\Sigma$ is a Toeplitz matrix ($\Sigma_{ij}=r^{|i-j|}$).
    The left panel varies correlation ($r$) while fixing signal strength at $\left|\bbeta_{0j}\right|=1$ for elements in $S_1$; the right panel fixes $r=0.2$ and varies signal strength from $1$ to $2$.
    In each scenario, there are 40 relevant features.
    The nominal FDR level is $q=0.1$.
    The power is assessed as the proportion of correctly identified relevant features.
    Each point averages over 100 replications. The SRE is computed using noninformative synthetic data with $M=20p$ and $\tau=p$.}
    \centering
    \includegraphics[scale=0.2]{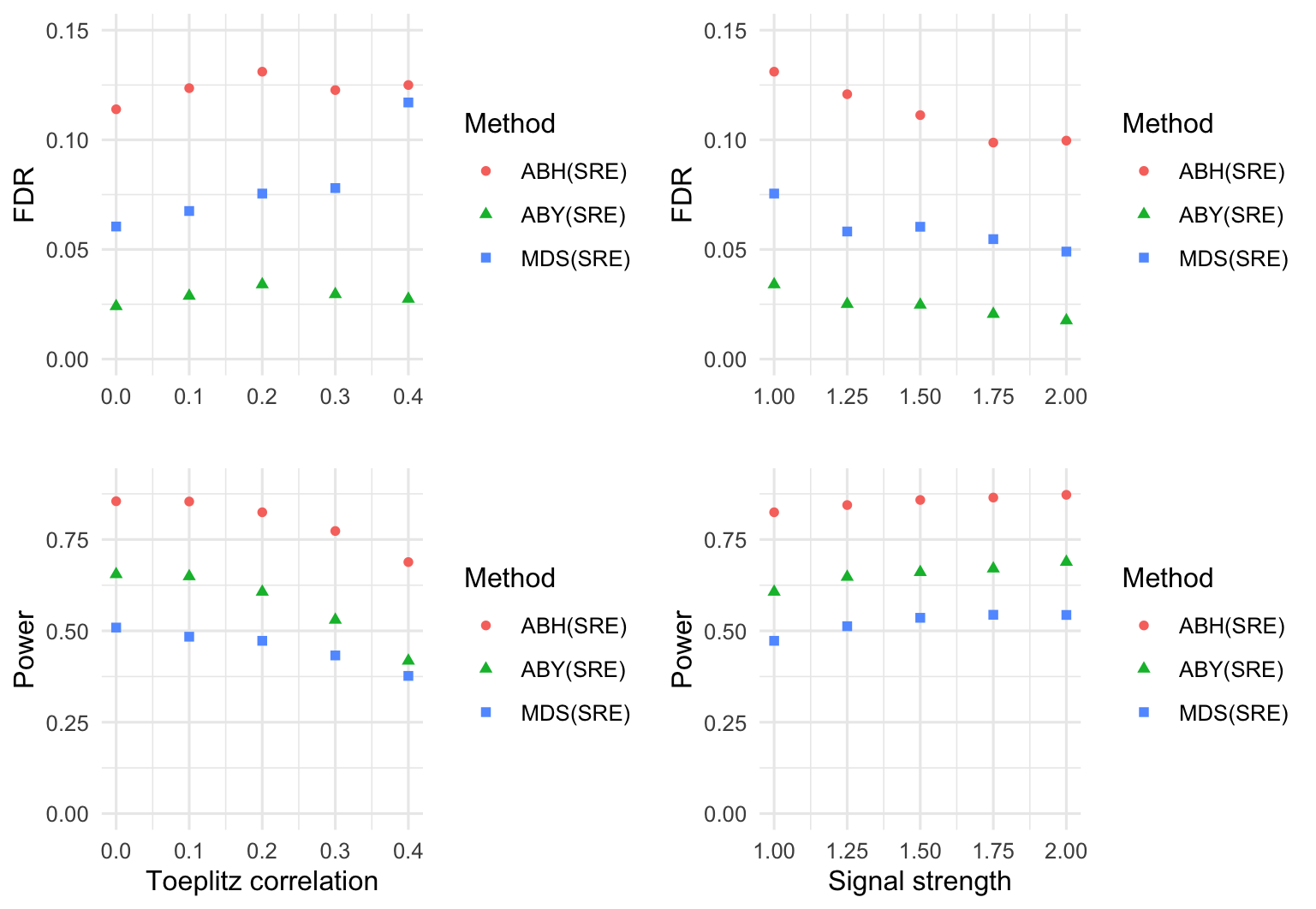}
   
    \label{fig:FDR_extra_noMLE}
\end{figure}

\textbf{Real data illustration.}
To demonstrate the performance of our SRE-based variable selection methods (MDS, ABY, and ABH), we analyze a single-cell RNA sequencing (scRNA-seq) dataset from \cite{hoffman2020single}, where the goal is to identify genes associated with the glucocorticoid response (GR).
This dataset consists of $n=2400$ gene expression samples from 2000 glucocorticoid-treated cells and 400 vehicle-treated control cells, with the binary response indicating glucocorticoid treatment.
We focus on the top 600 most variable genes ($p=600$) after centering and filtering out uncommon genes, following the procedure described in \cite{dai2023scale}.

Note that the separability of this dataset renders MLE-based methods inapplicable.
Using our SRE with noninformative synthetic data (generated as described previously), we apply the MDS, ABY, and ABH variable selection procedures.
All three methods identify HSPA1A and NFKBIA, while ABH selects two additional genes, namely EEF1A1 and RPL10.
These findings are supported by existing literature; see \Cref{supp:realdata}.

This example demonstrates the usefulness of our SRE-based variable selection methods, even in a setting where standard MLE approaches are inapplicable.

\section{Discussion}
\label{sec:discussion}

This paper introduces the synthetic-data regularization method inspired by Bayesian catalytic priors.
Rather than penalizing the parameter directly, the SRE augments the observed likelihood with a down-weighted likelihood based on synthetic data generated from a simpler model or auxiliary data sampled from a related population.
This method is easy to implement and equivariant under reparametrization.
With both theoretical and practical developments, we establish the synthetic-data regularization as a rigorous framework for improved frequentist inference.

Theoretically, we analyze the properties of the SRE in GLMs where no structural assumptions are imposed.
We prove that the SRE achieves minimax rate optimality across the asymptotic regimes covered by our theory, and we provide a precise asymptotic characterization in the high-dimensional linear regime.
These results clarify the roles of the tuning parameters in determining the behavior of the SRE.

Practically, building upon our theoretical results, we develop novel methodologies for implementation and inference.
Specifically, we propose methods for estimating the bias and variance parameters that govern the SRE's asymptotic behavior, which are then used to adjust high-dimensional inference.
The effectiveness of these adjusted inferences is demonstrated through simulations and real-data examples.

Several research questions remain open:
(1)
The synthetic-data regularization construction is broadly applicable, while the theory developed here focuses on GLMs.
Extending this theory to other model classes is an important direction for future work. 
(2) Our estimation method for signal strength is empirically accurate, but a theoretical guarantee requires further investigation.
(3) While the precise asymptotic characterization assumes Gaussian designs, numerical results in \Cref{sec:beyond_gaussian_empirical_justification} suggest this characterization holds under weaker conditions.
Proving this universality is promising but technically challenging.

\bigskip

\section{Acknowledgements}
D. Huang was partially supported by NUS Start-up Grant A-0004824-00-0 and by Singapore Ministry of Education AcRF Tier 1 Grants A-8000466-00-00 and A-8004149-00-00.

\newpage
 
\numberwithin{equation}{section}
\numberwithin{theorem}{section}
\numberwithin{lemma}{section}
\numberwithin{proposition}{section}
\numberwithin{corollary}{section}
\numberwithin{assumption}{section}
\numberwithin{condition}{section}
\numberwithin{definition}{section}
\numberwithin{remark}{section}
\numberwithin{figure}{section}
\numberwithin{table}{section}
\numberwithin{algorithm}{section}

\begin{appendix}

This appendix summarizes the estimation strategies for $\kappa_1$ and $\xi$   and the variable selection approaches mentioned in the main text. \Cref{supp:sec:methods} also details the rationale for the approximate leave-one-out cross-validation (LOOCV).
\Cref{supp:sec:numerical} collects all additional numerical experiments mentioned in the main text; see the detailed outline therein.
\Cref{supp:extension_GLM_section} extends the theoretical results developed in the main text from the logistic regression model to the generalized linear model (GLM) with the canonical link.
\Cref{supp:sec:proof} provides all proofs for the theorems presented in the main text.

\section{Methodology}\label{supp:sec:methods}

\subsection{Estimation of signal strength $\kappa_1$}

\Cref{algorithm:estimation_kappa1} summarizes the estimation method of signal strength described in \Cref{sec:estimation_kappa1}.

\begin{algorithm}
    \centering
 \caption{Estimation of signal strength $\kappa_1$}
\begin{tabular}{p{15cm}}
\textbf{Input}: \\
Observed Data: $\{\bX_i,Y_i\}_{i=1}^n$,  \\
noninformative synthetic data: $\{\bX_i^*,Y_i^*\}_{i=1}^M$ \\
Relationship function: $g_{\noverp}$ with $\noverp=n/p,\tau_0=0.25,m=20/\noverp$ \\
\textbf{Process}:
 \begin{enumerate}
        \item Compute $\widehat{\bbeta}_M$ with $\tau=0.25n$ and set $\widehat{\eta}_M=\|\widehat{\bbeta}_M\|_2$.
        \item

        Find the smallest minimizer, $\widehat{\kappa}_1$, of the function $\kappa\mapsto\left|g_{\noverp}(\kappa)-\widehat{\eta}_M\right|$ over a prespecified bounded grid.
        
    \end{enumerate}

\textbf{Output}:  $\widehat{\kappa}_1$ \\
\end{tabular}
\label{algorithm:estimation_kappa1}
\end{algorithm}

\subsection{Estimation of similarity $\xi$}

\Cref{algorithm:estimation_xi} summarizes the procedure to estimate $\xi$ described in \Cref{sec:est_xi1} of the main text.

\begin{algorithm}
    \centering
 \caption{Estimation of similarity $\xi$}
 \begin{tabular}{ll}
\textbf{Input}: \\
Target dataset $\{\bX_{i0},Y_{i0}\}_{i=1}^{n_0}$ and source dataset $\{\bX_{is},Y_{is}\}_{i=1}^{n_s}$ \\
\textbf{Process}: \\
1. Generate two noninformative synthetic datasets: $\left\{\boldsymbol{X}^*_{i0}\right\}_{i=1}^M,\left\{\boldsymbol{X}^*_{is}\right\}_{i=1}^M \stackrel{\text { i.i.d. }}{\sim} \mathcal{N}\left(\mathbf{0},  \mathbf{I}_p\right)$ \\
\hskip 1cm and $\left\{Y^*_{i0}\right\}_{i=1}^M$, $\left\{Y^*_{is}\right\}_{i=1}^M \stackrel{\text { i.i.d. }}{\sim}\text{Bern}(0.5)$, $M=20p$. \\
2. Compute the SREs $\widehat{\bbeta}_{M,0}$ with $\tau=0.25n_0$ and $\widehat{\bbeta}_{M,s}$ with $\tau=0.25n_s$ based on \eqref{eq: SRE_def}. \\
3. Apply \Cref{algorithm:estimation_kappa1} to obtain two estimates $\widehat{\kappa}_1,\widehat{\kappa}_2$. Find solutions of the system of equations \eqref{nonlinear_three_equation(modify)} \\
$(\widehat{\alpha}_{1*},\widehat{\sigma}_{1*},\widehat{\gamma}_{1*})$
based on parameters $(\noverp=n_0/p,\widehat{\kappa}_1,\tau_0=0.25,m=M/n_0)$ and $(\widehat{\alpha}_{2*},\widehat{\sigma}_{2*},\widehat{\gamma}_{2*})$ \\
based on parameters $(\noverp=n_s/p,\widehat{\kappa}_2,\tau_0=0.25,m=M/n_s)$ \\
\textbf{Output}:  $\widehat{\xi}= \langle \widehat{\bbeta}_{M,0},\widehat{\bbeta}_{M,s}\rangle/({ \widehat{\alpha}_{*1} \widehat{\alpha}_{*2} \widehat{\kappa}_1\widehat{\kappa}_2})$ \\

\end{tabular}

\label{algorithm:estimation_xi}
\end{algorithm}

\subsection{Approximate leave-one-out cross-validation (LOOCV)}
\label{appendix_sec:LOOCV}
We provide the rationale for  the approximated leave-one-out cross-validation (LOOCV) method described in \Cref{sec:apply_theory_estimate_tau_MSE} of the main text for tuning the parameter $\tau$.
To mitigate the extensive computational burden of the standard LOOCV, we design an approximation to speed up the computation of the validation error (VE) for each candidate value of $\tau$.
This approximation requires running the optimization in \eqref{eq: SRE_def} only once per candidate value.

Recall that the VE is measured using the deviance and it is given by
\begin{equation*}
    VE(\tau)= -\sum_{i=1}^n\left\{Y_i\bX_i^\top \widehat{\bbeta}_{M,-i} -\rho(\bX_i^\top \widehat{\bbeta}_{M,-i}) \right\}
\end{equation*}
where $\widehat{\bbeta}_{M,-i} $ denotes the optimizer of \eqref{eq: SRE_def} computed using all data except for the $i$-th observation.

\subsubsection{Part 1}
We consider the approximation of $\widehat{\bbeta}_{M,-i}^\top \bX_i$ with any $i\in [n]$. To ease the notation, we drop the subscript $M$.
Let $\mathcal{I}=\{1, \ldots, n\}$ be the indices of all observations and $\mathcal{I}_{-i}=\{1, \ldots, i-1, i+1, \ldots, n\}$ be the indices of all but the $i$-th observation. Now we can write out the first-order optimality condition for $\widehat{\bbeta}$ and $\widehat{\bbeta}_{-i}$:
\begin{align*}
	0&=\sum_{j \in \mathcal{I}} \bX_j\left(Y_j-\rho^{\prime}\left(\widehat{\boldsymbol{\beta}}^{\top} \bX_j\right)\right) +\frac{\tau}{M} \sum_{j=1}^M \bX_j^*\left(Y_j^*-\rho^{\prime}\left(\widehat{\boldsymbol{\beta}}^{\top} \bX_j^*\right)\right), \\
	0&= \sum_{j \in \mathcal{I}_{-i}} \bX_j\left(Y_j-\rho^{\prime}\left(\widehat{\boldsymbol{\beta}}_{-i}^{\top} \bX_j\right)\right) +\frac{\tau}{M} \sum_{j=1}^M \bX_j^*\left(Y_j^*-\rho^{\prime}\left(\widehat{\boldsymbol{\beta}}_{-i}^{\top} \bX_j^*\right)\right).
\end{align*}

Taking the difference between these two equations yields
\begin{align*}
	0&=\bX_i\left(Y_i-\rho^\prime\left(\widehat{\boldsymbol{\beta}}^{\top} \bX_i\right)\right)+\sum_{j \in \mathcal{I}_{-i}} \bX_j\left[\rho^{\prime}\left(\widehat{\boldsymbol{\beta}}_{-i}^{\top} \bX_j\right)-\rho^{\prime}\left(\widehat{\boldsymbol{\beta}}^{\top} \bX_j\right)\right] + \\
	& \frac{\tau}{M} \sum_{j=1}^M \bX_j^*\left[\rho^{\prime}\left(\widehat{\boldsymbol{\beta}}_{-i}^{\top} \bX_j^*\right)-\rho^{\prime}\left(\widehat{\boldsymbol{\beta}}^{\top} \bX_j^*\right)\right].
\end{align*}

We expect the difference between $\widehat{\boldsymbol{\beta}}_{-i}$ and $\widehat{\boldsymbol{\beta}}$ to be small, so we can well approximate the difference $\rho^{\prime}\left(\widehat{\boldsymbol{\beta}}_{-i}^{\top} \bX_j\right)-\rho^{\prime}\left(\widehat{\boldsymbol{\beta}}^{\top} \bX_j\right)$ and $\rho^{\prime}\left(\widehat{\boldsymbol{\beta}}_{-i}^{\top} \bX_j^*\right)-\rho^{\prime}\left(\widehat{\boldsymbol{\beta}}^{\top} \bX_j^*\right)$ using a Taylor expansion of $\rho^\prime$ around $\widehat{\boldsymbol{\beta}}^{\top} \bX_j$ and $\widehat{\boldsymbol{\beta}}^{\top} \bX_j^* $, respectively. In other words, we have
\begin{align*}
	0&\approx \bX_i\left(Y_i-\rho^{\prime}\left(\widehat{\boldsymbol{\beta}}^{\top} \bX_i\right)\right)+\sum_{j \in \mathcal{I}_{-i}} \rho^{\prime\prime}\left(\widehat{\boldsymbol{\beta}}^{\top} \bX_j\right)\bX_j\bX_j^{\top}\left(\widehat{\boldsymbol{\beta}}_{-i}-\widehat{\boldsymbol{\beta}}\right) + \\
	&\frac{\tau}{M}\sum_{j=1}^M  \rho^{\prime\prime}\left(\widehat{\boldsymbol{\beta}}^{\top} \bX_j^*\right)\bX_j^*\bX_j^{*\top}\left(\widehat{\boldsymbol{\beta}}_{-i}-\widehat{\boldsymbol{\beta}}\right)
\end{align*}

To simplify the notation, we introduce the following shorthands for the Hessian matrices appearing in the above display:

\begin{align*}
	& H_{\tau}=-\sum_{j \in \mathcal{I}} \rho^{\prime\prime}\left(\widehat{\boldsymbol{\beta}}^{\top} \bX_j\right)\bX_j\bX_j^{\top}, \\
 &H_{\tau,-i}=-\sum_{j \in \mathcal{I}_{-i}} \rho^{\prime\prime}\left(\widehat{\boldsymbol{\beta}}^{\top} \bX_j\right)\bX_j\bX_j^{\top},\\
	& H_{\tau}^*=-\frac{\tau}{M}\sum_{j=1}^M  \rho^{\prime\prime}\left(\widehat{\boldsymbol{\beta}}^{\top} \bX_j^*\right)\bX_j^*\bX_j^{*\top}.
\end{align*}
Admitting this second order approximation, we have
\begin{equation*}
	\bX_i\left(Y_i-\rho^{\prime}\left(\widehat{\boldsymbol{\beta}}^{\top} \bX_i\right)\right)\approx \left(H_{\tau,-i}+H_{\tau}^*\right)\left(\widehat{\boldsymbol{\beta}}_{-i}-\widehat{\boldsymbol{\beta}}\right),
\end{equation*}
or
\begin{equation*}
	\left(\widehat{\boldsymbol{\beta}}_{-i}-\widehat{\boldsymbol{\beta}}\right)\approx \left(H_{\tau,-i}+H_{\tau}^*\right)^{-1}\bX_i\left(Y_i-\rho^{\prime}\left(\widehat{\boldsymbol{\beta}}^{\top} \bX_i\right)\right).
\end{equation*}

Therefore, we can approximate the term $\widehat{\boldsymbol{\beta}}_{-i}^{\top} \bX_i$ by
\begin{equation}\label{eq:approximate-xbeta-tilde-l}
\tilde l_i:=\widehat{\boldsymbol{\beta}}^{\top} \bX_i+\bX_i^{\top}\left( H_{\tau,-i}+H_{\tau}^*\right)^{-1} \bX_i\left(Y_i-\rho^{\prime}\left(\widehat{\boldsymbol{\beta}}^{\top} \bX_i \right)\right).
\end{equation}

\subsubsection{Part 2}
The derivation above involves a matrix inversion for each $i\in [n]$.
To obtain the inverse of $H_{\tau,-i}+H_{\tau}^*$ for all $i$ efficiently, we can take advantage of the fact that they are each a rank one update from the $H_\tau+H_\tau^*$:
$$
H_{\tau,-i}+H_{\tau}^*=H_{\tau}+H_{\tau}^*+\rho^{\prime\prime}\left(\widehat{\boldsymbol{\beta}}^{\top} \bX_i\right)\bX_i\bX_i^{\top}.
$$
Applying the Sherman-Morrison inverse formula, we have for each $i$:
\begin{equation}\label{eq:Sherman-Morrison-H-sum}
    (H_{\tau,-i}+H_{\tau}^*)^{-1}=(H_{\tau}+H_{\tau}^*)^{-1}-\frac{(H_{\tau}+H_{\tau}^*)^{-1}\rho^{\prime\prime}\left(\widehat{\boldsymbol{\beta}}^{\top} \bX_i\right)\bX_i\bX_i^{\top} (H_{\tau}+H_{\tau}^*)^{-1}}{1+\rho^{\prime\prime}\left(\widehat{\boldsymbol{\beta}}^{\top} \bX_i\right)\bX_i^{\top}(H_{\tau}+H_{\tau}^*)^{-1}\bX_i}.
\end{equation}

\subsubsection{Synthesis}
Based on the derivation above, we are ready to approximate $VE(\tau)$ using the following $\widetilde{VE}(\tau)$:
\begin{equation}
	\label{validation_error_approx}
	\widetilde{VE}(\tau)= -\sum_{i=1}^n\left\{Y_i\tilde l_i-\rho(\tilde l_i) \right\}.
 \end{equation}
 We summarize the procedure for the approximated LOOCV in \Cref{algorithm:LOOCV}.

\begin{algorithm}
\centering
\caption{Approximated LOOCV}
\begin{tabular}{p{15cm}}
\textbf{Input}: \\
Data: $\{\bX_i,Y_i\}_{i=1}^n$ \\
Synthetic data: $\{\bX_j^*,Y_j^*\}_{j=1}^M$ \\
Sequence of candidate tuning parameters $\tau_k, k\in \{1,2,\cdots, K\}$  \\
\textbf{Process}: \\
 For each $\tau_k$:
    \begin{enumerate}
        \item compute $\widehat{\bbeta}_M$ according to \eqref{eq: SRE_def},
        \item compute
        $\tilde l_i:=\bX_i^\top \widehat{\bbeta}_{M}+\bX_i^{\top}\left( H_{\tau,-i}+H_{\tau}^*\right)^{-1} \bX_i\left(Y_i-\rho^{\prime}\left(\bX_i^\top \widehat{\bbeta}_{M} \right)\right),$
        for $i=\{1,2\ldots,n\}$,
        \item compute
        $\widetilde{VE}(\tau_k)= -\sum_{i=1}^n\left\{Y_i\tilde l_i-\rho(\tilde l_i) \right\}.$
    \end{enumerate}

\textbf{Output}:  $\widehat{\tau}_{cv}=\arg\min_{\tau_k}\widetilde{VE}(\tau_k)$ \\
\end{tabular}
\label{algorithm:LOOCV}
\end{algorithm}

\subsection{Estimation of signal strengths with general covariance structures}
\label{supp:est_kappa_general_covariance}

In \Cref{sec:estimation_kappa1}, we provide a method for estimating $\kappa_1$ when the covariance of the covariate vector is identity, where the key idea is to make use of the one-to-one correspondence between $\lim_{n\rightarrow \infty}\|\widehat{\bbeta}_M\|^2$ and $\kappa_1$ which is defined as $\kappa_1=\lim_{p\rightarrow \infty} \|\bbeta_0\|$.
Here we provide an extension to the case where the covariance of $\bX_i$ is a general covariance $\Sigma$.

Let $\Sigma^{1/2}$ be a symmetric square root of $\Sigma$. We can write $\bX_i=\Sigma^{1/2} \bZ_i$ with $\bZ_i\sim N(0,\mathbb \mathbb{I}_p)$ and $\bX_i^* =\Sigma^{1/2}\bZ_i^*$ with $\bZ_i^*\sim N(0,\mathbb \mathbb{I}_p)$. The expression in \eqref{eq: SRE_def} can be written as
\begin{align*}
    \widehat{\boldsymbol{\beta}}_{M}&=\arg \max _{\boldsymbol{\beta} \in \mathbb{R}^p} \cdot\sum_{i=1}^n \left[Y_i\boldsymbol{Z}_i^\top\Sigma^{1/2} \boldsymbol{\beta}-\rho\left(\boldsymbol{Z}_i^\top \Sigma^{1/2}\boldsymbol{\beta}\right)\right]+\frac{\tau}{M}\sum_{i=1}^M \left[Y^*_i{\boldsymbol{Z}_i^*}^\top \boldsymbol{\beta}-\rho\left({\boldsymbol{Z}_i^*}^\top\Sigma^{1/2}\boldsymbol{\beta}\right)\right] \\
   &= \arg \max _{\boldsymbol{\beta} \in \mathbb{R}^p} \cdot\sum_{i=1}^n \left[Y_i\boldsymbol{Z}_i^\top \Sigma^{1/2}\boldsymbol{\beta}-\rho\left(\boldsymbol{Z}_i^\top \Sigma^{1/2}\boldsymbol{\beta}\right)\right]+\frac{\tau}{M}\sum_{i=1}^M \left[Y^*_i{\boldsymbol{Z}_i^*}^\top \Sigma^{1/2}\boldsymbol{\beta}-\rho\left({\boldsymbol{Z}_i^*}^\top\Sigma^{1/2}\boldsymbol{\beta}\right)\right]
\end{align*}
If we consider the reparametrization for $\Sigma^{1/2}\boldsymbol{\beta}$, we can follow the same logic as in \Cref{sec:estimation_kappa1} to obtain the one-to-one correspondence between $\lim_{n\rightarrow\infty}\|\Sigma^{1/2}\widehat{\bbeta}_M\|^2$ and $\kappa_1^{\Sigma}:=\lim_{p\rightarrow \infty}\|\Sigma^{1/2}\bbeta_0\|$.
Suppose $\widehat{\eta}^2_{M}$ is an estimate for
$\lim_{n\rightarrow\infty}\|\Sigma^{1/2}\widehat{\bbeta}_M\|^2$.
Again, following the reasoning in Section \ref{sec:estimation_kappa1}, the estimate of the signal strength $\kappa_1^{\Sigma}$ is given by the solution $\widehat{\kappa_1^{\Sigma}}$ to the equation $g_{\noverp}(\kappa)=\widehat{\eta}_{M}$, where the function $g_{\noverp}(\cdot)$ is defined in Section \ref{sec:estimation_kappa1}.

It remains to find an estimator $\widehat{\eta}^2_{M}$ for $\lim_{n\rightarrow\infty}\|\Sigma^{1/2}\widehat{\bbeta}_M\|^2$.
Suppose $\bX\sim N(0,\Sigma)$ is independent of $\{\bX_i,Y_i\}_{i=1}^n$.
We have $\|\Sigma^{1/2}\widehat{\bbeta}_M\|^2=\operatorname{Var}_{\bX}(\bX^\top \widehat{\bbeta}_M)$. Using the leave-one-out method, $\operatorname{Var}(\bX^\top \widehat{\bbeta}_M)$ can be estimated by

$$
\frac{1}{n} \sum_{i=1}^n\left(\widehat{\bbeta}_{M,-i}^{\top} \bX_i\right)^2-\left(\frac{1}{n} \sum_{i=1}^n \widehat{\bbeta}_{M,-i}^{\top} \bX_i\right)^2
$$
where $\widehat{\bbeta}_{M,-i} $ denotes the optimum of \eqref{eq: SRE_def} computed using all data except for the $i$-th observation.
We can reduce the computational burden of $\widehat{\bbeta}_{M,-i}$ by the same approximation outlined in \Cref{appendix_sec:LOOCV}. Specifically, we recall the approximation for the term $\bX_i^\top \widehat{\bbeta}_{M,-i}$ in \eqref{eq:approximate-xbeta-tilde-l} that
$$
\tilde l_i:=\bX_i^\top \widehat{\bbeta}_{M}+\bX_i^{\top}\left( H_{\tau,-i}+H_{\tau}^*\right)^{-1} \bX_i\left(Y_i-\rho^{\prime}\left(\bX_i^\top \widehat{\bbeta}_{M} \right)\right),
$$
where $H_{\tau,-i}=-\sum_{j \in \mathcal{I}_{-i}} \rho^{\prime\prime}\left(\widehat{\boldsymbol{\beta}}_M^{\top} \bX_j\right)\bX_j\bX_j^{\top}$ and $H_{\tau}^*=-\frac{\tau}{M}\sum_{j=1}^M  \rho^{\prime\prime}\left(\widehat{\boldsymbol{\beta}}_M^{\top} \bX_j^*\right)\bX_j^*\bX_j^{*\top}$ denote the empirical Hessian matrix of the log likelihood based on leave-one-out data and synthetic data respectively.
The inversion of $\left( H_{\tau,-i}+H_{\tau}^*\right)^{-1}$ can be done using Sherman-Morrison inverse formula as in \eqref{eq:Sherman-Morrison-H-sum}.
Then our estimator for $\lim_{n\rightarrow\infty}\|\Sigma^{1/2}\widehat{\bbeta}_M\|^2$ is defined as
\begin{equation*}
    \widehat{\eta}^2_{M}=\frac{1}{n} \sum_{i=1}^n\left(\tilde l_i\right)^2-\left(\frac{1}{n} \sum_{i=1}^n \tilde l_i\right)^2.
\end{equation*}

\subsection{Variable selection}

\Cref{sec:apply_theory_variable_selection_FDR} has proposed a feature selection procedure that utilizes the SRE by adapting the method from \citet[Algorithm 3]{dai2023scale}.
We summarize this procedure in
\Cref{algorithm:DS_FDR}.
In \Cref{algorithm:DS_FDR}, the value $\tau=p$ is taken for convenience and can be replaced by other values.

\begin{algorithm}
    \centering
 \caption{Feature selection using data-splitting}
 \begin{tabular}{ll}
 
\textbf{Input}: \\
Observed Data $\{\bX_i,Y_i\}_{i=1}^n$, synthetic data $\{\bX^*_i,Y^*_i\}_{i=1}^{M}$ and FDR level $q \in(0,1)$ \\
\textbf{Process}: \\
1. Split the observed data into two equal-sized halves $\{\bX^{(1)}_i,Y^{(1)}_i\}_{i=1}^{n/2}$  and $\{\bX^{(2)}_i,Y^{(2)}_i\}_{i=1}^{n/2}$.  \\
2. Split the synthetic data into two equal-sized halves $\{\bX^{(1)*}_i,Y^{(1)*}_i\}_{i=1}^{M/2}$  and $\{\bX^{(2)*}_i,Y^{(2)*}_i\}_{i=1}^{M/2}$.  \\
3. Compute the SRE for each part of data using
a chosen value of $\tau$
in \eqref{eq: SRE_def}.\\ Denote the estimators by $\widehat{\bbeta}^{(1)}_M$ and $\widehat{\bbeta}^{(2)}_M$.\\
4. For $j \in[p]$, regress $\bX_j^{(1)}$ onto $\mathbb X_{-j}^{(1)}$, and regress $\bX_j^{(2)}$ onto $\mathbb X_{-j}^{(2)}$. Let  $(\widehat{v}_j^{(1)})^2=\frac{\operatorname{RSS}_j^{(1)}}{n / 2-p+1}$,\\
\hskip 1cm and $ (\widehat{v}_j^{(2)})^2=\frac{\operatorname{RSS}_j^{(2)}}{n / 2-p+1}$
where $\mathrm{RSS}_j$ is the residual sum of squares.   \\
5. Compute the mirror statistic for $j\in [p]$: $M_j=T_j^{(1)}T_j^{(2)}$, \\
\hskip 1cm where  $T_j^{(1)}=\widehat{v}_j^{(1)} \widehat{\bbeta}_{M,j}^{(1)}$ and $T_j^{(2)}=\widehat{v}_j^{(2)} \widehat{\bbeta}_{M,j}^{(2)}$.\\
6. Calculate the cutoff $\omega_q$ as $\omega_q=\inf \left\{t>0:  \frac{\#\left\{j: M_j<-t\right\}}{\#\left\{j: M_j>t\right\}} \leq q\right\}$.\\
7. Output the selection set: $\widehat{S}_{\omega_q}=\left\{j\in [p]: M_j>\omega_q\right\}$.
\end{tabular}

\label{algorithm:DS_FDR}
\end{algorithm}

Following the argument in \cite{dai2023scale}, we can show the procedure in \Cref{algorithm:DS_FDR} can asymptotically control FDR at any given desired level. A precise statement is summarized in \Cref{proposition:FDR_DS_MAP}. To theoretically justify DS, we define $S_{1, \text { strong }}$ to be the largest subset of $S_1$ such that
$$
\sqrt{n} \min _{j \in S_{1, \text { strong }}}\left|\beta_j^{\star}\right| \rightarrow \infty .
$$
 Let $p_{1, \text { strong }}=\left|S_{1, \text { strong }}\right|$ and recall that $p_0$ is the number of null features.

\begin{proposition}\label{proposition:FDR_DS_MAP}
   Suppose the conditions of \Cref{coro:arbitray_cov_exact_cat_M_MAP_noninformative} hold and $n > 2p$. Assume $p_0 \rightarrow \infty,$ and $ \liminf \frac{p_{1, \text { strong }}}{ p_0}>0$ as $n, p \rightarrow \infty$.
   Then,
$$
\limsup _{n, p \rightarrow \infty} \mathbb{E}\left[\frac{\#\left\{j: j \in S_0, j \in \widehat{S}_{\omega_q}\right\}}{\#\left\{j: j \in \widehat{S}_{\omega_q}\right\}}\right] \leq q .
$$
using the data-splitting method outlined in \Cref{algorithm:DS_FDR}.

\end{proposition}

\section{Numerical experiments }\label{supp:sec:numerical}

This section includes additional experiments. An outline is as follows.
\begin{itemize}
    \item \Cref{app:estimation-comparison} compares the numerical performance of the SRE with those of the ridge and Lasso estimators.
    \item \Cref{sec:experiment-var-M} provides a numerical illustration for the convergence indicated in \Cref{thm:stability_finite_M}.
  
  \item  \Cref{supp:sec:estimation_kappa1} illustrates the accuracy of our estimation of $\kappa_1$ as well as the solutions $(\alpha_*, \sigma_*)$.
\item \Cref{supp:extra_CI} shows the performance of the adjusted confidence intervals when MLE exists.
 \item \Cref{supp:numerical_est_cosine_similar} illustrates the accuracy of our estimation of $\xi$.

\item  \Cref{supp:FDR_numerical} replicates the experiments from \cite{dai2023scale} for feature selection in the cases where MLE exists and compares our methods with theirs.

\item   In \Cref{sec:beyond_gaussian_empirical_justification}, we numerically demonstrate that the results in \Cref{thm:exact_cat_M_MAP_noninformative(modify)} can be extended to general covariates with finite fourth moments.
\item In \Cref{supp:realdata}, we provide the supporting evidence for gene selection.
\item In \Cref{supp:sec:apply_theory_estimate_tau_MSE}, we provide a numerical experiment to illustrate proposed tuning parameter selection methods.
\item \Cref{app:negative_transfer} illustrates negative transfer when the source and target signals are anti-aligned.
\end{itemize}

\subsection{Comparison of SRE, Ridge, and Lasso}\label{app:estimation-comparison}

Empirical studies demonstrate that the SRE outperforms the MLE in estimation and prediction, especially when
$p$ is large relative to $n$.
For example,
\Cref{fig:enter-label} illustrates a high-dimensional experiment with $p>n$, where the SRE remains feasible while both the MLE and the Maximum Diaconis-Ylvisaker prior penalized likelihood (MDYPL) estimator \citep{sterzinger2023diaconis} fail to exist.
As an additional illustration, this section presents two simulation studies in which the SRE outperforms ridge and Lasso estimators with increasing dimensions.

\begin{figure}[!t]
  \caption{
    An example of SREs in logistic regression using $\widehat{\boldsymbol{\beta}}_{M}$ with varying $\tau$ ($n=200$ and $p=250$).
    Observed data:
   $\bX_i\sim N(\mathbf{0},\mathbf{I}_p)$ and $Y_i\sim \text{Bern}(\rho^\prime(\bX_i^\top \bbeta_0))$ with $\|\bbeta_0\|_2=2.5$.
   Synthetic data: $\bX_i^*\sim N(\mathbf{0},\mathbf{I}_p)$, $Y_i^*\sim \text{Bern}(0.5)$, and $M=20p$.
   }
    \centering\includegraphics[width=0.45\linewidth]{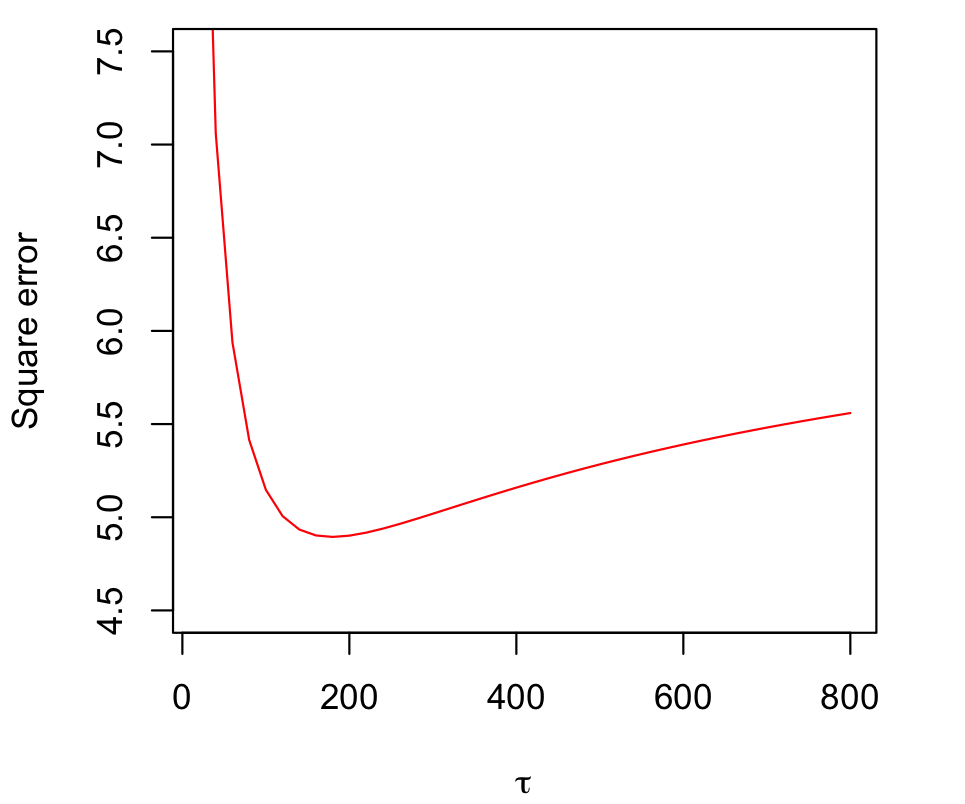}
  
    \label{fig:enter-label}
\end{figure}

\begin{table}[h]
    \centering
    \caption{Average Mean Squared Error  (with standard error in parentheses) over 50 independent trials for different dimensions \( p \).}
    \label{tab:errors}
    \begin{tabular}{c ccc}
        \hline
        Dimension \( p \) & SRE & Ridge & Lasso \\
        \hline
        50  & 8.319 (0.349) & 8.748 (0.351) & 9.840 (0.414) \\
        60  & 9.531 (0.309) & 9.699 (0.298) & 11.186 (0.398) \\
        70  & 10.208 (0.328) & 10.118 (0.310) & 11.650 (0.317) \\
        80  & 10.586 (0.274) & 10.486 (0.293) & 12.088 (0.367) \\
        90  & 11.616 (0.323) & 11.221 (0.296) & 12.750 (0.365) \\
        100 & 11.736 (0.270) & 14.649 (0.291) & 13.116 (0.301) \\
        110 & 12.122 (0.279) & 14.825 (0.303) & 13.429 (0.354) \\
        120 & 12.809 (0.235) & 14.955 (0.268) & 13.841 (0.287) \\
        130 & 12.992 (0.232) & 15.012 (0.255) & 14.222 (0.276) \\
        140 & 13.527 (0.240) & 15.125 (0.256) & 14.631 (0.276) \\
        150 & 13.770 (0.250) & 15.025 (0.255) & 14.505 (0.279) \\
        160 & 13.888 (0.205) & 14.934 (0.227) & 14.409 (0.246) \\
        170 & 14.276 (0.244) & 14.965 (0.226) & 14.579 (0.240) \\
        180 & 14.453 (0.234) & 15.053 (0.209) & 14.953 (0.223) \\
        190 & 14.704 (0.227) & 14.983 (0.213) & 14.771 (0.204) \\
        200 & 14.662 (0.252) & 14.937 (0.218) & 14.729 (0.229) \\
        \hline
    \end{tabular}
\end{table}

\begin{table}[h]
    \centering
    \caption{Average Mean Squared Error (with standard error in parentheses) over 50 independent trials for different dimensions \( p \).}
    \label{tab:experiment2}
    \begin{tabular}{c ccc}
        \hline
        Dimension \( p \) & SRE & Ridge & Lasso \\
        \hline
        50  & 1.057 (0.047) & 1.775 (0.119) & 1.067 (0.048) \\
        60  & 1.536 (0.073) & 2.102 (0.138) & 1.490 (0.062) \\
        70  & 2.147 (0.092) & 2.376 (0.114) & 2.103 (0.072) \\
        80  & 2.858 (0.138) & 2.837 (0.128) & 2.891 (0.136) \\
        90  & 3.551 (0.167) & 3.417 (0.151) & 3.884 (0.184) \\
        100 & 4.712 (0.244) & 12.607 (0.276) & 5.465 (0.204) \\
        110 & 5.569 (0.253) & 12.872 (0.294) & 6.508 (0.192) \\
        120 & 6.575 (0.312) & 13.069 (0.273) & 7.635 (0.210) \\
        130 & 7.495 (0.260) & 13.304 (0.264) & 8.651 (0.231) \\
        140 & 8.069 (0.280) & 13.292 (0.273) & 9.724 (0.323) \\
        150 & 8.521 (0.272) & 13.314 (0.266) & 10.088 (0.238) \\
        160 & 9.236 (0.218) & 13.174 (0.226) & 10.675 (0.228) \\
        170 & 9.864 (0.324) & 13.158 (0.249) & 11.037 (0.234) \\
        180 & 10.443 (0.315) & 13.266 (0.214) & 11.641 (0.243) \\
        190 & 10.762 (0.283) & 13.216 (0.206) & 12.168 (0.293) \\
        200 & 11.120 (0.253) & 13.206 (0.221) & 12.402 (0.282) \\
        \hline
    \end{tabular}
\end{table}

In the first experiment, we consider logistic regression, while the second experiment is based on linear regression.

For logistic regression, the response variable is generated as:
$$
Y_i \sim \text{Bern}\left(\frac{1}{1+\exp(-\bX_i^\top \bbeta_0)}\right).
$$
For linear regression, the response follows:
$$
Y_i \sim N(\bX_i^\top \bbeta_0,1).
$$

In both experiments, the observed sample size is fixed at $n=100$.
Each of the independent repetitions starts by simulating the true coefficient vector $\bbeta_0$, where each entry is independently drawn from
$$
\beta_{0,j} \sim N\left(0, \frac{16}{\sqrt{p}}\right).
$$
The observed $p$-dimensional covariate vectors $\boldsymbol{X}_i$ ($i=1,\dots,n$) are generated with independent entries as follows:
\begin{itemize}
    \item $X_{i,1} \sim \text{Bern}(0.1)$,
\item $X_{i,2} \sim \chi_1^2$ (Chi-square with 1 degree of freedom),
\item  $X_{i,3} \sim \chi_4^2$ (Chi-square with 4 degrees of freedom),
\item  For remaining entries ($j \geq 4$), $X_{i,j}$ follows a $t$-distribution with 4 degrees of freedom, mean 0, and variance 1.
\end{itemize}

Here, the first three entries of a covariate vector are designed to mimic real-world data characteristics, such as highly unbalanced categorical variables and skewed continuous distributions.

The synthetic data for the SRE are generated as follows.
For each entry $X_{j}^{*}$ of a synthetic covariate vector $\boldsymbol{X}^*$, $X_{j}^{*}$ is sampled from the marginal empirical distribution of observed $\{X_{i,j}\}_{i=1}^{n}$.
To accommodate the highly unbalanced binary covariate ($j=1$) in our simulations, half of the sampled $X_{1}^{*}$ will be replaced by i.i.d. random variables drawn from Bernoulli($p=0.5$); this follows the \textit{flattening} strategy proposed in the supplementary material of \cite{huang_catalytic_2020}.
To accommodate the skewness in continuous covariates ($j\geq 2$), half of the sampled $X_{j}^{*}$ will be replaced by i.i.d random variables drawn from a normal distribution with median and interquartile range (IQR) matched to those of the observed covariates. Specifically, the normal distribution is $N(\mu_j,\sigma_j^2)$, where $\mu_j$ is the sample median of $\{X_{i,j}\}_{i=1}^{n}$ and $\sigma_j$ is chosen properly such that
$1/4=\Phi\left( - {IQR_j}/{(2 \sigma_j)} \right)$,
where $IQR_j$ is the IQR of observed $\{X_{i,j}\}_{i=1}^{n}$ and $\Phi$ is the cumulative distribution function of standard normal. For logistic regression, the synthetic response is generated as:
$$Y^*\sim \text{Bern}(0.5).$$
For linear regression, the synthetic response follows:
$$Y^*\sim N(0,1).$$
The synthetic sample size is fixed at $M=1000$ across all scenarios.
 Tuning parameters for Ridge and Lasso are selected using their respective default cross-validation procedures in \texttt{glmnet} \citep{simon_regularization_2011}. 
 The tuning parameter for SRE is selected using \Cref{algorithm:LOOCV}.

\Cref{tab:errors,tab:experiment2} present the mean error (with standard error in parentheses) over 50 independent trials for various dimensions $p$.
In the logistic regression experiment (\Cref{tab:errors}), the SRE estimator consistently yields lower errors compared to both the Ridge and Lasso estimators across all considered dimensions.
In the linear regression experiment (\Cref{tab:experiment2}), SRE again demonstrates superior performance. Notably, as the dimensionality increases, Ridge's error rises markedly, while Lasso's performance remains slightly inferior to that of SRE.

Overall, these findings demonstrate the effectiveness of the SRE for estimation, particularly in high-dimensional settings with complicated covariate structures.
The results suggest that SRE may offer a more reliable estimation method compared to traditional Ridge and Lasso estimators.

\subsection{Stability of the SRE against M}\label{sec:experiment-var-M}

In this section, we present an experiment to demonstrate that, with fixed observed data, the SRE with a finite $M$ approaches its limit at the rate of $\frac{1}{M}$, as stated in \Cref{thm:stability_finite_M}.
We set $n = 1000$, $p = 250$, and $\tau = 500$, and gradually increase the   synthetic sample size  $M \in \{2^{k-1}p : k \in \{1, 2, 3, 4, 5, 6, 7\}\}$. The generation of observed and synthetic data is listed below. For observed data, we first sample regression coefficients ${\beta}_j \stackrel{\mathrm{i.i.d.}}{\sim} \mathcal{N}(0, 1/p)$ for $j\in [p]$, and then generate covariates $\boldsymbol{X}_i \stackrel{\mathrm{i.i.d.}}{\sim} \mathcal{N}(0, \mathbb{I}_p)$ and responses $Y_i \sim \text{Bern}(\rho^\prime(\boldsymbol{X}_i^\top \boldsymbol{\beta}))$ for $i \in [n]$.
For synthetic data, for each $i \in [M]$, generate $Y_i^* \stackrel{\mathrm{i.i.d.}}{\sim} \text{Bern}(0.5)$ and $\boldsymbol{X}_i^* \stackrel{\mathrm{i.i.d.}}{\sim} \mathcal{N}(0, \mathbb{I}_p)$.
Note that this synthetic data generation allows for a mathematical derivation of an exact formula for computing the SRE with infinite synthetic samples.

The SRE $\widehat{\boldsymbol{\beta}}_M$ is computed based on \eqref{eq: SRE_def}. For $\widehat{\boldsymbol{\beta}}_{\infty}$, since we know the synthetic data-generating distribution, we first rewrite \eqref{cat_betahat_Minfty} by finding an analytical expression of the expectation. Note that $Y^*\sim \text{Bern}(0.5)$ and $\bX^*\sim \mathcal{N}(0, \mathbb{I}_p)$, we have $\mathbb E(Y^*\bX^*)=\mathbf{0}$. We have
\begin{align*}
    \operatorname{pen}(\bbeta):=\mathbb E\left[\rho(\boldsymbol{X}^{*\top}\boldsymbol{\beta})-Y^*\boldsymbol{X}^{*\top}\boldsymbol{\beta}\right]&=  \mathbb E\left[\rho(\boldsymbol{X}^{*\top}\boldsymbol{\beta})\right]=\int_{-\infty}^{\infty} \rho(\|\bbeta\|_2 z)\frac{1}{\sqrt{2\pi}}\exp(-\frac{z^2}{2})dz.
\end{align*}
The function $ \operatorname{pen}(\bbeta)$ is convex in $\bbeta$, which is a direct consequence of the convexity of the function $\rho(\cdot)$ and the convexity of Euclidean norm. Then $\widehat{\bbeta}_{\infty}$ can be computed via following convex optimization:
\begin{align*}
    \widehat{\boldsymbol{\beta}}_{\infty}=\arg \min _{\boldsymbol{\beta} \in \mathbb{R}^p} \left\{\sum_{i=1}^n \left[\rho\left(\boldsymbol{X}_i^\top \boldsymbol{\beta}\right)-Y_i\boldsymbol{X}_i^\top \boldsymbol{\beta}\right]+\tau  \operatorname{pen}(\bbeta) \right\}.
\end{align*}
We denote the difference between $\widehat{\boldsymbol{\beta}}_{M}$ and $\widehat{\boldsymbol{\beta}}_{\infty}$ as $\text{Err} := \|\widehat{\boldsymbol{\beta}}_{M} - \widehat{\boldsymbol{\beta}}_{\infty}\|^2$. For one simulation, the relationship between Err and the value of $M$ is illustrated in \Cref{fig:vari_M}.
To confirm a linear dependence between $\log(\text{Err})$ and $\log(M)$, we fit a linear regression model for $\log(\text{Err})$ on $\log(M)$, where the least squares estimated slope is $-1.048$ with a small standard error of $0.006$.
This observation aligns with the rate of convergence between $\widehat{\boldsymbol{\beta}}_{M}$ and $\widehat{\boldsymbol{\beta}}_{\infty}$  established in \Cref{thm:stability_finite_M}.

\begin{figure}[ht]
	\centering
	\includegraphics[scale=0.5]{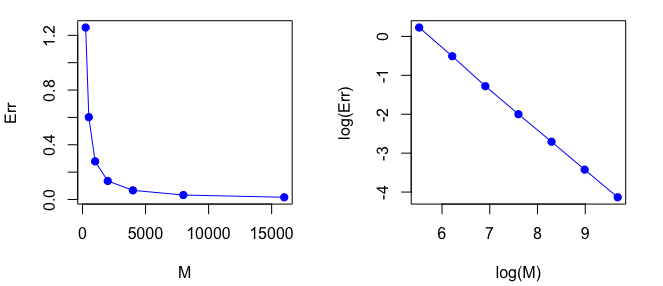}
	\caption{Convergence of the SRE as $M$ increases. The slope in the right figure is -1.048 (standard error 0.006) based on the least squares estimate.}
 \label{fig:vari_M}
\end{figure}

\subsection{Numerical illustration of estimating signal strength}\label{supp:sec:estimation_kappa1}

We demonstrate the accuracy of our estimation of $\kappa_1$ as well as the solutions $(\alpha_*, \sigma_*)$ via some empirical results.
We consider the same setting described in \Cref{sec:numerical_verify} but examine a sequence of dimensions $p = \{100, 400, 1600\}$.

We first investigate the estimation accuracy of $\kappa_1$. The results are displayed in \Cref{table:estimation_kappa1}.
From the table, it is evident that when $\noverp$ and $\kappa_1$ are held constant, both the estimation error and its standard deviation decrease as $p$ increases.
This trend is expected since $\widehat{\eta}_{M}$ converges to its limit $\eta_{M}$ as $p$ increases. Given $\kappa_1$ and $p$, the estimation error is smaller for larger $\noverp$, since the sample size is larger.
This observation aligns with the curves of $g_\delta(\cdot)$ in \Cref{fig:dictionary_M20}, where a larger value of $\noverp$ leads to a steeper slope and thus a more accurate estimate for $\kappa_1$,  the solution to $g_{\noverp }(\kappa)=\eta_M$.

  \begin{table}[ht]
\centering
\caption{Mean and standard deviation (in parentheses) of error $|\widehat{\kappa}_1-\kappa_1|$ based on 50 independent replications. }
\begin{tabular}{cccc}
\hline
\textbf{$\kappa_1$} & \textbf{$p$} & \textbf{$\noverp = 2$} & \textbf{$\noverp = 4$} \\ \hline \hline
                   & 100              & 0.363(0.315)             & 0.196(0.127)             \\
0.5               & 400              & 0.234(0.132)             & 0.128(0.102)             \\
                   & 1600             & 0.129(0.102)             & 0.060(0.045)             \\
\hline
                   & 100              & 0.397(0.285)             & 0.228(0.160)             \\
1                 & 400              & 0.227(0.165)             & 0.134(0.116)             \\
                   & 1600             & 0.104(0.114)             & 0.068(0.067)             \\
\hline
                   & 100              & 0.426(0.325)             & 0.294(0.240)             \\
1.5               & 400              & 0.230(0.214)             & 0.178(0.164)             \\
                   & 1600             & 0.154(0.159)             & 0.103(0.091)             \\
\hline
                   & 100              & 0.678(0.747)             & 0.396(0.305)             \\
2                 & 400              & 0.329(0.307)             & 0.209(0.255)             \\
                   & 1600             & 0.201(0.214)             & 0.135(0.121)             \\
\hline
\end{tabular}
\label{table:estimation_kappa1}
\end{table}

 Next, we investigate the estimation accuracy of $\left(\alpha_*, \sigma_*\right)$ in \Cref{table:estimation_alpha_sigma}.
 The true values $\left(\alpha_*, \sigma_*\right)$ are presented in \Cref{table:true_alpha_sigma}.
 We observe that the estimation errors for $(\alpha_*, \sigma_*)$ are relatively small compared to the true values, thus the estimates are quite accurate despite the estimation error of $\kappa_1$. Furthermore, we note that the errors decrease as $p$ increases, which aligns with the observed pattern in the estimation of $\kappa_1$.

\begin{table}[ht]
\centering
\caption{Mean and standard deviation (in parentheses) of the estimation error of true solutions of the system of equations ($\alpha_*,\sigma_*$) based on 50 independent replications. }
\begin{tabular}{|c|c|cc|cc|}
\hline
\textbf{$\kappa_1$} & \textbf{$p$} & \multicolumn{2}{c|}{\textbf{$\noverp = 2$}} & \multicolumn{2}{c|}{\textbf{$\noverp = 4$}} \\ \hline
                        &                  & $|\widehat{\alpha}_{*}-\alpha_*|$      & $|\widehat{\sigma}_{*}-\sigma_*|$     & $|\widehat{\alpha}_{*}-\alpha_*|$      & $|\widehat{\sigma}_{*}-\sigma_*|$    \\ \hline
0.5                     & 100              & 0.049(0.058)       & 0.007(0.017)       & 0.017(0.012)        & 0.006(0.004)       \\
                        & 400              & 0.028(0.020)       & 0.002(0.002)       & 0.011(0.010)        & 0.004(0.003)       \\
                        & 1600             & 0.015(0.011)       & 0.003(0.002)       & 0.005(0.004)        & 0.003(0.002)       \\ \hline
1                       & 100              & 0.066(0.051)       & 0.015(0.017)       & 0.027(0.018)        & 0.006(0.004)       \\
                        & 400              & 0.040(0.032)       & 0.009(0.008)       & 0.018(0.014)        & 0.003(0.003)       \\
                        & 1600             & 0.018(0.021)       & 0.006(0.006)       & 0.010(0.008)        & 0.003(0.002)       \\ \hline
1.5                     & 100              & 0.079(0.055)       & 0.026(0.026)       & 0.041(0.033)        & 0.004(0.003)       \\
                        & 400              & 0.044(0.040)       & 0.020(0.020)       & 0.025(0.023)        & 0.002(0.002)       \\
                        & 1600             & 0.029(0.029)       & 0.014(0.015)       & 0.015(0.012)        & 0.002(0.002)       \\ \hline
2                       & 100              & 0.110(0.092)       & 0.051(0.047)       & 0.052(0.039)        & 0.005(0.006)       \\
                        & 400              & 0.058(0.049)       & 0.031(0.028)       & 0.029(0.033)        & 0.004(0.005)       \\
                        & 1600             & 0.036(0.034)       & 0.021(0.018)       & 0.018(0.018)        & 0.003(0.003)       \\ \hline
\end{tabular}

\label{table:estimation_alpha_sigma}
\end{table}
\begin{table}[ht]
\centering
\caption{Solutions of system of equations ($\alpha_*,\sigma_*$) under different settings with noninformative synthetic data. }
\begin{tabular}{|c|c|c|c|c|}
\hline
$\noverp$ \textbackslash $\kappa_1$ & 0.5 & 1 & 1.5 & 2 \\
\hline
2 & (1.004, 1.735) & (0.932, 1.726) & (0.833, 1.708) & (0.740, 1.665) \\
4 & (0.890, 1.008) & (0.836, 1.021) & (0.773, 1.030) & (0.701, 1.031) \\
\hline
\end{tabular}

\label{table:true_alpha_sigma}
\end{table}

\subsection{Adjusted confidence intervals for $\delta=4$}
\label{supp:extra_CI}

We follow the same experimental setting in \Cref{supp:sec:estimation_kappa1} but consider the case where $\delta=4$.
Note that the MLE nearly always exists in this case.
In this experiment, we compare the coverage rates given by three methods: our adjusted confidence intervals (Adjusted SRE), the confidence intervals based on classical MLE asymptotics, and the adjusted confidence intervals based on the MLE (Adjusted MLE) as implemented in the R package \textbf{glmhd} \citep{zhao2020glmhd}.
The results are shown in \Cref{table:adjust_CI_delta4}.
As we can see, when the MLE exists, the coverage rate of the confidence interval provided by classical MLE asymptotics is lower than 0.95, but both adjusted confidence intervals provide the expected coverage.

\begin{table}[!htbp]
\centering
\caption{Coverage rates of 95\% confidence intervals based on classical MLE asymptotics and adjusted intervals with $\noverp=4$ (MLE exists). Average over 50 independent experiments.}
\begin{tabular}{|c|c|c|c|c|c|}
\hline
Method & $p$ & $\kappa_1=0.5$ & $\kappa_1=1$ & $\kappa_1=1.5$ & $\kappa_1=2$ \\ \hline
MLE Asymptotics      & 100        & 0.900                 & 0.884               & 0.857                  & 0.817                \\ \hline
MLE Asymptotics     & 400        & 0.902                 & 0.889               & 0.863                  & 0.827                \\ \hline
Adjusted MLE    & 100        & 0.943                 & 0.946               & 0.936                  & 0.931                \\ \hline
Adjusted MLE    & 400        & 0.948                 & 0.949               & 0.947                  & 0.944                \\ \hline
Adjusted SRE    & 100        & 0.943                 & 0.948               & 0.943                  & 0.944                \\ \hline
Adjusted SRE    & 400        & 0.951                 & 0.951               & 0.950                  & 0.949                \\ \hline
\end{tabular}

\label{table:adjust_CI_delta4}
\end{table}

\subsection{Numerical illustration of estimating $\xi$}
\label{supp:numerical_est_cosine_similar}

We conduct a series of experiments to examine the performance of our proposed method to estimate $\xi$, which is referred to as \textit{adjusted correlation}.
Specifically, we compare our estimate with the cosine similarity between $\widehat{\bbeta}_{M,0}$ and $\widehat{\bbeta}_{M,s}$, which is referred to as \textit{naive correlation}.
In this experiment, we enumerate $p\in\{100,400,1600\}$, set the target sample size to be $n_0=\delta_0p$, and set the source sample size to be $n_s=\delta_s p$.
The data are generated as follows.
For target data, we draw the coordinates of $\sqrt{p}\boldsymbol{\beta}_0$ independently from the scaled t-distribution with 3 degrees of freedom and variance equal to $1$, generate the covariates $\left\{\boldsymbol{X}_{i0}\right\}_{i=1}^{n_0}$ independently from $\mathcal{N}\left(\mathbf{0},  \mathbf{I}_p\right)$, and sample the response $Y_{i0}$ from $\text{Bern}\left(\rho^{\prime}(\boldsymbol{X}_{i0}^{\top} \boldsymbol{\beta}_0) \right)$.
For source data, the covariates and responses are generated in a similar manner as the target dataset, except that the coefficient is now $\boldsymbol{\beta}_s=\xi \boldsymbol{\beta}_0+ \sqrt{1-\xi^2}\tilde{\beps}$, where $\tilde \beps$ is an independent noise vector whose entries are independently generated from the scaled t-distribution with 3 degrees of freedom and variance equal to $1/p$.
The true cosine similarity $\xi$ is fixed at $0.9$ in this experiment.

\Cref{table:estimation_xi} presents the experimental results.
First, our proposed adjusted correlation outperforms the naive correlation across all settings.
Second, it is clear that as $p$ increases, the estimation error decreases, which follows the same pattern observed in the estimation of signal strength.
Moreover, a larger value of $\delta_0=n_0/p$ (the ratio of the target sample size to the dimension) results in a smaller estimation error.
In contrast, increasing $\delta_s$ from 4 to 16 while holding $\delta_0=2$ produces little reduction in the estimation error, illustrating that the accuracy of $\widehat{\xi}$ is primarily limited by the less informative sample.

\begin{table}[ht]
\centering
\caption{Mean and standard deviation (in parentheses) of the estimation error $|\widehat{\xi}-\xi|$ across various settings of $(\noverp_0, \noverp_s, p)$.  Average over 50 independent replications.}
\begin{tabular}{|c|c|c|c|c|}
\hline
\textbf{$\noverp_0$} & \textbf{$\noverp_s$} & \textbf{$p$} & \textbf{Adjusted Correlation} & \textbf{Naive Correlation} \\ \hline
 &    & 100 & 0.133(0.101) & 0.624(0.089) \\
2 & 4   & 400 & 0.128(0.076) & 0.634(0.064) \\
 &    & 1600 & 0.059(0.044) & 0.632(0.028) \\ \hline
 &    & 100 & 0.151(0.146) & 0.567(0.094) \\
2 & 10  & 400 & 0.102(0.062) & 0.563(0.069) \\
 &    & 1600 & 0.062(0.042) & 0.562(0.038) \\ \hline
 &    & 100 & 0.142(0.128) & 0.541(0.106) \\
2 & 16  & 400 & 0.114(0.075) & 0.538(0.070) \\
 &    & 1600 & 0.068(0.048) & 0.537(0.037) \\ \hline
 &    & 100 & 0.100(0.084) & 0.507(0.106) \\
4 & 4   & 400 & 0.079(0.050) & 0.513(0.065) \\
 &    & 1600 & 0.055(0.035) & 0.530(0.038) \\ \hline
 &    & 100 & 0.112(0.096) & 0.447(0.105) \\
4 & 10  & 400 & 0.062(0.054) & 0.442(0.065) \\
 &    & 1600 & 0.039(0.029) & 0.446(0.040) \\ \hline
 &    & 100 & 0.103(0.098) & 0.417(0.104) \\
4 & 16  & 400 & 0.070(0.046) & 0.404(0.066) \\
 &    & 1600 & 0.031(0.020) & 0.412(0.037) \\ \hline
\end{tabular}
\label{table:estimation_xi}
\end{table}

\subsection{Additional material for feature selection}
\label{supp:FDR_numerical}

Now we examine the experiments described in \citet[Section 5.1]{dai2023scale}, which consist of two experiments: one in a small-p-n setting ($p=60, n=500$) and the other in a large-p-n setting ($p=500, n=3000$). The number of relevant features, denoted as $p_1 = p - p_0$, is set to 30 in the small-p-n setting and 50 in the large-p-n setting.
We use the SRE to conduct the ABH, ABY and MDS procedures as described in \Cref{sec:apply_theory_variable_selection_FDR}.
In addition, we consider three competing methods that utilize the MLE: MDS, BHq, and $\mathrm{ABH}$.
The implementation of the MDS method follows \citet[Algorithm 3]{dai2023scale}.
The $\mathrm{BH}$ method utilizes classical asymptotic p-values calculated via the Fisher information, whereas the $\mathrm{ABH}$ method is based on adjusted asymptotic p-values computed via the \textbf{R} package \texttt{glmhd} \citep{zhao2020glmhd}.

\Cref{fig:FDR_reproduce_small_n_p} shows the experimental results for the small-p-n setting.
It is evident that our proposed procedures  ABH and MDS perform comparably to the alternatives: all methods control the FDR at the nominal level of 0.1, and their power appears close when focusing on either MLE or SRE. 
The BH procedure using the MLE often has a slightly higher power at the expense of a larger FDR. 
For ABY, it has the lowest FDR but it is too conservative and has the lowest power.

\Cref{fig:FDR_reproduce_large_p} shows the experimental results for the large-p-n setting.
It is seen that the BH procedure using the MLE does not provide satisfactory error control since its FDR exceeds the nominal level significantly.
The ABH procedure, whether using the MLE or the SRE, has a higher power than the other methods, albeit at the price of slight inflation of the FDR in some cases.
The MDS procedure using either the MLE or the SRE performs reasonably well in every case, since the FDR is close to or below the nominal level and the power is not much lower than that of the ABH procedure.

\begin{figure}[H]
    \centering
    \includegraphics[scale=0.25]{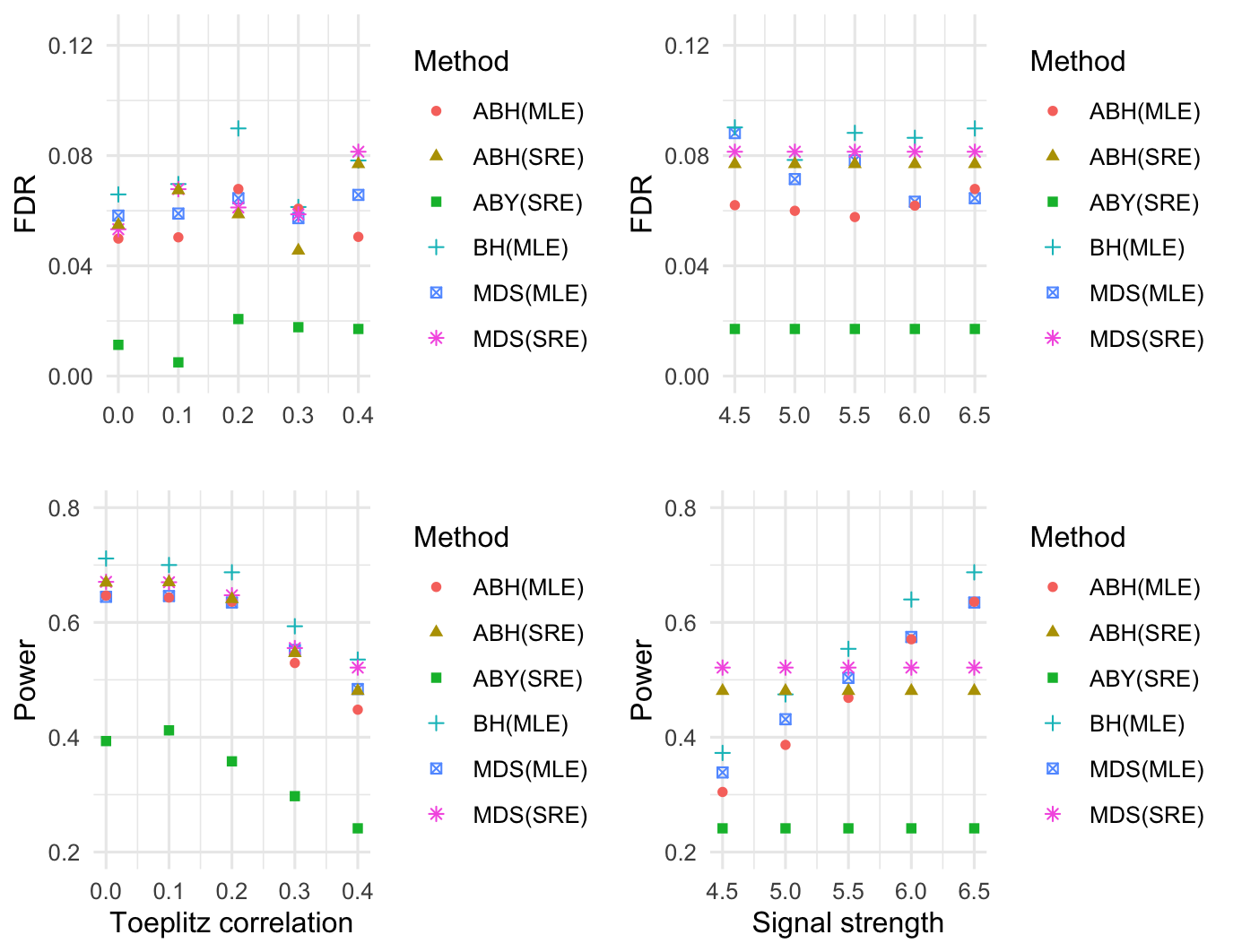}
    \caption{
Empirical FDRs and powers in the small-p-n setting. The covariates are independently drawn from a normal distribution $N(0, \Sigma)$, where $\Sigma$ has a Toeplitz correlation structure ($\Sigma_{i j}=r^{|i-j|}$).
In the left panel, we keep the signal strength constant at $\left|\beta_{0j}\right|=0.291$ for each $j$ in the set $S_1$ (same setting as in \citet{dai2023scale} without standardization on design matrix), while varying the correlation coefficient $r$.
In the right panel, we fix the correlation at $r=0.2$ and adjust the signal strength.
In each scenario, there are 30 relevant features.
The nominal FDR level is $q=0.1$.
The power is assessed as the proportion of correctly identified relevant features.
Each point represents the average of 50 replications. The SRE is computed using noninformative synthetic data with $M=20p$ and $\tau=p$.
}
    \label{fig:FDR_reproduce_small_n_p}
\end{figure}

\begin{figure}[H]
    \centering
    \includegraphics[scale=0.25]{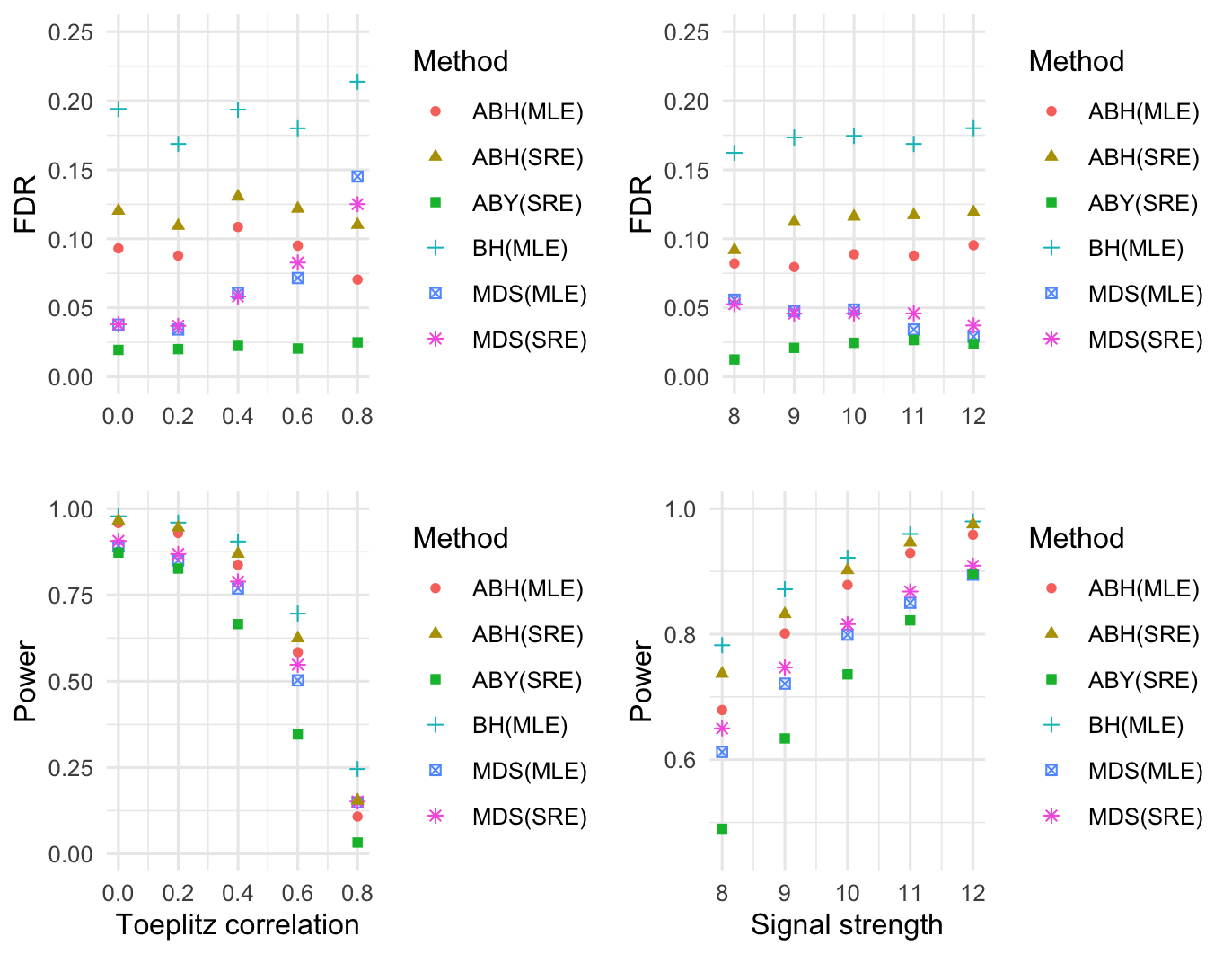}
    \caption{
    Empirical FDRs and powers in the large-p-n setting.
    The simulation of the covariate matrix follows the procedure described in \Cref{fig:FDR_reproduce_small_n_p}.
    In the left panel, we keep the signal strength constant at $\left|\beta_{0j}\right|=0.201$   for each $j$ in the set $S_1$ (same setting as in \citet{dai2023scale} without standardization on design matrix), while varying the correlation coefficient $r$.
In the right panel, we fix the correlation at $r=0.2$ and adjust the signal strength.
In each scenario, there are 50 relevant features.
The nominal FDR level is $q=0.1$.
Each point represents the average of 50 replications.
The SRE is computed using noninformative synthetic data with $M=20p$ and $\tau=p$.}
    \label{fig:FDR_reproduce_large_p}
\end{figure}

\subsection{Beyond Gaussian design empirical studies}
\label{sec:beyond_gaussian_empirical_justification}

In this section, we provide several numerical experiments to empirically justify that the Gaussian design condition used in \Cref{thm:exact_cat_M_MAP_noninformative(modify)} can be relaxed.

In the following experiments, the entries of the observed and synthetic covariate matrices are i.i.d. samples from a t-distribution with various degrees of freedom. The entries of the covariate matrix are scaled to have a mean of $0$ and a variance of $1$,
matching the first two moments of the standard Gaussian.
We compare
\begin{itemize}
    \item the averaged empirical squared error: $\|\widehat{\bbeta}_M-\bbeta_0\|^2$,
    \item the asymptotic squared error as derived from \Cref{thm:exact_cat_M_MAP_noninformative(modify)}: $(\alpha_*-1)^2\kappa_1^2+\sigma_*^2$,
\end{itemize}
where $(\alpha_*,\sigma_*,\gamma_*)$ is the solution of the system of equations \eqref{nonlinear_three_equation(modify)} based on ($\kappa_1,\noverp,\tau_0,m $), with $\kappa_1=1,\noverp=4,m=5$ in the current experiments.
We plot the empirical values as points and the theoretical values as a curve in \Cref{fig:design_t_dist}.
We observe that when the number of degrees of freedom is below 4, the alignment between empirical and theoretical values is not perfect. However, when the number of degrees of freedom is 4 or greater, the alignment becomes perfect. This observation suggests that our theoretical result can be extended beyond Gaussian design if a fourth moment condition is imposed.

Furthermore, we observe from \Cref{fig:design_t_dist} that when the number of degrees of freedom is as small as $3$, the 3rd moment does not exist but the theoretical value provides a reasonably good approximation for the empirical value. This also justifies the usefulness of our theory.

\begin{figure}[H]
    \centering
    \includegraphics[scale=0.5]{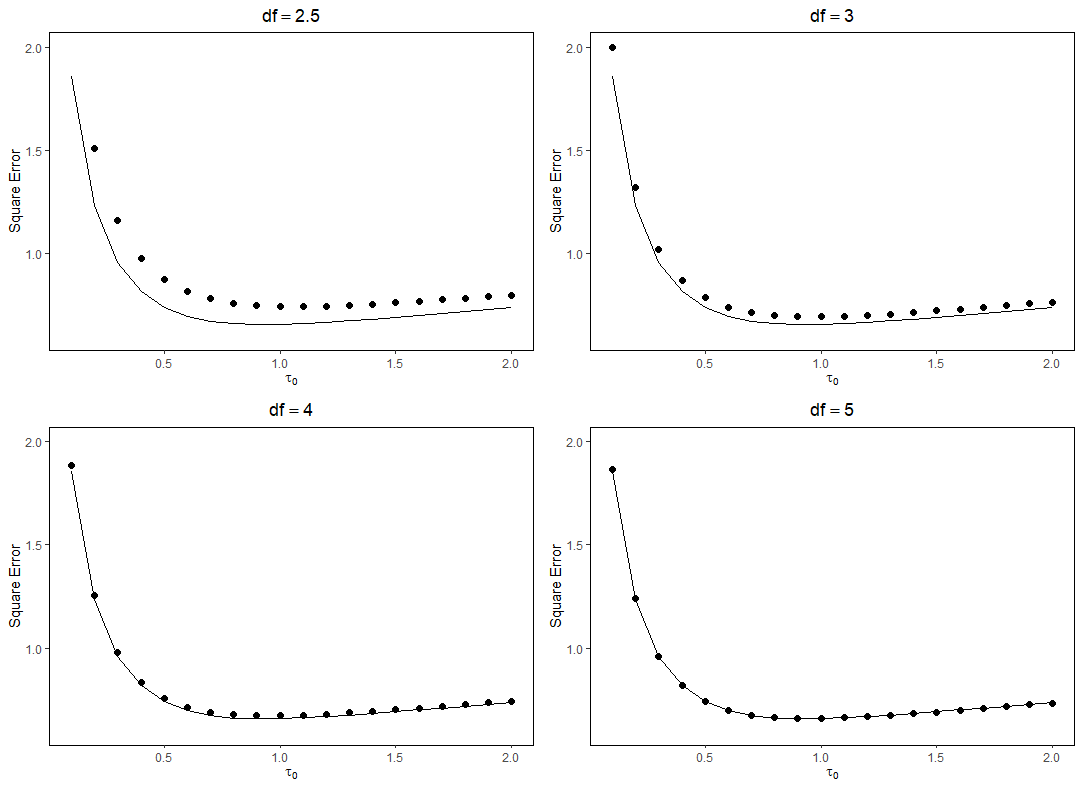}
    \caption{
    Performance of the SRE with noninformative synthetic data as a function of $\tau_0=\tau/n$.
  `df' represents the number of degrees of freedom of the t-distribution used to produce the covariate matrix.
  Each point is obtained by averaging the performance metric of the SRE over 100 simulation replications.
 The solid lines represent the corresponding theoretical prediction  derived from \Cref{thm:exact_cat_M_MAP_noninformative(modify)}. }
    \label{fig:design_t_dist}
\end{figure}

\subsection{Supporting Evidence for Gene Selection}\label{supp:realdata}

 In the analysis with the scRNA-seq dataset in \Cref{sec:apply_theory_variable_selection_FDR}, we applied the MDS, ABY, and ABH variable selection procedures using the SRE.
All three methods identify HSPA1A and NFKBIA, while ABH selects two additional genes, namely EEF1A1 and RPL10.
Below, we provide supporting evidence for these selections:
(1) HSPA1A inactivates GR through partial unfolding \citep{kirschke2014glucocorticoid}.
(2) NFKBIA is involved in GR activation \citep{deroo2001glucocorticoid}.
(3) EEF1A1 may modulate the cellular response to glucocorticoid treatment in breast cancer due to its role in cytoskeletal dynamics and apoptosis \citep{abbas2015eef1a}.
(4) RPL10 plays a role in tumor progression in epithelial ovarian cancer \citep{shi2018biological}, which may be associated with glucocorticoid treatment.

\subsection{Adjusting estimation by selection of tuning parameter}\label{supp:sec:apply_theory_estimate_tau_MSE}

The method in \Cref{sec:apply_theory_estimate_tau_MSE} can be
 naturally extended to cases where the SRE is constructed using informative auxiliary data.
More concretely, we can estimate the limit of the squared error by \eqref{eq:exact_cat_M_MAP_informative_square_error} using the estimation method for $(\kappa_1, \kappa_2, \xi)$ in \Cref{sec:est_xi1} and we call the resulting estimator \textbf{SESE(I)} where the suffix (I) denotes informative auxiliary data.
Similarly, we can select $\tau$ that minimizes the limit of the squared error based on the true value of $(\kappa_1, \kappa_2, \xi)$ and call the resulting estimator \textbf{STSE(I)}.
The procedure for leave-one-out cross-validation remains the same as before and the resulting estimator with informative auxiliary data is named SLCV(I).

We provide an experiment to illustrate these methods: SESE, STSE, and  SLCV that
are based on observed data and noninformative synthetic data; SESE(I), STSE(I), and SLCV(I) are based on observed data and informative auxiliary data.
We consider the scenarios where $p=400$, $n$ is either $2p$ or $4p$, and $\kappa_1$ is either $1$ or $2$.
The observed covariates and responses are generated according to the observed data generation process described in \Cref{sec:numerical_verify}.
The noninformative synthetic data are generated  $\{\bX_i\}_{i=1}^M\stackrel{i.i.d.}{\sim} N(\mathbf{0},\mathbf{I}_p), \{Y_i^*\}_{i=1}^M\stackrel{i.i.d.}{\sim} \text{Bern}(0.5)$ with $M=20 \cdot p$.
The informative auxiliary data are generated following the procedure described in \Cref{sec:numerical_verify} and we fix $\xi=0.9$, $\kappa_2=1$, and $M=10\cdot p$.
In each scenario, we repeat the experiments 50 times and evaluate the  squared error of each estimator.

The results across different scenarios are shown in \Cref{fig:compare_diff_method_MSE}.
In each scenario, both SESE and SLCV perform on par with the benchmark given by STSE, which indicates that our selection methods, either using theoretical limits with estimated signal strengths or using leave-one-out cross-validation, are effective in selecting the tuning parameter $\tau$.
In addition, the performance of the estimator using informative auxiliary data is significantly superior to that using noninformative synthetic data and there is little difference among SLCV(I), SESE(I), and STSE(I).
This suggests that in the presence of informative auxiliary data, our proposed selection methods can effectively utilize the information from the auxiliary data by selecting a suitable value of $\tau$.

\begin{figure}[H]
    \centering
    \includegraphics[scale=0.2]{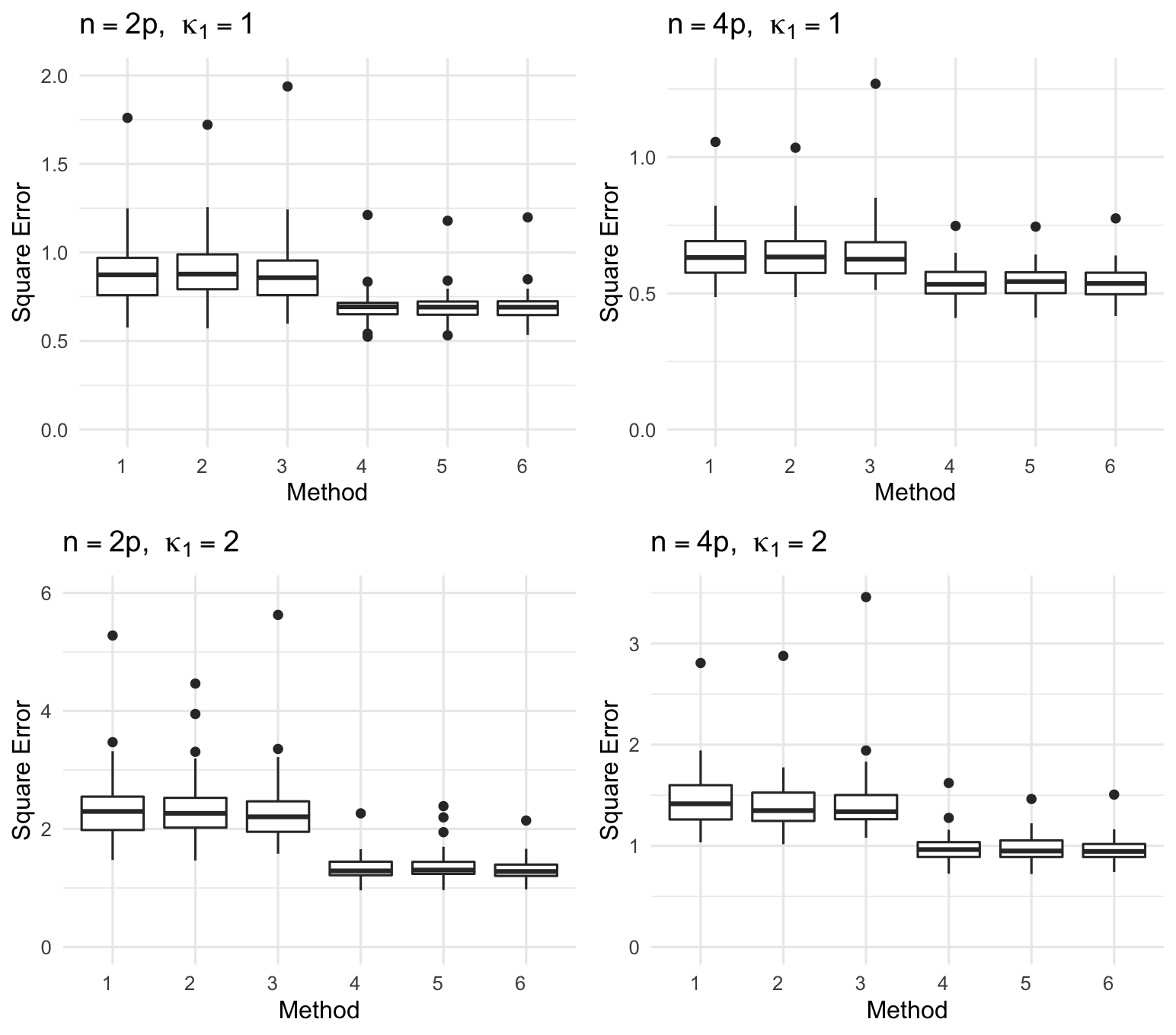}
    \caption{The box plot is constructed from 50 independent trials.
    The x-axis, labeled {1,2,3,4,5,6}, represents different estimators. Estimators 1 to 3 are based on noninformative synthetic data, specifically SLCV, SESE, and STSE; estimators 4 to 6 are based on informative auxiliary data, specifically SLCV(I), SESE(I),  and STSE(I).   }
    \label{fig:compare_diff_method_MSE}
\end{figure}

\subsection{Illustrations of negative transfer}
\label{app:negative_transfer}

The asymptotic characterization in \Cref{thm:exact_cat_M_MAP_informative} remains valid for any similarity parameter $\xi \in (-1,1)$.
In many applications where auxiliary data are from a source population similar to the target population, the regime $\xi \in [0,1)$ is more relevant.
When $\xi<0$, the source signal is anti-aligned with the target signal, so the auxiliary data may become harmful rather than helpful.
This subsection provides a numerical illustration of this phenomenon.

We use the same simulation setting as
in \Cref{sec:num_verify_infor_syn},
 except that the similarity parameter $\xi$ varies over both positive and negative values.
For each value of $\xi$ and each choice of $\tau_0=\tau/n \in \{0.5,1,2,5\}$, we compute the SRE $\widehat{\bbeta}_M$ and record both the squared error
$\|\widehat{\bbeta}_M-\bbeta_0\|_2^2$
and the cosine similarity $\frac{\langle \widehat{\bbeta}_M,\bbeta_0\rangle}{\|\widehat{\bbeta}_M\|_2\|\bbeta_0\|_2}$.
The circles in \Cref{fig:effect_xi} are empirical averages over $50$ independent runs, and the solid curves are the corresponding theoretical predictions from the asymptotic formulas in \Cref{eq:exact_cat_M_MAP_informative_square_error,eq:exact_cat_M_MAP_informative_similarity}.

\Cref{fig:effect_xi} shows close agreement between theory and simulation throughout the whole range of $\xi$.
Moreover, in this experiment, the squared error is consistently larger when $\xi<0$ than at the benchmark value $\xi=0$, which corresponds to a noninformative source.
At the same time, the cosine similarity is smaller for $\xi<0$, and for larger values of $\tau_0$ it can even become negative.
This indicates that when the source signal points in an opposite direction to the target signal, borrowing information from the source may deteriorate estimation accuracy.
In this sense, $\xi<0$ leads to negative transfer.

Therefore, although the theory applies more generally, the regime $\xi \in [0,1)$ is the most relevant one when the goal is to transfer information from a similar source.
In practice, the estimator of $\xi$ developed in \Cref{sec:est_xi1} can be used to assess whether such anti-alignment may be present.

\begin{figure}[H]
    \centering
    \includegraphics[width=0.85\linewidth]{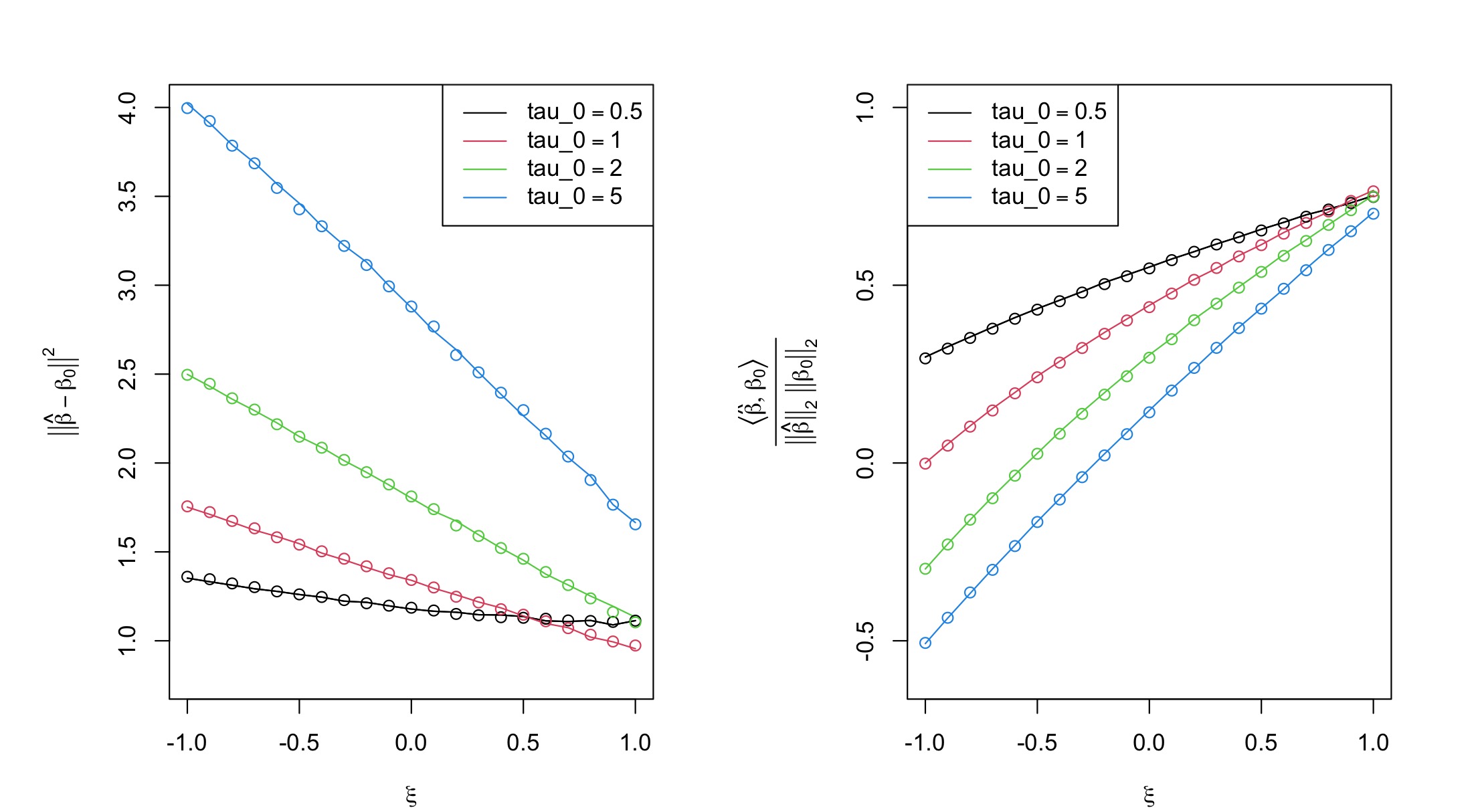}
    \caption{Effect of the similarity parameter $\xi$ for several values of $\tau_0=\tau/n$. Left: squared error; right: cosine similarity.
Circles: empirical averages over $50$ independent runs; solid curves: theoretical predictions.
    When $\xi<0$, the source signal is anti-aligned with the target signal, and the squared error is larger than at the benchmark $\xi=0$.}
    \label{fig:effect_xi}
\end{figure}

\section{Extension to generalized linear models (GLMs)}\label{supp:extension_GLM_section}

In this section, we extend the theoretical results developed in \Cref{sec:properties,sec:linear_asymptotic_regime} from the logistic regression model to the generalized linear model (GLM) with the canonical link.
Let $Y \in \mathcal{Y} \subset \mathbb{R}$ be a real-valued (response) variable and $\bX$ be a covariate vector of dimension $p$.
The conditional density of $Y$ given $\bX$ is assumed to be
\begin{equation}\label{GLM_distribution}
p_G(y \mid \bX,\bbeta_0)=h\left(y \right) \exp \left(y
 \bX^{\top} \boldsymbol{\beta}_0-\rho\left(\bX^{\top} \boldsymbol{\beta}_0\right)\right), y\in \mathcal{Y},
\end{equation}
where $\rho(\theta)$ and $h(y)$ are Borel functions associated with a particular GLM. Here, we consider a broader case rather than restricting $\rho(\theta)$ to $\log (1+\exp (\theta))$ as in the logistic regression setting.
Under a catalytic prior with some synthetic data, the SRE for this GLM is given by
\begin{equation}\label{GLM:cat_betahat}
\widehat{\boldsymbol{\beta}}_{M}^G=\arg \min _{\boldsymbol{\beta} \in \mathbf{R}^p} \sum_{i=1}^n \ell_G(Y_i,\bX_i^\top \boldsymbol{\beta})+\frac{\tau}{M}\sum_{i=1}^M \ell_G(Y^*_i,\bX_i^{*\top} \boldsymbol{\beta}) ,
\end{equation}
where $\ell_G(y, \theta):=\rho(\theta)-y \theta$ denotes the negative log-density and the subscript (superscript) refers to GLM.
Similarly, the pSRE with infinite synthetic data is given by
\begin{equation}\label{GLM:cat_betahat_Minfty}
    \widehat{\boldsymbol{\beta}}_{\infty}^G=\arg \min _{\boldsymbol{\beta} \in \mathbf{R}^p} \sum_{i=1}^n \ell_G(Y_i,\bX_i^\top \boldsymbol{\beta})+\tau
\mathbb E\left[\ell_G(Y^*,\bX^{*\top} \boldsymbol{\beta}) \right],
\end{equation}
where the expectation is taken over the synthetic data-generating distribution.

To present our theoretical result, we begin with some conditions on the model.
\begin{condition}\label{GLM_cond:link}
The density function of the GLM satisfies the following:
\begin{enumerate}
    \item For any $y\in \mathcal{Y}$ and $\bbeta\in \mathbb R^p$,  $p_G(y\mid \bX,\bbeta)\leq C_1$ for some universal constant $C_1$.
    
    \item
    For any $y\in \mathcal{Y}$ and $\theta\in \mathbb{R}$,
    $|\partial_\theta \ell_{G}(y,\theta)|\leq L_g$.

    \item
    $\rho$ is twice continuously differentiable and strictly convex.
    For any positive value $B$, there exists $c_\rho(B)>0$ such that $\rho^{\prime\prime}(\theta)$ is lower bounded by $c_\rho(B)$ for all $|\theta|\leq B$.
\end{enumerate}
\end{condition}

\begin{remark}
The requirements in \Cref{GLM_cond:link} are mild and commonly adopted in theoretical analysis on GLMs, as seen in \cite{van2008high,fan2010sure,huang_catalytic_2020}. The first requirement states that the probability density function should be bounded.  The second and third requirements generalize the properties of the log-likelihood function and log partition function, respectively, in logistic regression.

\end{remark}

For the synthetic data generation, we impose the following conditions.
\begin{condition}
\label{GLM_condition:synthetic_data}
The synthetic data are i.i.d. copies of $(\boldsymbol{X}^*, Y^*)$ such that the following statements hold:
	\begin{itemize}
	    \item The synthetic covariate vector $\boldsymbol{X}^*=\left(X_1^*, X_2^*, \cdots, X_p^*\right)$ satisfies \Cref{conditions:synthetic_X_Y} (C1)--(C3).

		\item For the synthetic response $Y^*$,  there are some constants $q\in (0,1/2] $ and $\varsigma>0$ such that
  
  $\min\{\mathbb P(Y^*\geq \rho^{\prime}(0)+\varsigma\mid \bX^*),\mathbb P(Y^*\leq \rho^{\prime}(0)-\varsigma\mid \bX^*)\}\geq q$.
	\end{itemize}
\end{condition}
\Cref{GLM_condition:synthetic_data} is an extension of \Cref{conditions:synthetic_X_Y} with no difference in the generation of synthetic covariates.
The requirement on the generation of responses ensures that synthetic responses do not highly skew towards one side of the domain $\mathcal{Y}$.
In logistic regression, this requirement becomes the same as in \Cref{conditions:synthetic_X_Y} if we take $\varsigma=0.5$.

We first present our results and defer the proofs in \Cref{proof_GLM_supp}.

\subsection{Existence of the SRE for GLM}

\begin{proposition}\label{GLM_thm:MAP_uniqueness}
		Assume \Cref{GLM_condition:synthetic_data} holds and there exists a positive constant $c_0$ such that
        $$
        \inf_{\bbeta\in\mathbb R^p, \|\bbeta\|>c_0} \frac{\tau}{M}\sum_{i=1}^M \ell_G(Y^*_i,\mathbf{X}_i^{*\top} \boldsymbol{\beta}) > (n+\tau)\rho(0)
        $$
        and
        $$
        \inf_{\bbeta\in\mathbb R^p, \|\bbeta\|>c_0} \sum_{i=1}^n \ell_G(Y_i,\mathbf{X}_i^{\top} \boldsymbol{\beta})>0 .
        $$

 Additionally, assume that the synthetic covariate matrix has full column rank. Under these conditions, the SRE in \eqref{GLM:cat_betahat} exists and is unique.
\end{proposition}

\subsection{Consistency of SRE when p diverges for GLM}

\begin{condition}\label{condition:GLM_moment_bound}
     $\{Y_i,\bX_i\}_{i=1}^n$ are independent and
    $\mathbb E [\operatorname{Var}(Y_i|\bX_i)\|X_i \|^2]\leq C_2 p$ for all $i\in [n]$.
\end{condition}

\begin{proposition}
	\label{thm:GLM_post_mode_consistency}
 Suppose $p/n\rightarrow 0$ and the tuning parameter is chosen such that $\tau \leq C_4 p$ for a constant $C_4$. 
 Under \Cref{condition:SubgaussianX,GLM_cond:link} and \Cref{condition:GLM_moment_bound}, the following statements hold:
 \begin{enumerate}[label=(\roman*)]
     \item
     Suppose the synthetic covariate matrix is of full rank and there is a constant $\Lambda$ such that
     $\| \frac{1}{M}\sum_{i=1}^{M} \boldsymbol{X}_i^*\boldsymbol{X}_i^{*\top }\| \leq \Lambda$,
     then
     $$\|\widehat{\boldsymbol{\beta}}^G_{M}-\boldsymbol{\beta}_0\|^2=O_p\left(\frac{p}{n}\right)$$
     \item Under covariate condition in \Cref{GLM_condition:synthetic_data}, we have
     $$ \|\widehat{\boldsymbol{\beta}}^G_{\infty }-\boldsymbol{\beta}_0\|^2=O_p\left(\frac{p}{n}\right)$$
 \end{enumerate}
\end{proposition}

\subsection{Nonasymptotic results in the linear asymptotic regime for GLM}
\begin{proposition}\label{thm:GLM_MAP_bounded}
Suppose \Cref{GLM_cond:link,GLM_condition:synthetic_data,condition:sufficient_regualrization} hold and $p > \omega_1 n$ for some positive constant $\omega_1$. Let $C_*$ be the constant $1+c_*\omega_1$. Assume  $\max\{\frac{1}{n}\sum_{i=1}^n \ell_G(Y_i,0), \frac{1}{M}\sum_{i=1}^M \ell_G(Y^*_i,0)\}\leq C_y$.
  Then,  the following statements hold:
\begin{enumerate}[label=(\roman*)]
    \item The estimator defined in the optimization \eqref{GLM:cat_betahat_Minfty} satisfies that
    $$\|\widehat{\bbeta}^G_{\infty}\|_2\leq \frac{C_*C_y}{\varsigma\eta_0 \nu} $$
	where $\eta_0,\nu$ are some positive constants that only depend on $(\kappa_{-},\kappa_{+},K_X,q)$ in \Cref{GLM_condition:synthetic_data}.
 \item The estimator defined in the optimization \eqref{GLM:cat_betahat} satisfies that
$$\|\widehat{\bbeta}^G_{M}\|_2\leq \frac{4C_*C_y}{\varsigma\eta_0 \nu}$$
with probability at least $1-2\exp(-\tilde{c}M)$ if  $M\geq \tilde{C}p$,
where $\tilde{c},\tilde{C}, \eta_0,\nu$ are positive constants that only depend on the constants $L_g$ and $(\kappa_{-},\kappa_{+},K_X,q)$ in \Cref{GLM_condition:synthetic_data}.
\end{enumerate}
\end{proposition}

\subsection{Stability of SRE due to finite $M$ for GLM}
\label{sec:GLM_stability _MAP_finite_M}

We study the influence of the synthetic sample size $M$ on the stability of the SRE.
Specifically, we establish a bound on the distance between the estimate $\widehat{\boldsymbol{\beta}}^G_{M}$ based on $M$ synthetic samples defined in \eqref{GLM:cat_betahat} and the estimate
$\widehat{\boldsymbol{\beta}}^G_{\infty}$ based on an infinite amount of synthetic data defined in \eqref{GLM:cat_betahat_Minfty}.
This bound decays to 0 linearly in $M$.

For the purpose here, we treat the observed data as fixed and consider the synthetic data the only source of randomness. For any $K>0$, we define $\mathcal B_K := \{\bbeta\in \mathbb R^p: \|\bbeta \|_2\leq K\}$.
Let $\widehat{\boldsymbol{\beta}}_{M}^{G,(K)}$ and   $\widehat{\boldsymbol{\beta}}_{\infty}^{G,(K)}$ be the constrained version of the SRE and pSRE over $\mathcal B_K$.

\begin{proposition}\label{GLM_thm:stability_finite_M}
Suppose that $\tau>0$ and the following holds
\begin{enumerate}[label=(\alph*)]
    \item the synthetic data are generated according to \Cref{GLM_condition:synthetic_data};
    \item  \Cref{GLM_cond:link} holds. 
\end{enumerate}
Let $\lambda_{n,K}\geq 0$ be a constant such that for any $\bbeta \in \mathcal B_K$, the smallest eigenvalue of $\sum_{i=1}^n \rho^{\prime\prime}(\bX_i^\top \bbeta)\bX_i \bX_i^\top$   is lower bounded by $\lambda_{n,K}$.
Then, the following statements hold:
 \begin{enumerate}[label=(\roman*)]
     \item
     There is a positive constant $\gamma$ that only depends on the constants $K$ and $\kappa_{-},\kappa_{+},K_X$ in \Cref{GLM_condition:synthetic_data} such that
     the smallest eigenvalue of $\mathbb{E}\left( \rho^{\prime\prime}(\bX^{*\top}\bbeta)\bX^{*}\bX^{*\top}  \right)$ is lower bounded by $\gamma$ for all $\bbeta \in \mathcal B_K$.
\item For any $\epsilon\in (0,1)$, it holds with probability at least $1-\epsilon$ that
\[\|\widehat{\boldsymbol{\beta}}^{G,(K)}_{M}-\widehat{\boldsymbol{\beta}}^{G,(K)}_{\infty}\|_2 \leq \frac{\tau C_1}{ \lambda_{n,K}+\tau \gamma/2 } \sqrt{\frac{p+\log(4/\epsilon)}{M}}.
	\]
where $C_1$ and $\gamma$ depend on $\kappa_{-},\kappa_{+},L_g,K$ and $K_X$ only.  In particular, since $\lambda_{n,K}\geq 0$, we have $\|\widehat{\boldsymbol{\beta}}^{G,(K)}_{M}-\widehat{\boldsymbol{\beta}}^{G,(K)}_{\infty}\|^{2} = O_p
\left( \frac{p}{M \gamma^2}\right)$.
 \end{enumerate}
\end{proposition}

\subsection{Exact asymptotics in the linear asymptotic regime for GLM}

\begin{proposition}\label{proposition:GLM_exact_cat_M_MAP_informative}
    Consider the optimization program \eqref{GLM:cat_betahat}, under \Cref{condition:proper_scaling}, and assume that the solution of the optimization program \eqref{GLM:cat_betahat} lies in a compact set. Assume $\left\{\boldsymbol{X}_i\right\}_{i=1}^n$ and $\left\{\boldsymbol{X}^*_i\right\}_{i=1}^M$ are i.i.d.\ samples from $\mathcal{N}\left(\mathbf{0}, \mathbf{I}_p\right)$. Let the responses $Y_i, Y_i^*$ be generated according to the GLM \eqref{GLM_distribution} with linear predictors $\boldsymbol{X}_i^{\top}\boldsymbol{\beta}_0$ and $\boldsymbol{X}_i^{*\top}\boldsymbol{\beta}_s$, respectively.
    Assume that $\Pi_2$ is a distribution on $\mathbb{R}$ with $\mathbb{E}_{\Pi_2}[\beta^2]=\kappa_1^2$ and that the empirical distribution of the entries of $\sqrt{p}\boldsymbol{\beta}_0 $ converges weakly to $\Pi_2$, i.e., $\frac{1}{p}\sum_{j=1}^p\noverpmass_{\sqrt{p}\beta_{0j} } \rightarrow \Pi_2$.
    Additionally, assume that
    $\lim_{p\rightarrow\infty}\|\boldsymbol{\beta}_0\|^2 = \kappa_1^2$, $\lim_{p\rightarrow\infty}\|\boldsymbol{\beta}_s\|^2 = \kappa_2^2$, and $\lim_{p\rightarrow\infty}\frac{1}{\|\boldsymbol{\beta}_0\|\|\boldsymbol{\beta}_s\|}\langle \boldsymbol{\beta}_0, \boldsymbol{\beta}_s \rangle = \xi \in [0,1)$. Then, as $p \rightarrow \infty$, for any locally-Lipschitz function $\Psi: \mathbb{R} \times \mathbb{R} \rightarrow \mathbb{R}$, we have,
    \[
    \frac{1}{p}\sum_{j=1}^p \Psi\left(\sqrt{p}[\widehat{\boldsymbol{\beta}}_{M,j}^G-\alpha_1^*\boldsymbol{\beta}_{0,j}-\frac{\alpha_2^*}{\sqrt{1-\xi^2}}(\bbeta_{s,j}-\xi \frac{\kappa_2}{\kappa_1}\bbeta_{0,j})] ,\sqrt{p} \boldsymbol{\beta}_{0,j}\right) \xrightarrow{\mathbb{P}} \mathbb{E}\left[\Psi\left( \sigma^*Z, \beta\right)\right]
    \]
    where $Z \sim \mathcal{N}(0,1)$ is independent of $ \beta  \sim \Pi_2$, and $(\alpha_1^*, \alpha_2^*, \sigma^*)$ depend on the GLM \eqref{GLM_distribution} and parameters $\varsigma, \kappa_1, \kappa_2, \tau_0, \xi$.
\end{proposition}

\section{Proofs}\label{supp:sec:proof}

\subsection{Preliminaries and basic properties}

In the following, we assume $K_X\geq 1$; otherwise we replace $K_X$ by $\max\{K_X,1\}$ and preserve \Cref{conditions:synthetic_X_Y}(C3).

The following lemma consists of some standard results; see, for example, \cite{vershynin2010introduction}.
\begin{lemma}[Standard Orlicz and Bernstein facts]
\label{lem:standard_psi_Bernstein_facts}
There exist universal constants $c,C>0$ such that:

\begin{enumerate}
\item[(i)] If $W$ is sub-gaussian, then $W^2$ is sub-exponential and
\[
\|W^2\|_{\psi_1}\le C \|W\|_{\psi_2}^2.
\]

\item[(ii)] If $W$ is sub-gaussian, then for every $q\ge 1$,
\[
\bigl(\mathbb{E}|W|^q\bigr)^{1/q}\le C \sqrt{q}\,\|W\|_{\psi_2}.
\]
In particular, $\mathbb{E}|W| \le C\|W\|_{\psi_2}$ and $\bigl(\mathbb{E}W^4\bigr)^{1/4}\le C\|W\|_{\psi_2}$.

\item[(iii)] If $Y_1,\ldots,Y_M$ are independent, mean zero, and satisfy $\|Y_i\|_{\psi_1}\le K$ for all $i$, then for every $t>0$,
\[
\mathbb{P}\!\left(\left|\frac{1}{M}\sum_{i=1}^M Y_i\right|\ge t\right)
\le
2\exp\!\left(-cM\min\left\{\frac{t^2}{K^2},\frac{t}{K}\right\}\right).
\]
\end{enumerate}
\end{lemma}

\bigskip

Our proof relies on properties of the synthetic data generating distribution.
In the supplementary material of \citet{huang_catalytic_2020}, Theorem 5.7 establishes those properties under more restricted conditions on the synthetic covariates, that is, independently distributed and uniformly bounded coordinates.
To relax their condition, we have to first establish similar results to their Propositions 5.11 and 5.12. In particular, we have the following lemma and propositions.

\begin{lemma}\label{lem:synthetic_X_4th_moment}
   
Under \Cref{conditions:synthetic_X_Y} (C3), it holds for every vector $\boldsymbol{u}\in \mathbb{R}^{p-1}$ that
$$
\begin{aligned}
\mathbb{E}\bigl(\boldsymbol{u}^{\top}\widetilde{\boldsymbol{X}}^*\bigr)^4 \leq \mu_4   \left[\mathbb{E}\bigl(\boldsymbol{u}^{\top}\widetilde{\boldsymbol{X}}^*\bigr)^2\right]^2,
\end{aligned}
$$
where $\mu_4:=2 K_X^4$.
\end{lemma}
This is a standard result for sub-Gaussian variables and its proof is omitted.

The following propositions are restatements of the three parts in  \Cref{prop:properties_synthetic_X}.

\begin{proposition}[Part 1 in \Cref{prop:properties_synthetic_X}]
\label{prop:design_op_norm}
Assume \Cref{conditions:synthetic_X_Y} (C1)--(C3).
Let $\widetilde{\mathbb{X}}^*\in\mathbb{R}^{M\times (p-1)}$ have i.i.d. rows
$\widetilde{\boldsymbol{X}}_1^{*\top},\ldots,\widetilde{\boldsymbol{X}}_M^{*\top}$ and set $d:=p-1$.
Then there exists a universal constant $C>0$ such that for all $t\ge 0$,
$$
\begin{aligned}
\|\mathbb{X}^*\|_{\mathrm{op}} \leq \sqrt{M}+ \kappa_{+}^{1/2}\left[\sqrt{M}+C K_X^2(\sqrt{d}+t)\right]
\end{aligned}
$$
with probability at least $1-2\exp(-t^2)$.

\end{proposition}

\begin{proposition}[Part 2 in \Cref{prop:properties_synthetic_X}]
    \label{prop:var_small_ball}
Assume \Cref{conditions:synthetic_X_Y} (C1)--(C3).
Then the following statements hold:

\begin{enumerate}
\item[(i)] For any $\boldsymbol{\beta}=(\beta_1,\widetilde{\boldsymbol{\beta}})\in\mathbb{R}^p$ with
$\|\boldsymbol{\beta}\|_2=1$,
$$
\begin{aligned}
\mathrm{Var}\!\left(\boldsymbol{X}^{*\top}\boldsymbol{\beta}\right)
= \mathrm{Var}\!\left(\widetilde{\boldsymbol{X}}^{*\top}\widetilde{\boldsymbol{\beta}}\right)
\leq \kappa_+ .
\end{aligned}
$$

\item[(ii)] Define
$$
\begin{aligned}
\eta_0:=\sqrt{\frac{\min\{1,\kappa_-\}}{2}},
\qquad
\rho_0:=\frac{\min\{1,\kappa_-\}^2}{32(1+\mu_4 \kappa_{+}^2)}.
\end{aligned}
$$
Then for every $\boldsymbol{\beta}\in\mathbb{R}^p$ with $\|\boldsymbol{\beta}\|_2=1$,
$$
\begin{aligned}
\mathbb{P}\left(\left|\boldsymbol{X}^{*\top}\boldsymbol{\beta}\right|>\eta_0\right)\geq \rho_0 .
\end{aligned}
$$

\item [(iii)]
 There exists a constant $r_0>0$ depending only on
$(\kappa_-,\kappa_+,K_X)$ such that if $M\ge r_0p$, then with probability at least
\[
1-2\exp\{-M\min(1,\rho_0^2/4)\},
\]
the synthetic covariate matrix $\mathbb X^*$ has full column rank and
\[
\inf_{\|\boldsymbol{\beta}\|_2=1}
\frac{1}{M}\sum_{i=1}^M
\left|\boldsymbol{X}_i^{*\top}\boldsymbol{\beta}\right|
\geq \frac{\eta_0\rho_0}{4}.
\]
\end{enumerate}
\end{proposition}

\begin{proposition}[Part 3 in \Cref{prop:properties_synthetic_X}]\label{prop:nonseparability}

   Suppose \Cref{conditions:synthetic_X_Y} holds. There exist positive constants $r_1$ and $c_1$ depending only on $q$,
    such that if $M\geq r_1 p$, then the synthetic data $\{(\boldsymbol {X}^*_i,{Y}^*_i)\}_{i=1}^M$ are not separable with probability at least $1-2e^{-c_1 M}$.
\end{proposition}

The rest of this subsection is devoted to proving the above propositions.

\begin{proof}[Proof of ~\Cref{prop:var_small_ball}]
(i) By \Cref{conditions:synthetic_X_Y}(C1),
$\mathbb{E}(\widetilde{\boldsymbol{X}}^*)=\mathbf{0}$, hence
$\mathbb{E}(\boldsymbol{X}^{*\top}\boldsymbol{\beta})=\beta_1$ and
$$
\begin{aligned}
\mathrm{Var}\!\left(\boldsymbol{X}^{*\top}\boldsymbol{\beta}\right)
&=\mathbb{E}\Bigl(\boldsymbol{X}^{*\top}\boldsymbol{\beta}-\beta_1\Bigr)^2
= \mathbb{E}\Bigl(\widetilde{\boldsymbol{X}}^{*\top}\widetilde{\boldsymbol{\beta}}\Bigr)^2 \\
&=\widetilde{\boldsymbol{\beta}}^{\top}\boldsymbol{\Sigma}^* \widetilde{\boldsymbol{\beta}}
\leq \lambda_{\max}(\boldsymbol{\Sigma}^*)\|\widetilde{\boldsymbol{\beta}}\|_2^2
\leq \kappa_+,
\end{aligned}
$$
where the last inequality is due to \Cref{conditions:synthetic_X_Y} (C2).

(ii) Fix any $\boldsymbol{\beta}=(\beta_1,\widetilde{\boldsymbol{\beta}})$ with $\|\boldsymbol{\beta}\|_2=1$.
Let $S:=\boldsymbol{X}^{*\top}\boldsymbol{\beta}=\beta_1+\widetilde{\boldsymbol{X}}^{*\top}\widetilde{\boldsymbol{\beta}}$.

First, since $\mathbb{E}(\widetilde{\boldsymbol{X}}^{*\top}\widetilde{\boldsymbol{\beta}})=0$, we have
$$
\begin{aligned}
\mathbb{E}(S^2)
&=\beta_1^2 + 2\beta_1\,\mathbb{E}(\widetilde{\boldsymbol{X}}^{*\top}\widetilde{\boldsymbol{\beta}})
+ \mathbb{E}\bigl(\widetilde{\boldsymbol{X}}^{*\top}\widetilde{\boldsymbol{\beta}}\bigr)^2 \\
&=\beta_1^2 + \widetilde{\boldsymbol{\beta}}^{\top}\boldsymbol{\Sigma}^*\widetilde{\boldsymbol{\beta}}\\
&\geq \beta_1^2 + \kappa_-\|\widetilde{\boldsymbol{\beta}}\|_2^2 \\
&=\beta_1^2 \cdot 1 + (1-\beta_1^2) \cdot \kappa_- \\
&\geq \min\{1,\kappa_-\},
\end{aligned}
$$
where the first inequality is due to \Cref{conditions:synthetic_X_Y} (C2) and the last inequality is due to the convex combination.

Second, by $(a+b)^4\leq 8(a^4+b^4)$,
$$
\begin{aligned}
\mathbb{E}(S^4)
&\leq 8\beta_1^4 + 8\,\mathbb{E}\bigl(\widetilde{\boldsymbol{X}}^{*\top}\widetilde{\boldsymbol{\beta}}\bigr)^4.
\end{aligned}
$$
If $\widetilde{\boldsymbol{\beta}}=\mathbf{0}$, then $S\equiv \beta_1$ and
$\mathbb{P}(|S|>\eta_0)=1\geq \rho_0$.
Otherwise, by \Cref{lem:synthetic_X_4th_moment},
$$
\begin{aligned}
\mathbb{E}\bigl(\widetilde{\boldsymbol{X}}^{*\top}\widetilde{\boldsymbol{\beta}}\bigr)^4
\leq \mu_4   \left[\mathbb{E}\bigl(\widetilde{\boldsymbol{\beta}}^{\top}\widetilde{\boldsymbol{X}}^*\bigr)^2\right]^2 \leq \mu_4 \left(\kappa_{+} \|\widetilde{\boldsymbol{\beta}}\|^2\right)^2 \leq \mu_4 \kappa_{+}^2.
\end{aligned}
$$

Note that $\beta_1^4\leq 1$. Hence $\mathbb{E}(S^4)\leq 8(1+\mu_4 \kappa_{+}^2)$.

Now apply the Paley--Zygmund inequality to $Z:=S^2$ with $\theta=1/2$:
$$
\begin{aligned}
\mathbb{P}\bigl(S^2 > \tfrac12\,\mathbb{E}(S^2)\bigr)
\geq \frac{1}{4}\cdot \frac{\mathbb{E}(S^2)^2}{\mathbb{E}(S^4)} .
\end{aligned}
$$
Since $\sqrt{\tfrac12\,\mathbb{E}(S^2)}\geq  \eta_0$, we have
$$
\begin{aligned}
\mathbb{P}\bigl(|S|>\eta_0\bigr)
\geq \frac{1}{4}\cdot \frac{\min\{1,\kappa_-\}^2}{8(1+\mu_4 \kappa_{+}^2)}
=\rho_0 .
\end{aligned}
$$

(iii) The proof follows directly from Lemma 5.8 in the supplement of \cite{huang_catalytic_2020}, which uses a concentration inequality for the sum of Bernoulli variables $\xi_i=\mathbf{1}\{\left|\boldsymbol{X}_i^{\top} \boldsymbol{\beta}\right|>\eta\}$ and a standard net argument.
The detail is omitted.

\end{proof}

\begin{proof}[Proof of \Cref{prop:design_op_norm}]
Consider the isotropic rows
$\widetilde{Z}_i:=\left(\boldsymbol{\Sigma}^{*}\right)^{-1/2}\widetilde{\boldsymbol{X}}_i^*$ and the transformed matrix $\widetilde{\mathbb{Z}}=\widetilde{\mathbb{X}}\left(\boldsymbol{\Sigma}^{*}\right)^{-1/2}$.

\Cref{conditions:synthetic_X_Y} (C3) implies that $\|\widetilde{Z}_i\|_{\psi_2}\leq K_X$.
Applying \cite[Theorem~4.6.1]{vershynin2018high}, there exists an absolute constant $C$ such that
$$
 \|\widetilde{\mathbb{Z}}\|_{\mathrm{op}} \leq \sqrt{M}+C K_X^2(\sqrt{d}+t)
$$
with probability at least $1-\exp(-t^2)$ for any $t\geq 0$.

Since the first column (for the intercept term) has operator norm $\|\mathbf{1}\|_2=\sqrt{M}$,
the desired result is proved by using
$$
\begin{aligned}
\|\mathbb{X}^*\|_{\mathrm{op}} & \leq \sqrt{M} + \|\widetilde{\mathbb{X}}^*\|_{\mathrm{op}}\\
&\leq \sqrt{M} + \|\left(\boldsymbol{\Sigma}^{*}\right)^{1/2}\|_{\mathrm{op}}  \|\widetilde{\mathbb{Z}}\|_{\mathrm{op}}
\end{aligned}
$$
\end{proof}

\begin{proof}[Proof of \Cref{prop:nonseparability}]

Without loss of generality, assume $M\geq p$.

Given the covariate vectors $\boldsymbol {X}^*_i$, let $\mathcal{S}_{\mathbb{X}^*}$ be the set of labelings in $\{0,1\}^M$ that are separable by a homogeneous hyperplane in $\mathbb{R}^p$.
By Function-Counting Theorem \citep[Theorem 1]{cover1965geometrical} and noting that points that are not in general position will only have a smaller number of possible labelings, we have
$$
\left|\mathcal{S}_{\mathbb{X}^*}\right|\leq C(M, p)=2 \sum_{k=0}^{p-1}\binom{M-1}{k}.
$$

For any labeling $\{y_i^*\}_{i=1}^M \in\{0,1\}^M$, \Cref{conditions:synthetic_X_Y}(C4) implies that for each $j$,
$$\mathbb{P}\left(Y_j^*=y_j^* \mid \{\bX_{i}^*\}_{i=1}^M\right) \leq 1- q.
$$
By \Cref{conditions:synthetic_X_Y},
the synthetic responses $\{Y_j^*\}_{j=1}^M$ are mutually conditionally independent given $\{\bX_{i}^*\}_{i=1}^M$.
Since $1-q\leq e^{-q}$, we have
$$
\begin{aligned}
\mathbb{P}\left( \text{ data are separable} \mid \{\bX_{i}^*\}_{i=1}^M\right) & = \sum_{\{y_i^*\}_{i=1}^M\in \mathcal{S}_{\mathbb{X}^*}} \mathbb{P}\left(Y_j^*=y_j^*, j\in [M] \mid \{\bX_{i}^*\}_{i=1}^M\right) \\
& \leq \left|\mathcal{S}_{\mathbb{X}^*}\right| (1-q)^M\\
& \leq 2 e^{-qM} \sum_{k=0}^{p-1}\binom{M-1}{k}.
\end{aligned}
$$

Let $\alpha=(p-1)/(M-1)$ and assume $\alpha\in (0, 1/2]$.
Recall $H(\alpha)=-\alpha \log \alpha-(1-\alpha) \log (1-\alpha)$
We can see the following holds:
$$
\sum_{k=0}^{p-1}\binom{M-1}{k} \leq e^{(M-1) H(\alpha)}.
$$
To see this, take $T=\{A \subseteq[M-1]:|A| \leq p-1\}$. Then $|T|=\sum_{i \leq p-1 }\binom{M-1}{i}$.
By \citet[Corollary~15.7.3]{alon2016probabilistic}, we have
$$
|T| \leq e^{\sum_{j=1}^{M-1} H\left(p_j\right)},
$$
where $p_j$ is the fraction of sets in $T$ containing $j$ ($j\in [M-1]$).
For this $T$, $p_j \leq \alpha$, and since $\alpha \leq 1 / 2$ we have $H\left(p_j\right) \leq H(\alpha)$. Therefore, $|T| \leq e^{(M-1) H(\alpha)}$.

Therefore,
$$
\begin{aligned}
\mathbb{P}\left( \text{ data are separable} \mid \{\bX_{i}^*\}_{i=1}^M\right)
&\le 2 \exp\!\left\{-qM + (M-1)H(\alpha)\right\}.
\end{aligned}
$$

Since $H(\cdot)$ is increasing on $(0,1/2]$, we can choose a constant $\alpha_q\in(0,1/2]$
such that $H(\alpha_q)\le q/2$.
Next, choose
$$
r_1 := \max\left\{2,\frac{1}{\alpha_q}\right\}.
$$
If $M\ge r_1 p$, then $M\ge 2p$, hence $\alpha=(p-1)/(M-1)\le 1/2$ and also
$$
\alpha=\frac{p-1}{M-1} \le \frac{p}{M}\le \frac{1}{r_1}\le \alpha_q.
$$
Consequently, $H(\alpha)\le H(\alpha_q)\le q/2$, and thus
$$
\begin{aligned}
\mathbb{P}\left( \text{ data are separable} \mid \{\bX_{i}^*\}_{i=1}^M\right)
&\le 2 \exp\!\left\{-qM + (M-1)\cdot \frac{q}{2}\right\} \\
&\le 2 e^{-(q/2)M}.
\end{aligned}
$$

Taking expectation over $\{\bX_i^*\}_{i=1}^M$ yields the desired conclusion with $c_1=q/2$.

\end{proof}

\subsection{Proof of Theorem \ref{thm:MAP_uniqueness}}
\label{proof_sec:logitic_existence}

We prove the existence of the SRE estimate.
Recall for logistic regression $\rho(t)=\log(1+\exp(t))$.
Using the elementary identity that $yt-\rho(t)=-\log(1+e^{(1-2y)t})$ for $y\in \{0,1\}$ and $t\in \mathbb{R}$, we can express the SRE using following optimization problem:

{
\begin{align*}
    \widehat{\bbeta}_{M}&=\arg \max _{\bbeta\in \mathbb R^p} \sum_{i=1}^{n} \left(Y_{i} \bX_i^{\top} \bbeta-\rho(\bX_i^{\top} \bbeta)\right)+ \frac{\tau}{M} \sum_{i=1}^{M} \left(Y^*_{i} \bX_i^{*\top} \bbeta-\rho(\bX_i^{*\top} \bbeta)\right)\\
	&=\arg \min_{\bbeta\in \mathbb R^p}\underbrace{\sum_{i=1}^{n} \log\left(1+\exp(-(2Y_i-1)\bX_i^{\top} \bbeta) \right)+ \frac{\tau}{M} \sum_{i=1}^{M} \log\left(1+\exp(-(2Y_i^*-1)\bX_i^{*\top} \bbeta) \right)}_{\overset{\text{Def.}}{=}\ell(\bbeta)} .
\end{align*}
}

Note that $\ell(\mathbf{0}) = (n+\tau)\log 2$.
Our goal is to demonstrate that the norm of the optima is finite. For any $\be\in \mathbb{S}^{p-1}$, define
$$\kappa(\be):=\min_{i\in [M]} (2Y_i^*-1)\bX_i^{*\top} \be,$$
which is a continuous function over $\mathbb{S}^{p-1}$.
Given that the synthetic data set $\{(\boldsymbol {X}^*_i,
{Y}^*_i)\}_{i=1}^M$ is not separable, we have $\kappa(\be)<0$. Based on Extreme Value Theorem and compactness of $\mathbb{S}^{p-1}$, $\kappa(\be)$ attains its maximum over $\mathbb{S}^{p-1}$, denoted by $\iota$. We have $\iota<0$.

Let $c_0 = \frac{M}{\tau(-\iota)} \cdot (n+\tau)\log 2$.
For any $\bbeta_1 \in \mathbb{R}^p\backslash \{\mathbf{0}\}$, there exists some $j \in [M]$ such that $(2Y^*_{j}-1) \bX_j^{*\top} \bbeta_1<0$. Take $\tilde j$ such that $(2Y^*_{\tilde j}-1) \bX_{\tilde j}^{*\top} \bbeta_1= \min_{j}(2Y^*_{j}-1) \bX_j^{*\top} \bbeta_1\leq \iota<0$.
For any $c>c_0$, we have
\begin{align*}
    \ell(c\bbeta_1/\|\bbeta_1\|_2) &>
\frac{\tau}{M}\log\left(1 + \exp\left[-c(2Y_{\tilde j}^*-1)\bX_{\tilde j}^{*\top} \frac{\bbeta_1}{\|\bbeta_1\|_2}\right]\right) \\
&>c\frac{\tau}{M}[-(2Y_{\tilde j}^*-1)\bX_{\tilde j}^{*\top} \frac{\bbeta_1}{\|\bbeta_1\|_2}|]\\
&>c\frac{\tau}{M}[-\iota]\\
&> (n+\tau)\log 2 = \ell(\mathbf{0}),
\end{align*}
where the first two inequalities are due to $\log(1+\exp(t))\geq \max(0,t)$ for all $t\in \mathbb{R}$, the third inequality is due to the definition of $\iota$ and $\tilde{j}$, and the last inequality is because $c>c_0$.
This suggests that the trivial estimator $\mathbf{0}$ results in a smaller loss compared to any other $\bbeta_1$ with norm larger than $c_0$. Therefore, the norm of the optima must be no larger than $c_0$.

The uniqueness of the optima is guaranteed by the strict convexity of $\ell(\bbeta)$, which can be verified straightforwardly by confirming that the Hessian matrix of $\ell(\bbeta)$ is positive definite since the synthetic covariate matrix is full rank.

\subsection{Proof of Theorem \ref{thm:stability_finite_M} and Proposition~\ref{prop:strong_convex_SRE}}
\label{proof_sec:logitic_stability}

We prove the stability of the SRE with respect to the finite synthetic sample size $M$.

Since the proof of \Cref{GLM_thm:stability_finite_M} is essentially the same, we present a unified argument for logistic regression and general GLMs.
For general GLMs, we additionally assume \Cref{GLM_cond:link} and introduce the constant $L_g$.
For logistic regression, \Cref{GLM_cond:link} holds automatically with $L_g=1$ and $c_\rho(t)=e^t/(1+e^t)^2$.

\subsubsection{Proof of Propositions~\ref{prop:strong_convex_SRE} and~\ref{GLM_thm:stability_finite_M}}

We present the full statement of \Cref{prop:strong_convex_SRE} and its proof.
\begin{proposition}[Lower curvature for synthetic Hessian]
\label{prop:lower_curvature}

Fix the radius $K>0$.
For logistic regression, assume \Cref{conditions:synthetic_X_Y} holds.
For general GLMs, assume both \Cref{GLM_condition:synthetic_data} and \Cref{GLM_cond:link} hold.
Let $D_K := K\sqrt{\frac{2(1+\kappa_+)}{\rho_0}}$. Recall the constants $\rho_0$ and $\eta_0$ in \Cref{prop:var_small_ball}.
Recall that $c_\rho(b)=\inf _{|t| \leq b} \rho^{\prime \prime}(t)>0$.
Define
\[
C_K := \frac{\rho_0\eta_0^2}{2}\, c_\rho(D_K).
\]
Define for $\bbeta\in \mathbb{R}^{p}$ that
\[
\boldsymbol{H}(\bbeta)
:=
\mathbb{E}\!\left(\rho^{\prime\prime}(\bX^{*\top}\bbeta)\bX^*\bX^{*\top}\right),
\qquad
\widehat{\boldsymbol{H}}_M(\bbeta)
:=
\frac{1}{M}\sum_{i=1}^M \rho^{\prime\prime}(\bX_i^{*\top}\bbeta)\bX_i^*\bX_i^{*\top}.
\]

Then:
\begin{enumerate}
\item[(a)]  For all $\bbeta\in\mathcal{B}_K$, it holds that $
\boldsymbol{H}(\bbeta)\succcurlyeq C_K\,\mathbf{I}_p$.

\item[(b)]  There exists a universal constant $C>0$ such that for any
$\epsilon\in(0,1)$, if
\[
M \ge C\,\frac{p+\log(1/\epsilon)}{\rho_0^2},
\]
then with probability at least $1-\epsilon$ (with respect to the synthetic sample),
\[
\inf_{\bbeta\in\mathcal{B}_K}\lambda_{\min}\bigl(\widehat{\boldsymbol{H}}_M(\bbeta)\bigr)
\ge \frac{1}{2}C_K
=
\frac{\rho_0\eta_0^2}{4}\,\rho^{\prime\prime}(D_K).
\]
\end{enumerate}
\end{proposition}

\begin{proof}[Proof of \Cref{prop:lower_curvature}]
We begin with two results.

\medskip
\noindent\textbf{Result 1: lower bounds via indicators.}
Fix $\bbeta\in\mathcal{B}_K$ and $\bv\in\mathbb{S}^{p-1}$.
Write
\[
\bv^\top \boldsymbol{H}(\bbeta)\bv
=
\mathbb{E}\!\left(\rho^{\prime\prime}(\bX^{*\top}\bbeta)(\bX^{*\top}\bv)^2\right),
\qquad
\bv^\top \widehat{\boldsymbol{H}}_M(\bbeta)\bv
=
\frac{1}{M}\sum_{i=1}^M \rho^{\prime\prime}(\bX_i^{*\top}\bbeta)(\bX_i^{*\top}\bv)^2.
\]

Based on part 3 of \Cref{GLM_cond:link},
on the event
$\{|\bX^{*\top}\bbeta|\le D_K\}$, we have $\rho^{\prime\prime}(\bX^{*\top}\bbeta)\ge c_\rho(D_K)>0$.
Furthermore, on $\{|\bX^{*\top}\bv|>\eta_0\}$ we have $(\bX^{*\top}\bv)^2\ge \eta_0^2$.
Therefore,
$$
\begin{aligned}
\rho^{\prime\prime}(\bX^{*\top}\bbeta)(\bX^{*\top}\bv)^2
&\ge
c_\rho(D_K)\eta_0^2\,
\mathbf{1}\bigl\{|\bX^{*\top}\bbeta|\le D_K,\ |\bX^{*\top}\bv|>\eta_0\bigr\}.
\end{aligned}
$$
Taking expectations yields
\begin{equation}\label{eq:population_hessian_fixed_lower}
    \begin{aligned}
\bv^\top \boldsymbol{H}(\bbeta)\bv
&\ge
c_\rho(D_K)\eta_0^2\,
\mathbb{P}\bigl(|\bX^{*\top}\bbeta|\le D_K,\ |\bX^{*\top}\bv|>\eta_0\bigr).
\end{aligned}
\end{equation}

Similarly, taking empirical averages yields
\begin{equation}\label{eq:empirical_hessian_fixed_lower}
    \begin{aligned}
\bv^\top \widehat{\boldsymbol{H}}_M(\bbeta)\bv
&\ge
c_\rho(D_K)\eta_0^2\,
\frac{1}{M}\sum_{i=1}^M
\mathbf{1}\bigl\{|\bX_i^{*\top}\bbeta|\le D_K,\ |\bX_i^{*\top}\bv|>\eta_0\bigr\}.
\end{aligned}
\end{equation}

\medskip
\noindent\textbf{Result 2: a uniform population lower bound for the intersection event.}
For any $\bbeta\in\mathcal{B}_K$ and $\bv\in\mathbb{S}^{p-1}$,
$$
\begin{aligned}
\mathbb{P}\bigl(|\bX^{*\top}\bbeta|\le D_K,\ |\bX^{*\top}\bv|>\eta_0\bigr)
&\ge
\mathbb{P}\bigl(|\bX^{*\top}\bv|>\eta_0\bigr)
-
\mathbb{P}\bigl(|\bX^{*\top}\bbeta|>D_K\bigr).
\end{aligned}
$$
By \Cref{prop:var_small_ball}, the first term is at least $\rho_0$.
Furthermore, write $\bbeta=(\beta_1,\widetilde{\bbeta})^\top$, we have
$$
\mathbb{E}[(\bX^{*\top}\bbeta)^2] = \beta_1^2 + \widetilde{\boldsymbol{\beta}}^{\top}\boldsymbol{\Sigma}^*\widetilde{\boldsymbol{\beta}} \leq K^2 +  \kappa_+ \|\widetilde{\boldsymbol{\beta}}\|^2 \leq K^2 (1+ \kappa_+).
$$
By Markov's inequality and the uniform variance bound,
$$
\begin{aligned}
\mathbb{P}\bigl(|\bX^{*\top}\bbeta|>D_K\bigr)
\le
\frac{\mathbb{E}[(\bX^{*\top}\bbeta)^2]}{D_K^2}
\le
\frac{K^2(1+\kappa_+)}{D_K^2}
=
\frac{\rho_0}{2}.
\end{aligned}
$$
By definition of $D_K$, we have
\begin{equation}\label{eq:population_hessian_fixed_probability}
\begin{aligned}
\inf_{\bbeta\in\mathcal{B}_K,\ \bv\in\mathbb{S}^{p-1}}
\mathbb{P}\bigl(|\bX^{*\top}\bbeta|\le D_K,\ |\bX^{*\top}\bv|>\eta_0\bigr)
\ge
\frac{\rho_0}{2}.
\end{aligned}
\end{equation}

\medskip
We can now prove Parts (a) and (b).

\noindent\textbf{Proof of (a): }

For all $\bbeta\in\mathcal{B}_K$ and all $\bv\in\mathbb{S}^{p-1}$, \eqref{eq:population_hessian_fixed_lower} and \eqref{eq:population_hessian_fixed_probability} together imply that
$$
\begin{aligned}
\bv^\top \boldsymbol{H}(\bbeta)\bv
&\ge
c_\rho(D_K)\eta_0^2\cdot \frac{\rho_0}{2}
=
C_K.
\end{aligned}
$$
Taking the infimum over $\bv\in\mathbb{S}^{p-1}$ gives
$\lambda_{\min}(\boldsymbol{H}(\bbeta))\ge C_K$ for all $\bbeta\in\mathcal{B}_K$, which is equivalent to
$\boldsymbol{H}(\bbeta)\succcurlyeq C_K\,\mathbf{I}_p$.

\medskip
\noindent\textbf{Proof of (b):}

Define a collection of indicator functions indexed by $(\bbeta,\bv)$ as
\[
f_{\bbeta,\bv}(\bx)
:=
\mathbf{1}\bigl\{|\bx^\top\bbeta|\le D_K,\ |\bx^\top\bv|>\eta_0\bigr\},
\qquad
\mathcal{F}:=\{f_{\bbeta,\bv}:\bbeta\in\mathcal{B}_K,\ \bv\in\mathbb{S}^{p-1}\}.
\]
Let $P^*$ be the law of $\bX^*$ and $P^*_M$ be the empirical measure of $\{\bX_i^*\}_{i=1}^M$. Then \eqref{eq:empirical_hessian_fixed_lower} implies
\begin{equation}\label{eq:synthetic-hessian-eigen-lower}
    \bv^\top \widehat{\boldsymbol{H}}_M(\bbeta)\bv
\ge
c_\rho(D_K)\eta_0^2\, P^*_M f_{\bbeta,\bv}.
\end{equation}

By \eqref{eq:population_hessian_fixed_probability}, we have
$$
\begin{aligned}
\inf_{f\in\mathcal{F}} P^* f
\ge
\frac{\rho_0}{2}.
\end{aligned}
$$

It remains to lower bound $\inf_{f\in\mathcal{F}}P^*_M f$ uniformly.
The class $\mathcal{F}$ is obtained from linear halfspaces by taking a finite number of unions and intersections.
More concretely,  $|\bx^\top \bbeta|\le D_K$ is an intersection of two halfspaces and $|\bx^\top\bv|>\eta_0$ is a union of two halfspaces, and we have
$$\begin{aligned}
\left\{\left|x^{\top} \beta\right| \leq D_K,\left|x^{\top} v\right|>\eta_0\right\}= & \left(\left\{x^{\top} \beta \leq D_K\right\} \cap\left\{-x^{\top} \beta \leq D_K\right\} \cap\left\{x^{\top} v>\eta_0\right\}\right) \\ & \cup\left(\left\{x^{\top} \beta \leq D_K\right\} \cap\left\{-x^{\top} \beta \leq D_K\right\} \cap\left\{-x^{\top} v>\eta_0\right\}\right). \end{aligned}
$$
In other words, $\mathcal{F}$ is the class of indicator functions for a 2-fold union of 3-fold intersections of halfspaces.
By Lemma 3.2.3 in \citet{blumer1989learnability},  $\mathcal{F}$ is a VC class with VC dimension
\[
v_{\mathcal{F}}\le C\,p.
\]

Therefore, by a standard bounded difference inequality and the VC bound on Rademacher complexity (see Theorem~4.10 and Equation (5.50) respectively in \citet{wainwright2019high}), there exists a universal constant $C>0$ such that for any $\epsilon\in(0,1)$,
with probability at least $1-\epsilon$,
$$
\begin{aligned}
\sup_{f\in\mathcal{F}}|P^*_M f-P^* f|
\le
C\sqrt{\frac{p}{M}}+ \sqrt{\frac{8\log(2/\epsilon)}{M}}.
\end{aligned}
$$
Therefore, if $M \ge C'\,(p+\log(1/\epsilon))/\rho_0^2$ with $C'$ large enough, then the right side of the above display is at most $\rho_0/4$.
On this event, we have
$$
\begin{aligned}
\inf_{f\in\mathcal{F}} P^*_M f
&\ge
\inf_{f\in\mathcal{F}} P^* f - \sup_{f\in\mathcal{F}}|P^*_M f-P^* f|
\ge
\frac{\rho_0}{2}-\frac{\rho_0}{4}
=
\frac{\rho_0}{4}.
\end{aligned}
$$
Consequently, for all $\bbeta\in\mathcal{B}_K$ and all $\bv\in\mathbb{S}^{p-1}$, \Cref{eq:synthetic-hessian-eigen-lower} implies that
$$
\begin{aligned}
\bv^\top \widehat{\boldsymbol{H}}_M(\bbeta)\bv
&\ge
c_\rho(D_K)\eta_0^2\cdot \frac{\rho_0}{4}
=
\frac{1}{2}C_K.
\end{aligned}
$$
Taking the infimum over $\bv\in\mathbb{S}^{p-1}$ yields
$\lambda_{\min}(\widehat{\boldsymbol{H}}_M(\bbeta))\ge C_K/2$ for all $\bbeta\in\mathcal{B}_K$,
and then taking the infimum over $\bbeta\in\mathcal{B}_K$ completes the proof.
\end{proof}

\bigskip

\subsubsection{Proof of Theorem~\ref{thm:stability_finite_M}}

Recall the definitions of the constrained estimators $\widehat{\bbeta}_{M}^{(K)}$ and $\widehat{\bbeta}_{\infty}^{(K)}$ in \eqref{eq: SRE_BL} in the main text.
Since the theorem is stated in terms of these constrained estimators, in the proof, we will drop the superscripts for readability. In other words, we write
\[
\widehat{\bbeta}_M := \widehat{\bbeta}_M^{(K)},\qquad
\widehat{\bbeta}_\infty := \widehat{\bbeta}_\infty^{(K)}.
\]

The theorem is an implication of the following two lemmas.

\begin{lemma}\label{lem:convexity-stability}
Recall the constant $C_K$ in \Cref{prop:lower_curvature}.
On the event
$$\mathcal{E}_{\mathrm{curv}}:=\Bigl\{\inf_{\bbeta\in\mathcal{B}_K}\lambda_{\min}\bigl(\widehat{\boldsymbol{H}}_M(\bbeta)\bigr)\ge C_K/2\Bigr\},$$
the following holds for the constrained estimators:
$$
\begin{aligned}
\|\widehat{\bbeta}_M-\widehat{\bbeta}_\infty\|_2
\le
\frac{\tau}{ \lambda_{n,K}+\tau C_K/2 }\,
\|\nabla\delta_M(\widehat{\bbeta}_\infty)\|_2.
\end{aligned}
$$
where $
\delta_{M}(\boldsymbol{\beta}) :=\mathbb{E}\left(Y^*\bX^{*\top}\bbeta-\rho(\bX^{*\top}\bbeta)\right) - \frac{1}{M}\sum_{i\leq M} \left(Y_i^*\bX_i^{*\top}\bbeta-\rho(\bX_i^{*\top}\bbeta)\right)
$.
\end{lemma}

\begin{lemma}[Concentration of $\nabla\delta_M(\bbeta)$ for a fixed $\bbeta$]
\label{lem:grad_delta_fixed_beta}
Suppose \Cref{conditions:synthetic_X_Y} or \Cref{GLM_condition:synthetic_data} holds.
For GLM, also assume \Cref{GLM_cond:link} holds (note that for logistic regression, $L_g=1$).
We only consider the randomness in the synthetic sample $\{(\bX_i^*,Y_i^*)\}_{i=1}^M$.
There exist universal constants $c,C>0$ such that for any vector $\bbeta_0\in\mathbb{R}^p$ not depending on the synthetic sample and for any $\epsilon\in(0,1)$, the following holds with probability at least
$1-\epsilon$:
$$
\begin{aligned}
\|\nabla\delta_M(\bbeta_0)\|_2
\le
C L_g K_X\sqrt{1+\kappa_+}\sqrt{\frac{p+\log(2/\epsilon)}{M}}.
\end{aligned}
$$
\end{lemma}

\begin{proof}[Proof of \Cref{thm:stability_finite_M}]

Fix $\epsilon\in(0,1)$ and let $\epsilon_1=\epsilon_2=\epsilon/2$.

By \Cref{prop:lower_curvature}(b), if
\[
M \ge C\,\frac{p+\log(1/\epsilon_1)}{\rho_0^2},
\]
then with probability at least $1-\epsilon_1$,
\[
\inf_{\bbeta\in\mathcal{B}_K}\lambda_{\min}\bigl(\widehat{\boldsymbol{H}}_M(\bbeta)\bigr)\ge \frac{1}{2}C_K.
\]
Denote this event by $\mathcal{E}_1$.

Next, note that $\widehat{\bbeta}_\infty$ does not depend on the synthetic sample.
Applying \Cref{lem:grad_delta_fixed_beta} with $\bbeta_0=\widehat{\bbeta}_\infty$ and $\epsilon_2$ yields that with probability at least $1-\epsilon_2$,
\[
\|\nabla\delta_M(\widehat{\bbeta}_\infty)\|_2
\le
C L_g K_X\sqrt{1+\kappa_+}\sqrt{\frac{p+\log(2/\epsilon_2)}{M}}
=
C L_g K_X\sqrt{1+\kappa_+}\sqrt{\frac{p+\log(4/\epsilon)}{M}}.
\]
Denote this event by $\mathcal{E}_2$.

On the intersection event $\mathcal{E}_1\cap\mathcal{E}_2$, \Cref{lem:convexity-stability} implies
\[
\|\widehat{\bbeta}_M-\widehat{\bbeta}_\infty\|_2
\le
\frac{\tau}{\lambda_{n,K}+\tau C_K/2}\,
C L_g K_X\sqrt{1+\kappa_+}\sqrt{\frac{p+\log(4/\epsilon)}{M}}.
\]
Finally, by the union bound, $\mathbb{P}(\mathcal{E}_1\cap\mathcal{E}_2)\ge 1-\epsilon$.
This completes the proof with $c_L=C_K/2$ and $C_1=C L_g K_X\sqrt{1+\kappa_+}$.
\end{proof}

The rest of this section is devoted to proving \Cref{lem:convexity-stability} and \Cref{lem:grad_delta_fixed_beta}.

\begin{proof}[Proof of \Cref{lem:convexity-stability}]

Define
\[
\begin{aligned}
g_1(\bbeta)
&=
\sum_{i=1}^n \Bigl[Y_i\,\bX_i^\top \bbeta-\rho\!\left(\bX_i^\top \bbeta\right)\Bigr],\\
g_2(\bbeta)
&=
\frac{1}{M}\sum_{i=1}^M \Bigl[Y^*_i\,{\bX_i^*}^\top \bbeta-\rho\!\left({\bX_i^*}^\top \bbeta\right)\Bigr],\\
g_3(\bbeta)
&=
\mathbb{E}\Bigl[Y^*\,\bX^{*\top}\bbeta-\rho\!\left(\bX^{*\top}\bbeta\right)\Bigr],
\end{aligned}
\]
and
\[
\delta_M(\bbeta):=g_3(\bbeta)-g_2(\bbeta),\qquad
S_M(\bbeta):=g_1(\bbeta)+\tau g_2(\bbeta),\qquad
S_\infty(\bbeta):=g_1(\bbeta)+\tau g_3(\bbeta).
\]
Consequently, we have
$$S_M=S_\infty-\tau\delta_M.$$

By definitions in \eqref{eq: SRE_BL}, we have
\[
\widehat{\bbeta}_M=\arg\max_{\bbeta\in\mathcal{B}_K} S_M(\bbeta),
\qquad
\widehat{\bbeta}_\infty=\arg\max_{\bbeta\in\mathcal{B}_K} S_\infty(\bbeta).
\]

\medskip
\noindent\textbf{Step 1: strong concavity of $S_M$ on $\mathcal{B}_K$.}
For any $\bbeta\in\mathcal{B}_K$,
\[
\begin{aligned}
-\nabla^2 g_1(\bbeta)
&=
\sum_{i=1}^n \rho''(\bX_i^\top \bbeta)\,\bX_i\bX_i^\top,\\
-\nabla^2 g_2(\bbeta)
&=
\frac{1}{M}\sum_{i=1}^M \rho''({\bX_i^*}^\top \bbeta)\,\bX_i^*{\bX_i^*}^\top
=
\widehat{\boldsymbol{H}}_M(\bbeta).
\end{aligned}
\]
Hence
\[
-\nabla^2 S_M(\bbeta)
=
\sum_{i=1}^n \rho''(\bX_i^\top \bbeta)\,\bX_i\bX_i^\top
+\tau\,\widehat{\boldsymbol{H}}_M(\bbeta).
\]
By the definition of $\lambda_{n,K}$ in \Cref{thm:stability_finite_M},
\[
\sum_{i=1}^n \rho''(\bX_i^\top \bbeta)\,\bX_i\bX_i^\top \succcurlyeq \lambda_{n,K}\,\mathbf{I}_p,
\qquad
\forall\,\bbeta\in\mathcal{B}_K.
\]
On the event $\mathcal{E}_{\mathrm{curv}}$, we have
\[
\inf_{\bbeta\in\mathcal{B}_K}\lambda_{\min}\bigl(\widehat{\boldsymbol{H}}_M(\bbeta)\bigr)
\ge
\frac{1}{2}C_K,
\]
and therefore
\[
-\nabla^2 S_M(\bbeta)
\succcurlyeq
\Bigl(\lambda_{n,K}+\tau C_K/2\Bigr)\mathbf{I}_p,
\qquad
\forall\,\bbeta\in\mathcal{B}_K.
\]
Thus $-S_M$ is strongly convex on $\mathcal{B}_K$ with parameter $\lambda_{n,K}+\tau C_K/2$.
By \citet[Theorem 2.1.9]{nesterov2013introductory}, for any $\bbeta,\bbeta'\in\mathcal{B}_K$,
\begin{equation}\label{eq:stability-strong-convexity}
\bigl\langle \nabla S_M(\bbeta)-\nabla S_M(\bbeta'),\, \bbeta'-\bbeta \bigr\rangle
\ge
\Bigl(\lambda_{n,K}+\tau C_K/2\Bigr)\|\bbeta'-\bbeta\|_2^2.
\end{equation}

\medskip
\noindent\textbf{Step 2: apply \eqref{eq:stability-strong-convexity} at the constrained maximizers.}
Let $\Delta:=\widehat{\bbeta}_M-\widehat{\bbeta}_\infty$.
Applying \eqref{eq:stability-strong-convexity} with
$(\bbeta,\bbeta')=(\widehat{\bbeta}_\infty,\widehat{\bbeta}_M)$ gives
\[
\Bigl(\lambda_{n,K}+\tau C_K/2\Bigr)\|\Delta\|_2^2
\le
\bigl\langle \nabla S_M(\widehat{\bbeta}_\infty)-\nabla S_M(\widehat{\bbeta}_M),\,\Delta \bigr\rangle.
\]
Since $\widehat{\bbeta}_M$ maximizes $S_M$ over $\mathcal{B}_K$ and $S_M$ is concave,
the first-order optimality condition implies
\[
\bigl\langle \nabla S_M(\widehat{\bbeta}_M),\,\bbeta-\widehat{\bbeta}_M \bigr\rangle \le 0,
\qquad \forall\,\bbeta\in\mathcal{B}_K.
\]
Taking $\bbeta=\widehat{\bbeta}_\infty$ yields
$\langle \nabla S_M(\widehat{\bbeta}_M),\,\Delta\rangle \ge 0$, and therefore
\[
\bigl\langle \nabla S_M(\widehat{\bbeta}_\infty)-\nabla S_M(\widehat{\bbeta}_M),\,\Delta \bigr\rangle
\le
\bigl\langle \nabla S_M(\widehat{\bbeta}_\infty),\,\Delta \bigr\rangle.
\]
Next, since $S_M=S_\infty-\tau\delta_M$, we have
$\nabla S_M(\widehat{\bbeta}_\infty)=\nabla S_\infty(\widehat{\bbeta}_\infty)-\tau\nabla\delta_M(\widehat{\bbeta}_\infty)$.
Because $\widehat{\bbeta}_\infty$ maximizes $S_\infty$ over $\mathcal{B}_K$ and $S_\infty$ is concave,
\[
\bigl\langle \nabla S_\infty(\widehat{\bbeta}_\infty),\,\bbeta-\widehat{\bbeta}_\infty \bigr\rangle \le 0,
\qquad \forall\,\bbeta\in\mathcal{B}_K,
\]
and taking $\bbeta=\widehat{\bbeta}_M$ gives
$\langle \nabla S_\infty(\widehat{\bbeta}_\infty),\,\Delta\rangle \le 0$.
Thus,
\[
\bigl\langle \nabla S_M(\widehat{\bbeta}_\infty),\,\Delta \bigr\rangle
\le
-\tau \bigl\langle \nabla\delta_M(\widehat{\bbeta}_\infty),\,\Delta \bigr\rangle
\le
\tau\|\nabla\delta_M(\widehat{\bbeta}_\infty)\|_2\,\|\Delta\|_2.
\]
Combining the above bounds, we obtain
\[
\Bigl(\lambda_{n,K}+\tau C_K/2\Bigr)\|\Delta\|_2^2
\le
\tau\|\nabla\delta_M(\widehat{\bbeta}_\infty)\|_2\,\|\Delta\|_2.
\]
If $\Delta=\mathbf{0}$ the claim is trivial; otherwise divide both sides by
$\bigl(\lambda_{n,K}+\tau C_K/2\bigr)\|\Delta\|_2$ to get
\[
\|\widehat{\bbeta}_M-\widehat{\bbeta}_\infty\|_2
\le
\frac{\tau}{\lambda_{n,K}+\tau C_K/2}\,
\|\nabla\delta_M(\widehat{\bbeta}_\infty)\|_2.
\]
\end{proof}

\medskip

\begin{proof}[Proof of \Cref{lem:grad_delta_fixed_beta}]
Throughout, $c,C>0$ denote universal constants whose values may change from line to line.
Fix $\bbeta_0$ and write
$$
\begin{aligned}
\nabla\delta_M(\bbeta_0)
=
\mathbb{E}\Bigl[\bigl(Y^*-\rho^{\prime}(\bX^{*\top}\bbeta_0)\bigr)\bX^*\Bigr]
-
\frac{1}{M}\sum_{i=1}^M \Bigl(Y_i^*-\rho^{\prime}(\bX_i^{*\top}\bbeta_0)\Bigr)\bX_i^*.
\end{aligned}
$$
Define mean zero random vectors
$$
\begin{aligned}
\bxi_i
:=
\mathbb{E}\Bigl[\bigl(Y^*-\rho^{\prime}(\bX^{*\top}\bbeta_0)\bigr)\bX^*\Bigr]
-
\Bigl(Y_i^*-\rho^{\prime}(\bX_i^{*\top}\bbeta_0)\Bigr)\bX_i^*,
\qquad i\in[M],
\end{aligned}
$$
so that $\nabla\delta_M(\bbeta_0)=\frac{1}{M}\sum_{i=1}^M \bxi_i$.

\paragraph{Step 1: a uniform $\psi_2$ bound for one dimensional projections.}
For any $\bv\in\mathbb{S}^{p-1}$, let $Z_i(\bv):=\bv^\top\bxi_i$. Then $Z_i(\bv)$ are i.i.d.\ and mean zero.
According to part 2 of \Cref{GLM_cond:link}, the log likelihood is differentiable in $\theta$ and Lipschitz with constant $L_g$, we have
$$\left|\rho'(\theta)-y\right|=\left|\partial_\theta \ell_G(y, \theta)\right| \leq L_g.$$
In particular, for logistic regression $L_g=1$.
Using \Cref{lem:standard_psi_Bernstein_facts}, we have
\begin{equation}\label{eq:Zi_psi2}
\begin{aligned}
\|Z_i(\bv)\|_{\psi_2}
&\le
\Bigl\|\Bigl(Y_i^*-\rho^\prime(\bX_i^{*\top}\bbeta_0)\Bigr)\bv^\top\bX_i^*\Bigr\|_{\psi_2}
+
\Bigl\|\mathbb{E}\Bigl[\Bigl(Y^*-\rho^\prime(\bX^{*\top}\bbeta_0)\Bigr)\bv^\top\bX^*\Bigr]\Bigr\|_{\psi_2}  \\
&\le
L_g\|\bv^\top\bX^*\|_{\psi_2}
+
L_g\mathbb{E}|\bv^\top\bX^*|
\le
C\|\bv^\top\bX^*\|_{\psi_2},
\end{aligned}
\end{equation}
where we have the fact that for a constant $c$, $\|c\|_{\psi_2}=|c|/\sqrt{\log 2}$.

Write $\bX^*=(1,\widetilde{\bX}^{*\top})^\top$ and $\bv=(v_1,\widetilde{\bv})$.
Then $\bv^\top\bX^*=v_1+\widetilde{\bv}^{\top}\widetilde{\bX}^*$ and thus
\begin{equation}
\label{eq:vX_psi2}
\begin{aligned}
\|\bv^\top\bX^*\|_{\psi_2}
\le
C\Bigl(|v_1|+\|\widetilde{\bv}^{\top}\widetilde{\bX}^*\|_{\psi_2}\Bigr)
\le
C\Bigl(1+\|\widetilde{\bv}^{\top}\widetilde{\bX}^*\|_{\psi_2}\Bigr).
\end{aligned}
\end{equation}
By \Cref{conditions:synthetic_X_Y}(C3),
$\|\widetilde{\bv}^{\top}\widetilde{\bX}^*\|_{\psi_2}\le K_X\|\widetilde{\bv}^{\top}\widetilde{\bX}^*\|_{L^2}$.
By \Cref{conditions:synthetic_X_Y}(C2),
$$
\begin{aligned}
\|\widetilde{\bv}^{\top}\widetilde{\bX}^*\|_{L^2}
=
\bigl(\widetilde{\bv}^\top \bSigma^*\widetilde{\bv}\bigr)^{1/2}
\le
\sqrt{\kappa_+}\,\|\widetilde{\bv}\|_2
\le
\sqrt{\kappa_+}.
\end{aligned}
$$
Therefore, \eqref{eq:Zi_psi2} becomes
$$
\begin{aligned}
\|Z_i(\bv)\|_{\psi_2}
\le
CL_g\|\bv^\top\bX^*\|_{\psi_2}
\le
CL_g\bigl(1+K_X\sqrt{\kappa_+}\bigr)
\le
C L_g K_X\sqrt{1+\kappa_+},
\end{aligned}
$$
where we used $K_X\ge 1$ without loss of generality.

\paragraph{Step 2: sub-gaussian concentration and a sphere net.}
By standard sub-gaussian concentration for averages (see \citet[Theorem 2.6.2]{vershynin2018high}), there is some universal constant $c$ such that for all $t>0$,
\begin{equation}\label{eq:concentration-Z_v}
\begin{aligned}
\mathbb{P}\left(\left|\frac{1}{M}\sum_{i=1}^M Z_i(\bv)\right|\ge t\right)
\le
2\exp\left(-c\,M\,\frac{t^2}{L_g^2K_X^2(1+\kappa_+)}\right).
\end{aligned}
\end{equation}
Let $\mathcal{V}$ be a $1/4$-net of $\mathbb{S}^{p-1}$ with $|\mathcal{V}|\le 9^p$ (which can be shown using a volume argument).
A standard net argument yields
\[
\|\nabla\delta_M(\bbeta_0)\|_2
=
\sup_{\bv\in\mathbb{S}^{p-1}}\bv^\top\nabla\delta_M(\bbeta_0)
\le 2\sup_{\bv\in\mathcal{V}} \bigl|\bv^\top\nabla\delta_M(\bbeta_0)\bigr|.
\]
For each $\bv\in\mathcal{V}$, we apply \eqref{eq:concentration-Z_v} $t = 2^{-1} C L_g K_X\sqrt{1+\kappa_+}\sqrt{\frac{p+\log(2/\epsilon)}{M}}$
with $C$ large enough such that the probability is upper bounded as
$$
2 \exp \left(-c M \frac{t^2}{L_g^2 K_X^2\left(1+\kappa_{+}\right)}\right)
= 2\exp\left( -c C^2(p+\log (2 / \epsilon)) \right)
\leq 9^{-p} \epsilon.
$$
We then apply a union bound over $\bv\in\mathcal{V}$, which yields that with probability at least $1- \epsilon$,
$$
\begin{aligned}
\|\nabla\delta_M(\bbeta_0)\|_2
=
\left\|\frac{1}{M}\sum_{i=1}^M \bxi_i\right\|_2
\le
C L_g K_X\sqrt{1+\kappa_+}\sqrt{\frac{p+\log(2/\epsilon)}{M}}.
\end{aligned}
$$
This completes the proof.
\end{proof}

\subsection{Proof of Theorem \ref{thm:post_mode_consistency}}
\label{proof_sec:logitic_consistency}

In this section, we establish the consistency of the SRE.
Let $\mathbb X=\left[\boldsymbol{X}_1, \boldsymbol{X}_2, \ldots \boldsymbol{X}_n\right]^{\top}$ be the covariate matrix of the observed data.

We first show that \Cref{condition:SubgaussianX} holds almost surely if the observed covariates are i.i.d. samples from a sub-Gaussian distribution.

\begin{lemma}
    Suppose $\boldsymbol{X}_i$'s are i.i.d. sub-gaussian random vectors with covariance matrix $\bSigma$.
We assume that the largest eigenvalue of $\bSigma$ is upper bounded by $\lambda_{\bSigma}^+ <\infty$ and the smallest eigenvalue is lower bounded by $\lambda_{\bSigma}^- >0$.
Furthermore, the sub-gaussian norm of $\boldsymbol{X}_i$ is upper bounded by $K<\infty$.
    If $p/n\rightarrow 0$, then with probability 1,  \Cref{condition:SubgaussianX} holds almost surely.
\end{lemma}

\begin{proof}

Without loss of generality, we assume $K=1$ for convenience.
We first show that there exists a positive constant $c_3$, such that the inequalities in \Cref{condition:SubgaussianX} hold with probability exceeding $1-n\exp(- c_3 n)$.

Following the Theorem 5.39 in \cite{vershynin2010introduction}, there are universal positive constants $C'_1$ and $C'_2$ such that for every $t \geq 0$, the following inequality holds for any subset $S\subset [n]$ with probability at least $1-2 \exp \left(- C'_1 t^2\right)$ :
$$
\left\|\frac{1}{|S|} \mathbb{X}_{S}^\top \mathbb{X}_{S}-\bSigma\right\| \leq \max \left(\delta, \delta^2\right)\|\bSigma\| \quad \text { where } \quad \delta=C'_2\sqrt{\frac{p}{|S|}}+\frac{t}{\sqrt{|S|}}.
$$
When this event holds, the smallest eigenvalue of $\mathbb{X}_{S}^\top \mathbb{X}_{S}$ is lower bounded by
$$
|S|\lambda_{\bSigma}^-\left[1-\max \left(\delta, \delta^2\right) \lambda_{\bSigma}^+ /\lambda_{\bSigma}^- \right]
$$
and the largest eigenvalue is upper bounded by
$$
|S|\lambda_{\bSigma}^+\left[1+\max \left(\delta, \delta^2\right)\right].
$$

Since $p/n\rightarrow 0$, we can take $c_3$ small enough and $n_0$ large enough so that if $t=\sqrt{2c_3 n/C'_1}$, $|S|>n/2$, and $n>n_0$, then $\max \left(\delta, \delta^2\right)<\min\left[1,\lambda_{\bSigma}^-/(2\lambda_{\bSigma}^+ ) \right]$.
If we choose the positive constant $c_1$ to be $\lambda_\bSigma^-/4$ and $c_2$ to be $2\lambda_{\bSigma}^+$, then
for any given $S \subseteq[n]$ with $|S|\geq n/2$, it holds that

$$\mathbb{P}\left\{
\lambda_{\min }\left(\sum_{i \in S} \boldsymbol{X}_i \boldsymbol{X}_i^{\top}\right)< c_1 n,
\text{ or }
\lambda_{\max }\left(\sum_{i \in S} \boldsymbol{X}_i \boldsymbol{X}_i^{\top}\right)> c_2 n,
\right\}  \leq 2\exp \left(-2c_3 n\right).$$

Define $H(\epsilon):=-\epsilon \log \epsilon-(1-\epsilon)\log(1-\epsilon)$. We choose a positive $\zeta$ to be sufficiently small such that $H(\zeta)<c_3$ and $\zeta<1/2$. By taking the union bound over subsets $S$ with $|S|>(1-\zeta)n$, we have
$$
\begin{aligned}
& \mathbb{P}\left\{\exists S \subseteq[n] \text { with }|S|\geq (1-\zeta)n\text { s.t. } \lambda_{\min }\left(\sum_{i \in S} \boldsymbol{X}_i \boldsymbol{X}_i^{\top}\right)<  c_1n
\text{ or }
\lambda_{\max }\left(\sum_{i \in S} \boldsymbol{X}_i \boldsymbol{X}_i^{\top}\right)> c_2 n,
\right\} \\
& \quad \leq \sum_{k = \lceil (1-\zeta)n \rceil}^{n} \left(\begin{array}{c}
n \\
k
\end{array}\right)  2\exp \left(- 2c_3 n\right) \\
& \quad \leq   \sum_{k = \lceil (1-\zeta)n \rceil}^{n} 2\exp \left(n H(\frac{n-k}{n})- 2c_3 n\right)\\
& \quad \leq   2 \zeta n \exp \left(n H(\zeta)- 2c_3 n\right)\\
& \quad \leq   n \exp \left( - c_3 n\right),
\end{aligned}
$$
where in the second inequality we use $\left(\begin{array}{c}n \\ k \end{array}\right) \leq e^{n H(k/n)}$ \citep[Example 11.1.3]{cover2012elements}, the third inequality is due to the monotonicity of $H(\epsilon)$ for $\epsilon\in (0,1/2)$, and the fourth is due to $H(\zeta)<c_3$ and $2\zeta<1$.

Lastly, by the first Borel-Cantelli lemma and the fact that $\sum_{n} ne^{-c_3 n}<\infty$, for the above choices of $c_1$, $c_2$, and $\zeta$, with probability 1, there exists a constant $N_0$, such that for any $n\geq N_0$,  the inequalities in \Cref{condition:SubgaussianX} hold.

\end{proof}

Next, we first present two useful lemmas that will be used for proving \Cref{thm:post_mode_consistency}.

\begin{lemma}\label{lemma:SB_beta_set_size}
Let $c_2$ be the constant in \Cref{condition:SubgaussianX}.
For any $\boldsymbol{\beta}\in \mathbb R^p$ and any $C>2$, define
$$
\mathcal{S}_{C}(\boldsymbol{\beta}):=\left\{i:\left|\boldsymbol{X}_i^{\top} \boldsymbol{\beta}\right| \leq \sqrt{C c_2}\|\boldsymbol{\beta}\|\right\}.
$$
Under \Cref{condition:SubgaussianX},
the cardinality of $\mathcal{S}_{C}(\boldsymbol{\beta})$ is uniformly bounded from below as
	$$
\left|\mathcal{S}_{C}(\boldsymbol{\beta})\right| \geq(1-\frac{1}{C})n, \quad \forall \boldsymbol{\beta}.
$$
\end{lemma}
\begin{proof}[Proof of \Cref{lemma:SB_beta_set_size}]
If \(\boldsymbol\beta=\mathbf 0\), then
\(\mathcal S_C(\boldsymbol\beta)=[n]\) and the result is trivial.
We therefore focus on \(\boldsymbol\beta\ne \mathbf 0\).
We first note that
$$
\|\mathbb X \boldsymbol{\beta}\|^2 \geq \sum_{i \notin \mathcal{S}_{C} (\boldsymbol{\beta})}\left|\boldsymbol{X}_i^{\top} \boldsymbol{\beta}\right|^2 \geq\left(n-\left|\mathcal{S}_{C} (\boldsymbol{\beta})\right|\right)(\sqrt{C c_2}\|\boldsymbol{\beta}\|)^2=n\left(1-\frac{\left|\mathcal{S}_{C}  (\boldsymbol{\beta})\right|}{n}\right) C c_2\|\boldsymbol{\beta}\|^2.
$$
Under \Cref{condition:SubgaussianX},
we have
$$
\|\mathbb X \boldsymbol{\beta}\|^2 \leq c_2 n\|\boldsymbol{\beta}\|^2, \quad \forall \boldsymbol{\beta}\ne \mathbf 0.
$$
The above two inequalities imply that $$\left(1-\frac{\left|\mathcal{S}_{C}  (\boldsymbol{\beta})\right|}{n}\right) C \leq 1, \quad \forall \boldsymbol{\beta}\ne \mathbf 0. $$
\end{proof}

\begin{lemma}\label{lem:bounded-syn-cov-op-norm}
Under \Cref{conditions:synthetic_X_Y}(C1)--(C3), let
\[
\lambda_M^*:=\lambda_{\max }\left(\frac{1}{M} \sum_{i=1}^M \boldsymbol{X}_i^* \boldsymbol{X}_i^{* \top}\right).
\]
Then there exists a constant $C_\lambda>0$, depending only on $(\kappa_+,K_X)$, such that for every $t\geq 0$,
\begin{equation}\label{eq:synthetic_covariance_tail_bound_for_consistency}
\lambda_M^*\frac{p}{n}
\leq
C_\lambda\left(
\frac{p}{n}+\frac{p^2}{Mn}+\frac{t^2p}{Mn}
\right)
\end{equation}
with probability at least $1-2\exp(-t^2)$.
Consequently, if $p=o(n)$ and $p^2/(Mn)=O(1)$, then
\[
\lambda_M^*\frac{p}{n}=O_p(1).
\]
\end{lemma}

\begin{proof}
Let $\mathbb{X}^*\in\mathbb{R}^{M\times p}$ be the synthetic covariate matrix and write $d=p-1$. By \Cref{prop:design_op_norm}, under \Cref{conditions:synthetic_X_Y}(C1)--(C3), there exists a universal constant $C>0$ such that, for every $t\geq 0$, with probability at least $1-2\exp(-t^2)$,
\[
\|\mathbb{X}^*\|_{\mathrm{op}}
\leq
\sqrt{M}+\kappa_+^{1/2}\{\sqrt{M}+CK_X^2(\sqrt d+t)\}.
\]
Therefore, on the same event,
\[
\frac{\|\mathbb{X}^*\|_{\mathrm{op}}}{\sqrt M}
\leq
1+\kappa_+^{1/2}
+C\kappa_+^{1/2}K_X^2\left(\sqrt{\frac dM}+\frac{t}{\sqrt M}\right).
\]
Since $d=p-1\leq p$, there exists a constant $C_\lambda>0$, depending only on $(\kappa_+,K_X)$, such that
\[
\lambda_M^*
=
\lambda_{\max}\left(\frac{1}{M}\mathbb{X}^{*\top}\mathbb{X}^*\right)
=
\frac{\|\mathbb{X}^*\|_{\mathrm{op}}^2}{M}
\leq
C_\lambda\left(1+\frac{p}{M}+\frac{t^2}{M}\right).
\]
Multiplying both sides by $p/n$ gives \eqref{eq:synthetic_covariance_tail_bound_for_consistency}.

We now prove the $O_p(1)$ statement. Fix any $\epsilon\in(0,1)$ and choose
\[
t_\epsilon=\sqrt{\log(4/\epsilon)}.
\]
Then $2\exp(-t_\epsilon^2)=\epsilon/2$. On the event in \eqref{eq:synthetic_covariance_tail_bound_for_consistency},
\[
\lambda_M^*\frac{p}{n}
\leq
C_\lambda\left(
\frac{p}{n}+\frac{p^2}{Mn}+\frac{t_\epsilon^2p}{Mn}
\right).
\]
Because $p=o(n)$, we have $p/n\to0$. Since $M\geq1$,
\[
\frac{p}{Mn}\leq \frac{p}{n}\to0.
\]
By the assumption \(p^2/(Mn)=O(1)\), there exists a finite constant \(A\) such that
\(p^2/(Mn)\le A\) and $p/n\leq 1$ for all sufficiently large \(n\). Hence, for all sufficiently large \(n\), the
right-hand side is bounded by
\[
C_\lambda\{1+A+\log(4/\epsilon)\},
\]
on an event with probability at least \(1-\epsilon/2\).

This completes the proof.
\end{proof}

Now we are ready to show $\|\widehat{\boldsymbol{\beta}}_{M}-\bbeta_0\|^2=O_p(p/n)$.
We write the gradient of the objective function in \eqref{eq: SRE_def} as
	$$F(\boldsymbol{\beta})=\sum_{i=1}^{n} \left(Y_i-\rho^{\prime}(\boldsymbol{X}_i^\top \boldsymbol{\beta})\right )\boldsymbol{X}_i +\frac{\tau}{M} \sum_{i=1}^{M}  \left(Y^*_i-\rho^{\prime}(\boldsymbol{X}_i^{*\top} \boldsymbol{\beta})\right )\bX^*_i. $$
	Then the point estimator  $\widehat{\boldsymbol{\beta}}_{M}$ is the root of $F(\boldsymbol{\beta})=0$.
 Based on \citet[Result 6.3.4]{ortega1970iterative},  it suffices to show that  for any $\epsilon>0$, there is some constant $\tilde B>0$ such that   $(\boldsymbol{\beta}-\boldsymbol{\beta}_0)^{\top} F(\boldsymbol{\beta})<0$ for all $\boldsymbol{\beta}$ satisfies $\|\boldsymbol{\beta}-\boldsymbol{\beta}_0\|^2=\tilde{B}p/n$ with probability $1-\epsilon$.

 By Taylor's theorem with integral remainder, we have
 $$\rho^{\prime}(\boldsymbol{X}_i^\top \boldsymbol{\beta})=\rho^{\prime}(\boldsymbol{X}_i^\top \boldsymbol{\beta}_0)+\int_0^{\boldsymbol{X}_i^\top (\boldsymbol{\beta}-\boldsymbol{\beta}_0)}\rho^{\prime\prime}(\boldsymbol{X}_i^\top \boldsymbol{\beta}_0+s)ds. $$

For any fixed $\bbeta$ with $\|\boldsymbol{\beta}-\boldsymbol{\beta}_0\|^2=\tilde{B}p/n$, we write

{\allowdisplaybreaks
\begin{align}\label{beta_time_gradient}
    (\boldsymbol{\beta}-\boldsymbol{\beta}_0)^\top F(\boldsymbol{\beta})&=(\boldsymbol{\beta}-\boldsymbol{\beta}_0)^\top \sum_{i=1}^{n} \left(Y_i-\rho^{\prime}(\boldsymbol{X}_i^\top \boldsymbol{\beta})\right )\boldsymbol{X}_i \nonumber \\
	&\quad +\frac{\tau}{M} \sum_{i=1}^{M}  \left(Y^*_i-\rho^{\prime}(\boldsymbol{X}_i^{*\top} \boldsymbol{\beta})\right )(\boldsymbol{\beta}-\boldsymbol{\beta}_0)^\top \bX^*_i \nonumber \\
	&=\underbrace{ \sum_{i=1}^n (\boldsymbol{\beta}-\boldsymbol{\beta}_0)^\top\boldsymbol{X}_i\left(Y_i-\rho^{\prime}(\boldsymbol{X}_i^\top \boldsymbol{\beta}_0)\right ) }_{Q_1(\boldsymbol{\beta})}\\
	&\quad \quad -  \underbrace{\sum_{i=1}^n (\boldsymbol{\beta}-\boldsymbol{\beta}_0)^\top\boldsymbol{X}_i \int_0^{\boldsymbol{X}_i^\top (\boldsymbol{\beta}-\boldsymbol{\beta}_0)}\rho^{\prime\prime}(\boldsymbol{X}_i^\top \boldsymbol{\beta}_0+s)ds}_{Q_2(\boldsymbol{\beta})}\nonumber\\
	& \quad \quad + \underbrace{\frac{\tau}{M} \sum_{i=1}^{M}  \left(Y^*_i-\rho^{\prime}(\boldsymbol{X}_i^{*\top} \boldsymbol{\beta})\right )(\boldsymbol{\beta}-\boldsymbol{\beta}_0)^\top \bX^*_i}_{Q_3(\boldsymbol{\beta})},\nonumber
\end{align}}

where the second equation follows from applying the Taylor expansion of $\rho^\prime(\cdot)$ to the observed data.
We will derive upper bounds on $Q_1$ and $Q_3$, and a lower bound on $Q_2$ in the following steps, where $\tilde{B}$ is a positive number to be determined that only depends on $\epsilon$ and the constants.

\textbf{Upper bound on $Q_3$}:
Let $\lambda_{M}^{*}:=\lambda_{\max}\left(\frac{1}{M}\sum_{i=1}^{M} \boldsymbol{X}_i^*\boldsymbol{X}_i^{*\top }\right)$.
It is straightforward to see that
\begin{equation}\label{Q3_upper_bd}
	\begin{aligned}
	|Q_3(\boldsymbol{\beta})|&\leq \frac{\tau}{M} \sum_{i=1}^{M}  |(\boldsymbol{\beta}-\boldsymbol{\beta}_0)^\top \bX^*_i|\\
	&\leq \tau \sqrt{(\boldsymbol{\beta}-\boldsymbol{\beta}_0)^\top \left(\frac{1}{M}\sum_{i=1}^{M} \boldsymbol{X}_i^*\boldsymbol{X}_i^{*\top }\right)    (\boldsymbol{\beta}-\boldsymbol{\beta}_0) }\\
	& \leq \tau \sqrt{\lambda_{M}^{*} \tilde{B} p / n}.
\end{aligned}
\end{equation}

\textbf{Upper bound on $Q_1$}:
By \Cref{condition:moment_X_bound} and the fact that $|Y_i-\rho^{\prime}\left(\boldsymbol{X}_i^{\top} \boldsymbol{\beta}\right)|\leq 1$,
$$\begin{aligned}
	\mathbb E\|\sum_{i=1}^n (Y_i-\rho^{\prime}(\boldsymbol{X}_i^\top \boldsymbol{\beta}_0))\boldsymbol{X}_i\|^2&\leq \mathbb E\left(\sum_{i=1}^n \|\boldsymbol{X}_i\|^2\right)\leq C_2 n p.
\end{aligned}$$
Note that $|Q_1(\boldsymbol{\beta})|\leq \|\boldsymbol{\beta}-\boldsymbol{\beta}_0\|\cdot\|\sum_{i=1}^n (Y_i-\rho^{\prime}(\boldsymbol{X}_i^\top \boldsymbol{\beta}_0))\boldsymbol{X}_i\|$.
For any $\epsilon_1>0$, we have
\begin{equation}\label{Q1_upper_bd}
\begin{aligned}
&\mathbf P\left( \exists \bbeta, \text{such that}  \|\boldsymbol{\beta}-\boldsymbol{\beta}_0\|^2=\tilde{B}p/n, Q_1(\boldsymbol{\beta})\geq \sqrt{C_2np}\|\boldsymbol{\beta}-\boldsymbol{\beta}_0\|/\sqrt{\epsilon_1} \right) \\
\leq & P\left( \|\sum_{i=1}^n (Y_i-\rho^{\prime}(\boldsymbol{X}_i^\top \boldsymbol{\beta}_0))\boldsymbol{X}_i\|  \geq \sqrt{C_2np}/\sqrt{\epsilon_1} \right) \\
\leq & \epsilon_1,
 \end{aligned}
\end{equation}
where the second inequality is due to Markov's inequality.

\textbf{Lower bound on $Q_2$}:
 
Fix \(C^\dagger>\max\{4,2/\zeta\}\).
Define
$$\mathcal{S}_{+}(\boldsymbol{\beta})=\left\{i:\left|\boldsymbol{X}_i^{\top} (\boldsymbol{\beta}-\boldsymbol{\beta}_{0})\right| \leq \sqrt{c_2 C^{\dagger}}\|\boldsymbol{\beta}-\boldsymbol{\beta}_{0}\|\right\} \cap \left\{i:\left|\boldsymbol{X}_i^{\top} \boldsymbol{\beta}_{0}\right| \leq \sqrt{c_2 C^{\dagger}}\|\boldsymbol{\beta}_{0}\|\right\},
$$
where $c_2$ is the constant in the upper bound in \Cref{condition:SubgaussianX}.
By \Cref{lemma:SB_beta_set_size}, we have
$\left|\mathcal{S}_{+}(\boldsymbol{\beta})\right| \geq(1-2/C^{\dagger})n$.
For any $i\in \mathcal{S}_{+}(\boldsymbol{\beta})$, $|\boldsymbol{X}_i^\top (\boldsymbol{\beta}-\boldsymbol{\beta}_0)|\leq \sqrt{c_2 C^{\dagger}} \|\boldsymbol{\beta}-\boldsymbol{\beta}_0\|$ and $|\boldsymbol{X}_i^\top \boldsymbol{\beta}_0 |\leq \sqrt{c_2 C^{\dagger}}\|\boldsymbol{\beta}_{0}\|$.
Furthermore, for any $n$ sufficiently large such that $\tilde{B}\frac{p}{n}\leq C_3^2$ (recall that \Cref{condition:constant_signal} states that $\|\bbeta_0\|\leq C_3$), it holds for $i\in \mathcal{S}_{+}(\boldsymbol{\beta})$ that
$$|\boldsymbol{X}_i^\top \boldsymbol{\beta}_0 |+|\boldsymbol{X}_i^\top (\boldsymbol{\beta}-\boldsymbol{\beta}_0)|\leq \sqrt{c_2 C^{\dagger}} \left(\|\bbeta_0\| + \|\boldsymbol{\beta}-\boldsymbol{\beta}_0\| \right) \leq \sqrt{c_2 C^{\dagger}} \left( C_3 + \sqrt{\tilde{B}\frac{p}{n} } \right)
\leq  2\sqrt{c_2 C^{\dagger}}C_3. $$

We can lower bound $Q_2$ as follows:
\begin{equation}\label{Q2_lower_bound}
	\begin{aligned}
	Q_2&\geq \sum_{i\in \mathcal{S}_{+}(\boldsymbol{\beta})}\left( |\boldsymbol{X}_i^\top (\boldsymbol{\beta}-\boldsymbol{\beta}_0)|^2\cdot \inf \left\{\rho^{\prime\prime}(\boldsymbol{X}_i^\top \boldsymbol{\beta}_0+t):|t|\leq |\boldsymbol{X}_i^\top (\boldsymbol{\beta}-\boldsymbol{\beta}_0|)  \right\}  \right)\\
	&\geq \|\boldsymbol{\beta}-\boldsymbol{\beta}_0\|^2 \cdot \lambda_{\min}\left(\sum_{i\in \mathcal{S}_{+}(\boldsymbol{\beta})}\boldsymbol{X}_i \boldsymbol{X}_i^\top \right)\rho^{\prime\prime}( 2\sqrt{c_2 C^{\dagger}}C_3) \\
	&\geq \tilde{B} \frac{p}{n}\cdot 0.5c_1n\cdot \rho^{\prime\prime}( 2\sqrt{c_2 C^{\dagger}}C_3)
	\end{aligned}
\end{equation}
where the second inequality is due to the symmetry and the monotonicity of $\rho^{\prime\prime}(\cdot)$ and the third inequality is due to the lower bound in \Cref{condition:SubgaussianX} and $|\mathcal{S}_{+}(\boldsymbol{\beta})|\geq 0.5n$ and $|\mathcal{S}_{+}(\boldsymbol{\beta})|\geq (1-\zeta)n$.

Let
\[
a_0:=0.5c_1\rho^{\prime\prime}(2\sqrt{c_2C^\dagger}C_3)>0.
\]
By \Cref{lem:bounded-syn-cov-op-norm}, we have
\[
C_4\sqrt{\lambda_M^*p/n}=O_p(1).
\]
Thus, for any $\epsilon_2>0$, there exists a finite constant $L_{\epsilon_2}>0$ such that, for all sufficiently large $n$,
\begin{equation}\label{eq:finite_synthetic_score_quantile_bound}
\mathbb P\left(C_4\sqrt{\lambda_M^*p/n}\leq L_{\epsilon_2}\right)
\geq 1-\epsilon_2.
\end{equation}

In view of \eqref{Q3_upper_bd}, \eqref{Q1_upper_bd}, \eqref{Q2_lower_bound}, and \eqref{eq:finite_synthetic_score_quantile_bound}, with probability at least $1-\epsilon_1-\epsilon_2$, uniformly over the sphere $\|\boldsymbol{\beta}-\boldsymbol{\beta}_0\|^2=\tilde Bp/n$, a stochastic upper bound of $Q_1-Q_2+Q_3$ is given by
\begin{align*}
& \sqrt{\frac{C_2\tilde{B}}{\epsilon_1} } p- a_0\tilde Bp + L_{\epsilon_2}\sqrt{\tilde B}p \\
={}&p\sqrt{\tilde{B}}\left( \sqrt{C_2/\epsilon_1}+L_{\epsilon_2}-a_0\sqrt{\tilde B}\right).
\end{align*}
If we choose
\[
\tilde B>
\left(
\frac{\sqrt{C_2/\epsilon_1}+L_{\epsilon_2}}{a_0}
\right)^2,
\]
then from \eqref{beta_time_gradient}, we have
\begin{align*}
&\mathbb{P}\left(
    \exists \bbeta, \text{s.t.} \|\boldsymbol{\beta}-\boldsymbol{\beta}_0\|^2=\tilde{B}p/n, (\boldsymbol{\beta}-\boldsymbol{\beta}_0)^{\top} F(\boldsymbol{\beta})\geq 0
    \right)
\leq \epsilon_1+\epsilon_2.
\end{align*}
Taking, for example, $\epsilon_1=\epsilon_2=\epsilon/2$, we conclude that for any $\epsilon>0$, there is $\tilde{B}$ such that
\begin{align*}
&  \mathbb{P}\left(\|\widehat{\boldsymbol{\beta}}_{M}-\boldsymbol{\beta}_0\|^2\leq \tilde{B}\frac{p}{n} \right) \\
\geq{} & \mathbb{P}\left((\boldsymbol{\beta}-\boldsymbol{\beta}_0)^{\top} F(\boldsymbol{\beta})<0 \text{ for all }\boldsymbol{\beta} \text{ satisfying } \|\boldsymbol{\beta}-\boldsymbol{\beta}_0\|^2=\tilde{B}p/n\right)\\
\geq{} & 1-\epsilon.
\end{align*}
\Cref{condition:SubgaussianX} implies that $\mathbb{X}^\top\mathbb{X}$ is invertible. As a result, the objective function in \eqref{eq: SRE_def} is strictly concave and $\widehat{\boldsymbol{\beta}}_{M}$ is the unique root of $F(\boldsymbol{\beta})=0$.
We conclude that $\|\widehat{\boldsymbol{\beta}}_{M}-\boldsymbol{\beta}_0\|^2=O_p(p/n)$.

 The proof of $\|\widehat{\boldsymbol{\beta}}_{\infty}-\boldsymbol{\beta}_0\|^2=O_p(\frac{p}{n})$ follows by a similar argument if we replace $F(\bbeta)$ by

 $$F_\infty(\boldsymbol{\beta})=\sum_{i=1}^{n} \left(Y_i-\rho^{\prime}(\boldsymbol{X}_i^\top \boldsymbol{\beta})\right )\boldsymbol{X}_i +
\tau \mathbb E \left(Y^*-\rho^{\prime}(\boldsymbol{X}^{*\top} \boldsymbol{\beta})\right )\bX^*. $$
Correspondingly, a modification of the upper bound on $Q_3$ will replace
$\lambda_{M}^{*}$ by $\lambda_{\infty}^{*}:=\lambda_{\max}\left(\mathbb E (\bX^*\bX^{*\top})\right)$.
The bounds for $Q_1$ and $Q_2$ remain the same.

\subsection{Proof of Theorem~\ref{thm:MAP_bounded}}
\label{proof:MAP_bounded_proportional}

This section proves that the SRE is bounded in the linear asymptotic regime.

We begin with the bound for the SRE $\widehat{\bbeta}_{M}$ with finite $M$, which is given by following minimization problem:
{\small
$$\begin{aligned}
	\widehat{\bbeta}_{M}&=\arg \min_{\bbeta\in \mathbb R^p} \left\{\sum_{i=1}^{n} \log\left(1+\exp(-(2Y_i-1)\bX_i^{\top} \bbeta) \right)+ \frac{\tau}{M} \sum_{i=1}^{M} \log\left(1+\exp(-(2Y_i^*-1)\bX_i^{*\top} \bbeta) \right) \right\}.
\end{aligned}$$}
Note that the objective function evaluated at $\widehat{\bbeta}_{M}$ is necessarily no greater than that evaluated at $\bbeta=\mathbf{0}$, which is $(n+\tau)\log (2)$. Together with an elementary inequality that $\max\{0,t\}\leq \log(1+\exp(t))$ for $t\in \mathbb R$, we have
\begin{equation} \label{loss_at_beta_less_than_l_0}
	\frac{\tau}{M} \sum_{i=1}^{M} \max\{0,-(2Y_i^*-1)\bX_i^{*\top} \widehat{\bbeta}_{M} \}\leq (n+\tau)\log(2) .
\end{equation}

Note that the left-hand side of \eqref{loss_at_beta_less_than_l_0} can be lower bounded by
$$
\|\widehat{\bbeta}_{M}  \| \left( \inf _{\|\boldsymbol{\beta}\|=1} \frac{1}{M} \sum_{i=1}^M \max\{0,-(2Y_i^*-1)\bX_i^{*\top} \bbeta \} \right),
$$
for which we have the following result.

\begin{lemma}\label{lemma:uniform_small_ball_probability}
	Under \Cref{conditions:synthetic_X_Y}, we have
\begin{equation}\label{small_ball_ReLi}
    \inf _{\|\boldsymbol{\beta}\|=1} \frac{1}{M} \sum_{i=1}^M \max\{0,-(2Y_i^*-1)\bX_i^{*\top} \bbeta \} \geq \frac{\eta_0 \nu}{4}
\end{equation}
with probability at least $1-\exp(-c_BM)-\exp\left(-\frac{\nu^2}{2} M +p\log\left(1+\frac{8C_B}{\eta_0\nu}\right)\right)$, where $c_B,C_B,\eta_0,\nu$ are positive constants that depend on the constants $\kappa_{-},\kappa_{+},K_X$ and $q$ in \Cref{conditions:synthetic_X_Y}.

Furthermore, for any $\bbeta\in \mathbb R^p$ with $\|\bbeta\|_2=1$,
\begin{equation}\label{eq:clip-sign-separation}
    \mathbb E \max\{0,-(2Y^*-1)\bX^{*\top} \bbeta \} \geq \eta_0 \nu.
\end{equation}
\end{lemma}
\bigskip

We defer the proof of \Cref{lemma:uniform_small_ball_probability}.

The conditions of \Cref{thm:MAP_bounded} imply that  $1+\frac{n}{\tau}\leq C_*$.
When \eqref{small_ball_ReLi} holds, we can conclude from \eqref{loss_at_beta_less_than_l_0} that $\|\widehat{\bbeta}_{M}\|_2\leq \frac{4C_*\log(2)}{\eta_0 \nu}$. Suppose $M\geq\frac{4\log(1+8C_B/(\eta_0\nu))}{\nu^2}p$.
\Cref{lemma:uniform_small_ball_probability} implies that
$$\mathbb P\left(\|\widehat{\bbeta}_{M}\|_2\leq \frac{4C_*\log(2)}{\eta_0 \nu} \right)\geq 1-2\exp(-\min(c_B,\nu^2/4) M).$$

For the pSRE $\widehat{\bbeta}_{\infty}$, we have the following analogy of \eqref{loss_at_beta_less_than_l_0}:
\begin{equation} \label{loss_at_beta_less_than_l_0_infinite}
	\tau\mathbb E \max\{0,-(2Y^*-1)\bX^{*\top} \widehat{\bbeta}_{\infty} \}\leq (n+\tau)\log(2) .
\end{equation}
The left-hand side can be lower bounded using \eqref{eq:clip-sign-separation} in \Cref{lemma:uniform_small_ball_probability}, which proves that
$$\|\widehat{\bbeta}_{\infty}\|_2\leq \frac{C_*\log(2)}{\eta_0 \nu}. $$
Therefore, we complete the proof of \Cref{thm:MAP_bounded}.

The rest of this section is devoted to prove \Cref{lemma:uniform_small_ball_probability}, which is in turn based on the following lemma.

\begin{lemma}\label{lemma:single_beta_small_ball}
	Suppose $\{\bX_i^*,Y_i^*\} $ are i.i.d. copies of $(\bX^*,Y^*)$ generated under \Cref{conditions:synthetic_X_Y}. There are positive constants $\eta_0$ and $\nu$ such that for any $\bbeta\in \mathbb R^p$ with $\|\bbeta\|_2=1$, it holds that
	$$\frac{1}{M}\sum_{i=1}^{M} \max\{0,-(2Y_i^*-1)\bX_i^{*\top} \bbeta \} \geq \frac{\eta_0 \nu}{2} $$
with probability at least $1-\exp\left(- \frac{M\nu^2}{2}\right)$.
\end{lemma}

\begin{proof}[Proof of \Cref{lemma:single_beta_small_ball}]

By \Cref{prop:var_small_ball},
there exist two positive constants $\eta_0,\rho_0\in (0,1)$ that only depend on $\kappa_{-},\kappa_{+},K_X$, such that for any $\bbeta\in \mathbb R^p$ with $\|\bbeta\|_2=1$
$$
\mathbb P(|\bX_i^{*\top} \bbeta|>\eta_0)\geq \rho_0 . $$

For any $i$, let $A_i$ denote the indicator of the event ${\{\max\left(0,-(2Y_i^*-1)\bX_i^{*\top}\bbeta\right)>\eta_0\}}$.
We will first find the lower bound of $\mathbb E(A_i)$ and then apply Hoeffding's inequality to guarantee $\sum_{i=1}^M A_i$ is stochastically large.
Note that $A_i=1$ if and only if $|\bX_i^{*\top}\bbeta|>\eta_0$ and the sign of $(1-2Y_i^*)$ is the same as the sign of $\bX_i^{*\top}\bbeta$.
By the law of total expectation, we have
 \begin{align*}
 	\mathbb P\left(\max\{0,-(2Y_i^*-1)\bX_i^{*\top} \bbeta \}>\eta_0 \right)&=\mathbb E\left[\mathbb E \left(\mathbf{1}\left\{\max\left(-(2Y_i^*-1)\bX_i^{*\top}\bbeta,0\right)>\eta_0 \right\} \mid \bX_i^*\right) \right]\\
 	&=\mathbb E\left[\mathbf{1}\{|\bX_i^{*\top}\bbeta|>\eta_0\} \mathbb {P} \left( (2Y_i^*-1)\bX_i^{*\top}\bbeta<0 |\bX_i^*\right)    \right]\\
 	&\geq  \mathbb P(|\bX_i^{*\top}\bbeta|>\eta_0) \min(q,1-q)\\
 	&\geq \min(q,1-q)\rho_0,
 \end{align*}
where the first inequality is due to \Cref{conditions:synthetic_X_Y}.
Denote by $\nu=  \min(q,1-q)\rho_0$.
We have shown that $\mathbb E(A_i)\geq  \nu$.
By Hoeffding's inequality,
$\mathbb P\left(\sum_{i=1}^M A_i < \frac{M\nu}{2}\right)\leq \exp(-\frac{M\nu^2}{2})$. Note that the event $\{\sum_{i=1}^M A_i \geq \frac{M\nu}{2} \}$ implies that
$\sum_{i=1}^{M} \max\{0,-(2Y_i^*-1)\bX_i^{*\top} \bbeta \} \geq \frac{M\nu}{2}\eta_0 $. Thus, we conclude that
\begin{align*}
	\mathbb P\left(\frac{1}{M}\sum_{i=1}^{M} \max\{0,-(2Y_i^*-1)\bX_i^{*\top} \bbeta \} < \frac{\nu}{2}\eta_0 \right)\leq \exp\left(-\frac{M\nu^2}{2}\right).
\end{align*}
\end{proof}

\begin{proof}[Proof of \Cref{lemma:uniform_small_ball_probability}]

Denote by $\mathbb X^*$ the synthetic covariate matrix.

By \Cref{prop:design_op_norm}, under \Cref{conditions:synthetic_X_Y},
the event $E_1:=\{\|\mathbb X^*\|\leq C_B\sqrt{M} \}$ holds with probability at least
$1-\exp(-c_B M)$, where $c_B,C_B$ are constants that only depend on $(\kappa_+, K_X)$.

We fixed a $\left(\frac{\eta_0\nu}{4C_B}\right)$-net $\mathcal N$ to cover the unit sphere $\mathbb{S}^{p-1}.$  By a volume argument, $|\mathcal N|\leq (1+\frac{8C_B}{\eta_0\nu})^p$.
Denote by $E_2$ the event that
$$\left\{\frac{1}{M}\sum_{i=1}^{M} \max\{0,-(2Y_i^*-1)\bX_i^{*\top} \bbeta \} \geq \frac{\eta_0 \nu}{2} \quad \text{for all }\bbeta_k \in \mathcal N \right\}. $$
By \Cref{lemma:single_beta_small_ball}, $E_2$ happens with probability at least $1-|\mathcal N|\exp \left( -\frac{M \nu^2}{2} \right)$.

Under the events $E_1$ and $E_2$, for any $\|\bbeta\|=1$, we can pick $\bbeta_1\in \mathcal N$ such that $\|\bbeta-\bbeta_1\|\leq \frac{\eta_0\nu}{4C_B}$. Then we derive
\begin{align*}
	\frac{1}{M}&\left(\sum_{i=1}^M  \max\{0,-(2Y_i^*-1)\bX_i^{*\top} \bbeta \}-\sum_{i=1}^M  \max\{0,-(2Y_i^*-1)\bX_i^{*\top} \bbeta_{1} \}\right)\\
	&\stackrel{(1)}{\leq}\frac{1}{M}\sum_{i=1}^M |\bX_i^{*\top} (\bbeta-\bbeta_1)|\stackrel{(2)}{\leq}\frac{1}{\sqrt{M}}\|\mathbb X^*(\bbeta-\bbeta_1)\| \stackrel{(3)}{\leq} \frac{\eta_0\nu}{4C_B\sqrt{M}}\|\mathbb X^*\|_{op}\leq \frac{\eta_0 \nu}{4}
\end{align*}
where the step (1) is due to the inequalities $\max(0,a)-\max(0,b)\leq |a-b|$ and $|2Y_i^*-1|=1$, the step (2) is due to the generalized mean inequality, and the step (3) is due to the definition of operator norm and the fact that $\|\bbeta-\bbeta_1\|\leq \frac{\eta_0 \nu}{4 C_B}$. We complete the proof by noticing that the union bound on the exception probabilities of  $E_1$ and $E_2$ is $\exp(-c_BM)+\exp\left(-\frac{\nu^2}{2} M +p\log\left(1+\frac{8C_B}{\eta_0\nu}\right)\right)$.
\end{proof}

\subsection{Proof of Theorem \ref{thm:exact_cat_M_MAP_noninformative(modify)} part (2) and Theorem \ref{thm:exact_cat_M_MAP_informative} part (2)}
\label{proof_sec:logitic_exact}

In this section, we provide the proof for \Cref{thm:exact_cat_M_MAP_informative} part (2) in the case where $\xi\in (-1,1)$.
We omit the proofs for \Cref{thm:exact_cat_M_MAP_noninformative(modify)} part (2) and the special case of $\xi=1$ in \Cref{thm:exact_cat_M_MAP_informative} part (2) because they follow a similar and simpler argument using a rank-one decomposition.

We recall the distributional conditions and streamline the notations. The observed covariates are
 $\left\{ \boldsymbol{X}_i\right\}_{i=1}^n \stackrel{\text { i.i.d. }}{\sim} \mathcal{N}\left(\mathbf{0},   \mathbf{I}_p\right)$  and the auxiliary covariates are
 $\left\{ \boldsymbol{X}^*_i\right\}_{i=1}^M \stackrel{\text { i.i.d. }}{\sim} \mathcal{N}\left(\mathbf{0},  \mathbf{I}_p\right)$.
 Additionally, the observed responses are
 $Y_i\sim \text{Bern}(\rho^\prime(\boldsymbol{X}_i^\top \boldsymbol{\beta}_0))$, the auxiliary responses are $Y_i^*\sim \text{Bern}(\rho^\prime(\boldsymbol{X}_i^{*\top} \boldsymbol{\beta}_s))$, the true coefficients $\boldsymbol{\beta}_0$ satisfies
 $\lim_{p\rightarrow\infty} \|\boldsymbol{\beta}_0\|^2=\kappa_1^2$, the auxiliary coefficients satisfy
$\lim_{p\rightarrow\infty} \|\boldsymbol{\beta}_s\|^2=\kappa_2^2$,
and $\lim_{p\rightarrow\infty}\frac{1}{\|\boldsymbol{\beta}_0\|\|\boldsymbol{\beta}_s\|}\langle \boldsymbol{\beta}_0,\boldsymbol{\beta}_s\rangle=\xi\in (-1,1)$.
   
To streamline the notation, we write $\mathbf{y}_1=(Y_i)_{i\in [n]}$ for the observed response vector and $\mathbf{y}_2=(Y_j^*)_{j\in [M]}$ for the auxiliary response vector.

In the following, we first present an overview of our proof, followed by an introduction of the main technical tools and a layout of lemmas. We then dive into the details of the proof.

\subsubsection{Road-map of the proof}

\smallskip

\textbf{First step: Reformulation of original problem.}
To make our optimization problem more suitable for exact asymptotic analysis, we execute a series of transformations on the original optimization problem. By integrating these transformation steps, we reach an equivalent formulation known as the Primal Optimization (PO) problem:
$$
\begin{aligned}
	\min _{\boldsymbol{\beta}\in \mathcal{S}_{\bbeta},  \mathbf{u}_1 \in \mathbb{R}^n, \mathbf{u}_2 \in \mathbb{R}^M} \max _{\bv \in \mathcal{S}_{\bv}}&\left(\frac{1}{n} \mathbf{1}^T \rho\left(\mathbf{u}_1\right)-\frac{1}{n} \mathbf{y}_1^T \mathbf{u}_1+\frac{\tau_0}{M} \mathbf{1}^T \rho\left(\mathbf{u}_2\right)-\frac{\tau_0}{M} \mathbf{y}_2^T \mathbf{u}_2\right.\\
	&\left. +\frac{1}{\sqrt{n}} \bv^T\left(\left[\begin{array}{l}
\mathbf{u}_1 \\
\mathbf{u}_2
\end{array}\right]-  \mathbf{H} \boldsymbol{\beta}_S\right)-\frac{1}{\sqrt{n}}   \bv^T \mathbf{H} \boldsymbol{\beta}_{S^{\perp}}\right)
\end{aligned}
$$
where $\mathbf{H}$ is a matrix with entries that are i.i.d. standard normal,  $\bbeta_{S}:=\mathbf{P}\bbeta$ and $\bbeta_{S^{\perp}}:=\mathbf{P}^\perp \bbeta$, where $\mathbf{P}$ is the projection matrix onto the column space spanned by $\bbeta_0$ and $\bbeta_s$ and $\mathbf{P}^\perp$ is the projection onto the orthogonal complement of that space.

\smallskip
\textbf{Second step: Reduction  to an Auxiliary Optimization (AO) problem.}
The particular form of PO allows us to use the Convex Gaussian Min-max Theorem \citep{thrampoulidis2015regularized}, which characterizes the exact asymptotic behavior of min-max optimization problems that are affine in Gaussian matrices.
This result enables us to characterize the properties of $\widehat{\bbeta}_M$ by studying the asymptotic behavior of the following, arguably simpler, Auxiliary Optimization (AO) problem:
$$
\begin{aligned}
	\min _{\boldsymbol{\beta}_S \in \mathcal{S}_{\bbeta}, \boldsymbol{\beta}_{S^{\perp}} \in \mathcal{S}_{\bbeta}, \mathbf{u}_1 \in \mathbb{R}^n, \mathbf{u}_2 \in \mathbb{R}^M} \max _{\bv \in \mathcal{S}_{\bv}}&\left(\frac{1}{n} \mathbf{1}^T \rho\left(\mathbf{u}_1\right)-\frac{1}{n} \mathbf{y}_1^T \mathbf{u}_1+\frac{\tau_0}{M} \mathbf{1}^T \rho\left(\mathbf{u}_2\right)-\frac{\tau_0}{M} \mathbf{y}_2^T \mathbf{u}_2\right.\\
	&\left. +\frac{1}{\sqrt{n}} \bv^T\left(\left[\begin{array}{l}
\mathbf{u}_1 \\
\mathbf{u}_2
\end{array}\right]- \mathbf{H} \boldsymbol{\beta}_S\right) -\frac{1}{ \sqrt{n}}\left(\bv^T \mathbf{h}\left\|\mathbf{P}^{\perp} \boldsymbol{\beta}\right\|+\|\bv\| \mathbf{g}^T \mathbf{P}^{\perp} \boldsymbol{\beta}\right)\right)
\end{aligned}
$$
where $\mathbf{h} \in \mathbb{R}^{n+M}$ and $\mathbf{g} \in \mathbb{R}^p$ have i.i.d. standard normal entries.

\smallskip
\textbf{Third step: Scalarization of the Auxiliary Optimization problem.}
We further simplify AO to an optimization over some scalar variables. Specifically, we demonstrate that the asymptotic behavior of AO can be captured through the following optimization problem:
{\small
$$\begin{aligned}
\min _{\substack{ \alpha_1\in \mathbb{R} , \alpha_2\in \mathbb{R}\\v,\sigma>0}}\max _{r>0} & \quad \left( -\frac{r\sigma}{\sqrt{\noverp}}+\frac{r}{2v} -\frac{1}{4rv}-\kappa_1^2\alpha_1\mathbb{E}(\rho^{\prime\prime}(\kappa_1 Z_1))-\frac{\tau_0^2}{4rvm}      \right.\\
 & -\tau_0 \kappa_2 \mathbb{E}(\rho^{\prime\prime}(\kappa_2 \xi Z_1+\kappa_2 \sqrt{1-\xi^2}Z_2))(\alpha_1\kappa_1\xi+\alpha_2\kappa_2\sqrt{1-\xi^2} )\\
 &+\mathbb{E}(M_{\rho(\cdot)}(\kappa_1\alpha_1 Z_1+\kappa_2\alpha_2 Z_2+\sigma Z_3+\frac{1}{rv}\text{Bern}(\rho^{\prime}(\kappa_1 Z_1)),\frac{1}{rv}))\\
 &+\tau_0\mathbb{E}(M_{\rho(\cdot)}(\kappa_1\alpha_1 Z_1+\kappa_2\alpha_2 Z_2+\sigma Z_3+\frac{\tau_0}{rvm}\text{Bern}(\rho^{\prime}(\kappa_2\xi Z_1+\kappa_2 \sqrt{1-\xi^2}Z_2)),\frac{\tau_0}{rvm}))
 \end{aligned}$$}
By checking the first-order optimality conditions of the above scalar optimization, we can derive the system of equations \eqref{nonlinear_four_equation}.

\subsubsection{Comparison to the Approximate Message Passing approach}\label{proof_sec:amp_comparison}

Another popular tool to establish the precise asymptotics is the Approximate Message Passing (AMP) technique \citep{donoho2009message,bayati2011dynamics}.
The AMP argument requires an iterative algorithm that approximates the estimator, where the iterates of the algorithm need to have known precise asymptotics.
The AMP argument then uses these known precise asymptotic distributions of the iterations to approximate the asymptotic distribution of the estimator.

In \cite{sterzinger2023diaconis}, the authors employ the AMP to study the precise asymptotics of the Maximum Diaconis-Ylvisaker prior penalized likelihood (MDYPL) estimator. The MDYPL estimator is defined as
\begin{equation}
\label{mdypl}
\widehat{\boldsymbol{\beta}}_{DY}=\arg\max_{\boldsymbol{\beta}} \sum_{j=1}^n\left\{\left(\alpha Y_j+(1-\alpha) \zeta^{\prime}\left(\boldsymbol{X}_j^{\top} \boldsymbol{\beta}_P\right)\right) \boldsymbol{x}_j^{\top} \boldsymbol{\beta}-\rho\left(\boldsymbol{X}_j^{\top} \boldsymbol{\beta}\right)\right\},
\end{equation}
where $\rho(t)=\log(1+e^t)$ and $\boldsymbol{\beta}_P$ is the prior mode.
It is clear that MDYPL is restricted to scenarios with $p<n$ and requires a full-rank covariate matrix.
If $p>n$ or if the design matrix is not of full rank,  the MDYPL estimator does not exist.

The AMP technique is suitable for the MDYPL estimator because
\begin{enumerate}
    \item[(a).] there is only one set of covariate vectors, and
    \item [(b).] the synthetic responses are set to be $0.5$ since \cite{sterzinger2023diaconis} assume the prior mode in \eqref{mdypl} is fixed as $\boldsymbol{\beta}_{P}=0$  (see Section 2 therein).
\end{enumerate}
These two conditions are crucial because they make the AMP recursion analytically tractable.

However, such an argument cannot be applied to the analysis of our SRE due to different sets of covariate vectors and general synthetic responses.
More concretely,
\begin{enumerate}
    \item [(a).] the SRE involves synthetic covariates that are not the same as the observed ones (see Equation \eqref{eq: SRE_def}), and
    \item [(b).] the synthetic responses are sampled from the model with a general $\boldsymbol{\beta}_s$, which is allowed to be correlated with the true coefficient $\boldsymbol{\beta}_0$ (see \Cref{condition:informative_syn_data} in \Cref{sec:exact_asymptocis_infor}).
\end{enumerate}
Therefore, the AMP argument used in \cite{sterzinger2023diaconis} does not apply to the analysis of the SRE.

Furthermore, when applying the AMP technique, the iterative algorithm for the SRE will quickly become too complicated due to the synthetic data, and deriving the precise asymptotic distributions of the iterations is very difficult even in the simplest case where $\boldsymbol{\beta}_s=0$.
CGMT provides an effective alternative to bypass these complexities.

\subsubsection{Introduction of Convex Gaussian Min-max Theorem}
Our analysis is based on the Convex Gaussian Min-max Theorem (CGMT), which we will briefly review here; detailed theory and application can be found in \cite{thrampoulidis2015regularized,thrampoulidis2016recovering,thrampoulidis2018precise}.
This technique connects a Primary Optimization (PO) problem with an Auxiliary Optimization (AO) problem, which is easy to analyze yet allows studying various aspects of the PO.
Specifically, we define the PO and AO problems as follows:
\begin{align}
\text{(PO)\quad}	&\Phi(\mathbf{G}):=\min _{\mathbf{w} \in \mathcal{S}_{\mathbf{w}}} \max _{\mathbf{u} \in \mathcal{S}_{\mathbf{u}}} \mathbf{u}^T \mathbf{G w}+\psi(\mathbf{u}, \mathbf{w}) \label{CGMT:PO}\\
\text{(AO)\quad}	&\phi(\mathbf{g}, \mathbf{h}):=\min _{\mathbf{w} \in \mathcal{S}_{\mathbf{w}}} \max _{\mathbf{u} \in \mathcal{S}_{\mathbf{u}}}\|\mathbf{w}\| \mathbf{g}^T \mathbf{u}-\|\mathbf{u}\| \mathbf{h}^T \mathbf{w}+\psi(\mathbf{u}, \mathbf{w}) \label{CGMT:AO}
\end{align}
where $\mathbf{G} \in \mathbb{R}^{m \times n}, \mathbf{g} \in \mathbb{R}^m, \mathbf{h} \in \mathbb{R}^n, \mathcal{S}_{\mathbf{w}} \subset \mathbb{R}^n, \mathcal{S}_{\mathbf{u}} \subset \mathbb{R}^m$ and $\psi: \mathbb{R}^n \times \mathbb{R}^m \rightarrow \mathbb{R}$. Let $\mathbf{w}_{\Phi}:=\mathbf{w}_{\Phi}(\mathbf{G})$ and $\mathbf{w}_\phi:=\mathbf{w}_\phi(\mathbf{g}, \mathbf{h})$ denote arbitrary optimal minimizers of \eqref{CGMT:PO} and \eqref{CGMT:AO}, respectively.

\begin{lemma}[\cite{thrampoulidis2016recovering}] \label{CGMT:nonsaym}
Let $\mathcal{S}_{\mathbf{w}}$ and $\mathcal{S}_{\mathbf{u}}$ be two convex and compact sets. Assume the function $\psi(\cdot, \cdot)$ is convex-concave on $\mathcal{S}_{\mathbf{w}} \times \mathcal{S}_{\mathbf{u}}$. Also assume that $\mathbf{G}, \mathbf{g}$, and $\mathbf{h}$ all have entries i.i.d. standard normal. Then  for all $\mu \in \mathbb{R}$, and $t>0$,
$$
\mathbb{P}(|\Phi(\mathbf{G})-\mu|>t) \leq 2 \mathbb{P}(|\phi(\mathbf{g}, \mathbf{h})-\mu| \geq t)
$$
The probabilities are taken with respect to the randomness in $\mathbf{G}, \mathbf{g}$, and $\mathbf{h}$.
\end{lemma}

\begin{lemma}[Asymptotic CGMT \cite{thrampoulidis2016recovering}]\label{CGMT:asym}
Let $\mathcal{S}$ be an arbitrary open subset of $\mathcal{S}_{\mathbf{w}}$ and $\mathcal{S}^c:=\mathcal{S}_{\mathbf{w}} / \mathcal{S}$. Denote $\Phi_{\mathcal{S}^c}(\mathbf{G})$ and $\phi_{\mathcal{S}^c}(\mathbf{g}, \mathbf{h})$ be the optimal costs of the optimizations in \eqref{CGMT:PO} and \eqref{CGMT:AO}, respectively, when the minimization over $\mathbf{w}$ is now constrained over $\mathbf{w} \in \mathcal{S}^c$.
Suppose that there exist constants $\bar{\phi}<\bar{\phi}_{\mathcal{S}^c}$ such that $\phi(\mathbf{g}, \mathbf{h}) \xrightarrow{\mathbb{P}} \bar{\phi}$, and $\phi_{\mathcal{S}^c}(\mathbf{g}, \mathbf{h}) \longrightarrow \bar{\phi}_{\mathcal{S}^c}$. Then, $\lim _{n \rightarrow \infty} \mathbb{P}\left(\mathbf{w}_{\Phi}(\mathbf{G}) \in \mathcal{S}\right)=1$.
\end{lemma}

In the following, we equate $\left(1/ \noverp, \tau_0, m\right)=(p, \tau, M) / n$ with the understanding that when $n$ is finite these numbers are ratios and converge to some constants as $n$ increases to infinity.

\subsubsection{Additional useful lemmas}

Through the reformulation and transformation of the original optimization problem \eqref{eq: SRE_def}, we will frequently use the following lemma to flip the optimization order:
\begin{lemma}\citep[Sion's minimax theorem]{sion1958general}
Let $X \subset \mathbb{R}^n$ and $Y \subset \mathbb{R}^m$ be two convex spaces, at least one of which is compact. 
If $f: X \times Y \rightarrow \mathbb{R}$ is a continuous function that is concave-convex, i.e.
$f(\cdot, y): X \rightarrow \mathbb{R}$ is concave for fixed $y$, and
$f(x, \cdot): Y \rightarrow \mathbb{R}$ is convex for fixed $x$.

Then we have that
$$
\sup _{x \in X} \inf _{y \in Y} f(x, y)=\inf _{y \in Y} \sup _{x \in X} f(x, y) .
$$
\end{lemma}

The following result is also useful in our proof.
\begin{lemma}\label{lemma:flip_min_max}
    Let $K$, $\sigma$, and $V$ be any positive numbers. Let $\bg$ be a vector with the same dimension as $\btheta$ to be minimized. It holds that
$$
\min_{\|\btheta\|=1} \max_{r\in [0,V]}\{ r\sigma \bg^\top \btheta + r K \}= \max_{r\in [0,V]} \min_{\|\btheta\|=1} \{ r\sigma \bg^\top \btheta + r K \}
$$
\end{lemma}
\begin{proof}[Proof of \Cref{lemma:flip_min_max}]
We consider the following cases.
    \begin{itemize}
    \item Suppose $K-\sigma \|\bg\|>0$:
 It is clear that $K+\sigma \bg^\top \btheta\geq K-\sigma \|\bg\|>0$ for any unit vector $\btheta$. Therefore, the two sides can be computed as follows:
    $$\begin{aligned}
        &\min_{\|\btheta\|=1} \max_{r\in [0,V]}\{ r\sigma \bg^\top \btheta + r K \}=\min_{\|\btheta\|=1} V(K+\sigma \bg^\top \btheta)=V(K-\sigma\|\bg\|);\\
        &\max_{r\in [0,V]} \min_{\|\btheta\|=1} \{ r\sigma \bg^\top \btheta + r K \}=\max_{r\in [0,V]} r(K-\sigma\|\bg\|)=V(K-\sigma\|\bg\|).
    \end{aligned}$$
    
    \item Suppose $K-\sigma \|\bg\|\leq 0$:
    The left-hand side is
    $$\begin{aligned}
        \min_{\|\btheta\|=1} \max_{r\in [0,V]}\{ r\sigma \bg^\top \btheta + r K \}&=\min\left\{\min_{\|\btheta\|=1,K+\sigma \bg^\top \btheta>0}\max_{r\in [0,V]}\{ r\sigma \bg^\top \btheta + r K \}, \right. \\
        &\hskip 2cm \left.  \min_{\|\btheta\|=1,K+\sigma \bg^\top \btheta \leq 0}  \max_{r\in [0,V]}\{ r\sigma \bg^\top \btheta + r K \}          \right\}\\
        &=\min\left\{ \min_{\|\btheta\|=1,K+\sigma \bg^\top \btheta>0} V(K+\sigma\bg^\top \btheta)      ,0 \right\}\\
        &=0.
    \end{aligned}$$
The right-hand side is
$$\max_{r\in [0,V]} \min_{\|\btheta\|=1} \{ r\sigma \bg^\top \btheta + r K \}  =\max_{r\in [0,V]} r(K-\sigma \|\bg\|)=0.
$$
\end{itemize}
In either case, the two sides are equal.
\end{proof}

The following lemma shows that $\|\widehat{\bbeta}_M\|$ is bounded with high probability when the MLE based on the auxiliary data exists asymptotically.
According to \cite{candes2020phase}, the inequality that
$\kappa_2 < \overline{\kappa}_{\mathrm{MLE}}(m\delta)$ ($\overline{\kappa}_{\mathrm{MLE}}$
is defined in the lemma)
is sufficient for the auxiliary dataset to be non-separable with high probability.

\begin{lemma}\label{lemma:MAP_bounded_gaussian_design}
    Consider a standard normal variable $Z$ with density function $\varphi(t)$ and an independent continuous random variable $V_{\kappa}$ with density function $2\rho^\prime(\kappa  t)\varphi(t)$. Using the notation $x_{+} = \max(x, 0)$, we define
$$
\overline{\kappa}_{\mathrm{MLE}} (r)=\sup\left\{\kappa\geq 0: \frac{1}{r} < \min_{t \in \mathbb{R}}~\mathbb{E}\left\{ [(Z-t V_{\kappa})_{+}]^2 \right\}\right\}, \quad \forall r>2.$$
$\overline{\kappa}_{\mathrm{MLE}}(r)$ is a nondecreasing function with respect to $r$.
Assume \Cref{condition:proper_scaling,condition:informative_syn_data} hold.
Suppose $m\delta>2$.
If $\kappa_2 < \overline{\kappa}_{\mathrm{MLE}}(m\delta)$, then there exist constants $c_1$, $C_1>0$, and a threshold $M_0>0$ depending only on $m\delta,\tau_0$ and $\kappa_2$. For all $M \geq M_0$, the SRE $\widehat{\boldsymbol{\beta}}_M$ satisfies
 
$$\mathbb{P}\left(\|\widehat{\bbeta}_{M}\|\leq c_1\right)\geq 1-C_1 M^{-\alpha}$$
where $\alpha > 1$ is a constant that depends on $m\delta$ and $\kappa_2$.
\end{lemma}

The condition $\kappa_2 < \overline{\kappa}_{\mathrm{MLE}}(m\delta)$ in \Cref{lemma:MAP_bounded_gaussian_design}
places the auxiliary logistic regression below the separability phase-transition boundary established in Theorem 4 in SI Appendix H of \citet{sur2019modern},
which ensures that with high probability, the auxiliary data are non-separable and the auxiliary MLE is bounded.
In order to show that $\left\|\widehat{\beta}_M\right\|$ is bounded with high probability, we reduce the large-norm event to the existence of an almost-separating direction in the auxiliary data.
We then follow the reasoning presented in Theorem 4 in SI Appendix H of \citet{sur2019modern}. For the sake of completeness, we provide brief arguments here.

\begin{proof}[Proof of \Cref{lemma:MAP_bounded_gaussian_design}]

When the entries of the synthetic covariate matrix $\mathbb X^*$ are independent $N(0,1)$ variables and $p/M<1/2$, the least singular value of $\mathbb X^*\in \mathbb R^{M\times p}$ satisfies
\begin{equation}
	\label{min_singular_Xstar_lower_bound_boundness}
	\sigma_{\min }(\mathbb X^*) \geq \frac{1}{4}\sqrt{M},
\end{equation}
with probability at least $1-2 \exp \left(-\frac{1}{2}\left(\frac{3}{4}-\frac{1}{\sqrt{2}}\right)^2 M\right)$, which follows from \citet[Corollary 5.35]{vershynin2010introduction}. Recall \eqref{loss_at_beta_less_than_l_0}, we have
$$\frac{1}{M} \sum_{i=1}^{M} \max\{0,-(2y_i^*-1)\mathbf{x}_i^{*\top} \widehat{\bbeta}_{M} \}\leq \frac{n+\tau}{\tau}\log(2)  $$
Under event $\mathcal{E}_3=\{\sigma_{\min }(\mathbb X^*) \geq \frac{1}{4}\sqrt{M}\}$,
if $\|\widehat{\bbeta}_{M}\|_2 > \frac{n+\tau}{\tau \eps^2}4\log(2)$ ($\eps>0$ will be specified later), then
\begin{align*}
    \frac{1}{M} \sum_{i=1}^{M} \max\{0,-(2y_i^*-1)\mathbf{x}_i^{*\top} \widehat{\bbeta}_{M} \}&\leq\frac{n+\tau}{\tau}\log(2)\\
    &\leq \frac{n+\tau}{\tau}\log(2)4 \sqrt{\frac{1}{M}}\frac{\|\mathbb X^* \widehat{\bbeta}_{M}\|_2}{\|\widehat{\bbeta}_{M}\|_2}\\
    &\leq \frac{1}{\sqrt{M}}\eps^2 \|(2\mathbf{y}_2 -1) \circ\mathbb X^* \widehat{\bbeta}_{M}\|_2
\end{align*}
where $\circ$ denotes the usual Hadamard product and $\mathbf{y}_2=(y_1^*,\cdots,y_M^*)$; the last inequality implies that
$$
(2\mathbf{y}_2-1) \circ\mathbb X^* \widehat{\bbeta}_{M} \,\in\, \{(2\mathbf{y}_2-1) \circ\mathbb X^* \boldsymbol{b}   \mid \boldsymbol{b} \in \mathbb{R}^p \} \cap \mathcal{A}_{\eps},
$$
where the set $\mathcal{A}_{\eps}$ is defined as
$$\mathcal{A}_{\eps}:=\left\{\boldsymbol{u} \in \mathbb{R}^{M} \mid  \sum_{j=1}^M \max\{-u_j,0\} \leq  \sqrt{M}\eps^2 \|\bu\|_2\right\}.$$

The above relationship implies the following inequality:
\begin{equation}\label{pflemma:lower_bound_set_A_geometry}
    \mathbb{P}\left(\mathcal{E}_3, \|\widehat{\bbeta}_{M}\|_2\leq \frac{n+\tau}{\tau \eps^2}4\log(2)\right) \geq \mathbb{P}\left(\mathcal{E}_3, \{(2\mathbf{y}_2-1) \circ\mathbb X^* \boldsymbol{b}   \mid \boldsymbol{b} \in \mathbb{R}^p \} \cap \mathcal{A}_{\eps}=\{\mathbf{0}\}\right).
\end{equation}

Therefore, it is sufficient to demonstrate that the probability of the complement of the right-hand side of \eqref{pflemma:lower_bound_set_A_geometry} decays polynomially fast.
Following the reasoning in Theorem 4 in SI Appendix H of \citet{sur2019modern}, we identify the existence of positive constants $M_0 := M_0(M/p, \kappa_2)$ and $\epsilon_0 := \epsilon_0(M/p, \kappa_2)$, ensuring that for all $M > M_0$,
\begin{equation}\label{pflemma:interset_prob_upper_bound}
    \mathbb{P}\left(\{(2\mathbf{y}_2-1) \circ\mathbb X^* \boldsymbol{b}   \mid \boldsymbol{b} \in \mathbb{R}^p \} \cap \mathcal{A}_{\eps_0}\neq \{\mathbf{0}\}\right)\leq C_1 M^{-\alpha}, 
\end{equation}
where $\alpha>1$ and $C_1>0$ are constants that depend only on $M/p$ and $\kappa_2$. By combining \eqref{pflemma:lower_bound_set_A_geometry}, \eqref{pflemma:interset_prob_upper_bound}, and the bound on the minimum singular value of $\mathbb{X}^*$, we conclude that constants $c_1, C_1, \alpha (>1)$, and $M_0$ exist, which depend on $M/p$, $\tau/n$, and $\kappa_2$. These constants ensure that, for all $M>M_0$,
$$\mathbb{P}\left(\|\widehat{\bbeta}_{M}\|_2\leq c_1\right)\geq 1-C_1 M^{-\alpha}.
$$
The proof is completed.
\end{proof}

The next lemma establishes a bound on the norm of a normal random vector.

\begin{lemma}\label{lemma:concetration_norm_gaussian_vector}
    Let $\bZ\in \mathbb R^n$ be a vector of i.i.d. standard normal variables, then we have
    $$\mathbb P\left(\|\bZ\|>2\sqrt{n} \right)\leq \exp(-n/2)$$
\end{lemma}
\begin{proof}
See \citet[Example 2.28]{wainwright2019high}.
\end{proof}

The next lemma is useful when we find the optimality condition for the scalar optimization problem.
\begin{lemma}[Identities for logistic link]
\label{useful_identity} Let $\rho^{\prime}(t):=\frac{e^t}{1+e^t}$ and $Z_1,Z_2\sim N(0,1)$  independently. For any $\kappa_1>0,\kappa_2>0$ and $\xi\in [-1,1]$, we have
$$
\begin{aligned}
	\mathbb{E}(\rho^{\prime}(\kappa_1 Z_1))=\frac{1}{2}&\quad\quad  \mathbb{E}(\rho^{\prime}(\kappa_1\xi Z_1+\kappa_2\sqrt{1-\xi^2}Z_2))=\frac{1}{2}\\
	\mathbb{E}(Z_1^2\rho^{\prime}(\kappa_1\xi Z_1+\kappa_2\sqrt{1-\xi^2}Z_2))=\frac{1}{2}&\quad\quad  \mathbb{E}(Z_1Z_2\rho^{\prime}(\kappa_1\xi Z_1+\kappa_2\sqrt{1-\xi^2}Z_2))=0
\end{aligned}
$$
\end{lemma}

\begin{proof}
    Direct consequences of the symmetry of the standard normal distribution.
\end{proof}

The next lemma summarizes the partial derivatives of the Moreau envelope function, which will be used frequently when we derive the system of equations from the first-order optimality condition in \eqref{deri: first order}.

\begin{lemma}[\cite{rockafellar2009variational}]
\label{supp:lemma:morea_lemma}
 Let $\Phi: \mathbb{R}^d \rightarrow \mathbb{R}$ be a convex function. For $\mathbf{v} \in \mathbb{R}^d$ and $t \in \mathbb{R}_{+}$, the Moreau envelope function is defined as,
$$
M_{\Phi(\cdot)}(\mathbf{v}, t)=\min _{\mathbf{x} \in \mathbb{R}^d} \Phi(\mathbf{x})+\frac{1}{2 t}\|\mathbf{x}-\mathbf{v}\|^2,
$$
and the proximal operator is the solution to this optimization, i.e.,
$$
\operatorname{Prox}_{t \Phi(\cdot)}(\mathbf{v})=\arg \min _{\mathbf{x} \in \mathbb{R}^d} t\Phi(\mathbf{x})+\frac{1}{2 }\|\mathbf{x}-\mathbf{v}\|^2 .
$$

The derivative of the Moreau envelope function can be computed as follows,
$$
\frac{\partial M_{\Phi(\cdot)}}{\partial \mathbf{v}}=\frac{1}{t}\left(\mathbf{v}-\operatorname{Prox}_{ t\Phi(\cdot)}(\mathbf{v})\right), \quad \frac{\partial M_{\Phi(\cdot)}}{\partial t}=-\frac{1}{2 t^2}\left(\mathbf{v}-\operatorname{Prox}_{t \Phi(\cdot)}(\mathbf{v})\right)^2
$$
\end{lemma}

\subsubsection{Reformulation and transformation}

The goal of this subsection is to reformulate the optimization for the estimator into a PO problem and define the associated AO problem.
We start with rewriting the optimization in \eqref{eq: SRE_def} as
$$
\min _{\boldsymbol{\beta} \in \mathbb{R}^p}  \left\{ \frac{1}{n} \mathbf{1}^T \rho\left( \mathbf{H}_1 \boldsymbol{\beta}\right)-\frac{1}{n} \mathbf{y}_1^T \mathbf{H}_1 \boldsymbol{\beta}+\frac{\tau_0}{M} \mathbf{1}^T \rho\left( \mathbf{H}_2 \boldsymbol{\beta}\right)-\frac{\tau_0}{M } \mathbf{y}_2^T \mathbf{H}_2 \boldsymbol{\beta} \right\}
$$
where the action of function $\rho(\cdot)$ on a vector is considered entry-wise, $\mathbf{y}_1 \in \mathbb{R}^n$ is the vector of observed responses and $\mathbf{y}_2 \in \mathbb{R}^M$ is the vector of auxiliary responses, $\mathbf{H}_1 \in \mathbb{R}^{n \times p}$ is $\left[\mathbf{x}_1, \ldots, \mathbf{x}_n\right]^T$ and $\mathbf{H}_2 \in \mathbb{R}^{M \times p}$ is $\left[\mathbf{x}_1^*, \ldots, \mathbf{x}_M^*\right]^T$. Let $\mathbf{H}=\left[\begin{array}{c}\mathbf{H}_1 \\ \mathbf{H}_2\end{array}\right]$.
Note the entries of $\mathbf{H}$ are i.i.d. standard normal variables.

Introducing two new variables $\mathbf{u}_1$ and $\mathbf{u}_2$, we further rewrite the optimization as
$$
\begin{aligned}
\min _{\boldsymbol{\beta} \in \mathbb{R}^p, \mathbf{u}_1 \in \mathbb{R}^n, \mathbf{u}_2 \in \mathbb{R}^M}    &   \left(\frac{1}{n} \mathbf{1}^T \rho\left(\mathbf{u}_1\right)-\frac{1}{n} \mathbf{y}_1^T \mathbf{u}_1+\frac{\tau_0}{M} \mathbf{1}^T \rho\left(\mathbf{u}_2\right)-\frac{\tau_0}{M} \mathbf{y}_2^T \mathbf{u}_2\right) \\
\text { s.t. } &  {\left[\begin{array}{c}
\mathbf{u}_1 \\
\mathbf{u}_2
\end{array}\right]=\mathbf{H} \boldsymbol{\beta} }.
\end{aligned}
$$
Using a Lagrange multiplier, we rewrite the above optimization as a min-max optimization
{\footnotesize
\begin{equation}\label{eq_minmax_unconstraint}
\min _{\boldsymbol{\beta} \in \mathbb{R}^p, \mathbf{u}_1 \in \mathbb{R}^n, \mathbf{u}_2 \in \mathbb{R}^M} \max _{\bv \in \mathbb{R}^{n+M}}\left(\frac{1}{n} \mathbf{1}^T \rho\left(\mathbf{u}_1\right)-\frac{1}{n} \mathbf{y}_1^T \mathbf{u}_1+\frac{\tau_0}{M} \mathbf{1}^T \rho\left(\mathbf{u}_2\right)-\frac{\tau_0}{M} \mathbf{y}_2^T \mathbf{u}_2+\frac{1}{\sqrt{n}} \bv^T\left(\left[\begin{array}{l}
\mathbf{u}_1 \\
\mathbf{u}_2
\end{array}\right]- \mathbf{H} \boldsymbol{\beta}\right)\right)
\end{equation}
}

We reformulate the original loss function into a new form that is tailored for the application of CGMT,
as the current mini-max optimization problem is affine in the Gaussian matrix $\mathbf{H}$.
To utilize CGMT, we need to further constrain the feasible sets of $ \boldsymbol{\beta}$, $\mathbf{u}_1$, $\mathbf{u}_2$ and $\bv$ in \eqref{eq_minmax_unconstraint} to be both compact and convex.
We constrain these feasible sets to be compact because this constraint is one of the technical conditions for switching the order of minimization and maximization in the minimax theorem.

\paragraph{Feasible sets for optimization}
We denote by $(\widehat{\bbeta}_{M},\widehat{\mathbf{u}}_1,\widehat{\mathbf{u}}_2)$ the solution to \eqref{eq_minmax_unconstraint}.
According to \Cref{lemma:MAP_bounded_gaussian_design}, there exist constants
 $\alpha > 1$, $c_1>0$, $C_1>0$, and a threshold $M_0$ depending only on  $\noverp_M=m\delta=\lim_{M\rightarrow \infty}\frac{M}{p}$ and $\kappa_2=\lim_{p\rightarrow \infty} \|\bbeta_s\|$ such that for all $M \geq M_0$, we have
 $$\mathbb{P}\left(\|\widehat{\bbeta}_{M}\|> c_1\right)\leq C_1 M^{-\alpha}.$$

Since $\sum_{M=1}^{\infty}C_1M^{-\alpha}<\infty$, the Borel--Cantelli lemma gives
\begin{equation}
    \label{borelcantelli_beta}
    \mathbb P\left(\{\|\widehat{\bbeta}_{M}\|> c_1\} \text{ happens infinitely often }  \right)=0.
\end{equation}

\eqref{borelcantelli_beta} allows us to safely constrain the sets $\mathcal{S}_{\boldsymbol{\beta}}$ in $\mathbb R^p$ to be bounded by some constants depending on $p$ for all $p$.
Specifically, we will choose deterministic compact feasible sets that contain the optimizer
eventually almost surely.
Let
\[
    \mathcal U_p=\operatorname{span}\{\bbeta_0,\bbeta_s\},\qquad
    \mathbf{P}=\mathbf{P}_{\mathcal U_p},\qquad \mathbf{P}^\perp=I_p-\mathbf{P} .
\]
Let $d_p=\dim(\mathcal U_p)\le 2$.
Let
$\mathbf{E}_p\in\mathbb R^{p\times d_p}$ be a matrix whose columns form an
orthonormal basis of $\mathcal U_p$. Fix a constant $R_\beta>c_1$ and define
\begin{equation}
    \label{eq:S_beta_product}
    \mathcal S_{\bbeta}
    =
    \left\{
    \bbeta\in\mathbb R^p:
    \|\mathbf{E}_p^\top\bbeta\|_\infty\le R_\beta,\|
\mathbf{P}^\perp\bbeta\|\le R_\beta
    \right\}.
\end{equation}
This set is convex and compact. Moreover, by \eqref{borelcantelli_beta},
\[
    \mathbb P\left(
    \widehat{\bbeta}_{M}\in\mathcal S_{\bbeta}
    \text{ for all sufficiently large }M
    \right)=1.
\]
Thus restricting the optimization to $\mathcal S_{\bbeta}$ does not affect
the asymptotic analysis.

The particular choice \eqref{eq:S_beta_product} is useful because it is
adapted to the decomposition induced by $P$.
Specifically,
\[
    \mathcal S_{\bbeta}
    =
    \mathbf{P}\mathcal S_{\bbeta}\oplus \mathbf{P}^\perp\mathcal S_{\bbeta},
\]
where
\[
    \mathbf{P}\mathcal S_{\bbeta}
    =
    \left\{
    \mathbf{E}_p a:\ a\in[-R_\beta,R_\beta]^{d_p}
    \right\},
    \qquad
    \mathbf{P}^\perp\mathcal S_{\bbeta}
    =
    \left\{
    \bbeta_\perp\in\mathcal U_p^\perp:\
    \|\bbeta_\perp\|\le R_\beta
    \right\},
\]
and $\mathcal U_p^\perp$ is the orthogonal complement of $\mathcal U_p$.
Consequently, optimizing over $\bbeta\in\mathcal S_{\bbeta}$ is equivalent
to optimizing over
\[
    \bbeta_S\in \mathbf{P}\mathcal S_{\bbeta},
    \qquad
    \bbeta_{S^\perp}\in \mathbf{P}^\perp\mathcal S_{\bbeta},
    \qquad
    \bbeta=\bbeta_S+\bbeta_{S^\perp}.
\]

Furthermore, based on the first-order optimality condition of the min-max
optimization in \eqref{eq_minmax_unconstraint}, the maximizer
$\widehat{\bv}$ of the inner problem satisfies
\[
\sqrt n\,\widehat{\bv}
=
\left[
\begin{array}{c}
\by_1-\rho^{\prime}(\widehat{\mathbf{u}}_1) \\[2mm]
\displaystyle \frac{\tau}{M}
\left\{\by_2-\rho^{\prime}(\widehat{\mathbf{u}}_2)\right\}
\end{array}
\right].
\]
Since the entries of $\rho^{\prime}(\widehat{\mathbf{u}}_1)$ and
$\rho^{\prime}(\widehat{\mathbf{u}}_2)$ are bounded by $1$, and the entries
of $\by_1,\by_2$ are either $0$ or $1$, we have
\[
    \|\widehat{\bv}\|^2
    \le
    \frac{1}{n}\|\by_1-\rho^{\prime}(\widehat{\mathbf{u}}_1)\|^2
    +
    \frac{\tau^2}{nM^2}
    \|\by_2-\rho^{\prime}(\widehat{\mathbf{u}}_2)\|^2
    \le
    1+\frac{\tau_0^2}{m},
\]
where $\tau_0=\tau/n$ and $m=M/n$. Since $m$ is bounded away from zero and
$\tau_0$ is bounded along the asymptotic sequence, we may choose a fixed
constant $R_v>0$ such that
\[
    R_v^2>1+\sup_p\frac{\tau_0^2}{m}.
\]
This radius will be used below in \eqref{AO_with_r_sigma_alpha}.
Define
\[
    \mathcal S_{\bv}
    =
    \left\{
    \bv\in\mathbb R^{n+M}:\ \|\bv\|\le R_v
    \right\}.
\]
We have $\widehat{\bv}\in\mathcal S_{\bv}$ for all sufficiently large $p$.
In the following, these deterministic feasible sets are denoted by
$\mathcal S_{\bbeta}$ and $\mathcal S_{\bv}$; their dependence on $p$ is
suppressed for notational simplicity.

We will show that the values of $\frac{1}{\sqrt{n}}\|\mathbf{u}_1\|$ and $\frac{1}{\sqrt{n}}\|\mathbf{u}_2\|$ can be constrained by some universal constants without affecting the original optimization problem.
This property is needed below in \eqref{domain_tilde_v_compact}.
Note that the first-order optimality condition with respect to $\bv$ implies
$$
\begin{aligned}
\left\|\left[\begin{array}{l}
\widehat{\mathbf{u}}_1 \\
\widehat{\mathbf{u}}_2
\end{array}\right]\right\| & = \left\|  \mathbf{H} \widehat{\bbeta}_M\right\| \\
& \leq \|\mathbf{H}\|_{op}\|\widehat{\bbeta}_M \|_2
\end{aligned}
$$
To show $\frac{1}{\sqrt{n}}\|\widehat{\mathbf{u}}_1\|$ and $\frac{1}{\sqrt{n}}\|\widehat{\mathbf{u}}_2\|$ are bounded by some universal constants, it suffices to show $\frac{1}{\sqrt{n}}\|\mathbf{H}\|_{op}$ is bounded by some universal constant for all sufficiently large sample sizes.
Using the standard upper bound on the operator norm of Gaussian random matrices \citep[Corollary 5.35]{vershynin2010introduction}, we have $\mathbb P( \|\mathbf{H}\|_{op}> \sqrt{n+M}+\sqrt{p}+\sqrt{2n} )\leq 2\exp(-n)$. Recalling that $M/n=m$ and $n/p=\noverp$, we have
$$
\sum_{n=1}^{\infty}\mathbb P\left(\frac{1}{\sqrt{n}}\|\mathbf{H}\|_{op}> \sqrt{1+m}+\sqrt{\frac{1}{\noverp}}+\sqrt{2}\right)\leq 2\sum_{n=1}^{\infty}\exp(-n)< \infty.
$$
By Borel--Cantelli lemma, we conclude that
 \begin{equation}
    \label{borel_cantelli_H_operator_norm}
    \mathbb P\left( \left\{\frac{1}{\sqrt{n}}\|\mathbf{H}\|_{op}> \sqrt{1+m}+\sqrt{\frac{1}{\noverp}}+\sqrt{2}\right\} \text{ happens infinitely many times}\right)=0.
\end{equation}
Thus, it is safe to constrain the feasible sets of $\mathbf{u}_1$ and $\mathbf{u}_2$ to be some closed balls with diverging radii $C\sqrt{n}$ for some sufficiently large constant $C$, which are denoted by $\mathcal{S}_{\bu_1}$ and $\mathcal{S}_{\bu_2}$, respectively.

\paragraph{Formulations of PO and AO}
In order to define the PO and AO problems in the context of \Cref{sec:linear_asymptotic_regime}, we need to decompose $\bbeta$ into a ``signal part'' and a ``noise part''.

Denoted by $S$ the space spanned by $\boldsymbol{\beta}_0$ and $\boldsymbol{\beta}_s$. Let $\mathbf{P}$ be the projection matrix onto $S$ and let $\mathbf{P}^{\perp}:=\mathbf{I}_p-\mathbf{P}$ be the projection matrices onto the orthogonal complement of $S$.
We use these projections to decompose $\bbeta$ as the sum of $\bbeta_{S}:=\mathbf{P}\bbeta$ and $\bbeta_{S^{\perp}}:=\mathbf{P}^\perp \bbeta$. Since the length and the direction of $\mathbf{P}\boldsymbol{\beta}$ and those of $\mathbf{P}^{\perp}\boldsymbol{\beta}$ are independent of each other, the optimization can be conducted over these directions and lengths separately.
Besides, we will define the feasible set $\mathcal{S}_{\bbeta}$ appropriately such that
 the images of projections,
$\mathbf{P}\mathcal{S}_{\bbeta}$ and $\mathbf{P}^{\perp}\mathcal{S}_{\bbeta}$, are convex,  compact, and bounded sets.
In light of these observations, the optimization can be rewritten as

\begin{equation}\label{PO_before}
    \begin{aligned}
	\min _{\substack{\boldsymbol{\beta}_S \in \mathbf{P}\mathcal{S}_{\bbeta}, \boldsymbol{\beta}_{S^{\perp}} \in \mathbf{P}^{\perp}\mathcal{S}_{\bbeta}\, \\ \mathbf{u}_1 \in \mathcal{S}_{\bu_1}, \mathbf{u}_2 \in \mathcal{S}_{\bu_2}}} \quad \max _{\bv \in \mathcal{S}_{\bv}}&\left(\frac{1}{n} \mathbf{1}^T \rho\left(\mathbf{u}_1\right)-\frac{1}{n} \mathbf{y}_1^T \mathbf{u}_1+\frac{\tau_0}{M} \mathbf{1}^T \rho\left(\mathbf{u}_2\right)-\frac{\tau_0}{M} \mathbf{y}_2^T \mathbf{u}_2\right.\\
	&\left. +\frac{1}{\sqrt{n}} \bv^T\left(\left[\begin{array}{l}
\mathbf{u}_1 \\
\mathbf{u}_2
\end{array}\right]-  \mathbf{H} \boldsymbol{\beta}_S\right)-\frac{1}{\sqrt{n}}  \bv^T \mathbf{H} \boldsymbol{\beta}_{S^{\perp}}\right).
\end{aligned}
\end{equation}

In addition, the objective function is jointly convex with respect to $\left(\boldsymbol{\beta}_S,\boldsymbol{\beta}_S^{\perp}, \boldsymbol{u}_1,\boldsymbol{u}_2\right)$, and is concave with respect to $\boldsymbol{v}$.
Based on  Sion's minimax theorem and the compactness of all the feasible sets, we can rewrite \eqref{PO_before} by flipping the min and max signs as follows

$$
\begin{aligned}
	\min _{\substack{  \boldsymbol{\beta}_{S^{\perp}} \in \mathbf{P}^{\perp}\mathcal{S}_{\bbeta} }} \quad \max _{\bv \in \mathcal{S}_{\bv}}\quad \min _{\substack{\boldsymbol{\beta}_S \in \mathbf{P}\mathcal{S}_{\bbeta} \\  \mathbf{u}_1 \in \mathcal{S}_{\bu_1}, \mathbf{u}_2 \in \mathcal{S}_{\bu_2}}}&\left(\frac{1}{n} \mathbf{1}^T \rho\left(\mathbf{u}_1\right)-\frac{1}{n} \mathbf{y}_1^T \mathbf{u}_1+\frac{\tau_0}{M} \mathbf{1}^T \rho\left(\mathbf{u}_2\right)-\frac{\tau_0}{M} \mathbf{y}_2^T \mathbf{u}_2\right.\\
	&\left. +\frac{1}{\sqrt{n}} \bv^T\left(\left[\begin{array}{l}
\mathbf{u}_1 \\
\mathbf{u}_2
\end{array}\right]-  \mathbf{H} \boldsymbol{\beta}_S\right)-\frac{1}{\sqrt{n}}   \bv^T \mathbf{H} \boldsymbol{\beta}_{S^{\perp}}\right).
\end{aligned}
$$

It is important to note that the vector of observed and auxiliary responses, $\left(\mathbf{y}_1, \mathbf{y}_2\right)$, is independent of $\mathbf{HP}^{\perp}$. This independence arises because $\mathbf{H}_1 \boldsymbol{\beta}_0=\mathbf{H}_1 \mathbf{P} \boldsymbol{\beta}_0$ and $\mathbf{H}_2 \boldsymbol{\beta}_s=\mathbf{H}_2 \mathbf{P} \boldsymbol{\beta}_s$. Given that $\mathbf{HP}$ and $\mathbf{HP}^\perp$ are independent of each other, and considering that $\mathbf{HP}^\perp$ has the same distribution as $\Tilde{\mathbf{H}}\mathbf{P}^\perp$, where $\Tilde{\mathbf{H}}$ denotes an independent copy of $\mathbf{H}$, we can conclude that the solution to the optimization problem above follows the same distribution of the solution to the following
$$
\begin{aligned}
	\min _{\substack{  \boldsymbol{\beta}_{S^{\perp}} \in \mathbf{P}^{\perp}\mathcal{S}_{\bbeta} }} \quad \max _{\bv \in \mathcal{S}_{\bv}}\quad \min _{\substack{\boldsymbol{\beta}_S \in \mathbf{P}\mathcal{S}_{\bbeta} \\  \mathbf{u}_1 \in \mathcal{S}_{\bu_1}, \mathbf{u}_2 \in \mathcal{S}_{\bu_2}}}&\left(\frac{1}{n} \mathbf{1}^T \rho\left(\mathbf{u}_1\right)-\frac{1}{n} \mathbf{y}_1^T \mathbf{u}_1+\frac{\tau_0}{M} \mathbf{1}^T \rho\left(\mathbf{u}_2\right)-\frac{\tau_0}{M} \mathbf{y}_2^T \mathbf{u}_2\right.\\
	&\left. +\frac{1}{\sqrt{n}} \bv^T\left(\left[\begin{array}{l}
\mathbf{u}_1 \\
\mathbf{u}_2
\end{array}\right]-  \mathbf{H} \boldsymbol{\beta}_S\right)-\frac{1}{\sqrt{n}}  \bv^T \Tilde{\mathbf{H}} \boldsymbol{\beta}_{S^{\perp}}\right).
\end{aligned}
$$

We are ready to define the PO problem as
\begin{equation}
    \label{PO_our}
   \text{PO:} \quad  \min _{ \boldsymbol{\beta}_{S^{\perp}} \in \mathbf{P}^{\perp}\mathcal{S}_{\bbeta}} \max_{\bv \in \mathcal{S}_{\bv}} \left\{ \frac{-1}{\sqrt{n}}  \bv^\top \Tilde{\mathbf{H}}\boldsymbol{\beta}_{S^{\perp}} + \psi(\boldsymbol{\beta}_{S^{\perp}},\bv) \right\},
\end{equation}
where $\psi(\boldsymbol{\beta}_{S^{\perp}},\bv)$ is defined as
\begin{align*}
    \psi(\bbeta_{S^{\perp}},\bv) := \min _{\substack{\boldsymbol{\beta}_S \in \mathbf{P}\mathcal{S}_{\bbeta} \\ \mathbf{u}_1 \in \mathcal{S}_{\bu_1}, \mathbf{u}_2 \in \mathcal{S}_{\bu_2}}} \left\{ \frac{1}{n} \mathbf{1}^T \rho\left(\mathbf{u}_1\right)-\frac{1}{n} \mathbf{y}_1^T \mathbf{u}_1+\frac{\tau_0}{M} \mathbf{1}^T \rho\left(\mathbf{u}_2\right)-\frac{\tau_0}{M} \mathbf{y}_2^T \mathbf{u}_2 \right. \\
    \left.
    +\frac{1}{\sqrt{n}}\bv^T\left(\left[\begin{array}{l}
\mathbf{u}_1 \\
\mathbf{u}_2
\end{array}\right]- \mathbf{H} \bbeta_S\right) \right\}.
\end{align*}
It is easy to see the objective function in \eqref{PO_our} is jointly convex with respect to $\left(\boldsymbol{\beta}_S,\boldsymbol{\beta}_S^{\perp}, \boldsymbol{u}_1,\boldsymbol{u}_2\right)$, and is concave with respect to $\boldsymbol{v}$.

Furthermore, we define the AO problem as follows
\begin{equation}
    \label{AO_our}
   \text{AO:\quad}  \min _{ \boldsymbol{\beta}_{S^{\perp}} \in \mathbf{P}^{\perp}\mathcal{S}_{\bbeta}} \max_{\bv \in \mathcal{S}_{\bv}}
   \left\{
   -\frac{1}{ \sqrt{n}}\left(\bv^T \mathbf{h}\left\|\boldsymbol{\beta}_{S^{\perp}}\right\|+\|\bv\| \mathbf{g}^T \boldsymbol{\beta}_{S^{\perp}} \right) + \psi(\boldsymbol{\beta}_{S^{\perp}},\bv)\right\},
\end{equation}
where $\mathbf{h} \in \mathbb{R}^{n+M}$ and $\mathbf{g} \in \mathbb{R}^p$ have i.i.d. standard normal entries and are independent with $\mathbf{H}$.

\subsubsection{Analyzing the auxiliary optimization}
\label{supp:sec:convergence_AO}

Since the objective function in \eqref{AO_our} is concave with respect to $\boldsymbol{v}$, and the objective function in the definition of  $\psi(\bbeta_{S^{\perp}},\bv)$ is
 jointly convex with respect to
$\left(\boldsymbol{\beta}_S, \boldsymbol{u}_1,\boldsymbol{u}_2\right)$,  and all the feasible sets of $\bbeta_S$,$\bv$ and  $\mathbf{u}_1,\mathbf{u}_2$ are compact and convex, we apply Sion's minimax theorem to rewrite \eqref{AO_our} by flipping the $\min_{\bbeta_S,\mathbf{u}_1,\mathbf{u}_2}$ and $\max_{\bv}$:

\begin{equation}\label{AO_after}
    \begin{aligned}
	\min _{\substack{\boldsymbol{\beta}_S \in \mathbf{P}\mathcal{S}_{\bbeta}, \boldsymbol{\beta}_{S^{\perp}} \in \mathbf{P}^{\perp}\mathcal{S}_{\bbeta}\, \\  \mathbf{u}_1 \in \mathcal{S}_{\bu_1}, \mathbf{u}_2 \in \mathcal{S}_{\bu_2}}} \quad \max _{\bv \in \mathcal{S}_{\bv}}&\left(\frac{1}{n} \mathbf{1}^T \rho\left(\mathbf{u}_1\right)-\frac{1}{n} \mathbf{y}_1^T \mathbf{u}_1+\frac{\tau_0}{M} \mathbf{1}^T \rho\left(\mathbf{u}_2\right)-\frac{\tau_0}{M} \mathbf{y}_2^T \mathbf{u}_2\right.\\
	&\left. +\frac{1}{\sqrt{n}} \bv^T\left(\left[\begin{array}{l}
\mathbf{u}_1 \\
\mathbf{u}_2
\end{array}\right]- \mathbf{H} \boldsymbol{\beta}_S\right) -\frac{1}{ \sqrt{n}}\left(\bv^T \mathbf{h}\left\|\boldsymbol{\beta}_{S^{\perp}}\right\|+\|\bv\| \mathbf{g}^T \boldsymbol{\beta}_{S^{\perp}} \right)\right).
\end{aligned}
\end{equation}

Ideally, we would like to solve the optimization in \eqref{AO_after} with respect to the directions of the vectors while fixing the norms of the vectors, so that we get a scalar optimization.
We first perform the maximization with respect to the direction of $\bv$.
The maximization with respect to $\bv$ in \eqref{AO_after} can be rewritten as
$$
\max_{\bv \in \mathcal{S}_{\bv}} \frac{1}{\sqrt{n}}\|\bv\| \mathbf{g}^T \boldsymbol{\beta}_{S^{\perp}} +\frac{1}{\sqrt{n}} \bv^T\left(\left[\begin{array}{c}
\mathbf{u}_1 \\
\mathbf{u}_2
\end{array}\right]- \mathbf{H} \boldsymbol{\beta}_{S} -{\left\| \boldsymbol{\beta}_{S^{\perp}} \right\|} \mathbf{h}\right) .
$$
For this maximization, we choose the direction of $\bv$ to be the same as the direction of the vector that it is multiplied to and introduce a variable $r:=\|\bv\|$ to denote the length of $\bv$. Additionally, the feasible set of $r$ is $[0, V]$ where $V$ comes from the compact set $\mathcal{S}_{v}$.
The maximization then becomes
$$
\max _{r\in [0,V]} \frac{r}{\sqrt{n}}\left(\mathbf{g}^T \boldsymbol{\beta}_{S^{\perp}}+\left\|\left[\begin{array}{l}
\mathbf{u}_1 \\
\mathbf{u}_2
\end{array}\right]- \mathbf{H} \boldsymbol{\beta}_{S} -{\left\| \boldsymbol{\beta}_{S^{\perp}} \right\|} \mathbf{h}\right\|\right)
$$
The AO is now given by
\begin{equation}\label{AO_with_r}
    \begin{aligned}
\min _{\substack{\boldsymbol{\beta}_S \in \mathbf{P}\mathcal{S}_{\bbeta}, \boldsymbol{\beta}_{S^{\perp}} \in \mathbf{P}^{\perp}\mathcal{S}_{\bbeta}\, \\ \mathbf{u}_1 \in \mathbb{R}^n, \mathbf{u}_2 \in \mathbb{R}^M}} \max _{r\in [0,V]} & \left\{\frac{1}{n} \mathbf{1}^T \rho\left(\mathbf{u}_1\right)-\frac{1}{n} \mathbf{y}_1^T \mathbf{u}_1+\frac{\tau_0}{M} \mathbf{1}^T \rho\left(\mathbf{u}_2\right)-\frac{\tau_0}{M} \mathbf{y}_2^T \mathbf{u}_2 \right. \\
+ & \left.\frac{r}{\sqrt{n}}\left( \mathbf{g}^T \boldsymbol{\beta}_{S^{\perp}}+\left\|\left[\begin{array}{c}
\mathbf{u}_1 \\
\mathbf{u}_2
\end{array}\right]- \mathbf{H} \boldsymbol{\beta}_{S}-{\left\|\boldsymbol{\beta}_{S^{\perp}}\right\|} \mathbf{h}\right\|\right)\right\}
\end{aligned}
\end{equation}

For further analyses, we need to compute the projection matrix $\mathbf{P}$ explicitly.
It is worth mentioning that in the literature,  the projection matrix is often equal to $\frac{\bbeta_0\bbeta_0^\top}{\|\bbeta_0\|^2}$, which has rank $1$.
In the current work, the projection matrix is slightly more complicated as it is the projection onto a two-dimensional space spanned by $\{\bbeta_0,\bbeta_s\}$.

Since $\boldsymbol{\beta}_0$ and $\boldsymbol{\beta}_s$ are linearly independent,
we can use the Gram-Schmidt process to find two orthogonal vectors $\be_1,\be_2$ such that $\operatorname{span}\{\boldsymbol{\beta}_0,\boldsymbol{\beta}_s\}=\operatorname{span}\{\be_1,\be_2\}$, and thus the projection matrix can be written as $\mathbf{P}=\be_1\be_1^T+\be_2\be_2^T$.
The expressions for $\be_1,\be_2$ are given by
\begin{equation}
\label{e1e2}
\left\{\begin{aligned}
\be_1 & := \frac{\boldsymbol{\beta}_0}{\|\boldsymbol{\beta}_0\|_2} ~,\\
\be_2 & := \frac{\boldsymbol{\beta}_s-\xi^{(p)}\frac{\kappa_2^{(p)}}{\kappa_1^{(p)}}\boldsymbol{\beta}_0}{\|\boldsymbol{\beta}_s-\xi^{(p)}\frac{\kappa_2^{(p)}}{\kappa_1^{(p)}}\boldsymbol{\beta}_0\|_2} ~, \\
\end{aligned}\right.
\end{equation}
with the following constants
\begin{equation}\label{eq:kappa-xi-finite}
\left\{\begin{aligned}
\kappa_1^{(p)} & := \|\boldsymbol{\beta}_0\|_2, \\
\kappa_2^{(p)} & := \|\boldsymbol{\beta}_s\|_2, \\
\xi^{(p)} & := \frac{1}{\|\boldsymbol{\beta}_0\|_2 \|\boldsymbol{\beta}_s\|_2} \langle \boldsymbol{\beta}_0, \boldsymbol{\beta}_s \rangle,
\end{aligned}\right.
\end{equation}
By SLLN, $(\kappa_1^{(p)}, \kappa_2^{(p)}, \xi^{(p)})$ converges to $(\kappa_1, \kappa_2, \xi)$ a.s. and we will drop the superscript $(p)$ in the following to ease the notation.

For any candidate $\bbeta$ in \eqref{AO_with_r}, since the length and the direction of  $\mathbf{P}\boldsymbol{\beta}$ and those of $ \mathbf{P}^{\perp}\boldsymbol{\beta}$ are independent with each other, we can optimize over the directions and the lengths separately. To see how this works, we decompose $\bbeta$ as follows:

\begin{equation}
	\label{logic:AO_distribution}
	\begin{aligned}
	\bbeta&=\mathbf{P}\bbeta+\mathbf{P}^{\perp} \bbeta \\
 &=(\be_1^T \bbeta)\be_1 + (\be_2^T \bbeta)\be_2 +\mathbf{P}^{\perp}\bbeta\\
	&=(\frac{\be_1^T \bbeta}{\|\bbeta_0\|_2})\bbeta_0+(\frac{\be_2^T \bbeta}{\|\boldsymbol{\beta}_s-\xi\frac{\kappa_2}{\kappa_1}\boldsymbol{\beta}_0\|_2})(\boldsymbol{\beta}_s-\xi\frac{\kappa_2}{\kappa_1}\boldsymbol{\beta}_0)+\|\mathbf{P}^{\perp}\bbeta\|\cdot\text{direction}(\mathbf{P}^{\perp}\bbeta).
\end{aligned}
\end{equation}

For the SRE $\widehat{\bbeta}_M$, the three scalar quantities $\frac{\be_1^T \widehat{\bbeta}_M}{\|\bbeta_0\|_2}, \frac{\be_2^T \widehat{\bbeta}_M}{\|\boldsymbol{\beta}_s-\xi\frac{\kappa_2}{\kappa_1}\boldsymbol{\beta}_0\|_2},\|\mathbf{P}^{\perp}\widehat{\bbeta}_M\| $ will be tracked in the asymptotics with a system of equations.
Using the above decomposition, we interpret $\bbeta_0$ as the true signal, $\left(\boldsymbol{\beta}_s-\xi\frac{\kappa_2}{\kappa_1}\boldsymbol{\beta}_0\right)$ as the bias induced by the auxiliary data, and $\mathbf{P}^{\perp}\widehat{\bbeta}_M $ as the noise, which will be approximated by a standard Gaussian vector.
The essence of the application of CGMT is to characterize the asymptotic behaviors of the scalar quantities aforementioned.

To be concrete, we introduce the scalars $\alpha_1:=\frac{\be_1^T \boldsymbol{\beta}}{ \kappa_1},\alpha_2:=\frac{\be_2^T \boldsymbol{\beta}}{ \kappa_2}$,  $\sigma:=\left\|\mathbf{P}^{\perp} \boldsymbol{\beta}\right\|$ and let $\btheta$ be the direction of $\mathbf{P}^{\perp} \boldsymbol{\beta}$.
In the following, we drop the feasible sets to ease the notation whenever there is no ambiguity.
The AO problem is now written as
$$   \begin{aligned}
\min _{\substack{\sigma \geq 0 \\ \mathbf{u}_1 \in \mathbb{R}^n, \mathbf{u}_2 \in \mathbb{R}^M \\  \alpha_1, \alpha_2\in \mathbb{R} }}\min_{\|\btheta\|_2=1}\max _{\substack{r\in [0,V]}} &\left(  \frac{1}{n} \mathbf{1}^T \rho\left(\mathbf{u}_1\right)-\frac{1}{n} \mathbf{y}_1^T \mathbf{u}_1+\frac{\tau_0}{M} \mathbf{1}^T \rho\left(\mathbf{u}_2\right)-\frac{\tau_0}{M} \mathbf{y}_2^T \mathbf{u}_2 \right.  \\
+ &\left. \frac{r}{\sqrt{n}}\left(  \sigma\mathbf{g}^T\btheta+\left\|\left[\begin{array}{c}
\mathbf{u}_1 \\
\mathbf{u}_2
\end{array}\right]-\kappa_1\alpha_1\bq_1-\kappa_2\alpha_2\bq_2-\sigma \mathbf{h}\right\|\right)\right),
\end{aligned}$$
where $\bq_1 :=\mathbf{H}\be_1,\bq_2 :=\mathbf{H}\be_2$.
Notice that $\bq_1$ and $\bq_2$ are independent and have i.i.d. standard normal entries (recall that $\mathbf{H}$ has i.i.d. standard normal entries and $ \langle \be_1,\be_2\rangle=0$).
In the next step, we exchange the order of the $\min_{\|\btheta\|=1}$ and $\max_{r\in [0,V]}$ in the above problem. This flipping is based on \Cref{lemma:flip_min_max}.
The AO problem can be reformulated as
\begin{equation}
    \label{AO_with_r_sigma_alpha}
   \begin{aligned}
\min _{\substack{\sigma \geq 0 \\ \mathbf{u}_1 \in \mathbb{R}^n, \mathbf{u}_2 \in \mathbb{R}^M \\  \alpha_1, \alpha_2\in \mathbb{R} }}\max _{\substack{r\in [0,V]}} \min_{\|\btheta\|_2=1} &\left(  \frac{1}{n} \mathbf{1}^T \rho\left(\mathbf{u}_1\right)-\frac{1}{n} \mathbf{y}_1^T \mathbf{u}_1+\frac{\tau_0}{M} \mathbf{1}^T \rho\left(\mathbf{u}_2\right)-\frac{\tau_0}{M} \mathbf{y}_2^T \mathbf{u}_2 \right.  \\
+ &\left. \frac{r}{\sqrt{n}}\left(  \sigma\mathbf{g}^T\btheta+\left\|\left[\begin{array}{c}
\mathbf{u}_1 \\
\mathbf{u}_2
\end{array}\right]-\kappa_1\alpha_1\bq_1-\kappa_2\alpha_2\bq_2-\sigma \mathbf{h}\right\|\right)\right),
\end{aligned}
\end{equation}
Optimizing this problem with respect to the direction of $\btheta$
yields the following
$$
\begin{aligned}
\min _{\substack{\sigma \geq 0 \\ \mathbf{u}_1 \in \mathbb{R}^n, \mathbf{u}_2 \in \mathbb{R}^M \\  \alpha_1, \alpha_2\in \mathbb{R} }} \max _{\substack{ r\in [0,V]}} &\left(  \frac{1}{n} \mathbf{1}^T \rho\left(\mathbf{u}_1\right)-\frac{1}{n} \mathbf{y}_1^T \mathbf{u}_1+\frac{\tau_0}{M} \mathbf{1}^T \rho\left(\mathbf{u}_2\right)-\frac{\tau_0}{M} \mathbf{y}_2^T \mathbf{u}_2 -\frac{r\sigma}{\sqrt{n}}\left\| \mathbf{P}^{\perp}\bg\right\|\right.  \\
+ & \left.r\frac{1}{\sqrt{n}}\left\|\left[\begin{array}{c}
\mathbf{u}_1 \\
\mathbf{u}_2
\end{array}\right]-\kappa_1\alpha_1\bq_1-\kappa_2\alpha_2\bq_2-\sigma \mathbf{h}\right\|\right).
\end{aligned}
$$

Next, we use the identity that $\|\mathbf{a}\|=\min _{\tilde{\nu}>0}\left(\frac{1}{2\tilde{\nu}}\|\mathbf{a}\|^2+\frac{\tilde{\nu}}{2 }\right)$, with optima $\widehat{\tilde{\nu}}=\|\mathbf{a}\|$, to replace the norm in the last display by a squared term:

\begin{equation}
	\label{optim:before_remove_u}
	\begin{aligned}
\min _{\substack{\sigma \geq 0 \\ \mathbf{u}_1 \in \mathbb{R}^n, \mathbf{u}_2 \in \mathbb{R}^M \\  \alpha_1, \alpha_2\in \mathbb{R} }} \max _{\substack{r\in [0,V]}} \min_{\tilde{\nu}>0}&\left(  \frac{1}{n} \mathbf{1}^T \rho\left(\mathbf{u}_1\right)-\frac{1}{n} \mathbf{y}_1^T \mathbf{u}_1+\frac{\tau_0}{M} \mathbf{1}^T \rho\left(\mathbf{u}_2\right)-\frac{\tau_0}{M} \mathbf{y}_2^T \mathbf{u}_2 -\frac{\sigma  r}{ \sqrt{n}}\left\|\mathbf{P}^{\perp} \mathbf{g}\right\|\right.  \\
+ & \frac{r\tilde{\nu}}{2}+\left. \frac{r}{2\tilde{\nu}}\left\|\frac{1}{\sqrt{n}}\left[\begin{array}{c}
\mathbf{u}_1 \\
\mathbf{u}_2
\end{array}\right]-\frac{1}{\sqrt{n}}\kappa_1\alpha_1\bq_1-\frac{1}{\sqrt{n}}\kappa_2\alpha_2\bq_2-\frac{1}{\sqrt{n}}\sigma \mathbf{h}\right\|^2\right)
\end{aligned}
\end{equation}

We shall show the above objective function is jointly convex in $(\mathbf{u}_1,\mathbf{u}_2,\alpha_1,\alpha_2,\sigma,\tilde{\nu})$ and concave in $r$. The concavity is easy since the objective function is linear in $r$. To show the joint convexity, we first note that the function {\small $h_1(\tilde{\btheta}):=1+\left\|\frac{1}{\sqrt{n}}\left[\begin{array}{c}
\mathbf{u}_1 \\
\mathbf{u}_2
\end{array}\right]-\frac{1}{\sqrt{n}}\kappa_1\alpha_1\bq_1-\frac{1}{\sqrt{n}}\kappa_2\alpha_2\bq_2-\frac{1}{\sqrt{n}}\sigma \mathbf{h}\right\|^2$ }is jointly convex in $\tilde{\btheta}:=(\mathbf{u}_1,\mathbf{u}_2,\alpha_1,\alpha_2,\sigma)$ since $h_1$ is quadratic over some linear functions.
We then note that the perspective function of $h_1(\tilde{\btheta})$ is
$$\begin{aligned}
g_1(\tilde{\btheta},\tilde{\nu})    &:= \tilde{\nu}+ \frac{1}{\tilde{\nu}}\left\|\frac{1}{\sqrt{n}}\left[\begin{array}{c}
\mathbf{u}_1 \\
\mathbf{u}_2
\end{array}\right]-\frac{1}{\sqrt{n}}\kappa_1\alpha_1\bq_1-\frac{1}{\sqrt{n}}\kappa_2\alpha_2\bq_2-\frac{1}{\sqrt{n}}\sigma \mathbf{h}\right\|^2\\
&=\tilde{\nu} \left(1+ \frac{1}{\tilde{\nu}^2}\left\|\frac{1}{\sqrt{n}}\left[\begin{array}{c}
\mathbf{u}_1 \\
\mathbf{u}_2
\end{array}\right]-\frac{1}{\sqrt{n}}\kappa_1\alpha_1\bq_1-\frac{1}{\sqrt{n}}\kappa_2\alpha_2\bq_2-\frac{1}{\sqrt{n}}\sigma \mathbf{h}\right\|^2  \right)\\
&=\tilde{\nu}h_1(\frac{\tilde{\btheta}}{\tilde{\nu}}),
\end{aligned}$$
which  is jointly convex in $(\tilde{\btheta},\tilde{\nu})$ since $h_1$ is convex in $\tilde{\btheta}$.
The joint convexity of the objective function follows from the joint convexity of $g_1(\tilde{\btheta},\tilde{\nu})$ and the convexity of $\rho(\cdot)$. To perform minimization over $\mathbf{u}_1,\mathbf{u}_2$, we  use Sion's minimax theorem to swap the order of minimization and maximization,  arrive at

$$
\begin{aligned}
\min _{\substack{\sigma \geq 0,\tilde{\nu}>0  \\  \alpha_1, \alpha_2\in \mathbb{R} }} \max _{\substack{r\in [0,V]}} \min _{\substack{  \mathbf{u}_1 \in \mathbb{R}^n, \mathbf{u}_2 \in \mathbb{R}^M   }}&\left(  \frac{1}{n} \mathbf{1}^T \rho\left(\mathbf{u}_1\right)-\frac{1}{n} \mathbf{y}_1^T \mathbf{u}_1+\frac{\tau_0}{M} \mathbf{1}^T \rho\left(\mathbf{u}_2\right)-\frac{\tau_0}{M} \mathbf{y}_2^T \mathbf{u}_2 -\frac{\sigma  r}{ \sqrt{n}}\left\|\mathbf{P}^{\perp} \mathbf{g}\right\|\right.  \\
+ & \frac{r\tilde{\nu}}{2}+\left. \frac{r}{2\tilde{\nu}}\left\|\frac{1}{\sqrt{n}}\left[\begin{array}{c}
\mathbf{u}_1 \\
\mathbf{u}_2
\end{array}\right]-\frac{1}{\sqrt{n}}\kappa_1\alpha_1\bq_1-\frac{1}{\sqrt{n}}\kappa_2\alpha_2\bq_2-\frac{1}{\sqrt{n}}\sigma \mathbf{h}\right\|^2\right)
\end{aligned}
$$

\paragraph{Minimization over \texorpdfstring{$\mathbf{u}_1,\mathbf{u}_2$}{\mathbf{u}_1,\mathbf{u}_2}:}
 
We now focus on the optimization over $\mathbf{u}_1 \in \mathbb{R}^n$ and $\mathbf{u}_2 \in \mathbb{R}^M$. Specifically,  we analyze the following problem:

\begin{equation}
	\label{optim:u1u2}
	\begin{aligned}
\min _{ \mathbf{u}_1 \in \mathbb{R}^n, \mathbf{u}_2 \in \mathbb{R}^M} &\left(  \frac{1}{n} \mathbf{1}^T \rho\left(\mathbf{u}_1\right)-\frac{1}{n} \mathbf{y}_1^T \mathbf{u}_1+\frac{\tau_0}{M} \mathbf{1}^T \rho\left(\mathbf{u}_2\right)-\frac{\tau_0}{M} \mathbf{y}_2^T \mathbf{u}_2 \right.  \\
 &+\left. \frac{r}{2\tilde{\nu}}\left\|\frac{1}{\sqrt{n}}\left[\begin{array}{c}
\mathbf{u}_1 \\
\mathbf{u}_2
\end{array}\right]-\frac{1}{\sqrt{n}}\kappa_1\alpha_1\bq_1-\frac{1}{\sqrt{n}}\kappa_2\alpha_2\bq_2-\frac{1}{\sqrt{n}}\sigma \mathbf{h}\right\|^2 \right).
\end{aligned}
\end{equation}

Note that the three vectors $\bq_1,\bq_2,\mathbf{h}$ are $n+M$ dimensional and have independent standard normal entries. Each of these vectors can be divided into two parts corresponding to $\mathbf{u}_1$ and $\mathbf{u}_2$ as
$$
\bq_1=\left[\begin{array}{c}
\bq_1^{up} \\
\bq_1^{down}
\end{array}\right],\quad \bq_2=\left[\begin{array}{c}
\bq_2^{up} \\
\bq_2^{down}
\end{array}\right],\quad \mathbf{h}=\left[\begin{array}{c}
\mathbf{h}^{up} \\
\mathbf{h}^{down}
\end{array}\right].
$$

For the terms involving $\mathbf{y}_1$ and $\mathbf{u}_1$, we use the following completion of squares:
{\scriptsize{\begin{equation}
	\label{optim:complete_square_y1u1}
	\begin{aligned}
		&-\frac{1}{n} \mathbf{y}_1^T \mathbf{u}_1+\frac{r}{2\tilde{\nu}}\left\|\frac{1}{\sqrt{n}}\mathbf{u}_1-\frac{1}{\sqrt{n}}\kappa_1\alpha_1\bq_1^{up}-\frac{1}{\sqrt{n}}\kappa_2\alpha_2\bq_2^{up}-\frac{1}{\sqrt{n}}\sigma \mathbf{h}^{up}\right\|^2+\frac{\sigma}{n}\mathbf{y}_1^T\mathbf{h}^{up}\\
		=&\frac{r}{2\tilde{\nu}}\left\|\frac{1}{\sqrt{n}}\mathbf{u}_1-\frac{1}{\sqrt{n}}\kappa_1\alpha_1\bq_1^{up}-\frac{1}{\sqrt{n}}\kappa_2\alpha_2\bq_2^{up}-\frac{1}{\sqrt{n}}\sigma \mathbf{h}^{up}-\frac{\tilde{\nu}}{r\sqrt{n}}\mathbf{y}_1\right\|^2-\frac{\tilde{\nu}}{2rn}\|\mathbf{y}_1\|^2-\frac{\kappa_1\alpha_1}{n}\mathbf{y}_1^T\bq_1^{up}-\frac{\kappa_2\alpha_2}{n}\mathbf{y}_1^T\bq_2^{up}.
	\end{aligned}
\end{equation}}}

Similarly, by completing the squares for the terms that involve $\mathbf{y}_2$ and $\mathbf{u}_2$, we have
{\scriptsize{\begin{equation}
	\label{optim:complete_square_y2u2}
	\begin{aligned}
		&-\frac{\tau_0}{M} \mathbf{y}_2^T \mathbf{u}_2+\frac{r}{2\tilde{\nu}}\left\|\frac{1}{\sqrt{n}}\mathbf{u}_2-\frac{1}{\sqrt{n}}\kappa_1\alpha_1\bq_1^{down}-\frac{1}{\sqrt{n}}\kappa_2\alpha_2\bq_2^{down}-\frac{1}{\sqrt{n}}\sigma \mathbf{h}^{down}\right\|^2+\frac{\tau_0\sigma}{M}\mathbf{y}_2^T\mathbf{h}^{down}\\
		=&\frac{\tau_0}{M}\left[ \frac{rm}{2\tau_0\tilde{\nu}}\left\|\mathbf{u}_2-\kappa_1\alpha_1\bq_1^{down}-\kappa_2\alpha_2\bq_2^{down}-\sigma \mathbf{h}^{down}-\frac{\tau_0\tilde{\nu}}{rm}\mathbf{y}_2\right\|^2-\frac{\tau_0\tilde{\nu}}{2rm}\|\mathbf{y}_2\|^2-\kappa_1\alpha_1\mathbf{y}_2^T\bq_1^{down}-\kappa_2\alpha_2\mathbf{y}_2^T\bq_2^{down}\right].
	\end{aligned}
\end{equation}}}

\eqref{optim:before_remove_u} can be rewritten as
{\scriptsize
\begin{equation*}
	\begin{aligned}
\min _{\substack{\sigma \geq 0,\tilde{\nu}>0  \\  \alpha_1, \alpha_2\in \mathbb{R} }} \max _{\substack{ r\in [0,V]}}\min _{\substack{ \mathbf{u}_1 \in \mathbb{R}^n, \mathbf{u}_2 \in \mathbb{R}^M   }} &\left( \frac{1}{n}\mathbf{1}^T \rho\left(\mathbf{u}_1\right)+ \frac{r}{2\tilde{\nu}n}\left\|\mathbf{u}_1-\kappa_1 \alpha_1 \bq_1^{up}-\kappa_2 \alpha_2 \bq_2^{up}-\sigma \mathbf{h}^{up}-\frac{\tilde{\nu}}{r}\mathbf{y}_1\right\|^2  \right. \\
& -\frac{\tilde{\nu}}{2rn}\|\mathbf{y}_1\|^2-\frac{\kappa_1\alpha_1}{n}\mathbf{y}_1^T\bq_1^{up}-\frac{\kappa_2\alpha_2}{n}\mathbf{y}_1^T\bq_2^{up}-\frac{\sigma}{n}\mathbf{y}_1^T\mathbf{h}^{up}\\
& +\frac{\tau_0}{M} \mathbf{1}^T \rho\left(\mathbf{u}_2\right)+\frac{\tau_0}{M} \frac{rm}{2\tau_0\tilde{\nu}}\left\|\mathbf{u}_2-\kappa_1 \alpha_1 \bq_1^{down}-\kappa_2 \alpha_2 \bq_2^{down}-\sigma \mathbf{h}^{down}-\frac{\tau_0\tilde{\nu}}{rm}\mathbf{y}_2\right\|^2\\
& +\frac{\tau_0}{M}\left[-\frac{\tau_0\tilde{\nu}}{2rm}\|\mathbf{y}_2\|^2-\kappa_1\alpha_1\mathbf{y}_2^T\bq_1^{down}-\kappa_2\alpha_2\mathbf{y}_2^T\bq_2^{down}-\frac{\tau_0\sigma}{M}\mathbf{y}_2^T\mathbf{h}^{down}\right]\\
&\left.-\frac{\sigma  r}{ \sqrt{n}}\left\|\mathbf{P}^{\perp} \mathbf{g}\right\|+ \frac{r\tilde{\nu}}{2} \right).
\end{aligned}
\end{equation*}
}

Now we can perform the minimization over $\mathbf{u}_1,\mathbf{u}_2$.
Based on the definition of the Moreau envelope, we can express the  minimization over $\mathbf{u}_1$ as
$$
\begin{aligned}
	\min _{ \mathbf{u}_1 \in \mathbb{R}^n}& \frac{1}{n}\mathbf{1}^T \rho\mathbf{u}_1+ \frac{r}{2\tilde{\nu}n}\left\|\mathbf{u}_1-\kappa_1 \alpha_1 \bq_1^{up}-\kappa_2 \alpha_2 \bq_2^{up}-\sigma \mathbf{h}^{up}-\frac{\tilde{\nu}}{r}\mathbf{y}_1\right\|^2\\
	&=\frac{1}{n}M_{\rho(\cdot)}\left(\kappa_1 \alpha_1 \bq_1^{up}+\kappa_2 \alpha_2\bq_2^{up}+\sigma \mathbf{h}^{up}+\frac{\tilde{\nu}}{r } \mathbf{y}_1, \frac{\tilde{\nu}}{r }\right),
\end{aligned}
$$
and the one over $\mathbf{u}_2$ as
$$
\begin{aligned}
	\min _{ \mathbf{u}_2 \in \mathbb{R}^M}& \frac{\tau_0}{M} \mathbf{1}^T \rho\left(\mathbf{u}_2\right)+\frac{\tau_0}{M} \frac{rm}{2\tau_0\tilde{\nu}}\left\|\mathbf{u}_2-\kappa_1 \alpha_1 \bq_1^{down}-\kappa_2 \alpha_2 \bq_2^{down}-\sigma \mathbf{h}^{down}-\frac{\tau_0\tilde{\nu}}{rm}\mathbf{y}_2\right\|^2\\
	&=\frac{\tau_0}{M}M_{\rho(\cdot)}\left(\kappa_1 \alpha_1 \bq_1^{down}+\kappa_2 \alpha_2\bq_2^{down}+\sigma \mathbf{h}^{down}+\frac{\tau_0\tilde{\nu}}{r m} \mathbf{y}_2, \frac{\tau_0\tilde{\nu}}{r m}\right).
\end{aligned}
$$
As a result, \eqref{optim:before_remove_u} can be simplified as
\begin{equation}
\label{eq_reduced_AO}
\min _{\substack{\sigma \geq 0,\tilde{\nu}>0  \\  \alpha_1, \alpha_2\in \mathbb{R} }} \max _{\substack{ r \in [0,V]}} \quad \mathcal{R}_n(\sigma,r,\tilde{\nu},\alpha_1,\alpha_2)
\end{equation}
where
\begin{align*}
    \mathcal{R}_n(\sigma,r,\tilde{\nu},\alpha_1,\alpha_2)&:=\frac{1}{n}M_{\rho(\cdot)}\left(\kappa_1 \alpha_1 \bq_1^{up}+\kappa_2 \alpha_2\bq_2^{up}+\sigma \mathbf{h}^{up}+\frac{\tilde{\nu}}{r } \mathbf{y}_1, \frac{\tilde{\nu}}{r }\right) \\
    &+\frac{\tau_0}{M}M_{\rho(\cdot)}\left(\kappa_1 \alpha_1 \bq_1^{down}+\kappa_2 \alpha_2\bq_2^{down}+\sigma \mathbf{h}^{down}+\frac{\tau_0\tilde{\nu}}{r m} \mathbf{y}_2, \frac{\tau_0\tilde{\nu}}{r m}\right)\\
    &-\frac{\tilde{\nu}}{2rn}\|\mathbf{y}_1\|^2-\frac{\kappa_1\alpha_1}{n}\mathbf{y}_1^T\bq_1^{up}-\frac{\kappa_2\alpha_2}{n}\mathbf{y}_1^T\bq_2^{up}-\frac{\sigma}{n}\mathbf{y}_1^T\mathbf{h}^{up}\\
    & +\frac{\tau_0}{M}\left[-\frac{\tau_0\tilde{\nu}}{2rm}\|\mathbf{y}_2\|^2-\kappa_1\alpha_1\mathbf{y}_2^T\bq_1^{down}-\kappa_2\alpha_2\mathbf{y}_2^T\bq_2^{down}-\frac{\tau_0\sigma}{M}\mathbf{y}_2^T\mathbf{h}^{down}\right]\\
&-\frac{\sigma  r}{ \sqrt{n}}\left\|\mathbf{P}^{\perp} \mathbf{g}\right\|+ \frac{r\tilde{\nu}}{2}.
\end{align*}
Since the partial minimization of a convex function over a convex feasible set preserves the convexity, the objective function $\mathcal{R}_n$ is jointly convex in $(\sigma,\tilde{\nu}, \alpha_1,\alpha_2)$ for any $r$.
By Danskin's theorem \citep{danskin1966theory}, $\mathcal{R}_n$ is concave in $r$ for any $(\sigma,\tilde{\nu},\alpha_1,\alpha_2)$.
In the following, we aim to find the limit of $\mathcal{R}_n$ and then show that the solution to $\mathcal{R}_n$ converges to the solution to the limit.

\paragraph{Limit of \texorpdfstring{$\mathcal{R}_n(\sigma,r,\tilde{\nu},\alpha_1,\alpha_2)$}{Rn(sigma,r,nu,alpha1,alpha2)}}

Fix any $(\sigma,r,\tilde{\nu},\alpha_1,\alpha_2)$.
Using SLLN (as well as the SLLN for the constants defined in \eqref{eq:kappa-xi-finite}), we have as $n\rightarrow\infty$,
{\scriptsize{\begin{equation}\label{envelope_converge_SLLN}
	\begin{aligned}
		&\frac{1}{n}M_{\rho(\cdot)}\left(\kappa_1 \alpha_1 \bq_1^{up}+\kappa_2 \alpha_2\bq_2^{up}+\sigma \mathbf{h}^{up}+\frac{\tilde{\nu}}{r } \mathbf{y}_1, \frac{\tilde{\nu}}{r }\right) \stackrel{a . s}{\longrightarrow} \mathbb{E}(M_{\rho(\cdot)}(\kappa_1\alpha_1 Z_1+\kappa_2\alpha_2 Z_2+\sigma Z_3+\frac{\tilde{\nu}}{r}\text{Bern}(\rho^{\prime}(\kappa_1 Z_1)),\frac{\tilde{\nu}}{r})),\\
		&\frac{\tau_0}{M}M_{\rho(\cdot)}\left(\kappa_1 \alpha_1 \bq_1^{down}+\kappa_2 \alpha_2\bq_2^{down}+\sigma \mathbf{h}^{down}+\frac{\tau_0\tilde{\nu}}{r m} \mathbf{y}_2, \frac{\tau_0\tilde{\nu}}{r m}\right)
\\
&\quad \hskip 3cm \stackrel{a . s}{\longrightarrow}  \tau_0\mathbb{E}(M_{\rho(\cdot)}(\kappa_1\alpha_1 Z_1+\kappa_2\alpha_2 Z_2+\sigma Z_3+\frac{\tau_0\tilde{\nu}}{rm}\text{Bern}(\rho^{\prime}(\kappa_2\xi Z_1+\kappa_2 \sqrt{1-\xi^2}Z_2)),\frac{\tau_0\tilde{\nu}}{rm})).
	\end{aligned}
\end{equation}}}

Recall that $\mathbf{y}_1=\text{Bern}(\rho^{\prime}(\mathbf{H}_1 \boldsymbol{\beta}_0 ))=\text{Bern}(\rho^{\prime}(\kappa_1 \bq_1^{up}))$, we have
{\small
$$
\frac{1}{n} \mathbf{y}_1^T \bq_1^{up}=\frac{1}{n} \sum_{i=1}^n y_{1i} q_{1i}^{up}=\frac{1}{n} \sum_{i=1}^n B e r\left(\rho^{\prime}\left(\kappa_1 q_{1i}^{up}\right)\right) q_{1i}^{up}\stackrel{a . s}{\longrightarrow}
 \mathbb{E}_Z\left[Z \cdot \rho^{\prime}(\kappa_1 Z)\right]=\kappa_1 \mathbb{E}_Z\left[\rho^{\prime \prime}(\kappa_1 Z)\right],
$$}
$$
\text{and\quad} \frac{1}{n}\|\mathbf{y}_1\|^2=\frac{1}{n} \sum_{i=1}^n y_{1i}^2 \underset{n \rightarrow \infty}{\stackrel{\mathrm{SLLN}}{\Longrightarrow}} \mathbb{E}\left[y_{1i}^2\right]=\mathbb{E}\left[y_{1i}\right]=\mathbb{E}_Z\left[\rho^{\prime}(\kappa_1 Z)\right]=\frac{1}{2},
$$
where the last equality follows from \Cref{useful_identity}.
The other two inner products $\frac{1}{n}\mathbf{y}_1^T\mathbf{h}^{up}$ and $\frac{1}{n}\mathbf{y}_1^T\bq_2^{up}$ are of order $1/\sqrt{n}$ since $\mathbf{y}_1$ is independent of both $\mathbf{h}$ and $\bq_2$, and we can ignore them in the limit.

Recall that $\mathbf{y}_2=\text{Bern}(\rho^{\prime}(\mathbf{H}_2\boldsymbol{\beta}_s))=\text{Bern}(\rho^{\prime}(\kappa_2 \xi\bq_1^{down}+\kappa_2\sqrt{1-\xi^2}\bq_2^{down} ))$, we have
$$
\begin{aligned}
	& \quad~ \frac{1}{M} \mathbf{y}_2^T \bq_1^{down}
    \\
    &=\frac{1}{M} \sum_{i=1}^M y_{2i} q_{1i}^{down}\\
    =&\frac{1}{M} \sum_{i=1}^M B e r\left(\rho^{\prime}\left(\kappa_2 \xi q_{1i}^{down}+\kappa_2\sqrt{1-\xi^2} q_{2i}^{down} \right)\right) q_{1i}^{down} \\
 \stackrel{a . s}{\longrightarrow} &\mathbb{E}\left[Z_1 \cdot \rho^{\prime}(\kappa_2 \xi Z_1+\kappa_2\sqrt{1-\xi^2}Z_2)\right]=\kappa_2\xi \mathbb{E}\left[\rho^{\prime \prime}(\kappa_2 \xi Z_1+\kappa_2\sqrt{1-\xi^2}Z_2)\right]
\end{aligned}
$$
and
{\scriptsize{$$
\begin{aligned}
	\frac{1}{M} \mathbf{y}_2^T \bq_2^{down}=\frac{1}{M} \sum_{i=1}^M y_{2i} q_{2i}^{down}=&\frac{1}{M} \sum_{i=1}^M B e r\left(\rho^{\prime}\left(\kappa_2 \xi q_{1i}^{down}+\kappa_2\sqrt{1-\xi^2} q_{2i}^{down} \right)\right) q_{2i}^{down} \\
 \stackrel{a . s}{\longrightarrow}   &\mathbb{E}\left[Z_2 \cdot \rho^{\prime}(\kappa_2 \xi Z_1+\kappa_2\sqrt{1-\xi^2}Z_2)\right]=\kappa_2\sqrt{1-\xi^2} \mathbb{E}\left[\rho^{\prime \prime}(\kappa_2 \xi Z_1+\kappa_2\sqrt{1-\xi^2}Z_2)\right],
\end{aligned}
$$}}
where $Z_1,Z_2\sim N(0,1)$ independently.

For the term $\frac{\sigma  r}{ \sqrt{n}}\left\|\mathbf{P}^{\perp} \mathbf{g}\right\|$,  since $\mathbf{g} \in \mathbb{R}^p$ has i.i.d. standard normal entries, we can approximate $\frac{\sigma  r}{ \sqrt{n}}\left\|\mathbf{P}^{\perp} \mathbf{g}\right\|$ with $\frac{\sigma  r}{ \sqrt{\noverp}}$ by SLLN  for any fixed $(\sigma,r)$,  where $\noverp:=\frac{n}{p}$ is the oversampling ratio.

Putting all these together, the point-wise limit of the objective function  $\mathcal{R}_n(\sigma,r,\tilde{\nu},\alpha_1,\alpha_2)$, denoted by $\mathcal{R}(\sigma,r,\tilde{\nu},\alpha_1,\alpha_2)$, can be expressed as follows:
\begin{equation}
    \label{supp:obj_five_convex_concave}
    \begin{aligned}
& \quad~  \mathcal{R}(\sigma,r,\tilde{\nu},\alpha_1,\alpha_2)\\
& =\lim_{n\rightarrow\infty}\mathcal{R}_n(\sigma,r,v,\alpha_1,\alpha_2)\\
&=\quad \left\{ -\frac{r\sigma}{\sqrt{\noverp}}+\frac{r\tilde{\nu}}{2} -\frac{\tilde{\nu}}{4r}-\kappa_1^2\alpha_1\mathbb{E}(\rho^{\prime\prime}(\kappa_1 Z_1))       \right.\\
 & -\frac{\tau_0^2\tilde{\nu}}{4rm}-\tau_0 \kappa_2 \mathbb{E}(\rho^{\prime\prime}(\kappa_2 \xi Z_1+\kappa_2 \sqrt{1-\xi^2}Z_2))(\alpha_1\kappa_1\xi+\alpha_2\kappa_2\sqrt{1-\xi^2} )\\
 &+\mathbb{E}\left[M_{\rho(\cdot)}\left(\kappa_1\alpha_1 Z_1+\kappa_2\alpha_2 Z_2+\sigma Z_3+\frac{\tilde{\nu}}{r}\text{Bern}(\rho^{\prime}(\kappa_1 Z_1)),\frac{\tilde{\nu}}{r}\right)\right]\\
 &\left.  +\tau_0\mathbb{E}\left[M_{\rho(\cdot)}\left(\kappa_1\alpha_1 Z_1+\kappa_2\alpha_2 Z_2+\sigma Z_3+\frac{\tau_0\tilde{\nu}}{rm}\text{Bern}(\rho^{\prime}(\kappa_2\xi Z_1+\kappa_2 \sqrt{1-\xi^2}Z_2)),\frac{\tau_0\tilde{\nu}}{rm}\right) \right] \right\}.
 \end{aligned}
\end{equation}

Since taking point-wise limit preserves the convexity and the concavity, we know that $\mathcal{R}(\sigma,r,\tilde{\nu},\alpha_1,\alpha_2)$ is concave in $r$ and jointly convex in $(\sigma,\tilde{\nu},\alpha_1,\alpha_2)$.

Define a scalar optimization based on $\mathcal{R}(\sigma,r,\tilde{\nu},\alpha_1,\alpha_2)$
\begin{equation}\label{optim:scalar}
\min _{\substack{\sigma \geq 0,\tilde{\nu}>0  \\  \alpha_1, \alpha_2\in \mathbb{R} }}    \max _{\substack{ r\in [0,V]}}  \mathcal{R}(\sigma,r,\tilde{\nu},\alpha_1,\alpha_2),
\end{equation}
and let $(\sigma_*,r_*,\tilde{\nu}_*,\alpha_{1*},\alpha_{2*})$ be the solution to the optimization in \eqref{optim:scalar}.
We will show below that optima of \eqref{eq_reduced_AO} will converge to $(\sigma_*,r_*,\tilde{\nu}_*,\alpha_{1*},\alpha_{2*})$.

\paragraph{Convergence of the optima}

In order to justify the convergence of the optima of $\mathcal{R}_n$, we should show that the domain for $(\sigma,r,\tilde{\nu},\alpha_1,\alpha_2)$ is uniformly bounded in the following sense:
\begin{equation}\label{scalar_compact_domain}
    \begin{aligned}
         & \sigma= \left\|\mathbf{P}^{\perp} \boldsymbol{\beta}\right\|\leq  {\|\bbeta\|} \leq c_1, \\
    & |\alpha_1|=\left|\frac{\be_1^T \boldsymbol{\beta}}{\kappa_1}\right|\leq \frac{\|\bbeta\|}{ \kappa_1}\leq c_1/\kappa_1,\\
    &|\alpha_2|=\left|\frac{\be_2^T \boldsymbol{\beta}}{ \kappa_2}\right|\leq \frac{\|\bbeta\|}{ \kappa_2}\leq c_1/\kappa_2,\\
    &r=\|\bv\|\leq V
    \end{aligned}
\end{equation}
The first three inequalities in \eqref{scalar_compact_domain} follow from the fact that the feasible set of $\bbeta$ is a closed ball centered at the origin and has a constant radius, as proved in \eqref{borelcantelli_beta}.
The last inequality regarding $r$ follows from the fact that the feasible set for the variable  $\bv$ is a closed ball with a constant radius.
For the scalar variable $\tilde{\nu}$, we recall its definition in
\begin{equation}\label{domain_tilde_v_compact}
    \begin{aligned}
   & \left\|\frac{1}{\sqrt{n}}\left[\begin{array}{c}
\mathbf{u}_1 \\
\mathbf{u}_2
\end{array}\right]-\frac{1}{\sqrt{n}}\kappa_1\alpha_1\bq_1-\frac{1}{\sqrt{n}}\kappa_2\alpha_2\bq_2-\frac{1}{\sqrt{n}}\sigma \mathbf{h}\right\|\\
=& \min_{\tilde{\nu}>0}\left\{\frac{\tilde{\nu}}{2}+\frac{1}{2\tilde{\nu}} \left\|\frac{1}{\sqrt{n}}\left[\begin{array}{c}
\mathbf{u}_1 \\
\mathbf{u}_2
\end{array}\right]-\frac{1}{\sqrt{n}}\kappa_1\alpha_1\bq_1-\frac{1}{\sqrt{n}}\kappa_2\alpha_2\bq_2-\frac{1}{\sqrt{n}}\sigma \mathbf{h}\right\|^2  \right\},
\end{aligned}
\end{equation}
where the optimal $\widehat{\tilde{\nu}}$ is equal to  $\frac{1}{\sqrt{n}}\left\|\left[\begin{array}{c}
\mathbf{u}_1 \\
\mathbf{u}_2
\end{array}\right]-\kappa_1\alpha_1\bq_1-\kappa_2\alpha_2\bq_2-\sigma \mathbf{h}\right\|$.
Therefore, we can, without changing the formulation, restrict the feasible set of $\tilde{\nu}$ to be an interval with the right end larger than $\widehat{\tilde{\nu}}$.
Since we have already shown $\left\|\left[\begin{array}{c}
\mathbf{u}_1 \\
\mathbf{u}_2
\end{array}\right]\right\|\leq C\sqrt{n}$ for large enough sample size $n$ in \eqref{borel_cantelli_H_operator_norm},
by the triangle inequality, it suffices to bound $\frac{1}{\sqrt{n}}\|\kappa_1\alpha_1\bq_1+\kappa_2\alpha_2\bq_2+\sigma \mathbf{h}\|$.
Recall $\bq_1,\bq_2$ and $\mathbf{h}$ are random vectors with independent standard Gaussian random variable as entries. By \Cref{lemma:concetration_norm_gaussian_vector} and \eqref{scalar_compact_domain}, we have
\begin{align*}
    & \mathbb P(\|\kappa_1\alpha_1\bq_1\|>2c_1\sqrt{n+M})\leq \exp(-(n+M)/2),\\
    & P(\|\kappa_2\alpha_2\bq_2\|>2c_1\sqrt{n+M})\leq \exp(-(n+M)/2),\\
    &P(\|\sigma\mathbf{h}\|>2c_1\sqrt{n+M})\leq \exp(-(n+M)/2).
\end{align*}
By union bound and Borel Cantelli lemma, we have
\begin{equation}\label{borem_cantelli_tilde_v}
    \mathbb P\left(\left\{\frac{1}{\sqrt{n}}\|\kappa_1\alpha_1\bq_1+\kappa_2\alpha_2\bq_2+\sigma \mathbf{h}\|>6c_1\sqrt{1+m}\right\} \text{ happens infinitely many times}  \right)=0
\end{equation}
Therefore, we can constrain the feasible set of $\tilde{\nu}$  to be  bounded.

Up to this point, we have shown that the objective function in \eqref{optim:before_remove_u} converges point-wise to the objective function $\mathcal{R}(\sigma,r,\tilde{\nu},\alpha_1,\alpha_2)$. Furthermore, we've established that both objective functions are joint convex with respect to $(\sigma,\tilde{\nu},\alpha_1,\alpha_2)$ and concave with respect to $r$, within a compact domain for these parameters.
Drawing on similar reasoning as presented in the proof of \citet[Lemma A.1]{dai2023scale} and in \citet[Appendix B.3.3]{javanmard2022precise}, which in turn make use of arguments from \citet[Lemma A.5]{thrampoulidis2018precise}, we can conclude that the optimal solutions in \eqref{optim:before_remove_u}, denoted as $(\widehat{\sigma},\widehat{r},\widehat{\tilde{\nu}},\widehat{\alpha}_1,\widehat{\alpha}_2)$, will uniformly converge to the optimal solution  $(\sigma_*,r_*,\tilde{\nu}_*,\alpha_{1*},\alpha_{2*})$ in  \eqref{optim:scalar}.

\subsubsection{Uniqueness of  the optima}  \label{sec:unique-saddle-R}

Although the objective function $\mathcal{R}(\sigma,r,\tilde{\nu},\alpha_1,\alpha_2)$ is jointly convex in $(\sigma,\tilde{\nu},\alpha_1,\alpha_2)$ and concave in $r$ over a compact domain, these properties alone do not guarantee that the optimization problem \eqref{optim:scalar} admits a unique solution $(\sigma_*,r_*,\tilde{\nu}_*,\alpha_{1*},\alpha_{2*})$. To ensure uniqueness, we must additionally verify: (1) for fixed $r>0$, $\mathcal{R}(\sigma,r,\tilde{\nu},\alpha_1,\alpha_2)$ is jointly strictly convex in $(\sigma,\tilde{\nu},\alpha_1,\alpha_2)$, and (2) for fixed $\sigma,r,\tilde{\nu},\alpha_1,\alpha_2$, $\mathcal{R}(\sigma,r,\tilde{\nu},\alpha_1,\alpha_2)$ is strictly concave in $r$.

We begin with a simplification of the objective function \eqref{supp:obj_five_convex_concave}. We first expand the last two terms involving the Moreau envelope. Observe that
\begin{align*}
   &  \mathbb{E}\left[M_{\rho(\cdot)}\left(\kappa_1\alpha_1 Z_1+\kappa_2\alpha_2 Z_2+\sigma Z_3+\frac{\tilde{\nu}}{r}\text{Bern}(\rho^{\prime}(\kappa_1 Z_1)),\frac{\tilde{\nu}}{r}\right)\right]\\
   =& \mathbb{E}\left[\rho^{\prime}(-\kappa_1 Z_1)M_{\rho(\cdot)}\left(\kappa_1\alpha_1 Z_1+\kappa_2\alpha_2 Z_2+\sigma Z_3 ,\frac{\tilde{\nu}}{r}\right)\right]\\
   &+  \mathbb{E}\left[\rho^{\prime}(\kappa_1 Z_1)M_{\rho(\cdot)}\left(\kappa_1\alpha_1 Z_1+\kappa_2\alpha_2 Z_2+\sigma Z_3+\frac{\tilde{\nu}}{r} ,\frac{\tilde{\nu}}{r}\right)\right]
\end{align*}
Using the definition of the Moreau envelope,
\begin{align*}
	 &\mathbb{E}\left[\rho^{\prime}(\kappa_1 Z_1)M_{\rho(\cdot)}\left(\kappa_1\alpha_1 Z_1+\kappa_2\alpha_2 Z_2+\sigma Z_3+\frac{\tilde{\nu}}{r} ,\frac{\tilde{\nu}}{r}\right)\right]\\
	 = & \mathbb{E}\left(\rho^{\prime}\left(\kappa_1 Z_1\right) \min _t\left[\rho(t)+\frac{r }{2\tilde{\nu}}\left( \kappa_1  \alpha_1 Z_1+ \kappa_2\alpha_2 Z_2+\sigma Z_3+\frac{\tilde{\nu}}{  r}-t\right)^2\right]\right) \\
	 = & \mathbb{E}\left(\rho^{\prime}\left(\kappa_1 Z_1\right) \min _t\left[\rho(t)-t+\frac{r }{2\tilde{\nu}}\left(\kappa_1  \alpha_1 Z_1+ \kappa_2\alpha_2 Z_2+\sigma Z_3-t\right)^2\right]\right) \\
	 & +\mathbb{E}\left(\rho^{\prime}\left(\kappa_1 Z_1\right)\left[\frac{\tilde{\nu}}{2r}+\left(\kappa_1  \alpha_1 Z_1+ \kappa_2\alpha_2 Z_2+\sigma Z_3\right)\right]\right) \\
	 =&  \mathbb{E}\left(\rho^{\prime}\left(\kappa_1 Z_1\right) \min _t\left[\rho(t)-t+\frac{r }{2\tilde{\nu}}\left(\kappa_1  \alpha_1 Z_1+ \kappa_2\alpha_2 Z_2+\sigma Z_3-t\right)^2\right]\right) \\
	 & +\frac{\tilde{\nu} }{4r}  +\kappa_1^2\alpha_1\mathbb{E}(\rho^{\prime\prime}(\kappa_1 Z_1))
\end{align*}
where in the last step we use $\mathbb{E}\left(\rho^{\prime}\left(\kappa Z_1\right)\right)=1 / 2$ and the Stein identity $\mathbb{E}\left(\rho^{\prime}\left(\kappa Z_1\right) Z_1\right)=\kappa \mathbb{E}_Z\left[\rho^{\prime \prime}(\kappa Z)\right]$.
A similar argument yields
\begin{align*}
&\mathbb{E}\left[M_{\rho(\cdot)}\left(\kappa_1\alpha_1 Z_1+\kappa_2\alpha_2 Z_2+\sigma Z_3+\frac{\tau_0\tilde{\nu}}{rm}\text{Bern}(\rho^{\prime}(\kappa_2\xi Z_1+\kappa_2 \sqrt{1-\xi^2}Z_2)),\frac{\tau_0\tilde{\nu}}{rm}\right) \right]\\
	=&\mathbb{E}\left[\rho^{\prime}(-\kappa_2\xi Z_1-\kappa_2 \sqrt{1-\xi^2}Z_2)M_{\rho(\cdot)}\left(\kappa_1\alpha_1 Z_1+\kappa_2\alpha_2 Z_2+\sigma Z_3 ,\frac{\tau_0\tilde{\nu}}{rm}\right)\right]\\
	&+ \mathbb{E}\left[\rho^{\prime}(\kappa_2\xi Z_1+\kappa_2 \sqrt{1-\xi^2}Z_2)M_{\rho(\cdot)}\left(\kappa_1\alpha_1 Z_1+\kappa_2\alpha_2 Z_2+\sigma Z_3 +\frac{\tau_0\tilde{\nu}}{rm},\frac{\tau_0\tilde{\nu}}{rm}\right)\right]\\
	=&\mathbb{E}\left[\rho^{\prime}(-\kappa_2\xi Z_1-\kappa_2 \sqrt{1-\xi^2}Z_2)M_{\rho(\cdot)}\left(\kappa_1\alpha_1 Z_1+\kappa_2\alpha_2 Z_2+\sigma Z_3 ,\frac{\tau_0\tilde{\nu}}{rm}\right)\right]\\
	&+ \mathbb{E}\left(\rho^{\prime}(\kappa_2\xi Z_1+\kappa_2 \sqrt{1-\xi^2}Z_2)\min _t\left[\rho(t)-t+\frac{rm }{2\tau_0\tilde{\nu}}\left(\kappa_1  \alpha_1 Z_1+ \kappa_2\alpha_2 Z_2+\sigma Z_3-t\right)^2\right]\right)\\
	&+\frac{\tau_0\tilde{\nu}}{4rm}+  \kappa_2 \mathbb{E}(\rho^{\prime\prime}(\kappa_2 \xi Z_1+\kappa_2 \sqrt{1-\xi^2}Z_2))(\alpha_1\kappa_1\xi+\alpha_2\kappa_2\sqrt{1-\xi^2} )
\end{align*}

Putting the pieces together, the objective function $\mathcal{R}(\sigma,r,\tilde{\nu},\alpha_1,\alpha_2)$ in \eqref{supp:obj_five_convex_concave} can be expressed as
\begin{align*}
	& \mathcal{R}(\sigma,r,\tilde{\nu},\alpha_1,\alpha_2)=\quad \left\{ -\frac{r\sigma}{\sqrt{\noverp}}+\frac{r\tilde{\nu}}{2}         \right.\\
  & +\mathbb{E}\left[\rho^{\prime}(-\kappa_1 Z_1)M_{\rho(\cdot)}\left(\kappa_1\alpha_1 Z_1+\kappa_2\alpha_2 Z_2+\sigma Z_3 ,\frac{\tilde{\nu}}{r}\right)\right]\\
   &+  \mathbb{E}\left(\rho^{\prime}\left(\kappa_1 Z_1\right) \min _t\left[\rho(t)-t+\frac{r }{2\tilde{\nu}}\left(\kappa_1  \alpha_1 Z_1+ \kappa_2\alpha_2 Z_2+\sigma Z_3-t\right)^2\right]\right) \\
    &+ \tau_0\mathbb{E}\left(\rho^{\prime}(-\kappa_2\xi Z_1-\kappa_2 \sqrt{1-\xi^2}Z_2)\min _t\left[\rho(t)+\frac{rm }{2\tau_0\tilde{\nu}}\left(\kappa_1  \alpha_1 Z_1+ \kappa_2\alpha_2 Z_2+\sigma Z_3-t\right)^2\right]\right)\\
	&+\left. \tau_0\mathbb{E}\left(\rho^{\prime}(\kappa_2\xi Z_1+\kappa_2 \sqrt{1-\xi^2}Z_2)\min _t\left[\rho(t)-t+\frac{rm }{2\tau_0\tilde{\nu}}\left(\kappa_1  \alpha_1 Z_1+ \kappa_2\alpha_2 Z_2+\sigma Z_3-t\right)^2\right]\right)\right\}
\end{align*}
Thus, to establish uniqueness of the optimizer, it remains to verify that (1) for fixed $r>0$, $\mathcal{R}(\sigma,r,\tilde{\nu},\alpha_1,\alpha_2)$ is jointly strictly convex in $(\sigma,\tilde{\nu},\alpha_1,\alpha_2)$, and (2) for fixed ($\sigma,\tilde{\nu},\alpha_1,\alpha_2$), $\mathcal{R}(\sigma,r,\tilde{\nu},\alpha_1,\alpha_2)$ is strictly concave in $r$.

\textbf{Task (1): for fixed $r>0$, $\mathcal{R}(\sigma,r,\tilde{\nu},\alpha_1,\alpha_2)$ is jointly strictly convex in $(\sigma,\tilde{\nu},\alpha_1,\alpha_2)$.}

Since  $\rho(t)-t=\log(1+e^t)-t$ is convex in $t$, and the perspective of a convex function is also convex, the term $\rho(t)-t+\frac{r}{2\tilde{\nu}}[\kappa_1\alpha_1 Z_1+\kappa_2\alpha_2 Z_2+\sigma Z_3 -t]^2$ is jointly convex in ($\tilde{\nu},\alpha_1,\alpha_2,\sigma,t$).
Since partial minimization and expectation both preserve convexity,  $$\mathbb{E}\left(\rho^{\prime}\left(\kappa Z_1\right) \min _t\left[\rho(t)-t+\frac{r }{2\tilde{\nu}}\left(\kappa_1  \alpha_1 Z_1+ \kappa_2\alpha_2 Z_2+\sigma Z_3-t\right)^2\right]\right)$$ is jointly convex in ($\tilde{\nu},\alpha_1,\alpha_2,\sigma$). The same reasoning establishes joint convexity for the other three expectation terms of $\mathcal R$.

To obtain strict convexity, it suffices to show that one expectation term is strictly convex in ($\sigma,\tilde{\nu},\alpha_1,\alpha_2$) for fixed $r>0$. We focus on
 \begin{equation*}
 	\mathbb{E}\left[\rho^{\prime}(-\kappa_1 Z_1)M_{\rho(\cdot)}\left(\kappa_1\alpha_1 Z_1+\kappa_2\alpha_2 Z_2+\sigma Z_3 ,\frac{\tilde{\nu}}{r}\right)\right].
 \end{equation*}
By \Cref{supp:lemma:strict_convex1}, it is enough to show that
 \begin{equation}
 	\mathbb{E}\left[ M_{\rho(\cdot)}\left(\kappa_1\alpha_1 Z_1+\kappa_2\alpha_2 Z_2+\sigma Z_3 ,\frac{\tilde{\nu}}{r}\right)\right]
 \end{equation}
is strictly jointly convex in ($\sigma,\tilde{\nu},\alpha_1,\alpha_2$).
The proof proceeds in two steps:
\begin{enumerate}
	\item Set $\tilde q:=\sqrt{\kappa_1^2\alpha_1^2+\kappa_2^2\alpha_2^2+\sigma^2}>0$ and define $L(\tilde q,\tilde \nu):= \mathbb{E}\left[ M_{\rho(\cdot)}\left(\tilde q Z ,\frac{\tilde{\nu}}{r}\right)\right]$, \Cref{supp:lemma:M_q_v_strict} shows  that $L(\tilde q,\tilde \nu)$ is jointly strictly convex in $(\tilde q,\tilde \nu)$.
		\item Applying the strict convexity of $L(\tilde q,\tilde \nu)$ and \Cref{supp:lemma:M_alpha12_sigma_v_strict} yields the strict convexity of $\mathbb{E}\left[ M_{\rho(\cdot)}\left(\kappa_1\alpha_1 Z_1+\kappa_2\alpha_2 Z_2+\sigma Z_3 ,\frac{\tilde{\nu}}{r}\right)\right]$ in ($\sigma,\tilde{\nu},\alpha_1,\alpha_2$).
\end{enumerate}
Thus, $\mathcal{R}$ is jointly strictly convex in $\left(\sigma, \tilde{\nu}, \alpha_1, \alpha_2\right)$ for fixed $r>0$.

\textbf{Task (2): for fixed $\sigma,\tilde{\nu},\alpha_1,\alpha_2$, the function $\mathcal{R}(\sigma,r,\tilde{\nu},\alpha_1,\alpha_2)$ is strictly concave in $r$.}
For any functions $A(t)$ and $B(t)$,
$$
\inf _t\left[A(t)+\left(\lambda r_1+(1-\lambda) r_2\right) B(t)\right] \geq \lambda \inf _t\left[A(t)+r_1 B(t)\right]+(1-\lambda) \inf _t\left[A(t)+r_2 B(t)\right],
$$
showing that the infimum of an affine function of $r$ is concave in $r$. Hence every "min" term in $\mathcal{R}$ is concave in $r$. To obtain strict concavity, it again suffices to study a single term, e.g.,
$$
\tilde{L}(r):=\mathbb{E}\left[\rho^{\prime}\left(-\kappa_1 Z_1\right) M_{\rho(\cdot)}\left(\kappa_1 \alpha_1 Z_1+\kappa_2 \alpha_2 Z_2+\sigma Z_3, \frac{\tilde{\nu}}{r}\right)\right] .
$$

By dominated convergence, we may differentiate under the expectation. \Cref{supp:lemma:morea_lemma} yields:
\begin{align*}
	&\frac{d \tilde L}{d r}=\mathbb E\left[ \frac{\tilde \nu}{2r^2} \rho^{\prime}(-\kappa_1 Z_1) \rho^{\prime}\left(Prox_{\rho}(\kappa_1\alpha_1 Z_1+\kappa_2\alpha_2 Z_2+\sigma Z_3;\frac{\tilde \nu }{r})\right)^2   \right]\\
	&\frac{d^2 \tilde L}{d r^2}=\mathbb E\left[ -\frac{\tilde \nu}{r^3} \rho^{\prime}(-\kappa_1 Z_1)  \frac{\left[\rho^{\prime}\left(Prox_{\rho}(Q(Z);\frac{\tilde \nu }{r})\right) \right]^2}{1+\frac{\tilde \nu}{r}\rho^{\prime\prime}\left(Prox_{\rho}(Q(Z);\frac{\tilde \nu }{r})\right) }  \right]<0
\end{align*}
where $Q(Z):=\kappa_1\alpha_1 Z_1+\kappa_2\alpha_2 Z_2+\sigma Z_3$. Because the second derivative is strictly negative for all $r>0$ and $\tilde \nu>0$, the function $\tilde{L}(r)$ is strictly concave in $r$, and hence so is $\mathcal{R}$.

\begin{lemma}\label{supp:lemma:strict_convex1}
Fix $r>0$ and let $Z_1,Z_2,Z_3 \stackrel{\mathrm{i.i.d.}}{\sim} N(0,1)$.
Assume that $\rho$ is convex and $\rho'(t)>0$ for all $t\in\mathbb R$.
Suppose that all expectations below are finite. If
\[
(\sigma,\tilde{\nu},\alpha_1,\alpha_2)
\mapsto
\mathbb{E}\left[
M_{\rho(\cdot)}
\left(
\kappa_1 \alpha_1 Z_1+\kappa_2 \alpha_2 Z_2+\sigma Z_3,
\frac{\tilde{\nu}}{r}
\right)
\right]
\]
is strictly convex on a convex domain contained in $\{\tilde{\nu}>0\}$, then
\[
(\sigma,\tilde{\nu},\alpha_1,\alpha_2)
\mapsto
\mathbb{E}\left[
\rho^{\prime}\left(-\kappa_1 Z_1\right)
M_{\rho(\cdot)}
\left(
\kappa_1 \alpha_1 Z_1+\kappa_2 \alpha_2 Z_2+\sigma Z_3,
\frac{\tilde{\nu}}{r}
\right)
\right]
\]
is also strictly convex on the same domain.
\end{lemma}

\begin{proof}[Proof of \Cref{supp:lemma:strict_convex1}]
We write
$\theta=(\sigma,\tilde{\nu},\alpha_1,\alpha_2)$, $z=(z_1,z_2,z_3)$, and  $Z=(Z_1,Z_2,Z_3)$.
Define
\[
G(z;\theta)
:=
M_{\rho(\cdot)}
\left(
\kappa_1 \alpha_1 z_1+\kappa_2 \alpha_2 z_2+\sigma z_3,
\frac{\tilde{\nu}}{r}
\right),
\]
and
\[
w(z):=\rho'(-\kappa_1 z_1).
\]
By assumption that $\rho'(t)>0$ for all $t\in\mathbb R$, $w(Z)>0$ almost surely.

We first note that, for every fixed $z$, the mapping $\theta\mapsto G(z;\theta)$ is convex.
Indeed, since $\rho$ is convex, the Moreau envelope
\(
(x,t)\mapsto M_{\rho(\cdot)}(x,t)
\)
is jointly convex in $(x,t)$ when $t>0$.
Since
\(
x=\kappa_1 \alpha_1 z_1+\kappa_2 \alpha_2 z_2+\sigma z_3\) and
\(t=\frac{\tilde{\nu}}{r}\)
are affine functions of $\theta$ for fixed $r>0$, the composition $\theta\mapsto G(z;\theta)$ is convex.

Take any $\lambda\in(0,1)$ and any two distinct points
\[
\theta^{(1)}
=
(\sigma^{(1)},\tilde{\nu}^{(1)},\alpha_1^{(1)},\alpha_2^{(1)})
\neq
(\sigma^{(2)},\tilde{\nu}^{(2)},\alpha_1^{(2)},\alpha_2^{(2)})
=
\theta^{(2)}.
\]
Define the convex combination
\[
\theta^{(m)}
:=
(1-\lambda)\theta^{(1)}+\lambda\theta^{(2)}.
\]
For each $z$, define the convexity gap
\[
\Delta(z)
:=
(1-\lambda)G(z;\theta^{(1)})
+
\lambda G(z;\theta^{(2)})
-
G(z;\theta^{(m)}).
\]
By the pointwise convexity of $G(z;\cdot)$, we have
\[
\Delta(z)\ge 0
\quad\text{for all }z.
\]

Let
\(F(\theta):=\mathbb E[G(Z;\theta)]\).
The assumed strict convexity of $F$ gives
\[
F(\theta^{(m)})
<
(1-\lambda)F(\theta^{(1)})+\lambda F(\theta^{(2)}),
\]
or equivalently,
\[
\mathbb E[\Delta(Z)]>0.
\]
Since $\Delta(Z)\ge 0$ almost surely, the last expectation inequality implies that
\[
\mathbb P\{\Delta(Z)>0\}>0.
\]
Because $w(Z)>0$ almost surely, it follows that
\[
\mathbb E[w(Z)\Delta(Z)]>0.
\]
Expanding the function $\Delta(\cdot)$, this is equivalent to
\[
\mathbb{E}\left[
w(Z)
G(Z;\theta^{(m)})
\right]
<
(1-\lambda)
\mathbb{E}\left[
w(Z)
G(Z;\theta^{(1)})
\right] + \lambda \mathbb{E}\left[
w(Z)
G(Z;\theta^{(2)})
\right].
\]
This proves the desired strict convexity.
\end{proof}

\begin{lemma}\label{supp:lemma:M_q_v_strict}
	For $\tilde q>0, \tilde\nu>0,r>0$, and $Z\sim N(0,1)$, $\rho(t)=\log(1+e^t)$. Then the function  $L(\tilde q,\tilde \nu):= \mathbb{E}\left[ M_{\rho(\cdot)}\left(\tilde q Z ,\frac{\tilde{\nu}}{r}\right)\right]$ is jointly strictly convex in ($\tilde q, \tilde\nu$).
\end{lemma}

\begin{proof}[Proof of \Cref{supp:lemma:M_q_v_strict}]
	For any $\tilde q>0, \tilde\nu>0 $, it suffices to show that
	$$
\Gamma(x, y):=L(\tilde q+x, \tilde\nu+y)-L(\tilde q, \tilde\nu)-L_1(\tilde q, \tilde\nu) x-L_2(\tilde q, \tilde\nu) y>0, \quad \text { for all } x >-\tilde q , y>-\tilde\nu,
$$
where   $L_1=\partial L / \partial \tilde q$ and $L_2=\partial L / \partial \tilde \nu$. First note that $M_{\rho}(a,b)$ is jointly convex in $(a,b)$, which implies that $\Gamma(x, y)$ is jointly convex in $(x,y)$. Moreover, $\Gamma(0, 0)=0$, so by the mean value theorem there exists some $t^*\in (0,1)$ such that

\begin{align*}
	\Gamma(x, y)-\Gamma(0, 0)&=[\nabla \Gamma(t^*x,t^*y)-\nabla \Gamma (0,0)]^\top (x,y)\\
	&=\frac{1}{t^*}\nabla \Gamma(t^*x,t^*y)^\top (t^*x,t^*y)
\end{align*}
Here we use
\begin{align*}
	&\nabla\Gamma(0,0)=\left[\begin{array}{c}
\frac{\partial \Gamma}{\partial x}(0,0) \\
\frac{\partial \Gamma}{\partial y}(0,0)
\end{array}\right]=\left[\begin{array}{c}
L_1(\tilde q+0,\tilde \nu+0)-L_1(\tilde q,\tilde \nu) \\
L_2(\tilde q+0,\tilde \nu+0)-L_2(\tilde q,\tilde \nu)
\end{array}\right]=\left[\begin{array}{c}
0 \\
0
\end{array}\right].
\end{align*}
So it suffices to show for any $x>-\tilde q,y>-\tilde \nu$
$$\nabla \Gamma (x,y)^\top (x,y)>0 \text { for all }(x, y) \neq(0,0) .$$
This is equivalent to show
$$ \left[L_1(\tilde q+x,\tilde \nu+y)-L_1(\tilde q,\tilde \nu)\right]x+\left[L_2(\tilde q+x,\tilde \nu+y)-L_2(\tilde q,\tilde \nu)\right]y>0$$
Using dominated convergence to interchange derivatives and expectation, it is equivalent to show
\begin{align*}
	& \left[L_1(\tilde q+x,\tilde \nu+y)-L_1(\tilde q,\tilde \nu)\right]x+\left[L_2(\tilde q+x,\tilde \nu+y)-L_2(\tilde q,\tilde \nu)\right]y\\
	&=x \mathbb E\left[Z\rho^{\prime}\left(Prox_{\rho}((\tilde q+x)Z;\frac{\tilde \nu +y}{r})\right)-Z   \rho^{\prime}\left(Prox_{\rho}(\tilde q Z;\frac{\tilde \nu }{r})\right)     \right]\\
	&\quad +y\mathbb E\left[\frac{-1}{2r}\rho^{\prime}\left(Prox_{\rho}((\tilde q+x)Z;\frac{\tilde \nu +y}{r})\right)^2+ \frac{1}{2r}   \rho^{\prime}\left(Prox_{\rho}(\tilde q Z;\frac{\tilde \nu }{r})\right)^2     \right]\\
	&=\mathbb E \left\{\left[\rho^{\prime}\left(Prox_{\rho}((\tilde q+x)Z;\frac{\tilde \nu +y}{r})\right)-   \rho^{\prime}\left(Prox_{\rho}(\tilde q Z;\frac{\tilde \nu }{r})\right)  \right] \times\right.\\
	&\quad\left. \left[xZ-\frac{y}{2r}\left(\rho^{\prime}\left(Prox_{\rho}((\tilde q+x)Z;\frac{\tilde \nu +y}{r})\right)+  \rho^{\prime}\left(Prox_{\rho}(\tilde q Z;\frac{\tilde \nu }{r})\right)     \right)\right]\right\}\\
	&= \left(\frac{\tilde \nu}{r}+\frac{y}{2r}\right)\mathbb E\left[\rho^{\prime}\left(Prox_{\rho}((\tilde q+x)Z;\frac{\tilde \nu +y}{r})\right)-   \rho^{\prime}\left(Prox_{\rho}(\tilde q Z;\frac{\tilde \nu }{r})\right)  \right]^2\\
	&\quad +\mathbb E\left[\rho^{\prime}\left(Prox_{\rho}\left((\tilde q+x)Z;\frac{\tilde \nu +y}{r}\right)\right)-   \rho^{\prime}\left(Prox_{\rho}\left(\tilde q Z;\frac{\tilde \nu }{r}\right)\right)  \right] \\
	&\quad \quad \times \left[ Prox_{\rho}\left((\tilde q+x)Z;\frac{\tilde \nu +y}{r}\right) -    Prox_{\rho}\left(\tilde q Z;\frac{\tilde \nu }{r}\right)  \right]>0
\end{align*}
where we use the identity $z-\operatorname{Prox}_{\rho}(z ; b)=b \rho^{\prime}(\operatorname{Prox}_{\rho}(z ; b)) $ in the last equation. The proof is completed by observing that
\begin{align*}
	& (1) \left(\frac{\tilde \nu}{r}+\frac{y}{2r}\right)>0,\\
	& (2) \rho^{\prime\prime}(t)>0 \Longrightarrow [\rho^{\prime}(t_1)-\rho^{\prime}(t_2)](t_1-t_2)>0,\\
	& (3) \mathbb E\left[\rho^{\prime}\left(Prox_{\rho}((\tilde q+x)Z;\frac{\tilde \nu +y}{r})\right)-   \rho^{\prime}\left(Prox_{\rho}(\tilde q Z;\frac{\tilde \nu }{r})\right)  \right]^2>0.
\end{align*}

\end{proof}

\begin{lemma}\label{supp:lemma:M_alpha12_sigma_v_strict}
	Let $\sigma>0$ and denote
    $\tilde q:=\sqrt{\kappa_1^2\alpha_1^2+\kappa_2^2\alpha_2^2+\sigma^2}$. Let $Z\sim N(0,1)$. Suppose that $L(\tilde q,\tilde \nu)= \mathbb{E}\left[ M_{\rho(\cdot)}\left(\tilde q Z ,\frac{\tilde{\nu}}{r}\right)\right]$  is jointly strictly convex in $(\tilde q,\tilde \nu)$. Then $$\mathbb{E}\left[ M_{\rho(\cdot)}\left(\kappa_1\alpha_1 Z_1+\kappa_2\alpha_2 Z_2+\sigma Z_3 ,\frac{\tilde{\nu}}{r}\right)\right]$$ is strictly convex in ($\sigma,\tilde{\nu},\alpha_1,\alpha_2$), where $Z_i\stackrel{i.i.d}{\sim} N(0,1)$.
\end{lemma}

\begin{proof}[Proof of \Cref{supp:lemma:M_alpha12_sigma_v_strict}]
	By dominated convergence theorem and \Cref{supp:lemma:morea_lemma},  for fixed $\tilde \nu>0$ we have
	\begin{equation}\label{supp:eq:strict_increasing_L_q}
		\frac{\partial L}{\partial \tilde q}= \mathbb E \left[ \frac{\tilde q \rho^{\prime\prime }\left(Prox_{\rho}(\tilde q Z;\frac{\tilde \nu }{r})\right)}{1+\frac{\tilde \nu}{r} \rho^{\prime\prime }\left(Prox_{\rho}(\tilde q Z;\frac{\tilde \nu }{r})\right)  } \right]>0,
	\end{equation}
which suggests that for fixed $\tilde \nu$, $L(\tilde q,\tilde \nu)$ is strictly increasing and strictly convex in $\tilde q$.

Now fix $\tilde{\nu}>0$. Write $\xi=\left(\alpha_1, \alpha_2, \sigma\right)$ and let $\lambda \in(0,1)$. Define the diagonal matrix $D= \operatorname{diag}\left(\kappa_1, \kappa_2, 1\right)$, so that $\tilde{q}(\xi)=\|D \xi\|_2$. Take two distinct vectors $\xi_1=\left(\alpha_1^{(1)}, \alpha_2^{(1)}, \sigma^{(1)}\right) \neq \xi_2= \left(\alpha_1^{(2)}, \alpha_2^{(2)}, \sigma^{(2)}\right)$. We will prove strict convexity of the map
$$
\xi \mapsto L(\tilde{q}(\xi), \tilde{\nu})
$$
by the definition of strict convexity.
\begin{itemize}
	\item \textbf{Case 1}: $\xi_1$ is not parallel to $\xi_2$. Then, by strict convexity of the Euclidean norm, we have
	\begin{align*}
			\tilde q((1-\lambda )\xi_1+\lambda \xi_2) &=\|D[(1-\lambda )\xi_1+\lambda \xi_2]\|_2\\
		&< (1-\lambda) \|D\xi_1\|_2+\lambda \|D\xi_2\|_2\\
		&=(1-\lambda) \tilde q(\xi_1)+\lambda \tilde q(\xi_2)
	\end{align*}
	Let us denote the convex combination of components as
$$
\tilde{\alpha}_1=(1-\lambda) \alpha_1^{(1)}+\lambda \alpha_1^{(2)}, \quad \tilde{\alpha}_2=(1-\lambda) \alpha_2^{(1)}+\lambda \alpha_2^{(2)}, \quad \tilde{\sigma}=(1-\lambda) \sigma^{(1)}+\lambda \sigma^{(2)} .
$$
Then we have
	\begin{align*}
	&\mathbb{E}\left[ M_{\rho(\cdot)}\left(\kappa_1\tilde\alpha_1 Z_1+\kappa_2\tilde\alpha_2 Z_2+\tilde\sigma Z_3 ,\frac{\tilde{\nu}}{r}\right)\right]\\
	&=\mathbb{E}\left[ M_{\rho(\cdot)}\left( \sqrt{\kappa_1^2\tilde \alpha_1^2+\kappa_2\tilde \alpha_2^2+\tilde \sigma^2} Z ,\frac{\tilde{\nu}}{r}\right)\right] \\
		&=\mathbb{E}\left[ M_{\rho(\cdot)}\left(\tilde q((1-\lambda )\xi_1+\lambda \xi_2) Z ,\frac{\tilde{\nu}}{r}\right)\right]\\
		&<\mathbb{E}\left[ M_{\rho(\cdot)}\left((1-\lambda) \tilde q(\xi_1)Z+\lambda \tilde q(\xi_2)Z ,\frac{\tilde{\nu}}{r}\right)\right]\\
		&<(1-\lambda) \mathbb{E}\left[ M_{\rho(\cdot)}\left(\tilde q(\xi_1) Z ,\frac{\tilde{\nu}}{r}\right)\right]+\lambda \mathbb{E}\left[ M_{\rho(\cdot)}\left(\tilde q(\xi_2) Z ,\frac{\tilde{\nu}}{r}\right)\right]\\
		&=(1-\lambda) \mathbb{E}\left[ M_{\rho(\cdot)}\left( \kappa_1\alpha_1^{(1)} Z_1+\kappa_2\alpha_2^{(1)} Z_2+\sigma^{(1)} Z_3  ,\frac{\tilde{\nu}}{r}\right)\right]\\
		& \quad +\lambda \mathbb{E}\left[ M_{\rho(\cdot)}\left( \kappa_1\alpha_1^{(2)} Z_1+\kappa_2\alpha_2^{(2)} Z_2+\sigma^{(2)} Z_3  ,\frac{\tilde{\nu}}{r}\right)\right],
	\end{align*}
	where the first inequality follows from the strictly increasing property, and the second inequality follows from strict convexity(it is possible that $\tilde q(\xi_1)=\tilde q(\xi_2)$).
	\item \textbf{Case 2}: when $\xi_1$ is parallel to $\xi_2$ but $\|\xi_1\|\neq \|\xi_2\|$, we have $\tilde q((1-\lambda )\xi_1+\lambda \xi_2)  =(1-\lambda) \tilde q(\xi_1)+\lambda \tilde q(\xi_2)$.\footnote{Note that $\sigma>0$ is assumed, so $\xi_1$ and $\xi_2$ cannot be in opposite directions. }
    But $\tilde q(\xi_1)\neq \tilde q(\xi_2)$, then by strict convexity, we have
	\begin{align*}
		&\mathbb{E}\left[ M_{\rho(\cdot)}\left(\tilde q((1-\lambda )\xi_1+\lambda \xi_2) Z ,\frac{\tilde{\nu}}{r}\right)\right]\\
		&=\mathbb{E}\left[ M_{\rho(\cdot)}\left((1-\lambda) \tilde q(\xi_1)Z+\lambda \tilde q(\xi_2)Z ,\frac{\tilde{\nu}}{r}\right)\right]\\
		&<(1-\lambda) \mathbb{E}\left[ M_{\rho(\cdot)}\left(\tilde q(\xi_1) Z ,\frac{\tilde{\nu}}{r}\right)\right]+\lambda \mathbb{E}\left[ M_{\rho(\cdot)}\left(\tilde q(\xi_2) Z ,\frac{\tilde{\nu}}{r}\right)\right]
	\end{align*}
\end{itemize}
Based on these two cases, we conclude that for fixed $\tilde \nu$, $\mathbb{E}\left[ M_{\rho(\cdot)}\left(\kappa_1\alpha_1 Z_1+\kappa_2\alpha_2 Z_2+\sigma Z_3 ,\frac{\tilde{\nu}}{r}\right)\right]$ is strictly convex in ($\alpha_1,\alpha_2,\sigma$).

For joint convexity including $\tilde \nu$, let $(\alpha_1^{(1)},\alpha_2^{(1)},\sigma^{(1)},\tilde \nu^{(1)})\neq (\alpha_1^{(2)},\alpha_2^{(2)},\sigma^{(2)},\tilde \nu^{(2)})$, and define
 $\tilde \alpha_1=(1-\lambda)\alpha_1^{(1)}+\lambda \alpha_1^{(2)}$ , $\tilde \alpha_2,\tilde \sigma, \tilde{\tilde\nu}$ similarly.

\begin{itemize}
	\item \textbf{Scenario 1}: If $(\alpha_1^{(1)},\alpha_2^{(1)},\sigma^{(1)})= (\alpha_1^{(2)},\alpha_2^{(2)},\sigma^{(2)}), \tilde\nu^{(1)}\neq \tilde\nu^{(2)}$
	\begin{align*}
		&\mathbb{E}\left[ M_{\rho(\cdot)}\left(\kappa_1\tilde\alpha_1 Z_1+\kappa_2\tilde\alpha_2 Z_2+\tilde\sigma Z_3 ,\frac{(1-\lambda)\tilde\nu^{(1)}+\lambda \tilde\nu^{(2)} }{r}\right)\right]\\
	&=\mathbb{E}\left[ M_{\rho(\cdot)}\left( \sqrt{\kappa_1^2\tilde \alpha_1^2+\kappa_2\tilde \alpha_2^2+\tilde \sigma^2} Z ,\frac{(1-\lambda)\tilde\nu^{(1)}+\lambda \tilde\nu^{(2)} }{r}\right)\right] \\
	&=\mathbb{E}\left[ M_{\rho(\cdot)}\left( (1-\lambda +\lambda)\sqrt{\kappa_1^2\tilde \alpha_1^2+\kappa_2\tilde \alpha_2^2+\tilde \sigma^2} Z ,\frac{(1-\lambda)\tilde\nu^{(1)}+\lambda \tilde\nu^{(2)} }{r}\right)\right] \\
	\text{(Lemma \ref{supp:lemma:M_q_v_strict})}\quad	&<(1-\lambda) \mathbb{E}\left[ M_{\rho(\cdot)}\left( \sqrt{\kappa_1^2\tilde \alpha_1^2+\kappa_2\tilde \alpha_2^2+\tilde \sigma^2} Z ,\frac{ \tilde\nu^{(1)}  }{r}\right)\right]\\
		&\quad +\lambda \mathbb{E}\left[ M_{\rho(\cdot)}\left( \sqrt{\kappa_1^2\tilde \alpha_1^2+\kappa_2\tilde \alpha_2^2+\tilde \sigma^2} Z ,\frac{ \tilde\nu^{(2)} }{r}\right)\right]\\
		&=(1-\lambda) \mathbb{E}\left[ M_{\rho(\cdot)}\left(\kappa_1 \alpha_1^{(1)} Z_1+\kappa_2\alpha_2^{(1)} Z_2+\sigma^{(1)} Z_3 ,\frac{ \tilde\nu^{(1)}  }{r}\right)\right]\\
		&\quad +\lambda \mathbb{E}\left[ M_{\rho(\cdot)}\left( \kappa_1 \alpha_1^{(2)} Z_1+\kappa_2\alpha_2^{(2)} Z_2+\sigma^{(2)} Z_3  ,\frac{ \tilde\nu^{(2)} }{r}\right)\right]
	\end{align*}
\item \textbf{Scenario 2}: If $(\alpha_1^{(1)},\alpha_2^{(1)},\sigma^{(1)})\neq  (\alpha_1^{(2)},\alpha_2^{(2)},\sigma^{(2)}), \tilde\nu^{(1)}= \tilde\nu^{(2)}$, this scenario is identical to the fixed $\tilde \nu$ case, which we have already established in Case 1 and Case 2 above.
\item \textbf{Scenario 3}: If $(\alpha_1^{(1)},\alpha_2^{(1)},\sigma^{(1)})\neq  (\alpha_1^{(2)},\alpha_2^{(2)},\sigma^{(2)})$ and  $ \tilde\nu^{(1)}\neq  \tilde\nu^{(2)}$, denote $\xi_1=(\alpha_1^{(1)},\alpha_2^{(1)},\sigma^{(1)}),$ $ \xi_2=(\alpha_1^{(2)},\alpha_2^{(2)},\sigma^{(2)})$. By the subadditivity of the   $L_2$ norm, we have
{\small $$ \sqrt{\kappa_1^2\tilde \alpha_1^2+\kappa_2\tilde \alpha_2^2+\tilde \sigma^2} =\left\|D \left(\begin{array}{c}
(1-\lambda) \alpha_1^{(1)}+\lambda \alpha_1^{(2)} \\
  (1-\lambda) \alpha_2^{(1)}+\lambda \alpha_2^{(2)}\\
  (1-\lambda) \sigma^{(1)}+\lambda \sigma^{(2)}
\end{array}\right)\right\|_2\leq (1-\lambda) \left\|D \left(\begin{array}{c}
 \alpha_1^{(1)} \\
  \alpha_2^{(1)}\\
 \sigma^{(1)}
\end{array}\right)\right\|_2+ \lambda \left\|D \left(\begin{array}{c}
 \alpha_1^{(2)} \\
  \alpha_2^{(2)}\\
 \sigma^{(2)}
\end{array}\right)\right\|_2$$}
We prove strict convexity by definition:
\begin{align*}
	&\mathbb{E}\left[ M_{\rho(\cdot)}\left(\kappa_1\tilde\alpha_1 Z_1+\kappa_2\tilde\alpha_2 Z_2+\tilde\sigma Z_3 ,\frac{(1-\lambda)\tilde\nu^{(1)}+\lambda \tilde\nu^{(2)} }{r}\right)\right]\\
	&=\mathbb{E}\left[ M_{\rho(\cdot)}\left( \sqrt{\kappa_1^2\tilde \alpha_1^2+\kappa_2\tilde \alpha_2^2+\tilde \sigma^2} Z ,\frac{(1-\lambda)\tilde\nu^{(1)}+\lambda \tilde\nu^{(2)} }{r}\right)\right]	\\
(Eq. \eqref{supp:eq:strict_increasing_L_q})\quad 	&\leq \mathbb{E}\left[ M_{\rho(\cdot)}\left( (1-\lambda) \|D\xi_1\|_2 Z+\lambda \|D\xi_2\|_2 Z  ,\frac{(1-\lambda)\tilde\nu^{(1)}+\lambda \tilde\nu^{(2)} }{r}\right)\right]\\
\text{(Lemma \ref{supp:lemma:M_q_v_strict})}\quad	&<(1-\lambda) \mathbb{E}\left[ M_{\rho(\cdot)}\left(   \|D\xi_1\|_2 Z   ,\frac{ \tilde\nu^{(1)}  }{r}\right)\right]+\lambda  \mathbb{E}\left[ M_{\rho(\cdot)}\left(   \|D\xi_2\|_2 Z   ,\frac{ \tilde\nu^{(2)}  }{r}\right)\right]\\
&=(1-\lambda) \mathbb{E}\left[ M_{\rho(\cdot)}\left(\kappa_1 \alpha_1^{(1)} Z_1+\kappa_2\alpha_2^{(1)} Z_2+\sigma^{(1)} Z_3 ,\frac{ \tilde\nu^{(1)}  }{r}\right)\right]\\
		&\quad +\lambda \mathbb{E}\left[ M_{\rho(\cdot)}\left( \kappa_1 \alpha_1^{(2)} Z_1+\kappa_2\alpha_2^{(2)} Z_2+\sigma^{(2)} Z_3  ,\frac{ \tilde\nu^{(2)} }{r}\right)\right]
\end{align*}

\end{itemize}

\end{proof}

\subsubsection{Finding the optimality condition of the limiting scalar optimization}
\label{subsubsection:Finding the optimality condition of the scalar optimization}
We characterize the solution to the optimization in \eqref{optim:scalar}.
To facilitate the analysis in the following, we reparametrize $\tilde{\nu}$ by introducing $v=1/\tilde{\nu}$.
The original scalar optimization becomes:
{\footnotesize{\begin{equation}
	\label{optim:final_after_remove_u_repara}
	\begin{aligned}
\min_{\substack{ \alpha_1\in \mathbb{R} , \alpha_2\in \mathbb{R}\\v,\sigma>0}} \max _{r>0}  & \quad \left\{ -\frac{r\sigma}{\sqrt{\noverp}}+\frac{r}{2v}       \right.\\
 & -\frac{1}{4rv}-\kappa_1^2\alpha_1\mathbb{E}(\rho^{\prime\prime}(\kappa_1 Z_1))-\frac{\tau_0^2}{4rvm}-\tau_0 \kappa_2 \mathbb{E}(\rho^{\prime\prime}(\kappa_2 \xi Z_1+\kappa_2 \sqrt{1-\xi^2}Z_2))(\alpha_1\kappa_1\xi+\alpha_2\kappa_2\sqrt{1-\xi^2} )\\
 &+\mathbb{E}(M_{\rho(\cdot)}(\kappa_1\alpha_1 Z_1+\kappa_2\alpha_2 Z_2+\sigma Z_3+\frac{1}{rv}\text{Bern}(\rho^{\prime}(\kappa_1 Z_1)),\frac{1}{rv}))\\
 &\left.  +\tau_0\mathbb{E}(M_{\rho(\cdot)}(\kappa_1\alpha_1 Z_1+\kappa_2\alpha_2 Z_2+\sigma Z_3+\frac{\tau_0}{rvm}\text{Bern}(\rho^{\prime}(\kappa_2\xi Z_1+\kappa_2 \sqrt{1-\xi^2}Z_2)),\frac{\tau_0}{rvm}) \right\}.
 \end{aligned}
\end{equation}}}

Let $C(r,v,\sigma,\alpha_1,\alpha_2)$ denote the objective function in \eqref{optim:final_after_remove_u_repara}, we aim to analyze the optima of $C(\cdot)$, i.e., ($r^*,v^*,\sigma^*,\alpha_1^*,\alpha_2^*$).

We shall use the first-order characterization only for interior saddle points.
For the problem in \eqref{optim:final_after_remove_u_repara}, call a tuple
$(r,v,\sigma,\alpha_1,\alpha_2,)$ an admissible interior solution if
\[
        \sigma>0,\qquad \gamma>0,\qquad r\in(0,V),
\]
and with $\tilde{\nu}=1/v$,
the transformed tuple $(\sigma,r,\tilde{\nu},\alpha_1,\alpha_2)$ belongs to the
relative interior of the domain of \eqref{optim:scalar}. In this case $v>0$ and
$\tilde{\nu}>0$, and the change of variables between $v$ and $\tilde{\nu}$ is one-to-one
and smooth.

Assuming that the optimizer lies in the interior of the domain, the smoothness of $C(\cdot)$ implies the following first-order optimality condition:
\begin{equation}
	\label{deri: first order}
	\nabla C=\mathbf{0}
\end{equation}

We next show that \eqref{deri: first order} will reduce to our system of nonlinear equations in \eqref{nonlinear_four_equation}.
We start by taking derivatives of the objective function $C(\cdot)$ w.r.t. $r$ and $v$ and setting them equal to zero. We state the following lemma which will be exploited in taking the derivatives.
For ease of notation, we adopt a shorthand $\gamma_0:=\tau_0\gamma/m$.
 
\begin{lemma}
	\label{deri:lemma_F1gammaF2gamma0}
	For fixed values of $\kappa_1,\kappa_2, \alpha$, and $\sigma$, let the functions $F_1: \mathbb{R}_{+} \rightarrow \mathbb{R}$ and $F_2: \mathbb{R}_{+} \rightarrow \mathbb{R}$ be defined as follows:
 {\scriptsize{\begin{equation}
		\label{deri:F1F2}
		\begin{aligned}
			&F_1(\gamma)=\mathbb{E}(M_{\rho(\cdot)}(\kappa_1\alpha_1 Z_1+\kappa_2\alpha_2 Z_2+\sigma Z_3+\gamma \text{Bern}(\rho^{\prime}(\kappa_1 Z_1)),\gamma)),\\
			&F_2(\gamma_0)=\mathbb{E}(M_{\rho(\cdot)}(\kappa_1\alpha_1 Z_1+\kappa_2\alpha_2 Z_2+\sigma Z_3+\gamma_0 \text{Bern}(\rho^{\prime}(\kappa_2\xi Z_1+\kappa_2 \sqrt{1-\xi^2}Z_2)),\gamma_0)).
		\end{aligned}
	\end{equation}}}
Then, the derivatives of $F_1(\cdot)$ and $F_2(\cdot)$ are as follows:
 {\scriptsize{$$
	\begin{aligned}
		&F_1^{\prime}(\gamma)=\frac{1}{4}-\frac{1}{\gamma^2} \mathbb{E}\left[\rho^{\prime}\left(-\kappa_1 Z_1\right)\left(\kappa_1\alpha_1 Z_1+\kappa_2\alpha_2 Z_2+\sigma Z_3-\operatorname{Prox}_{\gamma \rho(\cdot)}\left(\kappa_1\alpha_1 Z_1+\kappa_2\alpha_2 Z_2+\sigma Z_3\right)\right)^2\right]\\
		&F_2^{\prime}(\gamma_0)=\frac{1}{4}-\frac{1}{\gamma_0^2} \mathbb{E}\left[\rho^{\prime}\left(-\kappa_2 \xi Z_1-\kappa_2 \sqrt{1-\xi^2}Z_2\right)\left(\kappa_1\alpha_1 Z_1+\kappa_2\alpha_2 Z_2+\sigma Z_3-\operatorname{Prox}_{\gamma_0 \rho(\cdot)}\left(\kappa_1\alpha_1 Z_1+\kappa_2\alpha_2 Z_2+\sigma Z_3\right)\right)^2\right]
	\end{aligned}
	$$}}
\end{lemma}
Using Lemma~\ref{useful_identity}, the derivation of Lemma~\ref{deri:lemma_F1gammaF2gamma0} follows directly from the proof of Lemma~7 in \cite{salehi2019impact}. To make use of Lemma~\ref{deri:lemma_F1gammaF2gamma0}, we set the new variables $\gamma=\frac{1}{rv}$ and $\gamma_0=\frac{\tau_0}{rvm}$. Then we have
{\scriptsize{$$
\begin{aligned}
	\frac{\partial C}{\partial v}=&-\frac{r}{2v^2}+\frac{1}{v^2r\gamma^2}\mathbb{E}\left[\rho^{\prime}\left(-\kappa_1 Z_1\right)\left(\kappa_1\alpha_1 Z_1+\kappa_2\alpha_2 Z_2+\sigma Z_3-\operatorname{Prox}_{\gamma \rho(\cdot)}\left(\kappa_1\alpha_1 Z_1+\kappa_2\alpha_2 Z_2+\sigma Z_3\right)\right)^2\right] \\
	&+\frac{\tau_0^2}{mv^2r\gamma_0^2}\mathbb{E}\left[\rho^{\prime}\left(-\kappa_2 \xi Z_1-\kappa_2 \sqrt{1-\xi^2}Z_2\right)\left(\kappa_1\alpha_1 Z_1+\kappa_2\alpha_2 Z_2+\sigma Z_3-\operatorname{Prox}_{\gamma_0 \rho(\cdot)}\left(\kappa_1\alpha_1 Z_1+\kappa_2\alpha_2 Z_2+\sigma Z_3\right)\right)^2\right]
\end{aligned}
$$}}

Setting $\frac{\partial C}{\partial v}=0$ we can get
{\scriptsize{\begin{equation}
	\label{deri:dv=0}
	\begin{aligned}
		\frac{r^2\gamma^2}{2}&=\mathbb{E}\left[\rho^{\prime}\left(-\kappa_1 Z_1\right)\left(\kappa_1\alpha_1 Z_1+\kappa_2\alpha_2 Z_2+\sigma Z_3-\operatorname{Prox}_{\gamma \rho(\cdot)}\left(\kappa_1\alpha_1 Z_1+\kappa_2\alpha_2 Z_2+\sigma Z_3\right)\right)^2\right]\\
		&+m\mathbb{E}\left[\rho^{\prime}\left(-\kappa_2 \xi Z_1-\kappa_2 \sqrt{1-\xi^2}Z_2\right)\left(\kappa_1\alpha_1 Z_1+\kappa_2\alpha_2 Z_2+\sigma Z_3-\operatorname{Prox}_{\gamma_0 \rho(\cdot)}\left(\kappa_1\alpha_1 Z_1+\kappa_2\alpha_2 Z_2+\sigma Z_3\right)\right)^2\right]
	\end{aligned}
\end{equation}}}

Since $\frac{\partial C}{\partial v}$ and $\frac{\partial C}{\partial r}$ contain the same expectation term, we omit the computation of $\frac{\partial C}{\partial r}$. By setting $\frac{\partial C}{\partial r}=0$, we can get
\begin{equation}
	\label{deri:dtau=0}
	\sigma^2=\noverp r^2 \gamma^2, \quad
\end{equation}

\begin{lemma}
	\label{deri:lemma_F3F4sigma}
	For fixed values of $\kappa, \alpha$, and $\gamma$, let the functions $F_3: \mathbb{R}_{+} \rightarrow \mathbb{R}$ and $F_4: \mathbb{R}_{+} \rightarrow \mathbb{R}$ be defined as follows:
	\begin{equation}
		\label{deri:F3F4}
		\begin{aligned}
			&F_3(\sigma)=\mathbb{E}(M_{\rho(\cdot)}(\kappa_1\alpha_1 Z_1+\kappa_2\alpha_2 Z_2+\sigma Z_3+\gamma \text{Bern}(\rho^{\prime}(\kappa_1 Z_1)),\gamma))\\
			&F_4(\sigma)=\mathbb{E}(M_{\rho(\cdot)}(\kappa_1\alpha_1 Z_1+\kappa_2\alpha_2 Z_2+\sigma Z_3+\gamma_0 \text{Bern}(\rho^{\prime}(\kappa_2\xi Z_1+\kappa_2 \sqrt{1-\xi^2}Z_2)),\gamma_0))
		\end{aligned}
	\end{equation}
	then the derivatives of $F_3(\cdot)$ and $F_4(\cdot)$ are as follows:
	$$
	\begin{aligned}
		&F_3^{\prime}(\sigma)=\frac{\sigma}{\gamma}\left[1-2\mathbb{E}\left(\frac{\rho^{\prime}(-\kappa_1 Z_1)}{1+\gamma \rho^{\prime\prime}(\operatorname{Prox}_{\gamma \rho(\cdot)}\left(\kappa_1\alpha_1 Z_1+\kappa_2\alpha_2 Z_2+\sigma Z_3\right))}     \right) \right]\\
		&F_4^{\prime}(\sigma)=\frac{\sigma}{\gamma_0}\left[1-2\mathbb{E}\left(\frac{\rho^{\prime}(-\kappa_2 \xi Z_1-\kappa_2\sqrt{1-\xi^2}Z_2)}{1+\gamma_0 \rho^{\prime\prime}(\operatorname{Prox}_{\gamma_0 \rho(\cdot)}\left(\kappa_1\alpha_1 Z_1+\kappa_2\alpha_2 Z_2+\sigma Z_3\right))}     \right) \right]
	\end{aligned}
	$$
\end{lemma}
\Cref{deri:lemma_F3F4sigma} can be derived based on the derivative of the Moreau envelope and Stein's identity as follows:
$$
\begin{aligned}
&\frac{\partial}{\partial \sigma} \mathbb{E}(M_{\rho(\cdot)}(\kappa_1\alpha_1 Z_1+\kappa_2\alpha_2 Z_2+\sigma Z_3+\gamma \text{Bern}(\rho^{\prime}(\kappa_1 Z_1)),\gamma)) \\&=\frac{2}{\gamma} \mathbb{E}\left[Z_3 \rho^{\prime}\left(-\kappa_1 Z_1\right)(\kappa_1\alpha_1 Z_1+\kappa_2\alpha_2 Z_2+\sigma Z_3-\operatorname{Prox}_{\gamma \rho(\cdot)}\left(\kappa_1\alpha_1 Z_1+\kappa_2\alpha_2 Z_2+\sigma Z_3\right)) \right] \\
& =\frac{\sigma}{\gamma}-\frac{2}{\gamma}\mathbb{E}\left[Z_3 \rho^{\prime}\left(-\kappa_1 Z_1\right)\operatorname{Prox}_{\gamma \rho(\cdot)}\left(\kappa_1\alpha_1 Z_1+\kappa_2\alpha_2 Z_2+\sigma Z_3\right) \right]\\
&=\frac{\sigma}{\gamma}-\frac{2}{\gamma}\mathbb{E}\left(\frac{\sigma\rho^{\prime}(-\kappa_1 Z_1)}{1+\gamma \rho^{\prime\prime}(\operatorname{Prox}_{\gamma \rho(\cdot)}\left(\kappa_1\alpha_1 Z_1+\kappa_2\alpha_2 Z_2+\sigma Z_3\right))}     \right)
\end{aligned}
$$
The derivation of $F_4^{\prime}(\sigma)$ is similar, hence omitted. Based on \Cref{deri:lemma_F3F4sigma}, the derivative of $C(\cdot)$ with respect to $\sigma$ is given by
\begin{equation}
	\label{deri:dsigma}
\begin{aligned}
		\frac{\partial C}{\partial \sigma}&=-\frac{r}{\sqrt{\noverp}}+\frac{\sigma}{\gamma}\left[1-2\mathbb{E}\left(\frac{\rho^{\prime}(-\kappa_1 Z_1)}{1+\gamma \rho^{\prime\prime}(\operatorname{Prox}_{\gamma \rho(\cdot)}\left(\kappa_1\alpha_1 Z_1+\kappa_2\alpha_2 Z_2+\sigma Z_3\right))}     \right) \right]\\
		&+\tau_0\frac{\sigma}{\gamma_0}\left[1-2\mathbb{E}\left(\frac{\rho^{\prime}(-\kappa_2 \xi Z_1-\kappa_2\sqrt{1-\xi^2}Z_2)}{1+\gamma_0 \rho^{\prime\prime}(\operatorname{Prox}_{\gamma_0 \rho(\cdot)}\left(\kappa_1\alpha_1 Z_1+\kappa_2\alpha_2 Z_2+\sigma Z_3\right))}     \right) \right]
\end{aligned}
\end{equation}
Setting $\frac{\partial C}{\partial \sigma}=0$ and taking advantage of  \eqref{deri:dtau=0}, we are able to get
\begin{equation}
	\label{deri:dsigma=0}
	\begin{aligned}
		1-\frac{1}{\noverp}+m&=2\mathbb{E}\left(\frac{\rho^{\prime}(-\kappa_1 Z_1)}{1+\gamma \rho^{\prime\prime}(\operatorname{Prox}_{\gamma \rho(\cdot)}\left(\kappa_1\alpha_1 Z_1+\kappa_2\alpha_2 Z_2+\sigma Z_3\right))}     \right) \\
		&+2m\mathbb{E}\left(\frac{\rho^{\prime}(-\kappa_2 \xi Z_1-\kappa_2\sqrt{1-\xi^2}Z_2)}{1+\gamma_0 \rho^{\prime\prime}(\operatorname{Prox}_{\gamma_0 \rho(\cdot)}\left(\kappa_1\alpha_1 Z_1+\kappa_2\alpha_2 Z_2+\sigma Z_3\right))}     \right)
	\end{aligned}
\end{equation}
where we use the relationship $\gamma_0=\tau_0\gamma/m$. So far we have shown that the optimality conditions of $C(\cdot)$ are the same as the first and second non-linear equations \eqref{nonlinear_four_equation}. Next we take derivatives with respect to $\alpha_1 $ and $\alpha_2$. We first present a lemma on the derivative of the Moreau envelope with respect to $\alpha_1 $ and $\alpha_2$.
\begin{lemma}
	\label{deri:lemmaF5F6alpha1}
	For fixed values of $\kappa, \sigma$, and $\gamma$, let the functions $F_5: \mathbb{R} \rightarrow \mathbb{R}$, $F_6: \mathbb{R} \rightarrow \mathbb{R}$, $F_7: \mathbb{R} \rightarrow \mathbb{R}$, and $F_8: \mathbb{R} \rightarrow \mathbb{R}$ be defined as follows: 
	\begin{equation*}
		\begin{aligned}
			&F_5(\alpha_1)=\mathbb{E}(M_{\rho(\cdot)}(\kappa_1\alpha_1 Z_1+\kappa_2\alpha_2 Z_2+\sigma Z_3+\gamma \text{Bern}(\rho^{\prime}(\kappa_1 Z_1)),\gamma)),\\
			&F_6(\alpha_1)=\mathbb{E}(M_{\rho(\cdot)}(\kappa_1\alpha_1 Z_1+\kappa_2\alpha_2 Z_2+\sigma Z_3+\gamma_0 \text{Bern}(\rho^{\prime}(\kappa_2\xi Z_1+\kappa_2 \sqrt{1-\xi^2}Z_2)),\gamma_0)),\\
			&F_7(\alpha_2)=\mathbb{E}(M_{\rho(\cdot)}(\kappa_1\alpha_1 Z_1+\kappa_2\alpha_2 Z_2+\sigma Z_3+\gamma \text{Bern}(\rho^{\prime}(\kappa_1 Z_1)),\gamma)),\\
			&F_8(\alpha_2)=\mathbb{E}(M_{\rho(\cdot)}(\kappa_1\alpha_1 Z_1+\kappa_2\alpha_2 Z_2+\sigma Z_3+\gamma_0 \text{Bern}(\rho^{\prime}(\kappa_2\xi Z_1+\kappa_2 \sqrt{1-\xi^2}Z_2)),\gamma_0)).
		\end{aligned}
	\end{equation*}
Then, the derivatives of $F_5(\cdot),F_6(\cdot),F_7(\cdot)$, and $F_8(\cdot)$ are as follows:
 
{\footnotesize
\begin{align*}
    &\frac{\partial F_5}{\partial \alpha_1}=\kappa_1^2\mathbb{E}[\rho^{\prime\prime}(\kappa_1 Z_1)]+ \frac{\kappa_1^2\alpha_1}{\gamma}+\frac{2\kappa_1^2}{\gamma}\mathbb{E}\left[\rho^{\prime\prime}(-\kappa_1 Z_1)\operatorname{Prox}_{\gamma \rho(\cdot)}\left(\kappa_1\alpha_1 Z_1+\kappa_2\alpha_2 Z_2+\sigma Z_3\right) \right]\\
			&\hskip 3cm-\frac{2\kappa_1^2}{\gamma} \mathbb{E}\left(\frac{\alpha_1\rho^{\prime}(-\kappa_1 Z_1)}{1+\gamma \rho^{\prime\prime}(\operatorname{Prox}_{\gamma \rho(\cdot)}\left(\kappa_1\alpha_1 Z_1+\kappa_2\alpha_2 Z_2+\sigma Z_3\right))}     \right)\\
			&\frac{\partial F_6}{\partial \alpha_1}=\kappa_1\kappa_2\xi \mathbb{E}(\rho^{\prime\prime}(-\kappa_2 \xi Z_1-\kappa_2 \sqrt{1-\xi^2}Z_2))+\frac{\kappa_1^2\alpha_1}{\gamma_0}\\
			&\hskip 3cm+\frac{2\kappa_1\kappa_2\xi}{\gamma_0}\mathbb{E}\left[\rho^{\prime\prime}(-\kappa_2 \xi Z_1-\kappa_2 \sqrt{1-\xi^2}Z_2)\operatorname{Prox}_{\gamma_0 \rho(\cdot)}\left(\kappa_1\alpha_1 Z_1+\kappa_2\alpha_2 Z_2+\sigma Z_3\right) \right]\\
			&\hskip 3cm -\frac{2\kappa_1^2}{\gamma_0} \mathbb{E}\left(\frac{\alpha_1\rho^{\prime}(-\kappa_2 \xi Z_1-\kappa_2 \sqrt{1-\xi^2}Z_2)}{1+\gamma_0 \rho^{\prime\prime}(\operatorname{Prox}_{\gamma_0 \rho(\cdot)}\left(\kappa_1\alpha_1 Z_1+\kappa_2\alpha_2 Z_2+\sigma Z_3\right))}     \right)
\end{align*}}
{\footnotesize
\begin{align*}
    &\frac{\partial F_7}{\partial \alpha_2}= \frac{\kappa_2^2\alpha_2}{\gamma}-\frac{2\kappa_2^2}{\gamma} \mathbb{E}\left(\frac{\alpha_2\rho^{\prime}(-\kappa_1 Z_1)}{1+\gamma \rho^{\prime\prime}(\operatorname{Prox}_{\gamma \rho(\cdot)}\left(\kappa_1\alpha_1 Z_1+\kappa_2\alpha_2 Z_2+\sigma Z_3\right))}     \right)\\
			&\frac{\partial F_8}{\partial \alpha_2}=\kappa_2^2\sqrt{1-\xi^2} \mathbb{E}(\rho^{\prime\prime}(-\kappa_2 \xi Z_1-\kappa_2 \sqrt{1-\xi^2}Z_2))+\frac{\kappa_2^2\alpha_2}{\gamma_0}\\
			&\hskip 3cm+\frac{2\kappa_2^2\sqrt{1-\xi^2}}{\gamma_0}\mathbb{E}\left[\rho^{\prime\prime}(-\kappa_2 \xi Z_1-\kappa_2 \sqrt{1-\xi^2}Z_2)\operatorname{Prox}_{\gamma_0 \rho(\cdot)}\left(\kappa_1\alpha_1 Z_1+\kappa_2\alpha_2 Z_2+\sigma Z_3\right) \right]\\
			&\hskip 3cm -\frac{2\kappa_2^2}{\gamma_0} \mathbb{E}\left(\frac{\alpha_2\rho^{\prime}(-\kappa_2 \xi Z_1-\kappa_2 \sqrt{1-\xi^2}Z_2)}{1+\gamma_0 \rho^{\prime\prime}(\operatorname{Prox}_{\gamma_0 \rho(\cdot)}\left(\kappa_1\alpha_1 Z_1+\kappa_2\alpha_2 Z_2+\sigma Z_3\right))}     \right)
\end{align*}
}

\end{lemma}
The proof for $\frac{\partial F_5}{\partial \alpha_1}$ is shown below, other three derivatives can be derived in same way.
{\footnotesize
\begin{align*}
    \frac{\partial F_5}{\partial \alpha_1}&=\mathbb{E}\left[\rho^{\prime}(\kappa_1 Z_1)\frac{\kappa_1 Z_1}{\gamma}\left(\kappa_1 \alpha_1 Z_1+\kappa_2 \alpha_2 Z_2+\sigma Z_3+\gamma + \operatorname{Prox}_{\gamma \rho(\cdot)}\left(-\kappa_1\alpha_1 Z_1-\kappa_2\alpha_2 Z_2-\sigma Z_3\right)\right) \right]\\
	&+\mathbb{E}\left[\rho^{\prime}(-\kappa_1 Z_1)\frac{\kappa_1 Z_1}{\gamma}\left(\kappa_1 \alpha_1 Z_1+\kappa_2 \alpha_2 Z_2+\sigma Z_3- \operatorname{Prox}_{\gamma \rho(\cdot)}\left(\kappa_1\alpha_1 Z_1+\kappa_2\alpha_2 Z_2+\sigma Z_3\right)\right) \right]\\
	&=\kappa_1^2\mathbb{E}[\rho^{\prime\prime}(\kappa_1 Z_1)]+ \frac{\kappa_1^2\alpha_1}{\gamma}-\frac{2\kappa_1}{\gamma}\mathbb{E}\left[Z_1 \rho^{\prime}(-\kappa_1 Z_1)\operatorname{Prox}_{\gamma \rho(\cdot)}\left(\kappa_1\alpha_1 Z_1+\kappa_2\alpha_2 Z_2+\sigma Z_3\right) \right]\\
	&=\kappa_1^2\mathbb{E}[\rho^{\prime\prime}(\kappa_1 Z_1)]+ \frac{\kappa_1^2\alpha_1}{\gamma}+\frac{2\kappa_1^2}{\gamma}\mathbb{E}\left[\rho^{\prime\prime}(-\kappa_1 Z_1)\operatorname{Prox}_{\gamma \rho(\cdot)}\left(\kappa_1\alpha_1 Z_1+\kappa_2\alpha_2 Z_2+\sigma Z_3\right) \right]\\
			&\hskip 3cm-\frac{2\kappa_1^2}{\gamma} \mathbb{E}\left(\frac{\alpha_1\rho^{\prime}(-\kappa_1 Z_1)}{1+\gamma \rho^{\prime\prime}(\operatorname{Prox}_{\gamma \rho(\cdot)}\left(\kappa_1\alpha_1 Z_1+\kappa_2\alpha_2 Z_2+\sigma Z_3\right))}     \right)
\end{align*}
}
where we use $\rho^{\prime}(-x)=1-\rho^{\prime}(x)$, $\operatorname{Prox}_{\gamma \rho(\cdot)}\left(b+\gamma\right)=-\operatorname{Prox}_{\gamma \rho(\cdot)}\left(-b\right)$ and the derivative of the Moreau envelope in the first equality. For the second equality, we apply the Stein identity and \Cref{useful_identity}, and we use the Stein identity and the derivative of the proximal operator of $\rho(\cdot)$ in the last equality.

Now we are ready to state the result for $\frac{\partial C}{\partial \alpha_1}$ and $\frac{\partial C}{\partial \alpha_2}$ based on \Cref{deri:lemmaF5F6alpha1}, we use \eqref{deri:dsigma=0} to replace two expectations when we set partial derivative to zero, we have
{\footnotesize{\begin{equation}
	\label{deri:dalpha1dalpha2=01}
	\begin{aligned}
&		0=\frac{\partial C}{\partial \alpha_1}= \frac{\kappa_1^2 \alpha_1}{\noverp \gamma}+\frac{2\kappa_1^2}{\gamma}\mathbb{E}\left[\rho^{\prime\prime}(-\kappa_1 Z_1)\operatorname{Prox}_{\gamma \rho(\cdot)}\left(\kappa_1\alpha_1 Z_1+\kappa_2\alpha_2 Z_2+\sigma Z_3\right) \right]\\
&\hskip 2cm+\frac{2\tau_0\kappa_1\kappa_2\xi}{\gamma_0} \mathbb{E}\left[\rho^{\prime\prime}(-\kappa_2 \xi Z_1-\kappa_2 \sqrt{1-\xi^2}Z_2)\operatorname{Prox}_{\gamma_0 \rho(\cdot)}\left(\kappa_1\alpha_1 Z_1+\kappa_2\alpha_2 Z_2+\sigma Z_3\right) \right]\\
&0=\frac{\partial C}{\partial \alpha_2}= \frac{\kappa_2^2 \alpha_2}{\noverp \gamma}+\frac{2\kappa_2^2\tau_0\sqrt{1-\xi^2}}{\gamma_0}\mathbb{E}\left[\rho^{\prime\prime}(-\kappa_2 \xi Z_1-\kappa_2 \sqrt{1-\xi^2}Z_2)\operatorname{Prox}_{\gamma_0 \rho(\cdot)}\left(\kappa_1\alpha_1 Z_1+\kappa_2\alpha_2 Z_2+\sigma Z_3\right) \right]
	\end{aligned}
\end{equation}}}

Combine the result from  \eqref{deri:dv=0} \eqref{deri:dtau=0} \eqref{deri:dsigma=0} \eqref{deri:dalpha1dalpha2=01}, we have

{\scriptsize{\begin{equation}
\label{nonlinear_5_equation}
\left\{\begin{aligned}
\frac{\gamma^2r^2}{2 } & =\mathbb{E}\left[\rho^{\prime}\left(-\kappa_1 Z_1\right)\left(\kappa_1 \alpha_1 Z_1+\kappa_2 \alpha_2 Z_2 +\sigma Z_3-\operatorname{Prox}_{\gamma \rho(\cdot)}\left(\kappa_1 \alpha_1 Z_1+\kappa_2 \alpha_2 Z_2 +\sigma Z_3\right)\right)^2\right] \\
& \quad  +m \mathbb{E}\left[\rho^{\prime}\left(-\kappa_2 \xi Z_1-\kappa_2 \sqrt{1-\xi^2}Z_2\right)\left(\kappa_1 \alpha_1 Z_1+\kappa_2 \alpha_2 Z_2 +\sigma Z_3-\operatorname{Prox}_{\gamma_0 \rho(\cdot)}\left(\kappa_1 \alpha_1 Z_1+\kappa_2 \alpha_2 Z_2 +\sigma Z_3\right)\right)^2\right]\\
\sigma^2&=\noverp\gamma^2r^2 \\
1-\frac{1}{\noverp}+m  & =\mathbb{E}\left[\frac{2 \rho^{\prime}\left(-\kappa_1 Z_1\right)}{1+\gamma \rho^{\prime \prime}\left(\operatorname{Prox}_{\gamma \rho(\cdot)}\left(\kappa_1 \alpha_1 Z_1+\kappa_2 \alpha_2 Z_2 +\sigma Z_3\right)\right)}\right] \\
& \quad +m \mathbb{E}\left[\frac{2 \rho^{\prime}\left(-\kappa_2 \xi Z_1-\kappa_2 \sqrt{1-\xi^2}Z_2\right)}{1+\gamma_0 \rho^{\prime \prime}\left(\operatorname{Prox}_{\gamma_0 \rho(\cdot)}\left(\kappa_1 \alpha_1 Z_1+\kappa_2 \alpha_2 Z_2 +\sigma Z_3\right)\right)}\right] \\
-\frac{\alpha_1}{2 \noverp} & =\mathbb{E}\left[\rho^{\prime \prime}\left(-\kappa_1 Z_1\right) \operatorname{Prox}_{\gamma \rho(\cdot)}\left(\kappa_1 \alpha_1 Z_1+\kappa_2 \alpha_2 Z_2 +\sigma Z_3\right)\right] \\
&\quad +m \xi \frac{\kappa_2}{\kappa_1}\mathbb{E}\left[\rho^{\prime \prime}\left(-\kappa_2 \xi Z_1-\kappa_2 \sqrt{1-\xi^2}Z_2\right) \operatorname{Prox}_{\gamma_0 \rho(\cdot)}\left(\kappa_1 \alpha_1 Z_1+\kappa_2 \alpha_2 Z_2 +\sigma Z_3\right)\right]\\
-\frac{\alpha_2}{2 \noverp} & =m \sqrt{
1-\xi^2
} \mathbb{E}\left[\rho^{\prime \prime}\left(-\kappa_2 \xi Z_1-\kappa_2 \sqrt{1-\xi^2}Z_2\right) \operatorname{Prox}_{\gamma_0 \rho(\cdot)}\left(\kappa_1 \alpha_1 Z_1+\kappa_2 \alpha_2 Z_2 +\sigma Z_3\right)\right]
\end{aligned}\right.
\end{equation}}}

The equations in \eqref{nonlinear_5_equation} are the first-order conditions of the
reparametrized scalar problem in \eqref{optim:final_after_remove_u_repara}.
We now show that, among admissible interior solutions, these equations characterize the unique optimizer
of \eqref{optim:scalar}.

Suppose that $(\alpha_1,\alpha_2,\sigma,\gamma,r)$ is an admissible interior solution of
\eqref{nonlinear_5_equation} with $\sigma>0$ and $r\in (0, V)$.
Define
\[
        v=\frac{1}{r\gamma},
        \qquad
        \tilde{\nu}=r\gamma .
\]
Then the tuple $(\sigma,r,\tilde{\nu},\alpha_1,\alpha_2)$ satisfies the
first-order conditions of the scalar objective
$\mathcal R(\sigma,r,\tilde{\nu},\alpha_1,\alpha_2)$ in \eqref{optim:scalar}.

Let $x=(\sigma,\tilde{\nu},\alpha_1,\alpha_2)$. Since
$\mathcal R(\cdot,r)$ is convex in $x$ and $\mathcal R(x,\cdot)$ is concave in $r$, these
interior first-order conditions are sufficient for a saddle point. Indeed, for any feasible
$x'$ and $r'$,
\[
        \mathcal R(x',r)
        \ge
        \mathcal R(x,r)+\langle \nabla_x \mathcal R(x,r),x'-x\rangle
        =
        \mathcal R(x,r),
\]
and
\[
        \mathcal R(x,r')
        \le
        \mathcal R(x,r)+\partial_r\mathcal R(x,r)(r'-r)
        =
        \mathcal R(x,r).
\]
Therefore $(x,r)$ is a saddle point of $\mathcal{R}$ in \eqref{optim:scalar}.
\Cref{sec:unique-saddle-R} has established that the
saddle point of $\mathcal{R}$ is unique, so \eqref{nonlinear_5_equation} has at most one
admissible interior solution, and this solution is the one induced by the unique optimizer of
\eqref{optim:scalar}.

Finally, by the second equation $\sigma^2=\noverp \gamma^2 r^2$ in \eqref{nonlinear_5_equation},
we have
\[
        r=\frac{\sigma}{\sqrt{\noverp}\gamma},
        \qquad
        \gamma^2r^2=\frac{\sigma^2}{\noverp}.
\]
Substituting this identity into the remaining equations of \eqref{nonlinear_5_equation},
and using $\gamma_0=\tau_0\gamma/m$, gives the reduced four-equation system in
\eqref{nonlinear_four_equation}. Consequently, \eqref{nonlinear_four_equation} has a
unique admissible solution.

\subsubsection{Applying CGMT to connect PO and AO}

Recall in the process of simplifying AO, we decompose $\bbeta$ in \eqref{logic:AO_distribution} and obtain the equality that $\text{direction}( \mathbf{P}^{\perp} \widehat{\bbeta}^{AO})=\text{direction}(\mathbf{P}^{\perp}\bg)$.
Therefore, the solution of AO can be expressed as
\begin{equation}
	\label{AO:solution}
	\widehat{\bbeta}^{AO}=\widehat{\sigma} \btheta_{g}+\widehat{\alpha}_1 \kappa_1 \be_1+ {\widehat{\alpha}_2}\kappa_2 \be_2
\end{equation}
where $\|\btheta_g\|=1$ and direction$(\btheta_g)=\text{direction}(\mathbf{P}^{\perp}\bg)$, and $\bg\sim N(0,\mathbb \mathbb{I}_p)$  is independent of $(\be_1,\be_2)$. Based on   the convergence of optima $(\widehat{\sigma},\widehat{r},\widehat{\tilde{\nu}},\widehat{\alpha}_1,\widehat{\alpha}_2)\xrightarrow{a.s} (\sigma_*,r_*,\tilde{\nu}_*,\alpha_{1*},\alpha_{2*})$ and \eqref{e1e2}, we have
\begin{align}
    & \langle \widehat{\bbeta}^{AO},\be_1    \rangle \xrightarrow{a.s} \alpha_{1*}\kappa_1   \label{AO_cosine_sim_e1}  \\
    & \langle \widehat{\bbeta}^{AO},\be_2    \rangle \xrightarrow{a.s} \alpha_{2*}\kappa_2 \label{AO_cosine_sim_e2} \\
    & \|\mathbf{P}^\perp \widehat{\bbeta}^{AO} \|_2\xrightarrow{a.s} \sigma_* \label{P_perp_AO_norm}
\end{align}

To apply the asymptotic convergence of CGMT (\Cref{CGMT:asym}), for any $\epsilon>0$, we introduce three sets $\mathcal{S}_1,\mathcal{S}_2,\mathcal{S}_3$ as follows:
\begin{align*}
    &\mathcal{S}_1=\left\{\boldsymbol{\beta} \in \mathbb{R}^p:\left|  \langle  {\bbeta} ,\be_1    \rangle - \alpha_{1*}\kappa_1
\right|<\epsilon\right\},\\
&\mathcal{S}_2=\left\{\boldsymbol{\beta} \in \mathbb{R}^p:\left|  \langle  {\bbeta} ,\be_2    \rangle - \alpha_{2*}\kappa_2
\right|<\epsilon\right\},\\
&\mathcal{S}_3=\left\{\boldsymbol{\beta} \in \mathbb{R}^p:\left|  \|\mathbf{P}^\perp  {\bbeta}  \|_2- \sigma_*
\right|<\epsilon\right\}.
\end{align*}
The convergence in \eqref{AO_cosine_sim_e1} \eqref{AO_cosine_sim_e2} and \eqref{P_perp_AO_norm}  guarantees that as $n \rightarrow \infty, \widehat{\boldsymbol{\beta}}^{A O} \in \mathcal{S}_j$ with probability  $1$ for $j\in \{1,2,3\}$.
To extend such a statement to the PO solution, we will show $ \widehat{\boldsymbol{\beta}}^{P O} \in \mathcal{S}_j$ with probability approaching $1$ using \Cref{CGMT:asym}.
First, we recall the PO, AO, and the scalar optimization we defined  in \eqref{PO_our}, \eqref{AO_our}, and \eqref{optim:scalar}:

  \begin{align*}
     \text{(PO)}\quad  \Phi(\Tilde{\mathbf{H}})
=& \min _{\substack{\boldsymbol{\beta}  \in  \mathcal{S}_{\bbeta} \\ \mathbf{u}_1 \in \mathcal{S}_{\bu_1}, \mathbf{u}_2 \in \mathcal{S}_{\bu_2}}}  \max_{\bv \in \mathcal{S}_{\bv}}  \left\{ \frac{-1}{\sqrt{n}}  \bv^\top \Tilde{\mathbf{H}}\mathbf{P}^{\perp}\boldsymbol{\beta}  + \frac{1}{n} \mathbf{1}^T \rho\left(\mathbf{u}_1\right)-\frac{1}{n} \mathbf{y}_1^T \mathbf{u}_1+ \right. \\
  & \left.\frac{\tau_0}{M} \mathbf{1}^T \rho\left(\mathbf{u}_2\right)-\frac{\tau_0}{M} \mathbf{y}_2^T \mathbf{u}_2 +\frac{1}{\sqrt{n}}\bv^T\left(\left[\begin{array}{l}
\mathbf{u}_1 \\
\mathbf{u}_2
\end{array}\right]-  \mathbf{H} \mathbf{P}\bbeta\right) \right\}
  \end{align*}

\begin{align*}
     \text{(AO)}\quad  \phi(\bg,\bh)
=& \min _{\substack{\boldsymbol{\beta}  \in  \mathcal{S}_{\bbeta} \\ \mathbf{u}_1 \in \mathcal{S}_{\bu_1}, \mathbf{u}_2 \in \mathcal{S}_{\bu_2}}}  \max_{\bv \in \mathcal{S}_{\bv}}  \left\{  -\frac{1}{ \sqrt{n}}\left(\bv^T \mathbf{h}\left\|\mathbf{P}^{\perp}\boldsymbol{\beta} \right\|+\|\bv\| \mathbf{g}^T \mathbf{P}^{\perp}\boldsymbol{\beta}  \right)  + \frac{1}{n} \mathbf{1}^T \rho\left(\mathbf{u}_1\right)- \right. \\
  & \left. \frac{1}{n} \mathbf{y}_1^T \mathbf{u}_1+ \frac{\tau_0}{M} \mathbf{1}^T \rho\left(\mathbf{u}_2\right)-\frac{\tau_0}{M} \mathbf{y}_2^T \mathbf{u}_2 +\frac{1}{\sqrt{n}}\bv^T\left(\left[\begin{array}{l}
\mathbf{u}_1 \\
\mathbf{u}_2
\end{array}\right]-  \mathbf{H} \mathbf{P}\bbeta\right) \right\}
\end{align*}

$$
\text{(scalar optimization)} \quad \quad \bar \phi:=  \quad \min _{\substack{\sigma \geq 0 \\   \alpha_1, \alpha_2\in \mathbb{R},\tilde{\nu}>0 }} \max _{\substack{ r\in [0,V]}} \mathcal{R}(\sigma,r,\tilde{\nu},\alpha_1,\alpha_2)
$$
We start with showing that $ \widehat{\boldsymbol{\beta}}^{P O} \in \mathcal{S}_1$ with probability approaching 1.
Let $\mathcal{S}_1^c:=\mathcal{S}_{\bbeta}\setminus S_1$. Denote $\Phi_{\mathcal{S}_1^c}(\Tilde{\mathbf{H}})$ and $\phi_{\mathcal{S}_1^c}(\mathbf{g}, \mathbf{h})$ the optimal loss of the PO and AO, respectively, when the minimization over $\bbeta$ is constrained over $\bbeta \in \mathcal{S}_1^c$. In terms of AO, $\bbeta \in \mathcal{S}_1^c$ is equivalent to put constraints on $\alpha_1$, we can express $\phi_{\mathcal{S}_1^c}(\mathbf{g}, \mathbf{h})$ as follows under same argument,
\begin{align*}
  \phi_{\mathcal{S}_1^c}(\mathbf{g}, \mathbf{h})=  \min _{\substack{0 \leq \sigma \leq c_1,0< \tilde{\nu}\leq 6c_1  \\  |\alpha_1|\leq c_1/\kappa_1, |\alpha_2| \leq c_1/\kappa_2 \\ |\alpha_1-\alpha_{1*}|\kappa_1\geq \epsilon }} \max _{\substack{ r \geq 0}} \quad \mathcal{R}_n(\sigma,r,\tilde{\nu},\alpha_1,\alpha_2).
\end{align*}
Recall in \Cref{supp:sec:convergence_AO}, we show that $\phi(\bg,\bh)\xrightarrow{\mathbb{P}} \bar \phi$. Following a similar argument, we can show that there exists a constant $\bar \phi_{\mathcal{S}_1^c}$, defined as
$$\bar \phi_{\mathcal{S}_1^c}:=  \min _{\substack{0 \leq \sigma \leq c_1,0< \tilde{\nu}\leq 6c_1  \\  |\alpha_1|\leq c_1/\kappa_1, |\alpha_2| \leq c_1/\kappa_2 \\ |\alpha_1-\alpha_{1*}|\kappa_1\geq \epsilon }} \max _{\substack{ r \geq 0}} \quad \mathcal{R}(\sigma,r,\tilde{\nu},\alpha_1,\alpha_2),$$
such that $\phi_{\mathcal{S}_1^c}(\mathbf{g}, \mathbf{h})\xrightarrow{\mathbb{P}} \bar \phi_{\mathcal{S}_1^c}$.
Based on the uniqueness of the optima $(\sigma_*,r_*,\tilde{\nu}_*,\alpha_{1*},\alpha_{2*})$, we have $\bar \phi< \bar \phi_{\mathcal{S}_1^c}$.  Then based on \Cref{CGMT:asym}, we have
\begin{equation}\label{PO_S1}
    \lim_{n\rightarrow \infty} \mathbb P(\widehat{\boldsymbol{\beta}}^{P O} \in \mathcal{S}_1)= 1.
\end{equation}
By the same argument, we can show that \eqref{PO_S1} holds for $S_2$ and $S_3$.
Define $\alpha_1(p):=\langle \be_1,\widehat \bbeta ^{PO} \rangle /\|\bbeta_0\|$, $\alpha_2(p):=\langle \be_2,\widehat \bbeta ^{PO} \rangle /\|\bbeta_s\|$ and $\sigma(p):= \|\mathbf{P}^\perp \widehat{\bbeta}^{PO} \|_2 $.
Since we have proved that the events $\widehat{\boldsymbol{\beta}}^{P O} \in \mathcal{S}_j$ for $j=1,2,3$ happen with probability approaching $1$,
we arrive at the following results:
\begin{align}
    & \alpha_1(p) \xrightarrow{\mathbb{P}} \alpha_{1*}  \label{PO_cosine_sim_e1} , \\
    & \alpha_2(p) \xrightarrow{\mathbb{P}} \alpha_{2*}  \label{PO_cosine_sim_e2} ,\\
    & \sigma(p)\xrightarrow{\mathbb{P}} \sigma_* \label{P_perp_PO_norm}.
\end{align}

\subsubsection{Proving asymptotics with locally Lipschitz function}
In this section, we will show for any locally Lipschitz function $\Psi$,
\begin{equation}\label{supp:proof_pseudo_converge_MAP}
\frac{1}{p} \sum_{j=1}^p \Psi\left(\sqrt{p}[\widehat{\boldsymbol{\beta}}_{M,j}-\alpha_{1*}\bbeta_{0,j}-\frac{\alpha_{2*}}{\sqrt{1-\xi^2}}(\bbeta_{s,j}-\xi \frac{\kappa_2}{\kappa_1}\bbeta_{0,j})], \sqrt{p}\boldsymbol{\beta}_{0,j}\right) \xrightarrow{\mathbb{P}} \mathbb{E}[\Psi( \sigma_{*}Z, \beta)],
\end{equation}
where $\beta\sim \Pi$ is independent of $Z\sim N(0,1)$. Our proof is an extension of the proof in \cite{zhao2022asymptotic}, and we include the details below for completeness.
Recall that we can decompose the SRE as follows:
\begin{align*}
\widehat{\bbeta}_M&=\mathbf{P}\widehat{\bbeta}_M+\mathbf{P}^{\perp}\widehat{\bbeta}_M \\
	&=(\frac{\bbeta_0^T \widehat{\bbeta}_M}{\|\bbeta_0\|^2})\bbeta_0+(\frac{ (\boldsymbol{\beta}_s-\frac{\langle \bbeta_s,\bbeta_0 \rangle}{\|\bbeta_0\|^2}\boldsymbol{\beta}_0)^T \widehat{\bbeta}_M}{\|\boldsymbol{\beta}_s-\frac{\langle \bbeta_s,\bbeta_0 \rangle}{\|\bbeta_0\|^2}\boldsymbol{\beta}_0\|^2})(\boldsymbol{\beta}_s-\frac{\langle \bbeta_s,\bbeta_0 \rangle}{\|\bbeta_0\|^2}\boldsymbol{\beta}_0)+\mathbf{P}^{\perp}\widehat{\bbeta}_M\\
 &= \alpha_1(p)\bbeta_0+\frac{\alpha_2(p)\|\bbeta_s\|}{\|\boldsymbol{\beta}_s-\frac{\langle \bbeta_s,\bbeta_0 \rangle}{\|\bbeta_0\|^2}\boldsymbol{\beta}_0\|}(\boldsymbol{\beta}_s-\frac{\langle \bbeta_s,\bbeta_0 \rangle}{\|\bbeta_0\|^2}\boldsymbol{\beta}_0)+\sigma(p)\frac{\mathbf{P}^{\perp}\widehat{\bbeta}_M}{\|\mathbf{P}^{\perp}\widehat{\bbeta}_M\|}.
\end{align*}

To prove \eqref{supp:proof_pseudo_converge_MAP}, we first introduce some notations. Let $\bZ=(Z_1,\cdots ,Z_p)$ be a random vector with independent standard Gaussian entries.
We define vectors $\bT$, $\bT^{\textrm{approx}}$, and $\tilde \bZ^{\textrm{scaled}}$ whose entries are defined as follows:

\begin{equation}\label{supp:define_T_j_Z_j_scaled}
    \begin{aligned}
        &T_j:=\frac{\sqrt{p}\left(\widehat \beta_{M,j}-\alpha_{1*}\beta_{0,j}-\frac{\alpha_{2*} }{\sqrt{1-\xi^2}} ( \beta_{s,j}-\frac{\xi\kappa_2}{\kappa_1}  \beta _{0,j}) \right)}{\sigma_*},\\
 &T^{\textrm{approx}}_j:=\frac{\sqrt{p}\left(\widehat \beta_{M,j}-\alpha_1(p)\beta_{0,j}-\frac{\alpha_2(p)\|\bbeta_s\|}{\|\boldsymbol{\beta}_s-\frac{\langle \bbeta_s,\bbeta_0 \rangle}{\|\bbeta_0\|^2}\boldsymbol{\beta}_0\|} ( \beta_{s,j}-\frac{\langle \bbeta_s,\bbeta_0 \rangle}{\|\bbeta_0\|^2}  \beta _{0,j}) \right)}{\sigma(p)},\\
        &\tilde Z^{\textrm{scaled}}_j:=\frac{\sqrt{p}}{\|\mathbf{P}^{\perp}\bZ\|} \left(Z_j-(\frac{\bbeta_0^T \bZ}{\|\bbeta_0\|^2})\beta_{0,j}-(\frac{ (\boldsymbol{\beta}_s-\frac{\langle \bbeta_s,\bbeta_0 \rangle}{\|\bbeta_0\|^2}\boldsymbol{\beta}_0)^T \bZ}{\|\boldsymbol{\beta}_s-\frac{\langle \bbeta_s,\bbeta_0 \rangle}{\|\bbeta_0\|^2}\boldsymbol{\beta}_0\|^2})( \beta_{s,j}-\frac{\langle \bbeta_s,\bbeta_0 \rangle}{\|\bbeta_0\|^2}  \beta _{0,j}) \right).
    \end{aligned}
\end{equation}
We comment that $T_j$ corresponds to an entry that appears in \eqref{supp:proof_pseudo_converge_MAP} and  $\tilde Z^{\textrm{scaled}}_j$ is a scaled version of $\mathbf{P}^{\perp}\bZ$. Note that $\tilde Z^{\textrm{scaled}}_j$ does not depend on the samples so its limiting distribution can be easily characterized. To analyze $T_j$, we utilize the key that $T^{\textrm{approx}}_j$ approximates $T_j$ closely while sharing the same distribution as $\tilde Z^{\textrm{scaled}}_j$.

\bigskip

For any locally Lipschitz function $\Psi$, the proof of \eqref{supp:proof_pseudo_converge_MAP} is decomposed into four steps:
\begin{enumerate}
    \item Utilizing \eqref{PO_cosine_sim_e1},\eqref{PO_cosine_sim_e2} and \eqref{P_perp_PO_norm}, we can show
    \begin{equation}\label{supp:step1_converge_T_T_approx}
\frac{1}{p} \sum_{j=1}^p \Psi\left(\sigma_{\star} T_j, \sqrt{p}   \beta_{0,j}\right)-\frac{1}{p} \sum_{j=1}^p \Psi\left(\sigma_{\star} T_j^{\textrm{approx}}, \sqrt{p}   \beta_{0,j}\right) \xrightarrow{\mathbb{P}} 0 .
\end{equation}
\item Utilizing \eqref{PO_cosine_sim_e1},\eqref{PO_cosine_sim_e2} and \eqref{P_perp_PO_norm}, we can show
\begin{equation}\label{supp:step2_converge_Z_Z_scaled}
\frac{1}{p} \sum_{j=1}^p \Psi\left(\sigma_{\star} \tilde{Z}_j^{\textrm{scaled}}, \sqrt{p}   \beta_{0,j}\right)-\frac{1}{p} \sum_{j=1}^p \Psi\left(\sigma_{\star} Z_j , \sqrt{p}   \beta_{0,j}\right) \xrightarrow{\mathbb{P}} 0 .
\end{equation}
\item Using the law of large numbers, we can show
\begin{equation}\label{supp:step3_final_limit}
\frac{1}{p} \sum_{j=1}^p \Psi\left(\sigma_{\star} Z_j , \sqrt{p}  \beta_{0,j}\right) \xrightarrow{\mathbb{P}} \mathbb{E}\left[\Psi\left(\sigma_{\star} Z, \eta\right)\right].
\end{equation}
\item To close the gap between $\bT$ and $\tilde \bZ^{\textrm{scaled}}$, we show that $\bT^{\textrm{approx}} \stackrel{d}{=} \tilde \bZ^{\textrm{scaled}}$, which immediately implies
$$
\frac{1}{p} \sum_{j=1}^p \Psi\left(\sigma_{\star} T_j^{\textrm{approx}}, \sqrt{p}   \beta_{0,j}\right) \stackrel{d}{=} \frac{1}{p} \sum_{j=1}^p \Psi\left(\sigma_{\star} \tilde{Z}_j^{\textrm{scaled}}, \sqrt{p}  \beta_{0,j}\right) ,
$$
and both the RHS and the LHS converge to the same limit stated in \eqref{supp:step3_final_limit}.
\end{enumerate}

\textbf{Step 1: Prove \eqref{supp:step1_converge_T_T_approx}}

We control the difference between $ \Psi\left(\sigma_{\star} T_j, \sqrt{p}   \beta_{0,j}\right)$ and $\Psi\left(\sigma_{\star} T_j^{\textrm{approx}}, \sqrt{p}   \beta_{0,j}\right) $ based on the definition of locally Lipschitz function.

\begin{align*}
    & \left|\frac{1}{p} \sum_{j=1}^p \Psi\left(\sigma_{\star} T_j, \sqrt{p}   \beta_{0,j}\right)-\frac{1}{p} \sum_{j=1}^p \Psi\left(\sigma_{\star} T_j^{\textrm{approx}}, \sqrt{p}   \beta_{0,j}\right)\right| \\
\leq & \frac{L \sigma_{\star}}{p} \sum_{j=1}^p\left(1+\left\|\left(\sigma_{\star} T_j, \sqrt{p} \beta_{0,j}\right)\right\|+\left\|\left(\sigma_{\star} T_j^{\textrm{approx}}, \sqrt{p}  \beta_{0,j}\right)\right\|\right)\left|T_j-T_j^{\textrm{approx}}\right| \\
\leq & {L \sigma_{\star}} \left\{\frac{1}{p}\sum_{j=1}^p\left(1+\sigma_{\star}\left|T_j\right|+\sigma_{\star}\left|T_j^{\textrm{approx}}\right|+2 \sqrt{p}\left| \beta_{0,j}\right|\right)^2\right\}^{1 / 2}\left\{\frac{1}{p}\sum_{j=1}^p\left(T_j-T_j^{\textrm{approx}}\right)^2\right\}^{1 / 2},
\end{align*}
where the second inequality follows from the Cauchy-Schwarz inequality. We will show that the first term is stochastically bounded by a constant and the second term converges to zero. The second term satisfies

\begin{align*}
    \frac{1}{p}\sum_{j=1}^p\left(T_j-T_j^{\textrm{approx}}\right)^2 & = \sum_{j=1}^p \left( (\frac{\sigma_*}{\sigma(p)}-1)\widehat \beta_{M,j}-(\frac{\sigma_*}{\sigma(p)}\alpha_1(p)-\alpha_{1*}) \beta_{0,j} \right.\\
& \quad \quad \left. +\left(\frac{\sigma_*}{\sigma(p)}\frac{\alpha_2(p)\|\bbeta_s\|}{\|\boldsymbol{\beta}_s-\frac{\langle \bbeta_s,\bbeta_0 \rangle}{\|\bbeta_0\|^2}\boldsymbol{\beta}_0\|}\frac{\langle \bbeta_s,\bbeta_0 \rangle}{\|\bbeta_0\|^2}   -\frac{\alpha_{2*}\xi \kappa_2}{\kappa_1\sqrt{1-\xi^2}}\right)\beta_{0,j} \right. \\
&\quad \quad \left. -\left(\frac{\sigma_*}{\sigma(p)}\frac{\alpha_2(p)\|\bbeta_s\|}{\|\boldsymbol{\beta}_s-\frac{\langle \bbeta_s,\bbeta_0 \rangle}{\|\bbeta_0\|^2}\boldsymbol{\beta}_0\|}    -\frac{\alpha_{2*} }{ \sqrt{1-\xi^2}}\right)\beta_{s,j}  \right)^2\\
&\leq 2 \|\widehat \bbeta_M\|^2\left(\frac{\sigma_*}{\sigma(p)}-1\right)^2+2\|\bbeta_0\|^2\left(\frac{\sigma_*}{\sigma(p)}\alpha_1(p)-\alpha_{1*}\right)^2\\
& \quad \quad +2\|\bbeta_0\|^2\left(\frac{\sigma_*}{\sigma(p)}\frac{\alpha_2(p)\|\bbeta_s\|}{\|\boldsymbol{\beta}_s-\frac{\langle \bbeta_s,\bbeta_0 \rangle}{\|\bbeta_0\|^2}\boldsymbol{\beta}_0\|}\frac{\langle \bbeta_s,\bbeta_0 \rangle}{\|\bbeta_0\|^2}   -\frac{\alpha_{2*}\xi \kappa_2}{\kappa_1\sqrt{1-\xi^2}}\right)^2\\
&\quad \quad +2\|\bbeta_s\|^2\left(\frac{\sigma_*}{\sigma(p)}\frac{\alpha_2(p)\|\bbeta_s\|}{\|\boldsymbol{\beta}_s-\frac{\langle \bbeta_s,\bbeta_0 \rangle}{\|\bbeta_0\|^2}\boldsymbol{\beta}_0\|}    -\frac{\alpha_{2*} }{ \sqrt{1-\xi^2}}\right)^2 \xrightarrow{\mathbb{P}} 0,
\end{align*}
where the last convergence follows from \eqref{PO_cosine_sim_e1},\eqref{PO_cosine_sim_e2}, \eqref{P_perp_PO_norm} and the continuous mapping theorem, together with conditions that $\|\bbeta_s\|\rightarrow \kappa_2,\|\bbeta_0\|\rightarrow \kappa_1$ and $\frac{\langle \bbeta_s,\bbeta_0 \rangle}{\|\bbeta_0\|\|\bbeta_s\|}\rightarrow \xi$.

For the first term,
\begin{align*}
    & \frac{1}{p}\sum_{j=1}^p\left(1+\sigma_{\star}\left|T_j\right|+\sigma_{\star}\left|T_j^{\textrm{approx}}\right|+2 \sqrt{p}\left| \beta_{0,j}\right|\right)^2\\
    &\leq 4\frac{1}{p}  \left(p+p\sum_{j=1}^p   \beta_{0,j}^2+\sum_{j=1}^p \sigma_{\star}^2\left|T_j\right|^2+\sum_j \sigma_{\star}^2\left|T_j^{\textrm{approx}}\right|^2\right)\\
    &=4 +4\|\bbeta_0\|^2+ 4\sum_{j=1}^p \left(\widehat \beta_{M,j}-\alpha_{1*}\beta_{0,j}-\frac{\alpha_{2*} }{\sqrt{1-\xi^2}} ( \beta_{s,j}-\frac{\xi\kappa_2}{\kappa_1}  \beta _{0,j}) \right)^2   \\
    &\quad \quad +4\frac{\sigma_*^2}{\sigma^2(p)}\sum_{j=1}^p \left(\widehat \beta_{M,j}-\alpha_1(p)\beta_{0,j}-\frac{\alpha_2(p)\|\bbeta_s\|}{\|\boldsymbol{\beta}_s-\frac{\langle \bbeta_s,\bbeta_0 \rangle}{\|\bbeta_0\|^2}\boldsymbol{\beta}_0\|} ( \beta_{s,j}-\frac{\langle \bbeta_s,\bbeta_0 \rangle}{\|\bbeta_0\|^2}  \beta _{0,j}) \right)^2\\
    &\leq 4 +4\|\bbeta_0\|^2+ 8\left( \|\widehat{\bbeta}_M\|^2+ \alpha_{1*}^2\|\bbeta_0\|^2 +\frac{\alpha_{2*}^2\xi^2\kappa_2^2}{(1-\xi^2)\kappa_1^2}\|\bbeta_0\|^2 +\frac{\alpha_{2*}^2}{1-\xi^2}\|\bbeta_s\|^2 \right)\\
    &\quad + 8\frac{\sigma_*^2}{\sigma^2(p)} \left(\|\widehat \bbeta_M\|^2 +\alpha^2(p) \|\bbeta_0\|^2+  \frac{\alpha_2^2(p)\|\bbeta_s\|^2}{\|\boldsymbol{\beta}_s-\frac{\langle \bbeta_s,\bbeta_0 \rangle}{\|\bbeta_0\|^2}\boldsymbol{\beta}_0\|^2}\frac{\langle \bbeta_s,\bbeta_0 \rangle^2}{\|\bbeta_0\|^2} +\frac{\alpha_2^2(p)\|\bbeta_s\|^4}{\|\boldsymbol{\beta}_s-\frac{\langle \bbeta_s,\bbeta_0 \rangle}{\|\bbeta_0\|^2}\boldsymbol{\beta}_0\|^2} \right)\\
    &\xrightarrow{\mathbb{P}}4+4\kappa_1^2+16\left(2\alpha^2_{1*}\kappa_1^2+\frac{2\alpha^2_{2*} \kappa_2^2}{1-\xi^2}+\sigma_*^2+ \frac{\alpha^2_{2*}\xi^2 \kappa_2^2}{1-\xi^2} \right),
\end{align*}
which suggests that $\frac{1}{p}\sum_{j=1}^p\left(1+\sigma_{\star}\left|T_j\right|+\sigma_{\star}\left|T_j^{\textrm{approx}}\right|+2 \sqrt{p}\left| \beta_{0,j}\right|\right)^2=O_p(1)$.

\medskip

\textbf{Step 2: Prove \eqref{supp:step2_converge_Z_Z_scaled}}

We control the difference between $ \Psi\left(\sigma_{\star} Z_j, \sqrt{p}   \beta_{0,j}\right)$ and
$\Psi\left(\sigma_{\star} \tilde Z_j^{\textrm{scaled}}, \sqrt{p}   \beta_{0,j}\right)$ as follows:
\begin{align*}
    & \left|\frac{1}{p} \sum_{j=1}^p \psi\left(\sigma_{\star} \tilde{Z}_j^{ scaled }, \sqrt{p} \  \beta_{0,j}\right)-\frac{1}{p} \sum_{j=1}^p \psi\left(\sigma_{\star} Z_j, \sqrt{p} \  \beta_{0,j}\right)\right| \\
\leq & \frac{L \sigma_{\star}}{p} \sum_{j=1}^p\left(1+\left\|\left(\sigma_{\star} \tilde{Z}_j^{ scaled }, \sqrt{p} \  \beta_{0,j}\right)\right\|+\left\|\left(\sigma_{\star} Z_j, \sqrt{p} \  \beta_{0,j}\right)\right\|\right)\left|\tilde{Z}_j^{ scaled }-Z_j\right| \\
\leq & L \sigma_{\star}\left\{ \frac{1}{p}\sum_{j=1}^p\left(1+\left|\tilde{Z}_j^{ scaled }\right|+\left|Z_j\right|+2 \sqrt{p}\left|\  \beta_{0,j}\right|\right)^2\right\}^{1 / 2}\left\{\frac{1}{p}\sum_{j=1}^p\left(\tilde{Z}_j^{ scaled }-Z_j\right)^2\right\}^{1 / 2} .
\end{align*}
Similar to the proof of \eqref{supp:step1_converge_T_T_approx}, we show the second term is $o_p(1)$ and the first term is $O_p(1)$.

For the second term, we have
\begin{align*}
    \frac{1}{p}\sum_{j=1}^p\left(\tilde{Z}_j^{ scaled }-Z_j\right)^2&=\frac{1}{p}\sum_{j=1}^p\left(Z_j (\frac{\sqrt{p}}{\|\mathbf{P}^{\perp}\bZ\|}-1)-\frac{\sqrt{p}}{\|\mathbf{P}^{\perp}\bZ\|}(\frac{\bbeta_0^T \bZ}{\|\bbeta_0\|^2})\beta_{0,j} \right.\\
    &\left. \quad \quad -\frac{\sqrt{p}}{\|\mathbf{P}^{\perp}\bZ\|}(\frac{ (\boldsymbol{\beta}_s-\frac{\langle \bbeta_s,\bbeta_0 \rangle}{\|\bbeta_0\|^2}\boldsymbol{\beta}_0)^T \bZ}{\|\boldsymbol{\beta}_s-\frac{\langle \bbeta_s,\bbeta_0 \rangle}{\|\bbeta_0\|^2}\boldsymbol{\beta}_0\|^2})( \beta_{s,j}-\frac{\langle \bbeta_s,\bbeta_0 \rangle}{\|\bbeta_0\|^2}  \beta _{0,j}) \right)^2\\
    &\leq 2 \left(\frac{\sqrt{p}}{\|\mathbf{P}^{\perp}\bZ\|}-1 \right)^2 \frac{1}{p}\sum_{j=1}^p Z_j^2 + 2\frac{1}{p}\frac{{p}}{\|\mathbf{P}^{\perp}\bZ\|^2}\left[(\boldsymbol{\beta}_s-\frac{\langle \bbeta_s,\bbeta_0 \rangle}{\|\bbeta_0\|^2}\boldsymbol{\beta}_0)^T \bZ\right]^2 \\
    &\xrightarrow{\mathbb{P}}0,
\end{align*}
where the last convergence follows from the properties of the independent normal vector $\bZ$ that $\|\mathbf{P}^{\perp}\bZ\|/\sqrt{p} \xrightarrow{a.s}1$ and $(\boldsymbol{\beta}_s-\frac{\langle \bbeta_s,\bbeta_0 \rangle}{\|\bbeta_0\|^2}\boldsymbol{\beta}_0)^T \bZ=O_p(1)$.

The first term is constant order, which is a direct consequence of the facts that $\frac{1}{p}\|\bZ\|^2\xrightarrow{a.s}1 $ and $\frac{1}{p}\|\tilde \bZ^{\textrm{scaled}}\|^2=1$.

\medskip

\textbf{Step 3:} \eqref{supp:step3_final_limit}  follows from  \citet[Lemma C.1]{zhao2022asymptotic}.
\medskip

\textbf{Step 4:} Recall that $\tilde \bZ^{\textrm{scaled}}=\sqrt{p}\frac{\mathbf{P}^{\perp}\bZ}{\|\mathbf{P}^{\perp}\bZ\|}$ and $\bT^{\textrm{approx}}=\sqrt{p}\frac{\mathbf{P}^{\perp}\widehat \bbeta_M}{\|\mathbf{P}^{\perp}\widehat \bbeta_M\|}$, where $\bZ\sim N(0,\mathbb{I}_p)$.
It suffices to show that
\begin{equation}
    \label{P_prop_beta_M_normal}
    \frac{\mathbf{P}^{\perp}\widehat \bbeta_M}{\|\mathbf{P}^{\perp}\widehat \bbeta_M\|}\stackrel{d}{=}\frac{\mathbf{P}^{\perp}\bZ}{\|\mathbf{P}^{\perp}\bZ\|}.
\end{equation}

We write $\mathbf{P}=\mathbf{A}\mathbf{A}^{\top}$ where $\mathbf{A}$ is a $p \times 2$ matrix. This projects onto a 2-dimensional subspace of $\mathbb{R}^p$.
We write $\mathbf{P}^{\perp}=\mathbf{B} \mathbf{B}^{\top}$ where $B$ is a $p \times(p-2)$ matrix. This projects onto the orthogonal complement of the subspace spanned by $\mathbf{A}$, which is $(p-2)$-dimensional. We have $\mathbf{B}^\top \mathbf{A}=\mathbf{0}\in \mathbb R^{(p-2)\times 2}$, $\mathbf{A}^\top \mathbf{A}=\mathbf{I}_2$ and $\mathbf{B}^\top \mathbf{B}=\mathbf{I}_{p-2}$. For any   $(p-2) \times(p-2)$ orthonormal matrix $\mathbf{G}$, $\mathbf{B}\mathbf{G}\mathbf{B}^\top$ rotates the subspace spanned by the columns of $\mathbf{B}$.

Consider $\mathcal{U}:=\{ \mathbf{A}\mathbf{A}^{\top}+\mathbf{B}\mathbf{G}\mathbf{B}^\top:\mathbf{G} \text{ is } (p-2)\times (p-2) \text{(orthonormal matrix )}\} $, the set of all orthonormal matrices $\mathbf{U}\in\mathbb R^{p\times p}$ such that $\mathbf{U} \bbeta_0=\bbeta_0$, $\mathbf{U}\bbeta_s=\bbeta_s$ and perform rotation on the unit sphere lying in $span\{\bbeta_0,\bbeta_s\}^{\perp}$.
By the isotropy of $N(0,\mathbf{I}_p)$, the distribution of $\frac{\mathbf{P}^{\perp}\bZ}{\|\mathbf{P}^{\perp}\bZ\|}$ is $\mathcal{U}$-invariant, that is, it is the uniform distribution on the unit sphere lying in $span\{\bbeta_0,\bbeta_s\}^{\perp}$. Therefore, it suffices to show that the distribution of $\frac{\mathbf{P}^{\perp}\widehat \bbeta_M}{\|\mathbf{P}^{\perp}\widehat \bbeta_M\|}$ is also $\mathcal{U}$-invariant.

For any $\mathbf{U}\in \mathcal{U}$, there exists an orthonormal matrix $\mathbf{G}$ such that $\mathbf{U}=\mathbf{A}\mathbf{A}^{\top}+\mathbf{B}\mathbf{G}\mathbf{B}^\top$. We want to show
    $$\mathbf{U}\frac{\mathbf{P}^{\perp}\widehat \bbeta_M}{\|\mathbf{P}^{\perp}\widehat \bbeta_M\|}
    \stackrel{d}{=} \frac{\mathbf{P}^{\perp}\widehat \bbeta_M}{\|\mathbf{P}^{\perp}\widehat \bbeta_M\|} .$$
Since $\|\mathbf{P}^{\perp}\widehat \bbeta_M\|=\|\mathbf{U}\mathbf{P}^{\perp}\widehat \bbeta_M\|$, it suffices to show that $\mathbf{U}\mathbf{P}^{\perp}\widehat \bbeta_M\stackrel{d}{=}\mathbf{P}^{\perp}\widehat \bbeta_M$.

We first show that $\mathbf{U}\widehat \bbeta_M\stackrel{d}{=}\widehat \bbeta_M$.
Note that $\mathbf{U}\widehat \bbeta_M$ is the SRE in \eqref{eq: SRE_def} with observed covariates replaced by $\{\mathbf{U}\bX_i\}_{i=1}^n$ and auxiliary covariates replaced by $\{\mathbf{U}\bX^*_i\}_{i=1}^M$.
Since $\mathbf{U}$ is orthonormal, all these covariate vectors have i.i.d. \(N(0,1)\) entries. 
Since $\mathbf{U}\in \mathcal{U}$, we have $\bbeta_0^\top \mathbf{U}\bX_i=\bbeta_0^\top \bX_i$ for $i\leq n$ and $\bbeta_s^\top \mathbf{U}\bX_j^{*}=\bbeta_s^\top \bX_j^{*}$ for $j\leq M$.
Therefore, the joint distribution of the new observed data and new auxiliary data remains the same as the original joint distribution. As a result, the distribution of the SRE remains the same, i.e., $\mathbf{U}\widehat \bbeta_M\stackrel{d}{=}\widehat \bbeta_M$.

Consequently, we derive that
$$\mathbf{U}\widehat \bbeta_M\stackrel{d}{=}\widehat \bbeta_M   \quad \Longrightarrow \quad    \mathbf{B}\mathbf{B}^{\top}   \mathbf{U}\widehat \bbeta_M\stackrel{d}{=}\mathbf{B}\mathbf{B}^{\top} \widehat \bbeta_M        \Longrightarrow   \quad  \mathbf{B}\mathbf{G}\mathbf{B}^\top \widehat \bbeta_M \stackrel{d}{=}\mathbf{P}^{\perp}\widehat \bbeta_M.  $$
We complete the proof by observing $\mathbf{U}\mathbf{P}^{\perp}=\mathbf{B}\mathbf{G}\mathbf{B}^\top $.

\subsubsection{The convergence of the empirical distribution of $T_j$}

We note that \citet{zhao2022asymptotic} have proved $ \frac{1}{p} \sum_{j=1}^p \mathbf{1}\left\{  T_j \leq t\right\} \xrightarrow{\mathbb{P}} \Phi(t)$ for any fixed $t\in \mathbb{R}$. In this section, we extend their result to the following: for any fixed $t>0$,
\begin{equation}\label{supp:indicator_T_j_converge}
    \frac{1}{p} \sum_{j=1}^p \mathbf{1}\left\{-t \leq T_j \leq t\right\} \xrightarrow{\mathbb{P}} \mathbb P(|Z|\leq t),
\end{equation}
where $Z\sim N(0,1)$. Our proof is largely adapted from \cite{zhao2022asymptotic} and we present it here for completeness.

We continue to use the notations defined in \eqref{supp:define_T_j_Z_j_scaled}. Furthermore, we denote the indicator function $\mathbf{1}\left\{-t \leq s \leq t\right\} $ as   $I_t(s)$.
We will prove \eqref{supp:indicator_T_j_converge} by approximating $I_t(s)$ using a Lipschitz function $I^{\textrm{approx}}_{t,\epsilon/8}(s) $,  defined as:
$$
I^{\textrm{approx}}_{t,\epsilon/8}(s) = \begin{cases}0 & \text { if } s<-t-\epsilon/8 \text { or } s>t+\epsilon/8, \\ \frac{s+t+\epsilon/8}{\epsilon/8} & \text { if }-t-\epsilon/8 \leq s<-t, \\ 1 & \text { if }-t \leq s \leq t \\ \frac{t+\epsilon/8-s}{\epsilon/8} & \text { if } t<s \leq t+\epsilon/8,
\end{cases}
$$
where $\varepsilon$ is any fixed positive constant.
\Cref{fig:approximation_indicator} provides an illustration for this approximation.

\begin{figure}[H]
    \centering
    \includegraphics[scale=0.3]{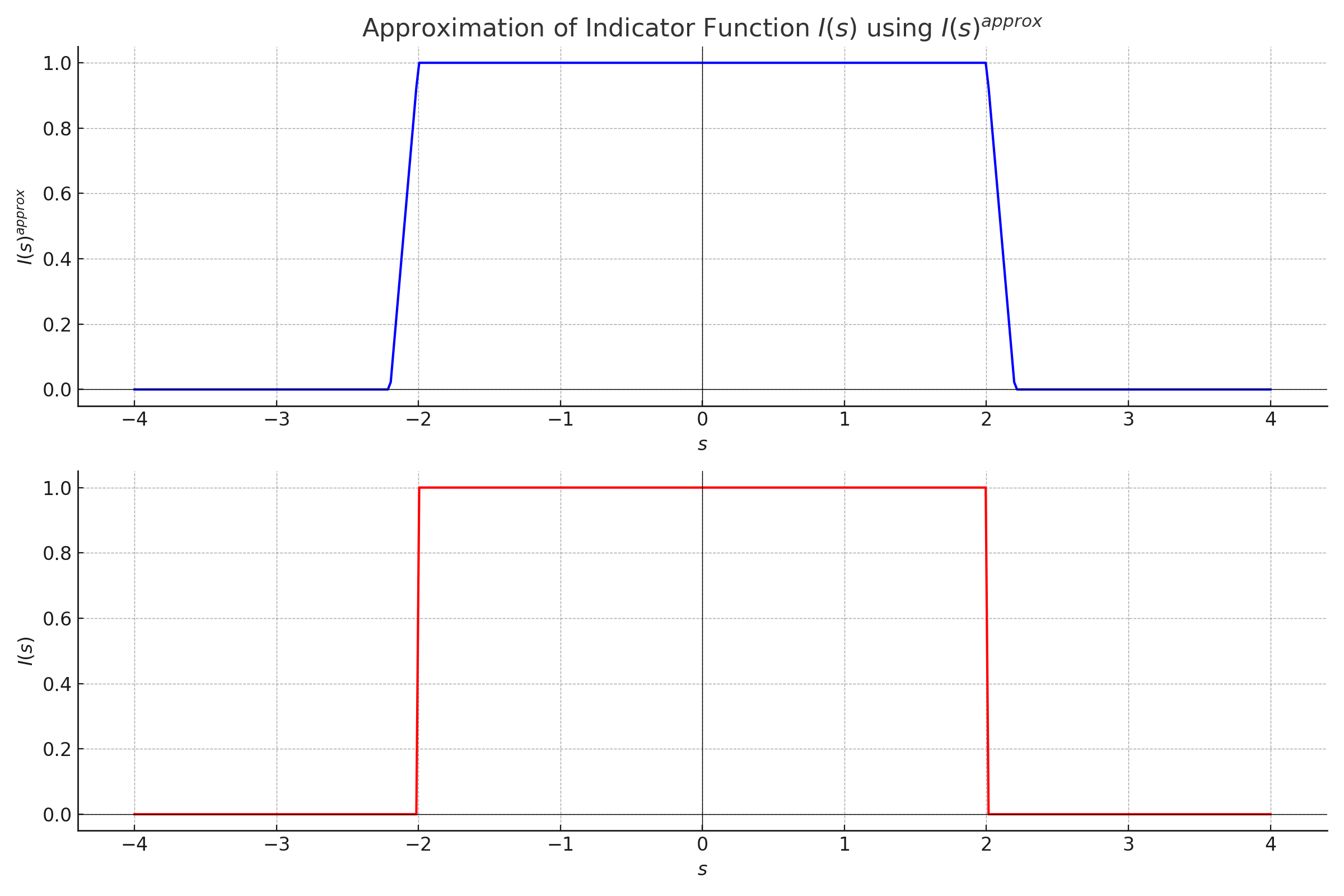}
    \caption{Comparison of the Indicator Function $I(s)$ and its Approximation $I^{\textrm{approx}}(s) $ with $\epsilon=1.6$}
    \label{fig:approximation_indicator}
\end{figure}
Note that $I^{\textrm{approx}}_{t,\epsilon/8}(s)>I_t(s)$ for every $s\in \mathbb R$.
We have

\begin{equation}\label{supp:proof_indicator_T_j_converge_1}
    \mathbb{P}\left(\frac{1}{p} \sum_{j=1}^p I_t(T_j)-P(|Z|\leq t)>\varepsilon\right) \leq \mathbb{P}\left(\frac{1}{p} \sum_{i=1}^p I^{\textrm{approx}}_{t,\epsilon/8}\left(T_j\right)-\mathbb P(|Z|\leq t)>\varepsilon\right).
\end{equation}

Based on the definition of $I^{\textrm{approx}}_{t,\epsilon/8}$, we have $\mathbb E[I^{\textrm{approx}}_{t,\epsilon/8}(Z)]-\mathbb P(|Z|\leq t)<\epsilon/4$. It follows that
{\scriptsize
\begin{equation}\label{supp:proof_indicator_T_j_converge_2}
\begin{aligned}
&\qquad \mathbb{P}\left(\frac{1}{p} \sum_{j=1}^p I^{\textrm{approx}}_{t,\epsilon/8}\left(T_j\right)-\mathbb P(|Z|\leq t)>\epsilon\right) \\
& \leq \mathbb{P}\left(\frac{1}{p} \sum_{j=1}^p I^{\textrm{approx}}_{t,\epsilon/8}\left(T_j\right)-\mathbb{E}\left[I^{\textrm{approx}}_{t,\epsilon/8}(Z)\right]>\epsilon / 2\right) \\
&\leq \mathbb{P}\left(\frac{1}{p} \sum_{j=1}^p\left[I^{\textrm{approx}}_{t,\epsilon/8}\left(T_j\right)-I^{\textrm{approx}}_{t,\epsilon/8}\left(T_j^{approx }\right)\right]>\epsilon / 4\right)  +\mathbb{P}\left(\frac{1}{p} \sum_{j=1}^p I^{\textrm{approx}}_{t,\epsilon/8}\left(T_j^{approx }\right)-\mathbb{E}\left[I^{\textrm{approx}}_{t,\epsilon/8}(Z)\right]>\epsilon / 4\right) \\
& =\mathbb{P}\left(\frac{1}{p} \sum_{j=1}^p\left[I^{\textrm{approx}}_{t,\epsilon/8}\left(T_j\right)-I^{\textrm{approx}}_{t,\epsilon/8}\left(T_j^{approx }\right)\right]>\epsilon / 4\right)  +\mathbb{P}\left(\frac{1}{p} \sum_{j=1}^p I^{\textrm{approx}}_{t,\epsilon/8}\left(\tilde{Z}_j^{\textrm{scaled}}\right)-\mathbb{E}\left[I^{\textrm{approx}}_{t,\epsilon/8}(Z)\right]>\epsilon / 4\right) \\
& \leq \mathbb{P}\left(\frac{1}{p} \sum_{j=1}^p\left[I^{\textrm{approx}}_{t,\epsilon/8}\left(T_j\right)-I^{\textrm{approx}}_{t,\epsilon/8}\left(T_j^{approx }\right)\right]>\epsilon / 4\right)  +\mathbb{P}\left(\frac{1}{p} \sum_{j=1}^p\left[I^{\textrm{approx}}_{t,\epsilon/8}\left(\tilde{Z}_j^{\textrm{scaled}}\right)-I^{\textrm{approx}}_{t,\epsilon/8}\left(Z_j\right)\right]>\epsilon / 8\right) \\
& +\mathbb{P}\left(\frac{1}{p} \sum_{j=1}^p I^{\textrm{approx}}_{t,\epsilon/8}\left(Z_j\right)-\mathbb{E}\left[I^{\textrm{approx}}_{t,\epsilon/8}(Z)\right]>\epsilon / 8\right),
\end{aligned}
\end{equation}
}
where the first equality follows from $\bT^{\textrm{approx}}\stackrel{d}{=}\tilde \bZ^{\textrm{scaled}}$. Since for any fixed $\epsilon>0$, $I^{\textrm{approx}}_{t,\epsilon/8}$ is a Lipschitz function, from \eqref{supp:step1_converge_T_T_approx}, \eqref{supp:step2_converge_Z_Z_scaled} and the law of large numbers, we have
\begin{equation}\label{supp:one_side_three_upper}
    \begin{aligned}
        & \lim_{n\rightarrow \infty}  \mathbb{P}\left(\frac{1}{p} \sum_{j=1}^p\left[I^{\textrm{approx}}_{t,\epsilon/8}\left(T_j\right)-I^{\textrm{approx}}_{t,\epsilon/8}\left(T_j^{approx }\right)\right]>\epsilon / 4\right) =0,\\
    &\lim_{n\rightarrow \infty} \mathbb{P}\left(\frac{1}{p} \sum_{j=1}^p\left[I^{\textrm{approx}}_{t,\epsilon/8}\left(\tilde{Z}_j^{\textrm{scaled}}\right)-I^{\textrm{approx}}_{t,\epsilon/8}\left(Z_j\right)\right]>\epsilon / 8\right)  =0     ,\\
    & \lim_{n\rightarrow \infty}\mathbb{P}\left(\frac{1}{p} \sum_{j=1}^p I^{\textrm{approx}}_{t,\epsilon/8}\left(Z_j\right)-\mathbb{E}\left[I^{\textrm{approx}}_{t,\epsilon/8}(Z)\right]>\epsilon / 8\right)=0     .
    \end{aligned}
\end{equation}
Combining \eqref{supp:proof_indicator_T_j_converge_1}, \eqref{supp:proof_indicator_T_j_converge_2}, and \eqref{supp:one_side_three_upper}, we have
\begin{equation}\label{supp:side1}
    \lim_{n\rightarrow \infty} \mathbb{P}\left(\frac{1}{p} \sum_{j=1}^p I_t(T_j)-P(|Z|\leq t)>\varepsilon\right)=0.
\end{equation}
For the other direction, we can use $I^{\textrm{approx}}_{t-\epsilon/8,\epsilon/8}(s)$ to approximate $I_t(s)$. By a similar argument, we get
\begin{equation}\label{supp:side2}
    \lim_{n\rightarrow \infty} \mathbb{P}\left(\frac{1}{p} \sum_{j=1}^p I_t(T_j)-P(|Z|\leq t)< -\varepsilon\right)=0.
\end{equation}
Since $\varepsilon$ is arbitrary, the proof of \eqref{supp:indicator_T_j_converge} is completed based on \eqref{supp:side1} and \eqref{supp:side2}.

\subsection{Proof of Theorem \ref{thm:exact_cat_M_MAP_noninformative(modify)} part (1)}

Part (1) of \Cref{thm:exact_cat_M_MAP_noninformative(modify)} follows from the following proposition.
\begin{proposition}\label{prop:asym_normal}
   Assume all conditions from \Cref{thm:exact_cat_M_MAP_noninformative(modify)}  hold, then for each coordinate $j \in[p]$ where the regression coefficient satisfies $\sqrt{p}  \beta_{0, j}=O(1)$, we have the following asymptotic normality:
   \begin{equation}\label{eq:asym_normal_one_coordinate}
       \frac{\sqrt{p}\left(\widehat{\beta}_{M, j}-\alpha_* \beta_{0, j}\right)}{\sigma_*} \xrightarrow{d} \mathcal{N}(0,1) .
   \end{equation}
   Furthermore, for any deterministic sequence of vectors $\boldsymbol{v} \in \mathbb{R}^p$ with unit norm $\|\boldsymbol{v}\|_2=1$ such that $\sqrt{p}\boldsymbol{v}^{\top} \boldsymbol{\beta}_0=$ $O(1)$, the following holds:
   \begin{equation}\label{eq:asym_normal_v_vector}
       \frac{\sqrt{p} \boldsymbol{v}^{\top}\left(\widehat{\boldsymbol{\beta}}_M-\alpha_* \boldsymbol{\beta}_0\right)}{\sigma_*} \xrightarrow{d} \mathcal{N}(0,1) .
   \end{equation}
Consequently, by applying the Cramér-Wold theorem, for any fixed index set $\mathcal{S} \subset$ $\{1, \ldots, p\}$ with $\sqrt{p}\left\|\boldsymbol{\beta}_{0, \mathcal{S}}\right\|_2=O(1)$, we obtain
$$
\frac{\sqrt{p}\left(\widehat{\boldsymbol{\beta}}_{M, \mathcal{S}}-\alpha_* \boldsymbol{\beta}_{0, \mathcal{S}}\right)}{\sigma_*} \stackrel{d}{\longrightarrow} \mathcal{N}\left(\mathbf{0}, \boldsymbol{I}_{|\mathcal{S}|}\right) .
$$

\end{proposition}

\begin{proof}[Proof of \Cref{prop:asym_normal}]

To establish \eqref{eq:asym_normal_v_vector}, it suffices to prove that equation \eqref{eq:asym_normal_one_coordinate} holds. Then, by leveraging the rotational invariance of the standard Gaussian distribution and considering an orthogonal matrix $\mathbf{U}$ with first row equal to $\bv$, \eqref{eq:asym_normal_v_vector} follows   directly from \eqref{eq:asym_normal_one_coordinate}.

We recall \eqref{P_prop_beta_M_normal} and rewrite it as
$$
\frac{\sqrt{p}\left(\widehat{\bbeta}_M- \alpha(p) \bbeta_0 \right)}{\sigma(p)}\stackrel{d}{=}\frac{\mathbf{P}^{\perp} \boldsymbol{Z}}{ \|\mathbf{P}^{\perp} \boldsymbol{Z} \|_2 / \sqrt{p}},
$$
where $\alpha(p)=\langle \widehat{\bbeta}_M,\bbeta_0 \rangle/\|\bbeta_0\|^2, \mathbf{P}=\bbeta_0\bbeta_0^\top/\|\bbeta_0\|^2, \mathbf{P}^{\perp}=\mathbf{I}-\mathbf{P} , \sigma(p)=\|\mathbf{P}^{\perp} \widehat{\bbeta}_M\|_2$ and  $\boldsymbol{Z}=\left(Z_1, \ldots, Z_p\right) \sim \mathcal{N}\left(\mathbf{0}, \boldsymbol{I}_p\right)$.
We expand the projection as
$$
\mathbf{P}^{\perp} \boldsymbol{Z}=\boldsymbol{Z}-\left\langle\boldsymbol{Z}, \frac{\boldsymbol{\beta}_0}{\|\boldsymbol{\beta}_0\|}\right\rangle \frac{\boldsymbol{\beta}_0}{\|\boldsymbol{\beta}_0\|},
$$
and note that the $j$th coordinate of the second term on the right-hand side is $o_p(1)$ since $\beta_{0,j}=o_p(1)$ while $\|\boldsymbol{\beta}_0\|=\Theta_p(1)$.
Therefore, the $j$th coordinate of $\mathbf{P}^{\perp} \boldsymbol{Z}$ is $Z_j+o_p(1)$.
Using the fact that $\left\|\mathbf{P}^{\perp} \boldsymbol{Z}\right\| / \sqrt{p} \xrightarrow{\text { a.s. }} 1$ and combining the convergence in \eqref{PO_cosine_sim_e1} and \eqref{P_perp_PO_norm}, Slutsky's theorem gives us
 $$\frac{\sqrt{p}\left(\widehat{\beta}_{M, j}-\alpha_* \beta_{0, j}\right)}{\sigma_*} \xrightarrow{d} \mathcal{N}(0,1) .$$
\end{proof}

\begin{remark}
The proof of \Cref{prop:asym_normal}
relies on equations \eqref{P_prop_beta_M_normal}, \eqref{PO_cosine_sim_e1}, and \eqref{P_perp_PO_norm}, which remain valid even without the condition $\frac{1}{p}\sum_{j=1}^p \noverpmass_{\sqrt{p}\beta_{0,j}} \rightsquigarrow \Pi$.
\end{remark}

\subsection{Proof of Theorem \ref{thm:exact_cat_M_MAP_informative} part (1)} Part (1) of Theorem \ref{thm:exact_cat_M_MAP_informative} follows from the following proposition.

\begin{proposition}\label{prop:asym_normal2}
   Assume all conditions from \Cref{thm:exact_cat_M_MAP_informative}  hold except  $\frac{1}{p}\sum_{j=1}^p \noverpmass_{\sqrt{p}\beta_{0,j}} \rightsquigarrow \Pi$,  then for each coordinate $j \in[p]$ where the regression coefficient satisfies $\sqrt{p}   \beta_{0, j}=O(1)$ and $\sqrt{p}   \beta_{s, j}=O(1)$, we have the following asymptotic normality:
$$
 \frac{\sqrt{p}\left(\widehat{\beta}_{M, j}-\alpha_{1*}\bbeta_{0,j}-\frac{\alpha_{2*}}{\sqrt{1-\xi^2}}(\bbeta_{s,j}-\xi \frac{\kappa_2}{\kappa_1}\bbeta_{0,j})\right)}{\sigma_*} \xrightarrow{d} \mathcal{N}(0,1) .
$$

\end{proposition}
The proof of \Cref{prop:asym_normal2} directly follows from the proof of \Cref{prop:asym_normal}.

\subsubsection{Deriving the limit for the squared error}\label{sec:exact_limit_square_error}
In this part, we derive the expressions for the limiting squared error and cosine similarity given in \eqref{eq:exact_cat_M_MAP_informative_square_error} and \eqref{eq:exact_cat_M_MAP_informative_similarity}.

Based on \Cref{thm:exact_cat_M_MAP_informative}, for any function $\Psi$ satisfying the stated regularity conditions, we have:
$$
\frac{1}{p} \sum_{j=1}^p \Psi\left(\sqrt{p}\left[\widehat{\boldsymbol{\beta}}_{M, j}-\alpha_{1 *} \beta_{0, j}-\frac{\alpha_{2 *}}{\sqrt{1-\xi^2}}\left(\beta_{s, j}-\xi \frac{\kappa_2}{\kappa_1} \beta_{0, j}\right)\right], \sqrt{p} \beta_{0, j}\right) \xrightarrow{\mathbb{P}} \mathbb{E}\left[\Psi\left(\sigma_* Z, \beta\right)\right],
$$
where $Z \sim \mathcal{N}(0,1)$ is independent of $\beta \sim \Pi(\beta)$, with $\mathbb{E}[\beta^2]=\kappa_1^2$.

We now derive the limiting squared error through the following four steps.

\noindent\textbf{Step 1:}
 Taking $\Psi(a, b)=\left(a+\left(\alpha_{1 *}-1\right) b\right)^2$, the RHS equals
 
 $\mathbb{E}\left[\left(\sigma_* Z+\left(\alpha_{1 *}-1\right) \beta\right)^2\right]=\sigma_*^2+\left(\alpha_{1 *}-1\right)^2 \kappa_1^2$,
 which implies
 \begin{equation}
     \label{eq1}
     \left\|\widehat{\boldsymbol{\beta}}_M-\boldsymbol{\beta}_0-\frac{\alpha_{2 *}}{\sqrt{1-\xi^2}}\left(\boldsymbol{\beta}_s-\xi \frac{\kappa_2}{\kappa_1} \boldsymbol{\beta}_0\right)\right\|_2^2\xrightarrow{\mathbb P}\sigma_*^2+\left(\alpha_{1 *}-1\right)^2 \kappa_1^2 .
 \end{equation}
\noindent\textbf{Step 2:}
 Taking $\Psi(a, b)=\left(a+ \alpha_{1 *}  b\right)^2$,
  the RHS  equals $ \sigma_*^2+\alpha_{1 *}^2 \kappa_1^2$,
 which implies
 \begin{equation}
     \label{eq2}
     \left\|\widehat{\boldsymbol{\beta}}_M- \frac{\alpha_{2 *}}{\sqrt{1-\xi^2}}\left(\boldsymbol{\beta}_s-\xi \frac{\kappa_2}{\kappa_1} \boldsymbol{\beta}_0\right)\right\|_2^2\xrightarrow{\mathbb P}\sigma_*^2+ \alpha_{1 *} ^2 \kappa_1^2 .
 \end{equation}
\noindent\textbf{Step 3:}
Taking $\Psi(a, b)=a   b$, the RHS equals 0,
 which implies that
 \begin{equation}
     \label{eq3}
    \left\langle \widehat{\boldsymbol{\beta}}_M-\alpha_{1*}\boldsymbol{\beta}_0-\frac{\alpha_{2 *}}{\sqrt{1-\xi^2}}\left(\boldsymbol{\beta}_s-\xi \frac{\kappa_2}{\kappa_1} \boldsymbol{\beta}_0\right),\bbeta_0\right\rangle\xrightarrow{\mathbb P}0 .
 \end{equation}

\noindent \textbf{Step 4:}
According to \Cref{condition:dist_condition(modify),condition:informative_syn_data}  that $\lim \|\bbeta_0\|^2=\kappa_1^2$, $\lim \|\boldsymbol{\beta}_s\|^2=\kappa_2^2$ and

$\lim \frac{1}{\|\boldsymbol{\beta}_0\|\|\boldsymbol{\beta}_s\|}\langle \boldsymbol{\beta}_0,\boldsymbol{\beta}_s\rangle=\xi$. We have
\begin{equation}
    \label{eq4}
    \left\langle \boldsymbol{\beta}_0,\boldsymbol{\beta}_s-\xi \frac{\kappa_2}{\kappa_1} \boldsymbol{\beta}_0 \right\rangle \rightarrow 0, \quad \left\|\frac{\alpha_{2 *}}{\sqrt{1-\xi^2}}\left(\boldsymbol{\beta}_s-\xi \frac{\kappa_2}{\kappa_1} \boldsymbol{\beta}_0\right)\right\|^2\rightarrow \alpha_{2*}^2\kappa_2^2
\end{equation}

Combining \eqref{eq1}--\eqref{eq4}, we obtain  $ \|\widehat{\bbeta}_M-\bbeta_0\|_2^2\xrightarrow{\mathbb{P}}(\alpha_{1*}-1)^2\kappa_1^2+\alpha_{2*}^2\kappa_2^2+\sigma_*^2$.

For cosine similarity of SRE, by \eqref{eq3}--\eqref{eq4}, the limit of numerator is
\begin{align*}
   \left\langle\widehat{\boldsymbol{\beta}}_M, \boldsymbol{\beta}_0\right\rangle&=  \left\langle \widehat{\boldsymbol{\beta}}_M-\alpha_{1*}\boldsymbol{\beta}_0-\frac{\alpha_{2 *}}{\sqrt{1-\xi^2}}\left(\boldsymbol{\beta}_s-\xi \frac{\kappa_2}{\kappa_1} \boldsymbol{\beta}_0\right),\bbeta_0\right\rangle\\
   & \quad +\alpha_{1 *}\left\|\boldsymbol{\beta}_0\right\|_2^2+\frac{\alpha_{2 *}}{\sqrt{1-\xi^2}}\left\langle\boldsymbol{\beta}_s-\xi \frac{\kappa_2}{\kappa_1} \boldsymbol{\beta}_0, \boldsymbol{\beta}_0\right\rangle\\
   &\xrightarrow{\mathbb P} \alpha_{1*}\kappa_1^2
\end{align*}
By the same logic, we have $\left\|\widehat{\boldsymbol{\beta}}_M\right\|_2^2 \xrightarrow{\mathbb{P}} \alpha_{1 *}^2 \kappa_1^2+\alpha_{2 *}^2 \kappa_2^2+\sigma_*^2$, then by Slutsky's theorem, we have
$$ \frac{\langle \widehat{\boldsymbol{\beta}}_M,\boldsymbol{\beta}_0 \rangle}{\|\widehat{\boldsymbol{\beta}}_M\|_2\|\boldsymbol{\beta}_0\|_2} \xrightarrow{\mathbb{P}}\frac{\alpha_{1*}\kappa_1}{\sqrt{\alpha_{1*}^2\kappa_1^2+\alpha_{2*}^2\kappa_2^2+\sigma_*^2}}.$$

\subsection{Limiting Predictive deviance and Generalization error}
\label{proof:deviance_generalization}

\Cref{thm:exact_cat_M_MAP_noninformative(modify)} also suggests the convergence of two quantities regarding the prediction performance of the SRE---specifically, the generalization error and the predictive deviance. 
Let $(\boldsymbol{X}_T, Y_T)$ be a pair of future data sampled from the same population as the observed data. 
Given the covariate vector $\boldsymbol{X}_T$ and the SRE $\widehat{\boldsymbol{\beta}}_{M}$, 
the binary prediction is $\widehat{Y}=\mathbf{1}\{\boldsymbol{X}_T^\top\widehat{\boldsymbol{\beta}}_M\geq 0 \}$. 
The following convergence of the generalization error holds:
\begin{align*}
    \mathbb{E}_{T}[\mathbf{1}\{\widehat{Y}\neq Y_T\}]  & \xrightarrow{\mathbb{P}} \mathbb{E}[\mathbf{1}\{Y_1\neq Y_2\}],   
\end{align*}
where $\mathbb{E}_{T}$ denotes the expectation over the randomness in $(\boldsymbol{X}_T,Y_T)$  and  $ Y_1= \mathbf{1} \{\sigma_*Z_1+\alpha_*\kappa_1 Z_2\geq 0 \}, Y_2\sim \text{Bern}(\rho^\prime(\kappa_1 Z_2))$ for i.i.d. standard normal variables $Z_1$ and $Z_2$. 
Furthermore, the predictive probability for $Y_T$ is $\rho^{\prime}(\boldsymbol{X}_T^\top\widehat{\boldsymbol{\beta}}_M)$ and we have the following convergence of the predictive deviance:
$$\begin{aligned}
    \mathbb{E}_{T}& \left[
       D(Y_T, \rho^{\prime}(\boldsymbol{X}_T^\top\widehat{\boldsymbol{\beta}}_M)) \right] \xrightarrow{\mathbb{P}}  \mathbb{E}\left[
       D( \rho^\prime(\kappa_1 Z_2) , \rho^\prime(\sigma_*Z_1+\alpha_*\kappa_1 Z_2))  \right], 
 \end{aligned}
$$
where the deviance is $D(a,b)=a\log(a/b)+(1-a)\log((1-a)/(1-b))$ with the  convention that $0\log(0):=0$. 


To begin with, we recall from \Cref{condition:dist_condition(modify)} and \Cref{thm:exact_cat_M_MAP_noninformative(modify)} that the following convergences hold: 
\begin{equation}\label{useful_result_deviance_generation}
    \begin{aligned}
        \|\bbeta_0\|_2^2 & \xrightarrow{\mathbb{P}}\kappa_1^2 ,\\
         \|\widehat \bbeta_M\|_2^2 & \xrightarrow{\mathbb{P}} \alpha_*^2\kappa_1^2+\sigma_*^2,\\
        \frac{\widehat{\bbeta}_M^\top \bbeta_0}{\|\widehat{\bbeta}_M\|_2\|\bbeta_0\|_2}& \xrightarrow{\mathbb{P}}\frac{\alpha_*\kappa_1}{\sqrt{\alpha_*^2\kappa_1^2+\sigma_*^2}}. 
    \end{aligned}
\end{equation}

\subsubsection{Limit of generalization error }
\label{supp:sec:limit_generalization_error}

Let $(\boldsymbol{X}_T, Y_T)$ be a pair of future data sampled from the same population as the observed data, i.e., $\bX_T\sim N(0,\mathbb{I}_p), Y_T\sim Bern(\rho^{\prime}(\bX_T^{\top}\bbeta_0))$. 
Given the covariate vector $\boldsymbol{X}_T$ and the SRE $\widehat{\boldsymbol{\beta}}_{M}$, the binary prediction is given by $\widehat{Y}=\mathbf{1}\{\boldsymbol{X}_T^\top\widehat{\boldsymbol{\beta}}_M\geq 0 \}$. 
We will use $\mathbb{E}_{T}$ to denote the expectation w.r.t. $(\boldsymbol{X}_T, Y_T)$. 
Therefore, $\mathbb{E}_{T}[\mathbf{1}\{\widehat{Y}\neq Y_T\}]$ is a random variable where randomness comes from $\widehat{\bbeta}_M$.

We first simplify $\mathbb{E}_{T}[\mathbf{1}\{\widehat{Y}\neq Y_T\}]$ as follows:
\begin{equation}\label{proof_simplif_generalization_error}
    \begin{aligned}
         \mathbb{E}_{T}[\mathbf{1}\{\widehat{Y}\neq Y_T\}]&=\mathbb E_{\bX_T}\left[ \mathbb E_{T} \left(\mathbf{1}\left\{Y_T\neq \mathbf{1}\{\boldsymbol{X}_T^\top\widehat{\boldsymbol{\beta}}_M\geq 0 \}\right\}\mid \bX_T\right)  \right]\\
    &=\mathbb E_{\bX_T} \left[\rho^{\prime}(\bX_T^\top \bbeta_0) \mathbf{1}\{\boldsymbol{X}_T^\top\widehat{\boldsymbol{\beta}}_M<0\} +(1-\rho^{\prime}(\bX_T^\top \bbeta_0))\mathbf{1}\{\boldsymbol{X}_T^\top\widehat{\boldsymbol{\beta}}_M\geq 0\}\right].
    \end{aligned}
\end{equation}
 
The evaluation of the second equation in \eqref{proof_simplif_generalization_error} relies on the following characterizations of $\bX_T^\top \bbeta_0$ and $\bX_T^\top \widehat{\bbeta}_M$. Let $Z_1, Z_2$ be two independent standard normal random variables. 
We introduce two random variables:
\begin{align*}
    & W_1:=\|\bbeta_0\|_2 Z_1,
\\
& W_2:= \frac{1}{\|\bbeta_0\|_2}\bbeta_0^\top \widehat{\bbeta}_M Z_1 +\sqrt{\|\widehat \bbeta_M\|_2^2-\left(\frac{1}{\|\bbeta_0\|_2}\bbeta_0^\top \widehat{\bbeta}_M\right)^2}Z_2 .
\end{align*}
This construction of $(W_1,W_2)$ preserves the conditional distribution of $(\bX_T^\top \bbeta_0,\bX_T^\top \widehat{\bbeta}_M)$ given the actual observed data, i.e.,
\begin{align*}
    & W_1\sim N(0,\|\bbeta_0\|^2),\quad  W_2 \sim N(0,\|\widehat \bbeta_M\|^2),\quad \operatorname{Cov}(\bX_T^\top \bbeta_0,W_2)= \bbeta_0^\top \widehat{\bbeta}_M,\\
    & \bX_T^\top \bbeta_0\sim N(0,\|\bbeta_0\|^2),\quad  \bX_T^\top \widehat{\bbeta}_M \sim N(0,\|\widehat \bbeta_M\|^2),\quad \operatorname{Cov}(\bX_T^\top \bbeta_0,\bX_T^\top \widehat{\bbeta}_M)= \bbeta_0^\top \widehat{\bbeta}_M. 
\end{align*}
Since $(W_1,W_2)\stackrel{D}{=}(\bX_T^\top \bbeta_0,\bX_T^\top \widehat{\bbeta}_M)$ conditional on the observed data, we can evaluate the second equation in \eqref{proof_simplif_generalization_error} as follows:
{\allowdisplaybreaks
\begin{align}\label{generalization_error_random}
&\mathbb E_{\bX_T} \left[\rho^{\prime}(\bX_T^\top \bbeta_0) \mathbf{1}\{\boldsymbol{X}_T^\top\widehat{\boldsymbol{\beta}}_M<0\} +(1-\rho^{\prime}(\bX_T^\top \bbeta_0))\mathbf{1}\{\boldsymbol{X}_T^\top\widehat{\boldsymbol{\beta}}_M\geq 0\}\right]\nonumber\\
        =&\mathbb E_{(W_1,W_2)} \left[\rho^{\prime}(W_1) \mathbf{1}\{W_2<0\} +(1-\rho^{\prime}(W_1))\mathbf{1}\{W_2\geq 0\} \right]\nonumber\\
        =&\mathbb E_{(Z_1,Z_2)} \left[\rho^{\prime}(\|\bbeta_0\|_2 Z_1) \mathbf{1}\left\{\frac{1}{\|\bbeta_0\|_2}\bbeta_0^\top \widehat{\bbeta}_M Z_1 +\sqrt{\|\widehat \bbeta_M\|_2^2-\left(\frac{1}{\|\bbeta_0\|_2}\bbeta_0^\top \widehat{\bbeta}_M\right)^2}Z_2<0\right\} \right.\nonumber\\
        & + \left.  (1-\rho^{\prime}(\|\bbeta_0\|_2 Z_1))\mathbf{1}\left\{\frac{1}{\|\bbeta_0\|_2}\bbeta_0^\top \widehat{\bbeta}_M Z_1 +\sqrt{\|\widehat \bbeta_M\|_2^2-\left(\frac{1}{\|\bbeta_0\|_2}\bbeta_0^\top \widehat{\bbeta}_M\right)^2}Z_2\geq 0\right\} \right]\\
        =&\mathbb E_{Z_1} \left[\rho^{\prime}(\|\bbeta_0\|_2 Z_1) \Phi\left(-\frac{\bbeta_0^\top \widehat{\bbeta}_M}{\sqrt{\|\bbeta_0\|^2\|\widehat \bbeta_M\|^2-(\bbeta_0^\top \widehat{\bbeta}_M)^2}}Z_1 \right)  \right.\nonumber\\
        & + \left.  (1-\rho^{\prime}(\|\bbeta_0\|_2 Z_1)) \Phi\left(\frac{\bbeta_0^\top \widehat{\bbeta}_M}{\sqrt{\|\bbeta_0\|^2\|\widehat \bbeta_M\|^2-(\bbeta_0^\top \widehat{\bbeta}_M)^2}}Z_1 \right) \right]\nonumber\\
        &= \mathbb E_{Z_1} [\rho^{\prime}(a_1 Z_1) \Phi(-a_2 Z_1)]+\mathbb E_{Z_1}[(1-\rho^{\prime}(a_1 Z_1)) \Phi(a_2 Z_1) ]\nonumber,   
\end{align}
}
where we use the shorthands $a_1:= \|\bbeta_0\|_2$ and $a_2:=\bbeta_0^\top \widehat{\bbeta}_M/{\sqrt{\|\bbeta_0\|^2\|\widehat \bbeta_M\|^2-(\bbeta_0^\top \widehat{\bbeta}_M)^2}} $ to  simplify the notation. 

Next we will study the convergence of $\mathbb E_{Z_1} [\rho^{\prime}(a_1 Z_1) \Phi(-a_2 Z_1)]$; the convergence of $\mathbb E_{Z_1} [(1-\rho^{\prime}(a_1 Z_1)) \Phi(a_2 Z_1)]$ can be shown using the same argument. Note that
\begin{align*}
    \mathbb E_{Z_1} \left[\rho^{\prime}(a_1 Z_1) \Phi\left(-a_2 Z_1 \right)  \right]
    = \int_{-\infty}^{\infty} \left[\rho^{\prime}(a_1 z) \Phi\left(-a_2 z \right) \phi(z)dz  \right].
\end{align*}
We will show $ \mathbb E_{Z_1} \left[\rho^{\prime}(a_1 Z_1) \Phi\left(-a_2 Z_1 \right)  \right]$ converges in probability to $ \mathbb E_{Z_1} \left[\rho^{\prime}(\kappa_1 Z_1) \Phi\left(-\alpha_*\kappa_1/\sigma_* Z_1 \right)  \right]$. 
Let $\bv=(v_1,v_2)$  be a two-dimensional vector. We define the function  $h(\bv,z):= \rho^{\prime}(v_1z)\Phi(-v_2 z) \phi(z)$, which is continuous with respect to $\bv$ for any $z\in \mathbb R$. Furthermore, $|h(\bv,z)|\leq \phi(z)$ for any $z\in \mathbb R$. By the dominated convergence theorem, the function $g(\bv):=\int_{-\infty}^{\infty} h(\bv,z)dz$ is   continuous   with respect to $\bv$. According to \eqref{useful_result_deviance_generation} and applying Slutsky's theorem, we conclude that
\begin{align*}
    \ba:=\left(\begin{array}{l}
 a_1 \\
 a_2
\end{array}\right) = \left(\begin{array}{l}
\|\bbeta_0\|_2 \\
\bbeta_0^\top \widehat{\bbeta}_M/{\sqrt{\|\bbeta_0\|^2\|\widehat \bbeta_M\|^2-(\bbeta_0^\top \widehat{\bbeta}_M)^2}}
\end{array}\right) \xrightarrow{\mathbb{P}} \left(\begin{array}{l}
 \kappa_1 \\
 {\alpha_*\kappa_1}/{\sigma_*}
\end{array}\right):= \ba_*.
\end{align*}
By the continuous mapping theorem, we have $g(\ba)\xrightarrow{\mathbb{P}} g(\ba_*)$, i.e., 
\begin{equation}\label{generalization_error_limit_1}
    \mathbb E_{Z_1} \left[\rho^{\prime}(a_1 Z_1) \Phi\left(-a_2 Z_1 \right)  \right]\xrightarrow{\mathbb{P}} \mathbb E_{Z_1} \left[\rho^{\prime}(\kappa_1 Z_1) \Phi\left(-\alpha_*\kappa_1/\sigma_* Z_1 \right)  \right].
\end{equation}
Similarly, we can show
\begin{equation}\label{generalization_error_limit_2}
    \mathbb E_{Z_1} \left[(1-\rho^{\prime}(a_1 Z_1)) \Phi\left(a_2 Z_1 \right)  \right]\xrightarrow{\mathbb{P}} \mathbb E_{Z_1} \left[(1-\rho^{\prime}(\kappa_1 Z_1)) \Phi\left(\alpha_*\kappa_1/\sigma_* Z_1 \right)  \right].
\end{equation}
Based on \eqref{proof_simplif_generalization_error}, \eqref{generalization_error_random}, \eqref{generalization_error_limit_1} and \eqref{generalization_error_limit_2}, the following convergence of the generalization error holds:
\begin{equation}\label{generalization_error_convergeence}
    \mathbb{E}_{T}[\mathbf{1}\{\widehat{Y}\neq Y_T\}]   \xrightarrow{\mathbb{P}}  \mathbb E_{Z_1} \left[\rho^{\prime}(\kappa_1 Z_1) \Phi\left(-\alpha_*\kappa_1/\sigma_* Z_1 \right)  \right] +\mathbb E_{Z_1} \left[(1-\rho^{\prime}(\kappa_1 Z_1)) \Phi\left(\alpha_*\kappa_1/\sigma_* Z_1 \right)  \right].
\end{equation}
To further simplify, the right-hand side of \eqref{generalization_error_convergeence} can be expressed as $\mathbb E[\mathbf{1}\{Y_1\neq Y_2 \}]$, where  $ Y_1= \mathbf{1} \{\sigma_*Z_1+\alpha_*\kappa_1 Z_2\geq 0 \}$, $ Y_2\sim \text{Bern}(\rho^\prime(\kappa_1 Z_2))$.  

\subsubsection{Limit of predictive deviance}

We will use a similar argument as in \Cref{supp:sec:limit_generalization_error} to show the following convergence of the predictive deviance:
       $$\begin{aligned}
    \mathbb{E}_{T} \left[
       D(Y_T, \rho^{\prime}(\boldsymbol{X}_T^\top\widehat{\boldsymbol{\beta}}_M)) \right] \xrightarrow{\mathbb{P}}  \mathbb{E}\left[
       D( \rho^\prime(\kappa_1 Z_1) , \rho^\prime(\sigma_*Z_2+\alpha_*\kappa_1 Z_1))  \right], 
 \end{aligned}
 $$
where the deviance is $D(a,b)=a\log(a/b)+(1-a)\log((1-a)/(1-b))$ with the convention that $0\log0 :=0$. 
 To prove this convergence, we first simplify $\mathbb{E}_{T} \left[
       D(Y_T, \rho^{\prime}(\boldsymbol{X}_T^\top\widehat{\boldsymbol{\beta}}_M)) \right] $ as follows:
\begin{equation}\label{deviance_simplification}
   \begin{aligned}
        \mathbb{E}_{T} \left[
       D(Y_T, \rho^{\prime}(\boldsymbol{X}_T^\top\widehat{\boldsymbol{\beta}}_M)) \right]&=-\mathbb E_T \left[Y_T \log(\rho^{\prime}(\boldsymbol{X}_T^\top\widehat{\boldsymbol{\beta}}_M))+(1-Y_T)\log(1-\rho^{\prime}(\boldsymbol{X}_T^\top\widehat{\boldsymbol{\beta}}_M))\right]  \\
       &=\mathbb E_T[\log(1+\exp(\boldsymbol{X}_T^\top\widehat{\boldsymbol{\beta}}_M))  -Y_T \boldsymbol{X}_T^\top\widehat{\boldsymbol{\beta}}_M ]\\
       &= \mathbb E_{\bX_T }[\log(1+\exp(\boldsymbol{X}_T^\top\widehat{\boldsymbol{\beta}}_M))]-\mathbb E_{\bX_T}[\rho^{\prime}(\bX_T^\top \bbeta_0)\boldsymbol{X}_T^\top\widehat{\boldsymbol{\beta}}_M], 
   \end{aligned}
\end{equation}
where the first equation follows from $0\log 0 +1 \log 1=0$. 
Based on the characterizations of $\bX_T^\top \bbeta_0$ and $\bX_T^\top \widehat{\bbeta}_M$ we used in \Cref{supp:sec:limit_generalization_error}, the right-hand side of the last equation in \eqref{deviance_simplification} is equal to
\begin{equation}\label{deviance_limit}
    \begin{aligned}
        &\mathbb E_{W_2}\left[\log(1+\exp(W_2))\right]-\mathbb E_{(W_1,W_2)}\left[\rho^{\prime}(W_1)W_2 \right]\\
     =&\mathbb E_{Z_1}\left[\log(1+\exp(\|\widehat{\bbeta}_M\|Z_1))\right]\\
     &-\mathbb E_{(Z_1,Z_2)}\left[\rho^{\prime}(\|\bbeta_0\|_2 Z_1)\left(\frac{1}{\|\bbeta_0\|_2}\bbeta_0^\top \widehat{\bbeta}_M Z_1 +Z_2\sqrt{\|\widehat \bbeta_M\|_2^2-\left(\frac{1}{\|\bbeta_0\|_2}\bbeta_0^\top \widehat{\bbeta}_M\right)^2} \right)\right]\\
     =&\mathbb E_{Z_1}\left[\log(1+\exp(\|\widehat{\bbeta}_M\|_2Z_1))\right]-\mathbb E_{Z_1}\left[\rho^{\prime}(\|\bbeta_0\|_2Z_1)\frac{1}{\|\bbeta_0\|_2}\bbeta_0^\top \widehat{\bbeta}_M Z_1 \right].
    \end{aligned}
\end{equation}
To apply the continuous mapping theorem, we define two functions:  $h_2(x,z)=\log(1+\exp(x z))\phi(z)$ and $h_3(y_1,y_2,z)=\rho^{\prime}(y_1z)y_2 z \phi(z)$.
We need to show that the function $g_2(x)= \int_{-\infty}^{\infty} h_2(x,z) dz $ is continuous with respect to $x>0$ and that the function $g_3(y_1,y_2)=\int_{-\infty}^{\infty} h_3(y_1,y_2,z)dz$ is continuous with respect to $(y_1,y_2)$, where $y_1>0, y_2\in \mathbb R$. Note that $h_2$ and $h_3$ are continuous. Furthermore, based on the uniform boundedness that $\|\widehat{\boldsymbol{\beta}}_M\|\leq c_1$ indicated in \eqref{borelcantelli_beta} and the two inequalities $\log(1+\exp(t))\leq |t|+\log(2)$ and $|\rho^{\prime}(t)|\leq 1/4$, we conclude that there exists a large constant $c_1>0$ independent of the sample size $n$ such that $|h_2(x,z)|\leq c_1|z|\phi(z)+\log(2)\phi(z)$ and $|h_3(y_1,y_2,z)|\leq c_1 |z|\phi(z)$. By the dominated convergence theorem, the function $g_2(x)$ is   continuous   with respect to $x\in (0,c_1]$ and $g_3(y_1,y_2)$ is continuous with respect to $(y_1,y_2)$, where $y_1>0, -c_1<y_2 <c_1$. According to \eqref{useful_result_deviance_generation} and applying Slutsky's theorem, we conclude that
\begin{align*}
    &\mathbb E_{Z_1}\left[\log(1+\exp(\|\widehat{\bbeta}_M\|_2Z_1))\right]-\mathbb E_{Z_1}\left[\rho^{\prime}(\|\bbeta_0\|_2Z_1)\frac{1}{\|\bbeta_0\|_2}\bbeta_0^\top \widehat{\bbeta}_M Z_1 \right]\\
      \xrightarrow{\mathbb{P}}& \mathbb E_{Z_1}\left[\log(1+\exp(\sqrt{\sigma_*^2+\alpha_*^2\kappa_1^2}Z_1))\right]-\mathbb E_{Z_1}\left[\rho^{\prime}(\kappa_1 Z_1)\alpha_*\kappa_1 Z_1 \right].
\end{align*}
 Furthermore, we can express the limiting value as $ \mathbb{E}\left[
       D( \rho^\prime(\kappa_1 Z_1) , \rho^\prime(\sigma_*Z_2+\alpha_*\kappa_1 Z_1))  \right]$ by verifying the equivalence through similar steps to those in \eqref{deviance_simplification}.

 \subsection{Proofs for GLM}\label{proof_GLM_supp}

\Cref{GLM_thm:stability_finite_M} has been proved in \Cref{proof_sec:logitic_stability}.
The proof of \Cref{GLM_thm:MAP_uniqueness} directly follows the proof in \Cref{proof_sec:logitic_existence}.
The proof of \Cref{thm:GLM_MAP_bounded} directly follows the proof in \Cref{proof:MAP_bounded_proportional}.
For \Cref{thm:GLM_post_mode_consistency}, the proof follows the proof in \Cref{proof_sec:logitic_consistency} where we replace the inequality $|y-\rho^{\prime}\left(\boldsymbol{x}^{\top} \boldsymbol{\beta}\right)|\leq 1$ by
$|\partial_\theta \ell_{G}(y,\boldsymbol{x}^{\top} \boldsymbol{\beta})|\leq L_g$.

The proof of \Cref{proposition:GLM_exact_cat_M_MAP_informative}
 follows from the proof in \Cref{proof_sec:logitic_exact}; the only difference is the application of the strong law of large numbers to $\mathbf{y}_1$ and $\mathbf{y}_2$, because the response distributions differ.
Accordingly, the forms of PO and AO remain unchanged. 
The asymptotic behavior of the SRE is tracked by the optima of the following optimization problem:
\begin{equation}\label{GLM_AO_limit}
    \max_{r}\min_{\sigma,\tilde{\nu}, \alpha_1,\alpha_2}  \mathcal{R}(\sigma,r,\tilde{\nu},\alpha_1,\alpha_2)
\end{equation}
with
\begin{equation}
    \begin{aligned}
         \mathcal{R}(\sigma,r,\tilde{\nu},\alpha_1,\alpha_2)&:=\mathbb{E}(M_{\rho(\cdot)}(\kappa_1\alpha_1 Z_1+\kappa_2\alpha_2 Z_2+\sigma Z_3+\frac{\tilde{\nu}}{r}Y_1,\frac{\tilde{\nu}}{r})) \\
    &+ \tau_0\mathbb{E}(M_{\rho(\cdot)}(\kappa_1\alpha_1 Z_1+\kappa_2\alpha_2 Z_2+\sigma Z_3+\frac{\tau_0\tilde{\nu}}{rm}Y_2,\frac{\tau_0\tilde{\nu}}{rm}))\\
    & +\tau_0\left[-\frac{\tau_0\tilde{\nu}}{2rm}\mathbb E(Y_2^2)-\kappa_1\alpha_1\mathbb E(Y_2Z_1)-\kappa_2\alpha_2\mathbb E(Y_2Z_2)\right]\\
&-\frac{r\sigma}{\sqrt{\delta}}+\frac{r\tilde{\nu}}{2}  -\frac{\tilde{\nu}}{2r}\mathbb E (Y_1^2)-\kappa_1\alpha_1 \mathbb E(Y_1Z_1),
    \end{aligned}
\end{equation}
where $Z_1,Z_2,Z_3$ are independent standard Gaussian random variables,  $Y_1\mid Z_1$ is distributed according to the corresponding GLM with linear predictor equal to $\kappa_1 Z_1$ and $Y_2\mid Z_1,Z_2$ is distributed according to the corresponding GLM with linear predictor equal to $\kappa_2\xi Z_1+\kappa_2 \sqrt{1-\xi^2}Z_2$.

\end{appendix}

\bibliographystyle{agsm}

\bibliography{bibfile}

\end{document}